%% file: PhDNeaime.tex
\RequirePackage{etex}
\RequirePackage{easybmat}
\documentclass[]{report}
\usepackage{minitoc}
\usepackage[titletoc]{appendix}
\usepackage[utf8]{inputenc}
\usepackage[T1]{fontenc}
\usepackage{graphicx}
\usepackage{amsmath}
\usepackage{listings}
\usepackage{moreverb}
\usepackage{eurosym}
\usepackage{fancyvrb}
\usepackage{afterpage}
\pdfoutput=1
\usepackage[final]{pdfpages}
\usepackage[nottoc]{tocbibind}
\usepackage{listings}
\usepackage{amsthm}
\usepackage{amsfonts}
\usepackage{geometry}
\usepackage{algorithm2e}
\usepackage{hyperref}
\usepackage{fixltx2e}
\usepackage{amssymb}
\usepackage{tikz}
\usepackage{tkz-graph}
\usetikzlibrary{calc}
\usepackage{float}
\usepackage{color}
\usepackage{lipsum}

\usepackage{multicol}
\usepackage{tikz-cd}
\usetikzlibrary{arrows, decorations.text, decorations.pathmorphing, decorations.markings, positioning, shapes}
\usetikzlibrary{decorations.markings}
\usepackage{feynmp}
\tikzset{
	fermion/.style = {draw = black, postaction = {decorate},decoration = {markings,mark = at position .55 with {\arrow{>}}}},
	vertex/.style = {draw,shape = circle,fill = black,minimum size = 2pt,inner sep = 0pt},
	fermionbar/.style={draw=black, postaction={decorate},
		decoration={markings,mark=at position .55 with {\arrow{<}}}}}
\usepackage{epsfig,graphics}
\usepackage{epigraph}

\newcommand*{\mathcolor}{}
\def\mathcolor#1#{\mathcoloraux{#1}}

\newtheorem{theorem}{Theorem}[section]
\newtheorem{lemma}[theorem]{Lemma}
\newtheorem{proposition}[theorem]{Proposition}
\newtheorem{cor}[theorem]{Corollary}
\newtheorem{definition}[theorem]{Definition}
\newtheorem{example}[theorem]{Example}
\newtheorem{conjecture}[theorem]{Conjecture}
\newtheorem{remark}[theorem]{Remark}

\newtheorem{convention}[theorem]{Convention}

\newtheorem*{theoremf}{Théorème}
\newtheorem*{lemmaf}{Lemme}
\newtheorem*{propositionf}{Proposition}

\newtheorem*{definitionf}{Définition}
\newtheorem*{examplef}{Exemple}
\newtheorem*{conjecturef}{Conjecture}
\newtheorem*{remarkf}{Remarque}

\newcommand\encircle[1]{%
  \tikz[baseline=(X.base)] 
    \node (X) [draw, shape=circle, inner sep=0] {\strut #1};}

\makeatletter
\def\v@rt#1#2{\m@th\ooalign{$\hfil#1|\hfil$\crcr$#1#2$}}
\def\captr{\mathrel{\mathpalette\v@rt\cap}}
\makeatother

\makeatletter

\newcommand{\Rmnum}[1]{\expandafter\@slowromancap\romannumeral #1@}
\makeatother

\newcommand\blankpage{
	\newpage
    \null
    \thispagestyle{empty}
    \addtocounter{page}{-3}
    \newpage}
    
\begin{document}

\clearpage
\thispagestyle{empty}
\hfill
\clearpage

\thispagestyle{empty}

\includegraphics[scale=0.6]{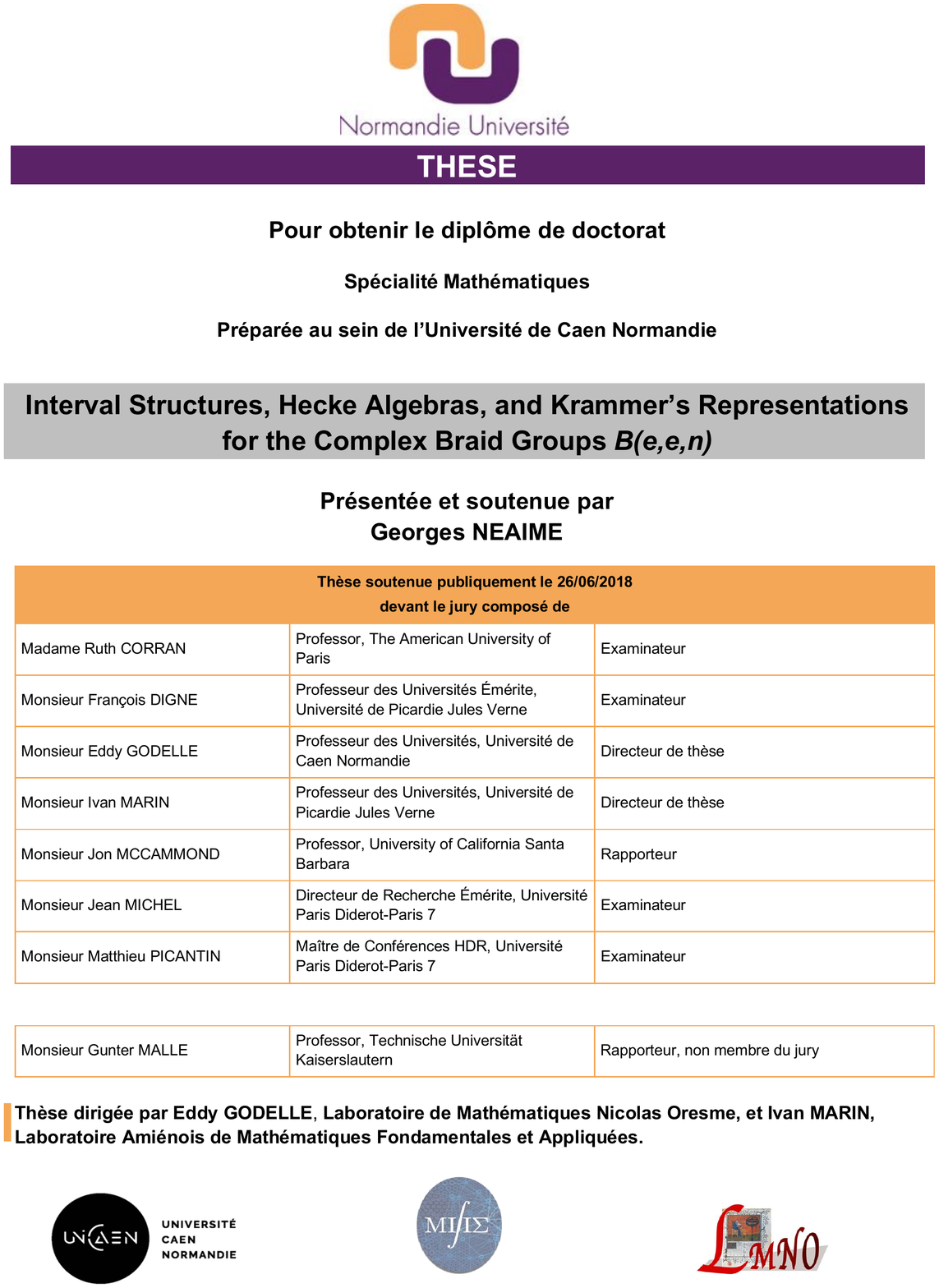}

\blankpage

\dominitoc

\include{Acknowledgments2}

\input{abstractthesis}

\newpage~
\thispagestyle{empty}

\tableofcontents

\include{chap11}

\include{chap22}

\include{chap33}

\include{chap44}

\include{chap55}

\include{annexe}

\include{compterendu}

\newpage~
\thispagestyle{empty}

\bibliographystyle{plain}
\bibliography{PhDNeaime.bib}

\newpage
\thispagestyle{empty}
\include{abstractthesis}
\newpage
\thispagestyle{empty}

\end{document}

%% file: Acknowledgments2.tex
\chapter*{Remerciements}

\addstarredchapter{Remerciements}

\minitoc

Assis à mon bureau, dans la salle S3 117, du Laboratoire de Mathématiques Nicolas Oresme, en train de boire ma tasse de café, comme chaque matin, pour se booster avant de commencer le travail, quand quelqu’un frappe à ma porte.\\
\begin{tabular}{ll}
& --- Salut mon chou !\\
& --- Salut Arnaud, ça va ?\\
& --- Ouiii, très bien, et toi ?\\
& --- Nickel !\\
& ---  Bonne journée ! Boujous !\\
& --- Merci, à toi aussi !
\end{tabular}\\
J’ai refermé la porte du bureau en la laissant un peu entrouverte. \mbox{Je dois me concentrer !} Il ne reste plus que quelques semaines avant ma soutenance ! J’ai allumé mon PC et ouvert le fichier contenant ma thèse. J’ai commencé à défiler doucement les pages de ma dissertation. D’une page à la suivante, un sentiment de nostalgie a commencé à monter en moi et à m’envahir. Tout d’un coup, j’ai vu défiler à la vitesse de mon curseur les trois années de ma thèse.\\

L’aventure a commencé le 22 juin 2015 quand j’ai reçu un mail de J. Guaschi à 16h43 : « La réunion du conseil de l’École Doctorale vient de se terminer. Vous avez eu un contrat doctoral de l’Université de Caen Basse Normandie. » Ça y est, je suis accepté en thèse sous la direction de Eddy Godelle et Ivan Marin ! Il était convenu que je sois doctorant associé au Laboratoire Amiénois de Mathématiques Fondamentales et Appliquées où je passerai ma première année de thèse. La suite de la thèse sera au Laboratoire de Mathématiques Nicolas Oresme à Caen.\\

Dans le bureau C013 au fond du couloir du premier étage du LAMFA, assis sur une chaise rouge devant une table ronde, j’ai toujours été très concentré devant le tableau blanc sur lequel Ivan était en train de m’éclaircir sur mon sujet de thèse. Monoïdes. Représentations. Garside. Krammer. $B(e,e,n)$. BMW. Fidélité. Dans ce même bureau, des fois assis, des fois debout, des fois silencieux, des fois bavard, des fois triste, des fois content, j’ai emprisonné les questions stupides, les pensées et les émotions d’un jeune doctorant en Mathématiques.\\

Ivan, souviens-toi de ce jeune garçon qui a voulu préparer une thèse avec toi avant même son Master 2. Souviens-toi de ce jeune garçon qui a présenté son premier beamer devant toi au groupe de travail « Tresses » que tu as organisé pour tes étudiants. Souviens-toi de ce jeune garçon et de sa folie pendant nos soirées à l’Atelier d’Alex. Souviens-toi quand tu es venu à Caen un dimanche soir pour lui parler de GBNP. Souviens-toi des lettres que tu lui as envoyées pour l’aider à obtenir un modèle de la Krammer. Souviens-toi de son état lors de votre rdv Skype un lundi après-midi juste avant d’envoyer la dissertation aux rapporteurs. Tu étais toujours à côté de lui en le voyant grandir et réussir. Tu le vois aujourd’hui debout, devant son jury de thèse, en train de défendre ses travaux de recherche. Je ne sais pas comment exprimer ce lien entre Ivan et ses étudiants mais je pense qu'Eirini l’a très bien fait dans ses remerciements de thèse : « I would like to thank him because apart from being a teacher and a colleague, he has also been a friend and he has most successfully managed to keep all these features in balance. » Merci Ivan d’être un ami et un patron !\\

Avant de commencer ma thèse, j’ai connu Eddy le 1 avril 2015 quand il est venu à Amiens pour me rencontrer pour la première fois. À partir de la fin de ma première année de thèse, je passais souvent dans son bureau pour lui exposer les idées qui virevoltaient dans ma tête. Mes preuves étaient parfois longues et moches. Il m’a incité à remuer mes pensées afin de trouver les bonnes idées pour simplifier mes démonstrations et rendre mon résultat plus lisible et plus élégant. Eddy m’a aussi appris à soigner ma rédaction afin de bien écrire mon premier article : « Cette définition avant ce lemme, cette proposition est un théorème, tu es sûr que ça c’est une remarque, c’est moche, cette phrase est sans verbe, tu as oublié un point ici, pas clair… » Bref, du rouge partout ! Et je dois tout rédiger à nouveau !
%
Il a aussi écouté plusieurs fois, avec attention, la chaîne Georges-Radio pendant 45 minutes dans la salle des séminaires du LMNO. Ava a été aussi présente dans la salle avec Eddy et moi. Elle préparait ses devoirs alors que je faisais la répétition de ma soutenance de thèse. Elle a pourtant remarqué que je parle assez vite. J’espère pouvoir parler plus lentement le jour de ma soutenance ! Je me souviens de la toute première répétition, à Milan, de ma soutenance devant Salim qui m'a donné beaucoup de remarques et de bons conseils.\\

Mon jury est exceptionnel ! Ruth Corran, François Digne, Jon McCammond, Jean Michel, Matthieu Picantin et mes deux directeurs de thèse, je vous remercie énormément. Je tiens aussi à remercier mes deux rapporteurs Gunter Malle et Jon McCammond. J’ai rencontré Jon pour la première fois à Caen lors de la conférence Winter Braids. Impressionné par son mini-cours et le dynamisme de ses exposés, j’ai demandé à Eddy si Jon pourrait être mon rapporteur et j’ai été très content quand j’ai appris qu’il a accepté de rapporter ma thèse et de venir à ma soutenance. Jon m’a aussi invité à l’UCSB (University of California Santa Barbara) pour présenter mes travaux de recherche dans le cadre du séminaire « Discrete Geometry and Combinatorics ». Malheureusement, je n’ai pas pu voyager aux USA à cause des procédures administratives (formulaire 5535) très lentes pour l’obtention d’un visa. Je pense aussi à Barbara Baumeister qui m’a fourni l’opportunité de préparer un postdoc et est venue de Bielefeld pour assister à ma soutenance.\\

Installé au bureau des doctorants au LAMFA, j’ai connu d’excellents camarades. Je revois mon frère Alexandre en train d’écouter mes histoires invraisemblables, ma sœur Eirini en train de me donner les bons conseils et m’encourager, Pierrot et nos discussions sans fin, Anne-So qui veut fumer une clope, Ruxi en train de m’apprendre les gros mots en chinois, Sylvain que j’empêche de travailler, Aktham et son bon café…  Vous m’avez fait un pot de départ à la fin de ma première année. Je me rappelle les messages chaleureux avec lesquels vous m’aviez dit au revoir :\\
\begin{tabular}{ll}
& —	Mon cher porte-parole, on continuera à Caen et à Amiens nos délires.\\
&—	C’est la vie pas le paradis.\\
& —	Je vous félicite pour l’obtention du numéro de M. Marin.\\
& —	Tes talents de comédien.\\
& —	Merci pour les chaînes de trombones, ça me manquera.
\end{tabular}\\

J’ai reçu il y a quelques semaines une invitation pour venir à Amiens faire un séminaire doctorants au LAMFA. Dans le mail diffusé aux doctorants, Alexandre a écrit une longue biographie sur l’orateur du séminaire. J’ai envie de partager avec vous quelques passages drôles de cette biographie :\\
\begin{small}\textit{« Des études universitaires en France et au Liban l'ont amené à étudier au près des tous meilleurs en Master 2 à l'Université Paris Diderot dans le cadre du master ATNA de l'UPJV. Lors de cette année, il a pu suivre des cours de Michel Broué et David Hernandez, il a aussi et surtout pu suivre des cours de Olivier Brunat et Ivan Marin\footnote{Olivier et Ivan sont les directeurs de thèse d'Alexandre.} qui sont des mathématiciens exceptionnels sans égal dans la communauté mondiale (ils sont tous les deux supérieurs à tous les autres mais non comparables entre eux).\\
    Son histoire avec Pr. Marin commence lors d'une réunion avant le début du master ATNA où c'est le coup de foudre mathématiques pour Georges et il décide en quelques minutes que Pr. Marin sera son directeur de thèse un an plus tard et le lui fait savoir de manière non équivoque. Un Master 2 brillant et une centaine de mails aux différents mathématiciens mentionnés par le futur directeur du laboratoire plus tard, il trouve un financement donné par l'Université de Caen et le spécialiste des structures monoïdales Eddy Godelle pour faire une thèse en co-direction avec Ivan Marin sur "trouver un point sur la variété". Il effectue la première année de sa thèse à un rythme insoutenable en enchaînant les exposés de qualité et en étant véritablement le cœur de l'équipe des doctorants cette année-là.\\
    Mais voilà, Georges est doux, Georges est frais mais Georges n'est vraiment pas pratique donc nous avons dû nous débarrasser de lui en Juin 2016 pour l'envoyer dans les contrées Normandes. Il vit depuis dans cette région obscure et on peut parfois apercevoir des tags\\ "I love LAMFA" sur les monuments historiques ou les copies de ses étudiants. Il sortira de cette région ce mercredi pour vous faire un exposé sur ce que la réclusion et l'abandon de la Picardie lui ont permis de chercher corps et âme faute de pouvoir faire autre chose. »}\end{small}\\

Arrivé à Caen, j’ai connu les meilleurs camarades du monde. Vu que je vous aime beaucoup, j’ai écrit de longs remerciements pour vous occuper pendant ma soutenance. Je vous imagine au fond de la salle des thèses en train de lire ces pages comme des innocents alors que je suis en train de parler !\\

Je tiens tout d’abord à remercier le personnel administratif du LMNO et tous les profs qui m’ont aidé pendant ces trois années de thèse surtout les membres de l’équipe GRAAL : Groupes, Représentations, Algèbre, Analyse, Logique. Je remercie aussi tous les gens du LMNO qui sont venus à ma soutenance. Je tiens à remercier particulièrement Anita et Axelle avec qui j’ai beaucoup papoté et qui m’ont aidé à préparer mes voyages en conférences, Marie qui m’a aidé avec énormément de gentillesse dans toute la procédure de soutenance, Emmanuelle qui m’a aidé à intégrer le monde de l’enseignement et à faire partie de l’équipe de vulgarisation mathématique, Christian, mon voisin de bureau, qui est toujours prêt à discuter avec moi et à me donner les bons conseils et Sonia, peux-tu me donner le secret pour garder constamment ton sourire sympathique ? Au risque de n’évoquer que les personnes qu’on côtoie à la fin d’une thèse, je tiens aussi à remercier tout le personnel administratif et les profs du LAMFA qui m’ont accompagné pendant ma première année.\\

Arnaud, tu es la seule personne, à ma connaissance, qui prend le reçu mais oublie de prendre ses billets du distributeur ! Tu as une folie et une intelligence qui font de toi un très bon camarade et c’est pour ça que je t’aime beaucoup !  Coumba, ma très ($\times$1M) très chère Coumba, la classe totale ! C’est le 26 juin et je n’ai pas encore fini le lemme 4.2.8 ! Rubén, tu connais toutes les aventures de Jorge el hombre, Jorge le farceur, Jorge le passéiste et surtout Jorge le machiste et l’ensemble de ses femmes $\{$Raymonde, Gertrude, Hermonde, Pascale, Yvette, Paola, Nathanaëlle, Adriana, Alma, Palma, Eva, Frida, Luna, Mercedes, Selena, Victoria, Bianca, Carmen, Lola, Gabriela$\}$.\\

Frank, doctorant (1991-??) en Mathématiques Appliquées et Dalal, nouvelle doctorante (2018-??) en Mathématiques Fondamentales me supportent dans le bureau 117. Étienne (English below starting from page 5) est notre nouveau représentant à l’École Doctorale. Pour les doctorants qui arrivent au LMNO, n’hésitez pas à poser directement vos questions à Étienne. Tuyen, tu as fini ton Skype pour venir manger avec nous ce midi ? Moustafa, très gentil, il serre la main à tout le monde quand il arrive au bureau et Nasser, je n’ai pas pu encaisser ton chèque de 1000\euro{}, tu m'en feras un autre s’il te plait ? Profe Pablo, tu as fini la preuve de notre théorème \mbox{(le monstre) ?} César, mon chouchou, tu manques au labo ! Guillaume, un jour je serais moi aussi végétarien. Emilie, tu vas venir à ma soutenance ? Julien, foot ce vendredi ? On se verra au bureau samedi ? Sarra, désolé de te faire croire que le verre de tequila était de l’eau. Anne-Sophie, Laetitia et Tuanlong, vous êtes toujours très sympatiques. Daniele, merci pour les deux dictionnaires, ça me permettra d’apprendre l’Allemand et l’Italien en même temps ! Raoul, j’ai gagné le pari ? Angelot, si tu penses toujours que les libanais sont fous, on se verra le 14 juin… Léonard, j’espère que tu ne me diras pas non la prochaine fois que je te propose de m’épouser. Ma chère Vlerë, rue Saint Pierre se rappellera toujours nos pas lourds quand on était bourré en train de \mbox{chanter :} padam padam. Attention, Vlerë a des photos et des vidéos très compromettantes de moi, elle va les rendre publiques quand je deviendrai un jour célèbre !\\

Mon père et mon oncle m’ont toujours aidé financièrement pour accomplir mon projet professionnel. Ma sœur Chrystelle est venue du Liban spécialement pour ma soutenance et m’a préparé avec l’aide de Samira un délicieux pot de thèse ! Ma petite sœur Claire, je te fais un gros bisou, ça pique non ?\\

« Bonjour joujou » est le message qui s’affiche sur mon téléphone chaque matin. Ma mère me l’envoyait inlassablement depuis des années. Ce petit message a rythmé ma vie. Avec ma mère, j’ai appris à grandir, à rêver et à ne pas abandonner mes rêves. Cette thèse de doctorat lui est dédiée.

%% file: abstractthesis.tex
\newgeometry{left=2cm,bottom=2cm,right=2cm,top=2cm}

\noindent \textbf{Interval Structures, Hecke Algebras, and Krammer's Representations for the Complex Braid Groups $B(e,e,n)$}\\

\noindent \textbf{Abstract:} We define geodesic normal forms for the general series of complex reflection groups $G(de,e,n)$. This requires the elaboration of a combinatorial technique in order to determine minimal word representatives and to compute the length of the elements of $G(de,e,n)$ over some generating set. Using these geodesic normal forms, we construct intervals in $G(e,e,n)$ that give rise to Garside groups. Some of these groups correspond to the complex braid group $B(e,e,n)$. For the other Garside groups that appear, we study some of their properties and compute their second integral homology groups. Inspired by the geodesic normal forms, we also define new presentations and new bases for the Hecke algebras associated with the complex reflection groups $G(e,e,n)$ and $G(d,1,n)$ which lead to a new proof of the BMR (Broué-Malle-Rouquier) freeness conjecture for these two cases. Next, we define a BMW (Birman-Murakami-Wenzl) and Brauer algebras for type $(e,e,n)$. This enables us to construct explicit Krammer's representations for some cases of the complex braid groups $B(e,e,n)$. We conjecture that these representations are faithful. Finally, based on our heuristic computations, we propose a conjecture about the structure of the BMW algebra.\\

\noindent \textbf{Keywords:} Reflection Groups, Braid Groups, Hecke Algebras, Geodesic Normal Forms, Garside Structures, Interval Garside Structures, Homology, BMR Freeness Conjecture, Brauer Algebras, BMW Algebras, Krammer's Representations.

\begin{center}***\end{center}

\noindent \textbf{Structures d'Intervalles, Algèbres de Hecke et Représentations de Krammer des Groupes de Tresses Complexes $B(e,e,n)$}\\

\noindent \textbf{R\'esum\'e :} Nous définissons des formes normales géodésiques pour les séries générales des groupes de réflexions complexes $G(de,e,n)$. Ceci nécessite l'élaboration d'une technique combinatoire afin de déterminer des décompositions réduites et de calculer la longueur des éléments de $G(de,e,n)$ sur un ensemble générateur donné. En utilisant ces formes normales géodésiques, nous construisons des intervalles dans $G(e,e,n)$ qui permettent d'obtenir des groupes de Garside. Certains de ces groupes correspondent au groupe de tresses complexe $B(e,e,n)$. Pour les autres groupes de Garside, nous étudions certaines de leurs propriétés et nous calculons leurs groupes d'homologie sur $\mathbb{Z}$ d'ordre $2$. Inspirés par les formes normales géodésiques, nous définissons aussi de nouvelles présentations et de nouvelles bases pour les algèbres de Hecke associées aux groupes de réflexions complexes $G(e,e,n)$ et $G(d,1,n)$ ce qui permet d'obtenir une nouvelle preuve de la conjecture de liberté de BMR (Broué-Malle-Rouquier) pour ces deux cas. Ensuite, nous définissons des algèbres de BMW (Birman-Murakami-Wenzl) et de Brauer pour le type $(e,e,n)$. Ceci nous permet de construire des représentations de Krammer explicites pour des cas particuliers des groupes de tresses complexes $B(e,e,n)$. Nous conjecturons que ces représentations sont fidèles. Enfin, en se basant sur nos calculs heuristiques, nous proposons une conjecture sur la structure de l'algèbre de BMW.\\

\noindent \textbf{Mots cl\'es :} Groupes de Réflexions, Groupes de Tresses, Algèbres de Hecke, Formes Normales Géodésiques, Structures de Garside, Structures d'Intervalles, Homologie, Conjecture de Liberté de BMR, Algèbres de Brauer, Algèbres de BMW, Représentations de Krammer.

\restoregeometry

%% file: chap11.tex
\chapter{Introduction}\label{ChapterIntroduction}
 
\minitoc

\bigskip

The aim of this chapter is to give an outline of the thesis and present its main results. We include the necessary preliminaries about the complex reflection groups, the complex braid groups, and the Hecke algebras.

\section{Complex reflection groups}

In this section, we provide the definition of a finite complex reflection group and recall the classification of Shephard and Todd of finite irreducible complex reflection groups. We also recall the definition of Coxeter groups and the classification of finite irreducible Coxeter groups.\\

Let $n$ be a positive integer and let $V$ be a $\mathbb{C}$-vector space of dimension $n$.

\begin{definition}\label{DefinitionReflection}

An element $s$ of $GL(V)$ is called a reflection if \\$Ker(s-1)$ is a hyperplane and $s^d=1$ for some $d \geq 2$. 

\end{definition}

Let $W$ be a finite subgroup of $GL(V)$. 

\begin{definition}\label{DefinitionReflectionGroup}

$W$ is a complex reflection group if $W$ is generated by the set $\mathcal{R}$ of reflections of $W$.

\end{definition}

We say that $W$ is irreducible if $V$ is an irreducible linear representation of $W$. Every complex reflection group can be written as a direct product of irreducible ones (see Proposition 1.27 of \cite{UnitaryReflectionGroupsLehrerTaylor}). Therefore, the study of complex reflection groups reduces to the irreducible case that was classified by Shephard and Todd in 1954 (see \cite{ShephardTodd}). The classification is as follows.

\begin{proposition}\label{PropositionClassificationShephardTodd}

Let $W$ be an irreducible complex reflection group. Then, up to conjugacy, $W$ belongs to one of the following cases:

\begin{itemize}

\item The infinite series $G(de,e,n)$ depending on three positive integer parameters\\ $d$, $e$, and $n$ (see Definition \ref{DefinitionG(de,e,n)} below).
\item The $34$ exceptional groups $G_4$, $\cdots$, $G_{37}$.

\end{itemize}

\end{proposition}

We provide the definition of the infinite series $G(de,e,n)$ and some of their properties. For the definition of the $34$ exceptional groups, see \cite{ShephardTodd}.

\begin{definition}\label{DefinitionG(de,e,n)}

$G(de,e,n)$ is the group of $n \times n$ matrices consisting of monomial matrices (each row and column has a unique nonzero entry), where
\begin{itemize}
\item all nonzero entries are $de$-th roots of unity and
\item the product of the nonzero entries is a $d$-th root of unity.
\end{itemize} 

\end{definition}

\begin{remark}

The group $G(1,1,n)$ is irreducible on $\mathbb{C}^{n-1}$ and the group $G(2,2,2)$ is not irreducible so it is excluded from the classification of Shephard and Todd.

\end{remark}

Let $s_i$ be the transposition matrix $(i, i+1)$ for $1 \leq i \leq n-1$, $t_e = \begin{pmatrix}
0 & \zeta_e^{-1} & 0\\
\zeta_e & 0 & 0\\
0 & 0 & I_{n-2}
\end{pmatrix}$, and $u_d = \begin{pmatrix}
\zeta_d & 0\\
0 & I_{n-1}
\end{pmatrix}$, where $I_k$ denotes the identity $k \times k$ matrix and $\zeta_l$ the $l$-th root of unity that is equal to $exp(2i\pi\l/l)$. The following result can be found in Section 3 of Chapter 2 in \cite{UnitaryReflectionGroupsLehrerTaylor}.

\begin{proposition}
The set of generators of the complex reflection group $G(de,e,n)$ are as follows.
\begin{itemize}
\item The group $G(e,e,n)$ is generated by the reflections $t_e$, $s_1$, $s_2$, $\cdots$, $s_{n-1}$.
\item The group $G(d,1,n)$ is generated by the reflections $u_d$, $s_1$, $s_2$, $\cdots$, $s_{n-1}$.
\item For $d \neq 1$ and $e \neq 1$, the group $G(de,e,n)$ is generated by the reflections $u_d$, $t_{de}$, $s_1$, $s_2$, $\cdots$, $s_{n-1}$.
\end{itemize}
\end{proposition}

Now, let $W$ be a finite real reflection group. By a theorem of Coxeter, it is known that every finite real reflection group is isomorphic to a Coxeter group. The definition of Coxeter groups by a presentation with generators and relations is as follows. 

\begin{definition}\label{DefinitionCoxeterGroups}

Assume that $W$ is a group and $S$ be a subset of $W$. For $s$ and $t$ in $S$, let $m_{st}$ be the order of $st$ if this order is finite, and be $\infty$ otherwise. We say that $(W,S)$ is a Coxeter system, and that $W$ is a Coxeter group, if $W$ admits the presentation with generating set $S$ and relations: \begin{itemize}

\item quadratic relations: $s^2=1$ for all $s \in S$ and
\item braid relations: $\underset{m_{st}}{\underbrace{sts\cdots}}=\underset{m_{st}}{\underbrace{tst\cdots}}$  for $s,t \in S$, $s \neq t$, and $m_{st} \neq \infty$.

\end{itemize}

\end{definition}

The presentation of a Coxeter group can be described by a diagram where the nodes are the generators that belong to $S$ and the edges describe the relations between these generators. We follow the standard conventions for Coxeter diagrams. The classification of finite irreducible Coxeter groups consists of:

\begin{tabular}{ll}

-- Type $A_n$ (the symmetric group $S_{n+1}$): &  

\begin{tikzpicture}
\draw[fill=black] 
(0,0)                         
      circle [radius=.1] node [label=above:$s_1$] (1) {} 
(1,0) 
      circle [radius=.1] node [label=above:$s_2$] (2) {} 
(3,0)
      circle [radius=.1] node [label=above:$s_n$] (3) {}; 

\draw[-] (1) to (2);
\draw[dashed,-] (2) to (3);

\end{tikzpicture}\\

-- Type $B_n$: & 

\begin{tikzpicture}

\draw[fill=black] 
(0,0)                         
      circle [radius=.1] node [label=above:$s_1$] (1) {}
(1,0) 
      circle [radius=.1] node [label=above:$s_2$] (2) {} 
(2,0)
      circle [radius=.1] node [label=above:$s_3$] (3) {}
(4,0)
	  circle [radius=.1] node [label=above:$s_n$] (4) {}; 

\draw[thick,double,-] (1) to (2);
\draw[-] (2) to (3);
\draw[dashed,-] (3) to (4);

\end{tikzpicture}\\

-- Type $D_n$: &

\begin{tabular}{ll}
\begin{tikzpicture}

\draw[fill=black] 
(-1,0.5)                         
      circle [radius=.1] node [label=left:$s_1$] (1) {}
(-1,-0.5) 
      circle [radius=.1] node [label=left:$s_2$] (2) {} 
(0,0)
      circle [radius=.1] node [label=below:$s_3$] (3) {}
(1,0)
	  circle [radius=.1] node [label=below:$s_4$] (4) {}
(3,0)
	  circle [radius=.1] node [label=below:$s_n$] (5) {}; 

\draw[-] (1) to (3);
\draw[-] (2) to (3);
\draw[-] (3) to (4);
\draw[dashed,-] (4) to (5);

\end{tikzpicture} & \\
 & 

\end{tabular}\\

-- Type $I_2{(e)}$ (the dihedral group): & \begin{tikzpicture}

\draw[fill=black] 
(0,0)                         
      circle [radius=.1] node [label=above:$s_1$] (1) {}
    
(1,0) 
      circle [radius=.1] node [label=above:$s_2$] (2) {};

\draw[-,label=$e$] (1) to (2);

\node at (0.5,0.15) {$e$};

\end{tikzpicture}\\

 & \\

-- $H_3$, $F_4$, $H_4$, $E_6$, $E_7$, and $E_8$. & \\

\end{tabular}

\begin{remark}

Using the notations of the groups that appear in the classification of Shephard and Todd (see Proposition \ref{PropositionClassificationShephardTodd}), we have: type $A_{n-1}$ is $G(1,1,n)$, type $B_n$ is $G(2,1,n)$, type $D_n$ is $G(2,2,n)$, and type $I_2(e)$ is $G(e,e,2)$. For the exceptional groups, we have $H_3 = G_{23}$, $F_4 = G_{28}$, $H_4 = G_{30}$, $E_6 = G_{35}$, $E_7 = G_{36}$, and $E_8 = G_{37}$.

\end{remark}

\section{Complex braid groups}

Let  $W$ be a Coxeter group. We define the Artin-Tits group $B(W)$ associated to $W$ as follows.

\begin{definition}\label{DefinitionArtinTits}

The Artin-Tits group $B(W)$ associated to a Coxeter group $W$ is defined by a presentation with generating set $\widetilde{S}$ in bijection with the generating set $S$ of the Coxeter group and the relations are only the braid relations: $\underset{m_{st}}{\underbrace{\tilde{s}\tilde{t}\tilde{s}\cdots}}=\underset{m_{st}}{\underbrace{\tilde{t}\tilde{s}\tilde{t}\cdots}}$ for $\tilde{s}, \tilde{t} \in \widetilde{S}$ and $\tilde{s} \neq \tilde{t}$, where $m_{st} \in \mathbb{Z}_{\geq 2}$ is the order of $st$ in $W$.

\end{definition}

Consider $W = S_n$, the symmetric group with $n \geq 2$.
The Artin-Tits group associated to $S_n$ is the `classical' braid group denoted by $B_n$. The following diagram encodes all the information about the generators and relations of the presentation of $B_n$.

\begin{center}
\begin{tikzpicture}

\node[draw, shape=circle, label=above:$\tilde{s}_1$] (1) at (0,0) {};
\node[draw, shape=circle, label=above:$\tilde{s}_2$] (2) at (1,0) {};
\node[draw, shape=circle,label=above:$\tilde{s}_{n-2}$] (n-2) at (4,0) {};
\node[draw,shape=circle,label=above:$\tilde{s}_{n-1}$] (n-1) at (5,0) {};

\draw[thick,-] (1) to (2);
\draw[thick,dashed,-] (2) to (n-2);
\draw[thick,-] (n-2) to (n-1);

\end{tikzpicture}
\end{center}

Brou\'e, Malle, and Rouquier \cite{BMR} managed to associate a complex braid group to each complex reflection group. This generalizes the definition of the Artin-Tits groups associated to real reflection groups. We will provide the construction of these complex braid groups. All the results can be found in \cite{BMR}.\\

Let $W < GL(V)$ be a finite complex reflection group. Let $\mathcal{R}$ be the set of reflections of $W$. There exists a corresponding hyperplane arrangement and hyperplane complement:  $\mathcal{A} = \{Ker(s-1)\ |\ s \in \mathcal{R} \}$ and $X = V\setminus\bigcup\mathcal{A}$. The complex reflection group $W$ acts naturally on $X$. Let $X/W$ be its space of orbits and $p: X \rightarrow X/W$ the canonical surjection. By Steinberg’s theorem (see \cite{SteinbergFreeAction}), this action is free. Therefore, it defines a Galois covering $X \rightarrow X/W$, which gives rise to the following exact sequence. Let $x \in X$, we have
 
$$1 \longrightarrow \pi_1(X,x) \longrightarrow \pi_1(X/W, p(x)) \longrightarrow W \longrightarrow 1.$$

\noindent This allows us to give the following definition.

\begin{definition}\label{DefinitionBraidGroup}

We define $P := \pi_1(X,x)$ to be the pure complex braid group and\\ $B := \pi_1(X/W, p(x))$ the complex braid group attached to $W$.

\end{definition}

Let $s \in \mathcal{R}$ and $H_s$ its corresponding hyperplane. We define a loop $\sigma_s \in B$ induced by a path $\gamma$ in $X$. Choose a point $x_0$ `close to $H_s$ and far from the other reflecting hyperplanes'. Define $\gamma$ to be the path in $X$ that is equal to $s.({\widetilde{\gamma}}^{-1}) \circ \gamma_0 \circ \widetilde{\gamma}$ where $\widetilde{\gamma}$ is any path in $X$ from $x$ to $x_0$, $s.({\widetilde{\gamma}}^{-1})$ is the image of $\widetilde{\gamma}^{-1}$ under the action of $s$, and $\gamma_0$ is a path in $X$ from $x_0$ to $s.x_0$ around the hyperplane $H_s$. The path $\gamma$ is illustrated in Figure \ref{FigureBraidedReflections} below.
 
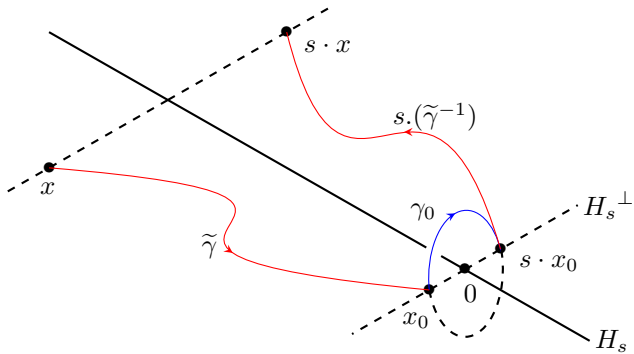
\begin{figure}[H]
\centering
\begin{tikzpicture}
[x={(0.866cm,-0.5cm)}, y={(0.866cm,0.5cm)}, z={(0cm,1cm)}, scale=0.9,
>=stealth, %
inner sep=0pt, outer sep=2pt,%
axis/.style={thick,->},
wave/.style={thick,color=#1,smooth},
polaroid/.style={fill=black!60!white, opacity=0.3},
]

\coordinate (O) at (0, 0, 0);
\draw[thick,dashed] (O) -- +(0,  2.7, 0) ;
\draw[thick,dashed] (O) -- +(0,  -2.7, 0);
\draw(0.3,1.9,0) node [right]{$s\cdot x$};
\draw(0.4,-2.2,0) node [left] {$x$};
 \draw[thick] (-2,0,0) -- (O);
\draw[thick] (O) -- +(4.35, 0,   0) ;

\draw [thick] (4.6,0,0)-- +(2.5, 0,   0) node [right] {$H_s$};
\draw (0,2,0)  node{\textbullet} ;
\draw (0,-2,0)  node{\textbullet} ;
\draw[thick,dashed] (5,0,0) -- +(0,  1.9, 0) node[right]{${H_{s}}^{\perp}$};
\draw (5, 0.6, 0) node{\textbullet}; 
\draw[thick,dashed] (5,0,0) -- +(0,  -1.9, 0) ;
\draw (5, -0.6, 0) node{\textbullet} ;
\draw (5.6,-1.1,0.1) node [left] {$x_0$};
\draw (5.8, 1.2,-0.1) node [left] {$s\cdot x_0$};
\draw (5,0,0) node{\textbullet};
\draw (5.5,-0.2,0) node[left]{0};

\draw[fermion, blue] (5,-0.6,0) arc [start angle=190,end angle=23, x radius=0.55cm, y radius=1cm];
\draw[thick, dashed] (5,-0.6,0) arc [start angle=-162, end angle=21, x radius=0.55cm, y radius=1cm];

\path[fermion,draw,red] (0,-2,0) .. controls (4,1,0) and (0.65,-3.7,0) .. (5,-0.6,0);
\path [fermionbar, draw,red](0,2,0).. controls (4,-1,0) and (0.65,3.7,0) .. (5,0.6,0);

\node at (2.5, -1.8,0) {$\widetilde{\gamma}$};
\node at (2.5, 2,0) {$s.({\widetilde{\gamma}}^{-1})$};
\node at (3.8, 0.5,0) {$\gamma_0$};

\end{tikzpicture}
\caption{The path $\gamma$ in $X$.} \label{FigureBraidedReflections}
\end{figure}

\noindent We call $\sigma_s$ a braided reflection associated to $s$. The importance of this construction will appear in Proposition \ref{PropositionGeneratorsBraidedReflections} below. We have the following property for braided reflections (see \cite{BMR} for a proof).

\begin{proposition}\label{PropositionBraidedReflectionsConjugate}

Let $s_1$ and $s_2$ be two reflections that are conjugate in $W$ and let $\sigma_1$ and $\sigma_2$ be two braided reflections associated to $s_1$ and $s_2$, respectively. We have $\sigma_1$ and $\sigma_2$ are conjugate in $B$.

\end{proposition}

A reflection $s$ is called distinguished if its only nontrivial eigenvalue is $exp(2i\pi/o(s))$, where $o(s)$ is the order of $s$. One can associate a braided reflection $\sigma_s$ to each distinguished reflection $s$. In this case, we call $\sigma_s$ a distinguished braided reflection associated to $s$. We have the following result (see \cite{BMR} for a proof). 
 
\begin{proposition}\label{PropositionGeneratorsBraidedReflections}

The complex braid group $B$ is generated by the distinguished braided reflections associated to all the distinguished reflections in $W$.

\end{proposition}
 
\begin{remark}\label{RemarkBrieskorn}

By a theorem of Brieskorn \cite{BrieskornArtinTitsFundamentalGroup}, the complex braid group associated to a finite Coxeter group $W$ is isomorphic to the Artin-Tits group $B(W)$ defined by a presentation with generators and relations in Definition \ref{DefinitionArtinTits}.

\end{remark} 
 
An important property of the complex braid groups is that they can be defined by presentations with generators and relations. This generalizes the case of complex braid groups associated to finite Coxeter groups (see Remark \ref{RemarkBrieskorn}). Presentations for the complex braid groups associated to $G(de,e,n)$ and some of the exceptional irreducible reflection groups can be found in \cite{BMR}. Presentations for the complex braid groups associated to the rest of the exceptional irreducible reflection groups are given in \cite{BessisMichelPresentationsExcepBraidGroups}, \cite{BessisKPi1}, and \cite{MalleMichelRepresentationsHecke}.
 
\section{Hecke algebras}
 
Extending earlier results in \cite{BroueMalleZyklotomischeHeckealgebren}, Brou\'e, Malle, and Rouquier \cite{BMR} managed to generalize in a natural way the definition of the Hecke (or Iwahori-Hecke) algebra for real reflection groups to arbitrary complex reflection groups. Actually, they defined these Hecke algebras by using their definition of the complex braid groups (see Definition \ref{DefinitionBraidGroup}). The aim of this section is to provide the definition of the Hecke algebra as well as some of its properties.\\

Let $W$ be a complex reflection group and $B$ its complex braid group as defined in the previous section. Let $R=\mathbb{Z}[a_{s,i},a_{s,0}^{-1}]$ where $s$ runs among a representative system of conjugacy classes of distinguished reflections in $W$ and $0 \leq i \leq o(s)-1$, where $o(s)$ is the order of $s$ in $W$. One can choose one distinguished braided reflection $\sigma_s$ for each distinguished reflection $s$. The definition of a distinguished braided reflection was given in the previous section.

\begin{definition}\label{DefinitionHeckeAlgebra}

The Hecke algebra $H(W)$ associated to $W$ is the quotient of the group algebra $RB$ by the ideal generated by the relations $${\sigma_s}^{o(s)} = \sum\limits_{i=0}^{o(s)-1} a_{s,i}{\sigma_s}^{i}$$ where $\sigma_s$ is the distinguished braided reflection associated to $s$, where $s$ belongs to a representative system of conjugacy classes of distinguished reflections in $W$.

\end{definition}

\begin{remark}

By Proposition \ref{PropositionBraidedReflectionsConjugate} and the fact that the relations defining the Hecke algebra are polynomial relations on the braided reflections, the previous definition of the Hecke algebra does not depend on the choice of the representative system of conjugacy classes of distinguished reflections of $W$, nor on the choice of one distinguished braided reflection for each distinguished reflection. It also coincides with the usual definition of the Hecke algebra (see \cite{BMR} and \cite{MarinBMRFreenessConjecture}).

\end{remark}
 
\begin{remark}

When $W$ is a finite Coxeter group with generating set $S$, the Hecke algebra (also known as the Iwahori-Hecke algebra) associated to $W$ is defined over $R=\mathbb{Z}[a_{1},a_{0}^{-1}]$ by a presentation with a generating set $\Sigma$ in bijection with $S$ and the relations are $\underset{m_{st}}{\underbrace{\sigma_s \sigma_t \sigma_s \cdots}}=\underset{m_{st}}{\underbrace{\sigma_t \sigma_s \sigma_t \cdots}}$ along with the polynomial relations $\sigma_s^2 = a_1 \sigma_s + a_0$ for all $s \in S$.

\end{remark}
 
An important conjecture (the BMR freeness conjecture) about $H(W)$ was stated in \cite{BMR}:

\begin{conjecture}\label{ConjectureOfLibertyBMR}

The Hecke algebra $H(W)$ is a free $R$-module of rank $|W|$.

\end{conjecture}

Even without being proven, this conjecture has been used by some authors as an assumption. For example, Malle used it to prove that the characters of $H(W)$ take their values in a specific field (see \cite{MalleCharactersHecke}). Note that the validity of this conjecture implies that $H(W) \otimes_R F$ is isomorphic to the group algebra $FW$, where $F$ is an algebraic closure field of the field of fractions of $R$ (see \cite{MarinBMRFreenessConjecture}).\\

Conjecture \ref{ConjectureOfLibertyBMR} is known to hold for real reflection groups (see Lemma 4.4.3 of \cite{GeckPfeifferBook}). Thanks to Ariki and Ariki-Koike (see \cite{ArikiHecke} and \cite{ArikiKoikeHecke}), the BMR freeness conjecture holds for the infinite series $G(de, e, n)$. Marin proved the conjecture for $G_4$, $G_{25}$, $G_{26}$, and $G_{32}$ in \cite{MarinCubicHeckeAlgebra5Strands} and \cite{MarinBMRFreenessConjecture}. Marin and Pfeiffer proved it for $G_{12}$, $G_{22}$, $G_{24}$, $G_{27}$, $G_{29}$, $G_{31}$, $G_{33}$, and $G_{34}$ in \cite{MarinPfeifferBMRFreeness}. In her PhD thesis and in the article that followed (see \cite{ChavliThesis} and \cite{ChavliArticlePhD}), Chavli proved the validity of this conjecture for $G_{5}$, $G_{6}$, $\cdots$, $G_{16}$. Recently, Marin proved the conjecture for $G_{20}$ and $G_{21}$ (see \cite{MarinG20G21}) and finally Tsushioka for $G_{17}$, $G_{18}$, and $G_{19}$ (see \cite{TsuchiokaBMRfreenessConjecture}). Therefore, after almost a quarter of a century, a case-by-case proof of this conjecture has been established. Hence we have:

\begin{theorem}\label{PropositionBMRFreenessTheorem}

The Hecke algebra $H(W)$ is a free $R$-module of rank $|W|$.

\end{theorem}

The reason for recalling the BMR freeness conjecture is that in Chapter \ref{ChapterHeckeAlgebras}, a new proof of this conjecture is provided for the general series of complex reflection groups of type $G(e,e,n)$ and $G(d,1,n)$.

\section{Motivations and main results}

It is widely believed that the complex braid groups share similar properties with Artin-Tits groups. One would like to extend what is known for Artin-Tits groups to the complex braid groups. Denote by $B(de,e,n)$ the complex braid group associated to the complex reflection group $G(de,e,n)$.\\

In this PhD thesis, we are interested in constructing interval Garside structures for $B(e,e,n)$ as well as explicit Krammer's representations that would probably be helpful to prove that the groups $B(e,e,n)$ are linear. The next subsection gives an idea about the interval structures obtained for $B(e,e,n)$. In the second subsection, we talk about the construction of a Hecke algebra associated to each complex reflection group of type $G(e,e,n)$ and $G(d,1,n)$ and its basis obtained from reduced words in the corresponding complex reflection group. The last subsection is about the construction of Krammer's representations for $B(e,e,n)$.

\subsection{Garside monoids for $B(e,e,n)$}

Brieskorn-Saito \cite{GarsideBrieskornSaito} and Deligne \cite{GarsideDeligne} obtained nice combinatorial results for finite-type Artin-Tits groups that generalize some results of Garside in \cite{GarsideThesis} and \cite{GarsideArt}. They described an Artin-Tits group as the group of fractions of a monoid in which divisibility behaves nicely, and there exists a fundamental element whose set of divisors encodes the entire structure of the group. In modern terminology, these objects are called Garside monoids and groups. Many problems like the Word and Conjugacy problems can be solved in Garside structures with some additional properties about the corresponding Garside group. For a detailed study about Garside monoids and groups, see \cite{GarsideBookPatrick}. Unfortunately, the presentations of Broué, Malle, and Rouquier for $B(de,e,n)$ (except for $B(d,1,n)$) do not give rise to Garside structures as it is shown in \cite{CorranPhD} and \cite{CorranLeeLee}. Therefore, it is interesting to search for (possibly various) Garside structures for these groups. For instance, it is shown by Bessis and Corran \cite{BessisCorranNonCrossingPartitions} in 2006, and by Corran and Picantin \cite{CorranPicantin} in 2009 that $B(e,e,n)$ admits Garside structures. It is also shown in \cite{CorranLeeLee} that $B(de,e,n)$ admits quasi-Garside structures (the set of divisors of the fundamental element is infinite). We are interested in constructing Garside structures for $B(e,e,n)$ that derive from intervals in the associated complex reflection group $G(e,e,n)$.\\

We use the presentation of Corran and Picantin for $G(e,e,n)$ obtained in \cite{CorranPicantin}. The generators and relations of this presentation can be described by the following diagram. Details about this presentation can be found in the next chapter. In \cite{CorranPicantin}, it is shown that if we remove the quadratic relations of this presentation, we get a presentation of the complex braid group $B(e,e,n)$ that we call the presentation of Corran and Picantin of $B(e,e,n)$.

\begin{figure}[H]
\begin{center}
\begin{tikzpicture}[yscale=0.8,xscale=1,rotate=30]

\draw[thick,dashed] (0,0) ellipse (2cm and 1cm);

\node[draw, shape=circle, fill=white, label=above:$\mathbf{t}_0$] (t0) at (0,-1) {\begin{tiny} 2 \end{tiny}};
\node[draw, shape=circle, fill=white, label=above:$\mathbf{t}_1$] (t1) at (1,-0.8) {\begin{tiny} 2 \end{tiny}};
\node[draw, shape=circle, fill=white, label=right:$\mathbf{t}_2$] (t2) at (2,0) {\begin{tiny} 2 \end{tiny}};
\node[draw, shape=circle, fill=white, label=above:$\mathbf{t}_i$] (ti) at (0,1) {\begin{tiny} 2 \end{tiny}};
\node[draw, shape=circle, fill=white, label=above:$\mathbf{t}_{e-1}$] (te-1) at (-1,-0.8) {\begin{tiny} 2 \end{tiny}};

\draw[thick,-] (0,-2) arc (-180:-90:3);

\node[draw, shape=circle, fill=white, label=below left:$\mathbf{s}_3$] (s3) at (0,-2) {\begin{tiny} 2 \end{tiny}};

\draw[thick,-] (t0) to (s3);
\draw[thick,-,bend left] (t1) to (s3);
\draw[thick,-,bend left] (t2) to (s3);
\draw[thick,-,bend left] (s3) to (te-1);

\node[draw, shape=circle, fill=white, label=below:$\mathbf{s}_4$] (s4) at (0.15,-3) {\begin{tiny} 2 \end{tiny}};
\node[draw, shape=circle, fill=white, label=below:$\mathbf{s}_{n-1}$] (sn-1) at (2.2,-4.9) {\begin{tiny} 2 \end{tiny}};
\node[draw, shape=circle, fill=white, label=right:$\mathbf{s}_{n}$] (sn) at (3,-5) {\begin{tiny} 2 \end{tiny}};

\node[fill=white] () at (1,-4.285) {$\cdots$};
\end{tikzpicture}
\end{center}
\caption{Diagram for the presentation of Corran-Picantin of $G(e,e,n)$.}
\end{figure}
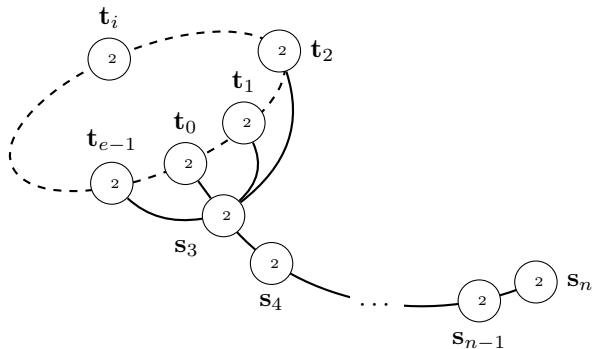

The first step is to define geodesic normal forms (words of minimal length) for all elements of $G(e,e,n)$ over the generating set of the presentation of Corran and Picantin. This is the main result of Section 2.1 of Chapter \ref{ChapterNormalFormsG(de,e,n)} where the geodesic normal forms are defined by an algorithm using the matrix form of the elements of $G(e,e,n)$. Moreover, these normal forms are generalized to the case of $G(de,e,n)$ in Section 2.2 of the same chapter.\\

We will provide an idea about the interval structures for $B(e,e,n)$ obtained in \mbox{Chapter \ref{ChapterIntervalGarside}} that follows essentially \cite{GeorgesNeaimeIntervals}. First, we equip $G(e,e,n)$ with a left and a right division. The description of a normal form for all elements of $G(e,e,n)$ allows us to determine the balanced elements (the set of left divisors coincide with the set of right divisors) of maximal length and to characterize their divisors (see Theorem \ref{PropAllBalancedofMaxLength}). We get $e-1$ balanced elements of maximal length which we denote by \mbox{$\lambda$, $\lambda^2$, $\cdots$, $\lambda^{e-1}$.} Suppose $\lambda^k$ is a balanced element for $1 \leq k \leq e-1$ and $[1,\lambda^k]$ is the interval of the divisors of $\lambda^k$ in the complex reflection group. We manage to prove that the interval $[1,\lambda^k]$ is a lattice for both left and right divisions (see Corollary \ref{CorBothPosetsLattices}). This is done by proving a property in Proposition \ref{PropMatsumoto} for each interval $[1,\lambda^k]$ that is similar to Matsumoto's property for real reflection groups, see \cite{MatsumotoTheorem}. By a theorem of Michel, see Section 10 of \cite{CorseJeanMichel}, we obtain a Garside monoid $M([1,\lambda^k])$ attached to each interval $[1,\lambda^k]$. Moreover, we prove that $M([1,\lambda^k])$ is isomorphic to a monoid that we denote by $B^{\oplus k}(e,e,n)$ and that is defined by a presentation similar to the presentation of Corran and Picantin, see Definition \ref{DefofB+keen}. By $B^{(k)}(e,e,n)$, we denote its group of fractions. One of the important results obtained is the following (see Theorem \ref{TheoremBkIsomBiff}).

\begin{center}
$B^{(k)}(e,e,n)$ is isomorphic to $B(e,e,n)$ if and only if \mbox{$k\wedge e=1$.}
\end{center}

When $k \wedge e \neq 1$, each group $B^{(k)}(e,e,n)$ is described as an amalgamated product of $k \wedge e$ copies of the complex braid group $B(e',e',n)$ with $e'=e/e \wedge k$, over a common subgroup which is the Artin-Tits group $B(2,1,n-1)$. Furthermore, we compute the second integral homology group of $B^{(k)}(e,e,n)$ using the Dehornoy-Lafont complexes \cite{DehLafHomologyofGaussianGroups} and the method of \cite{CalMarHomologyComputations} in order to deduce that $B^{(k)}(e,e,n)$ is not isomorphic to $B(e,e,n)$ when $k \wedge e \neq 1$. This is done in Section 3.4 of Chapter \ref{ChapterIntervalGarside}.\\

The Garside monoids $B^{\oplus k}(e,e,n)$ have been implemented using the development version of the CHEVIE package for GAP3 (see \cite{CHEVIEJMichel} and \cite{GAP3CPMichelNeaime}). In Appendix A, we explain this implementation and provide an algorithm that computes the integral homology groups of $B^{(k)}(e,e,n)$ by using their Garside structures.

\subsection{Hecke algebras for $G(e,e,n)$ and $G(d,1,n)$}

Ariki and Koike \cite{ArikiKoikeHecke} proved Conjecture \ref{ConjectureOfLibertyBMR} for the case of $G(d,1,n)$. Ariki defined in \cite{ArikiHecke} a Hecke algebra for $G(de,e,n)$ by a presentation with generators and relations. He also proved that it is a free module of rank $|G(de,e,n)|$. The Hecke algebra defined by Ariki is isomorphic to the Hecke algebra defined by Brou\'e, Malle, and Rouquier in \cite{BMR} for $G(de,e,n)$ (the details why this is true can be found in Appendix A.2 of \cite{RostamArikiBMR}). Hence one gets a proof of Conjecture \ref{ConjectureOfLibertyBMR} for the case of $G(de,e,n)$.\\

In Chapter \ref{ChapterHeckeAlgebras}, we define the Hecke algebra $H(e,e,n)$ associated to $G(e,e,n)$ as the quotient of the group algebra $R_0(B(e,e,n))$ by some polynomial relations, where $R_0$ is a polynomial ring. It is also described by a presentation with generators and relations by using the presentation of Corran and Picantin of $B(e,e,n)$. Similarly, we define the Hecke algebra $H(d,1,n)$ associated to $G(d,1,n)$ over a polynomial ring $R_0$ and describe it by a presentation with generators and relations by using the presentation of the complex braid group $B(d,1,n)$.\\


Next, we use the geodesic normal forms obtained in Chapter \ref{ChapterNormalFormsG(de,e,n)} for $G(e,e,n)$ and $G(d,1,n)$ in order to provide a nice description for a basis of the corresponding Hecke algebra that is probably simpler than the one obtained by Ariki for the case of $G(e,e,n)$ and by Ariki-Koike for the case of $G(d,1,n)$. Note that a basis for the Hecke algebra associated with $G(d,1,n)$ is also given in \cite{BremkeMalle}. Getting a basis for the Hecke algebra from geodesic normal forms is not very surprising since a spanning set for the Hecke algebra in the case of real reflection groups is made from reduced word representatives of the elements of the Coxeter group (see Lemma 4.4.3 of \cite{GeckPfeifferBook}).\\

By Proposition 2.3 $(ii)$ of \cite{MarinG20G21}, the construction of these bases provide a new proof of Theorem \ref{PropositionBMRFreenessTheorem} for the general series of complex reflection groups of type $G(e,e,n)$ and $G(d,1,n)$. We use Proposition 2.3 $(i)$ of \cite{MarinG20G21} to reduce our proof to find a spanning set of the Hecke algebra over $R_0$ of cardinal equal to the size of the corresponding complex reflection group. We get Theorems \ref{TheoremNewBasis} and \ref{TheoremNewBasisH(d,1,n)} that are the main results of Chapter \ref{ChapterHeckeAlgebras}.

\subsection{Krammer's representations for $B(e,e,n)$}

Both Bigelow \cite{BigelowBraidGroupsLinear} and Krammer \cite{KrammerGroupB4, KrammerBraidGroupsLinear} proved that the classical braid group $B_n$ is linear, that is there exists a faithful linear representation of finite dimension of the classical braid group. This result has been extended to all Artin-Tits groups associated to finite Coxeter groups by Cohen and Wales \cite{CohenWalesLinearity} and Digne \cite{DigneLinearity} by generalizing Krammer's representation as well as Krammer’s faithfulness proof. Paris proved in \cite{ParisArtinMonoidInjectGroup} that Artin-Tits monoids inject in their groups (see Definition \ref{DefinitionArtinTits}). This is done by constructing a faithful linear (infinite dimentional) representation for Artin-Tits monoids. The construction of the representation and the proof of the injectivity are based on a generalization of the methods used by Cohen and Wales, Digne, and Krammer. Moreover, a simple proof of the faithfulness of these representations was given by H\'ee in \cite{HeePreuveSimpleFidelite}. Note that for the case of the classical braid group $B_n$, the representations of $B_n$ occur in earlier work of Lawrence \cite{LawrenceRepresentation}.  \\

Consider a complex braid group $B(de,e,n)$. For $d > 1$, $e \geq 1$, and $n \geq 2$, it is known that the group $B(de,e,n)$ injects in the Artin-Tits group $B(de,1,n)$, see \cite{BMR}. Since $B(de,1,n)$ is linear, we have $B(de,e,n)$ is linear for $d > 1$, $e \geq 1$, and $n \geq 2$. Recall that $B(1,1,n)$ is the classical braid group, $B(2,2,n)$ is an Artin-Tits group, and $B(e,e,2)$ is the Artin-Tits group associated to the dihedral group $I_2(e)$. All of them are then linear. The only remaining cases in the infinite series are when $d=1$, $e > 2$, and $n > 2$. This arises the following question: \begin{center} Is $B(e,e,n)$ linear for all $e > 2$ and $n > 2$?\end{center}

A positive answer for this question is conjectured to be true by Marin in \cite{MarinLinearity} where a generalization of the Krammer's representation has been constructed for the case of $2$-reflection groups. The representation is defined analytically over the field of (formal) Laurent series by the monodromy of some differential forms. It has been generalized by Chen in \cite{ChenFlatConnectionsBrauerAlgebras} to arbitrary reflection groups.\\

Zinno \cite{ZinnoBMW} observed that the Krammer's representation of the classical braid group $B_n$ factors through the BMW (Birman-Murakami-Wenzl) algebra introduced in \cite{BirmanWenzlBMW, MurakamiBMW}. In \cite{CohenGijsbersWalesBMW}, Cohen, Gijsbers, and Wales defined a BMW algebra for Artin-Tits groups of type ADE and showed that the faithful representation constructed by Cohen and Wales in \cite{CohenWalesLinearity} factors through their BMW algebra. In \cite{ChenBMW2017}, Chen defined a BMW algebra for the dihedral groups, based on which he defined a BMW algebra for any Coxeter group extending his previous work in \cite{ChenOldBMW}. He also found a representation of the Artin-Tits groups associated to the dihedral groups. He conjectured that this representation is isomorphic to the generalized Krammer's representation defined by Marin in \cite{MarinLinearity} for the case of the dihedral groups.\\

Inspired by all the previous works, we define a BMW algebra for type $(e,e,n)$ that we denote by BMW$(e,e,n)$, see Definitions \ref{DefinitionBMWB(e,e,n)eOdd} and \ref{DefinitionBMWB(e,e,n)eEven} of Chapter 5. These definitions are inspired from the monoid of Corran and Picantin of $B(e,e,n)$ and from the definition of the BMW algebras for the dihedral groups given by Chen in \cite{ChenBMW2017} and the definition of the BMW algebras of type ADE given by Cohen, Gijsbers, and Wales in \cite{CohenGijsbersWalesBMW}. Moreover,  we describe BMW$(e,e,n)$ as a deformation of a certain algebra that we call the Brauer algebra of type $(e,e,n)$ and we denote it by Br$(e,e,n)$, see Definitions \ref{DefinitionBr(e,e,n)eodd} and \ref{DefinitionBr(e,e,n)eeven} of Chapter 5. We prove in Proposition \ref{PropositionChenIsomBr(e,e,3)} that Br$(e,e,3)$ is isomorphic to the Brauer-Chen algebra defined by Chen in \cite{ChenFlatConnectionsBrauerAlgebras} when $e$ is odd.\\

We are able to construct explicit linear (finite dimensional and absolutely irreducible) representations that are good candidates to be called the Krammer's representations for the complex braid groups $B(3,3,3)$ and $B(4,4,3)$. They are irreducible representations of the BMW algebras BMW$(3,3,3)$ and BMW$(4,4,3)$, respectively. We use the package GBNP (\cite{GBNP}) of GAP4 and the platform MATRICS (\cite{MATRICS}) for our heuristic computations. In Chapter 5, we explain how to construct these representations. We conjecture that they are faithful, see Conjecture \ref{ConjectureRho3and4faithful}.\\

Our method uses the computation of a Gröbner basis from the list of polynomials that describe the relations of BMW$(e,e,n)$. These computations are very heavy for $e \geq 5$ when $n = 3$. However, we were able to compute the dimension of BMW$(5,5,3)$ and BMW$(6,6,3)$ over a finite field for many specializations of the parameters of the BMW algebra. This enables us to propose Conjecture \ref{ConjectureStructureBMW} about the structure and dimension of BMW$(e,e,3)$. In Appendix B, we provide the algorithms that enable us to construct the Krammer's representations and to propose Conjecture \ref{ConjectureStructureBMW}.

%% file: chap22.tex
\chapter{Geodesic normal forms for $G(de,e,n)$}\label{ChapterNormalFormsG(de,e,n)}
 
\minitoc
 
\bigskip 
 
The aim of this chapter is to define geodesic normal forms for the elements of the general series of complex reflection groups $G(de,e,n)$. This requires the elaboration of a combinatorial technique to determine a reduced expression decomposition of an element over the generating set of the presentation of Corran-Picantin \cite{CorranPicantin} in the case of  $G(e,e,n)$ and Corran-Lee-Lee \cite{CorranLeeLee} in the case of $G(de,e,n)$ with $d > 1$. We start by studying the case of $G(e,e,n)$ and then the general case of $G(de,e,n)$. The reason for studying the case of $G(e,e,n)$ separately is that we use it to prove the general case and we will essentially use the results of the case of $G(e,e,n)$ in the next chapters. We therefore wanted to separate it from the other cases of the general series of complex reflection groups.

\section{Geodesic normal forms for $G(e,e,n)$}

Recall that $G(e,e,n)$ is the group of $n \times n$ monomial matrices where all nonzero entries are $e$-th roots of unity and such that their product is equal to $1$. We start by recalling the presentation of Corran-Picantin for $G(e,e,n)$. Then, we define an algorithm that produces a word representative for each element of $G(e,e,n)$ over the generating set of the presentation of Corran-Picantin. Finally, we prove that these word representatives are geodesic. Hence we get geodesic normal forms for the groups $G(e,e,n)$. 

\subsection{Presentation for $G(e,e,n)$}

Let $e \geq 1$ and $n > 1$. We recall the presentation of the complex reflection group $G(e,e,n)$ given in \cite{CorranPicantin}.

\begin{definition}\label{DefinitionPresentationCorranPicantin}

The complex reflection group $G(e,e,n)$ can be defined by a presentation with set of generators: $\textbf{X} = \{\textbf{t}_i \ |\ i \in \mathbb{Z}/e\mathbb{Z}\} \cup \{\textbf{s}_3, \textbf{s}_4, \cdots, \textbf{s}_n \}$ and relations as follows.

\begin{enumerate}

\item $\textbf{t}_i \textbf{t}_{i-1} = \textbf{t}_j \textbf{t}_{j-1}$ for $i, j \in \mathbb{Z}/e\mathbb{Z}$,
\item $\textbf{t}_i \textbf{s}_3 \textbf{t}_i = \textbf{s}_3 \textbf{t}_i \textbf{s}_3$ for $i \in \mathbb{Z}/e\mathbb{Z}$,
\item $\textbf{s}_j \textbf{t}_i = \textbf{t}_i \textbf{s}_j$ for $i \in \mathbb{Z}/e\mathbb{Z}$ and $4 \leq j \leq n$,
\item $\textbf{s}_i \textbf{s}_{i+1} \textbf{s}_i = \textbf{s}_{i+1} \textbf{s}_i \textbf{s}_{i+1}$ for $3 \leq i \leq n-1$, 
\item $\textbf{s}_i \textbf{s}_j = \textbf{s}_j \textbf{s}_i$ for $|i-j| > 1$, and
\item $\textbf{t}_i^2=1$ for $i \in \mathbb{Z}/e\mathbb{Z}$ and $\textbf{s}_j^2=1$ for $3 \leq j \leq n$.

\end{enumerate}

\end{definition}

The matrices in $G(e,e,n)$ that correspond to the set of generators $\mathbf{X}$ of this presentation are given by  $\mathbf{t}_i \longmapsto t_i:= \begin{pmatrix}

0 & \zeta_{e}^{-i} & 0\\
\zeta_{e}^{i} & 0 & 0\\
0 & 0 & I_{n-2}\\

\end{pmatrix}$ \mbox{for $0 \leq i \leq e-1$}, and $\mathbf{s}_j \longmapsto s_j:= \begin{pmatrix}

I_{j-2} & 0 & 0 & 0\\
0 & 0 & 1 & 0\\
0 & 1 & 0 & 0\\
0 & 0 & 0 & I_{n-j}\\

\end{pmatrix}$ for $3 \leq j \leq n$. To avoid confusion, we use normal letters for matrices and bold letters for words over $\mathbf{X}$. Denote by $X$ the set $\{t_0,t_1, \cdots, t_{e-1},s_3, \cdots, s_n\}$.\\

This presentation can be described by the following diagram. The dashed circle describes Relation $1$ of Definition \ref{DefinitionPresentationCorranPicantin}. The other edges used to describe all the other relations follow the standard conventions for Coxeter groups.

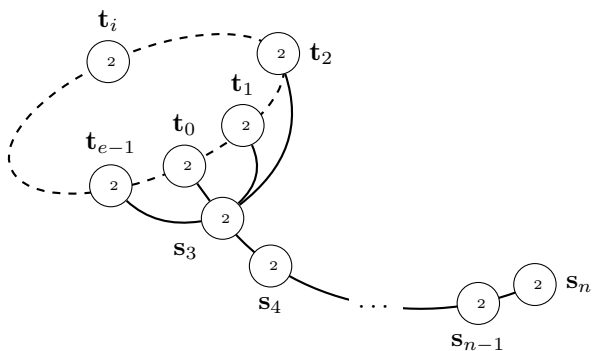
\begin{figure}[H]
\begin{center}
\begin{tikzpicture}[yscale=0.8,xscale=1,rotate=30]

\draw[thick,dashed] (0,0) ellipse (2cm and 1cm);

\node[draw, shape=circle, fill=white, label=above:$\mathbf{t}_0$] (t0) at (0,-1) {\begin{tiny} 2 \end{tiny}};
\node[draw, shape=circle, fill=white, label=above:$\mathbf{t}_1$] (t1) at (1,-0.8) {\begin{tiny} 2 \end{tiny}};
\node[draw, shape=circle, fill=white, label=right:$\mathbf{t}_2$] (t2) at (2,0) {\begin{tiny} 2 \end{tiny}};
\node[draw, shape=circle, fill=white, label=above:$\mathbf{t}_i$] (ti) at (0,1) {\begin{tiny} 2 \end{tiny}};
\node[draw, shape=circle, fill=white, label=above:$\mathbf{t}_{e-1}$] (te-1) at (-1,-0.8) {\begin{tiny} 2 \end{tiny}};

\draw[thick,-] (0,-2) arc (-180:-90:3);

\node[draw, shape=circle, fill=white, label=below left:$\mathbf{s}_3$] (s3) at (0,-2) {\begin{tiny} 2 \end{tiny}};

\draw[thick,-] (t0) to (s3);
\draw[thick,-,bend left] (t1) to (s3);
\draw[thick,-,bend left] (t2) to (s3);
\draw[thick,-,bend left] (s3) to (te-1);

\node[draw, shape=circle, fill=white, label=below:$\mathbf{s}_4$] (s4) at (0.15,-3) {\begin{tiny} 2 \end{tiny}};
\node[draw, shape=circle, fill=white, label=below:$\mathbf{s}_{n-1}$] (sn-1) at (2.2,-4.9) {\begin{tiny} 2 \end{tiny}};
\node[draw, shape=circle, fill=white, label=right:$\mathbf{s}_{n}$] (sn) at (3,-5) {\begin{tiny} 2 \end{tiny}};

\node[fill=white] () at (1,-4.285) {$\cdots$};
\end{tikzpicture}
\end{center}
\caption{Diagram for the presentation of Corran-Picantin of $G(e,e,n)$.}\label{FigureDiagramCPG(e,e,n)}
\end{figure}

\begin{remark}\

\begin{enumerate}

\item \mbox{For $e=1$ and $n \geq 2$, we get the classical presentation of the symmetric group $S_n$.}
\item For $e=2$ and $n \geq 2$, we get the classical presentation of the Coxeter group of type $D_n$.
\item For $e \geq 2$ and $n=2$, we get the dual presentation of the dihedral group $I_2(e)$, see \cite{PicantinPresentationsDualMonoids}. 

\end{enumerate}

\end{remark}

\begin{remark}\label{RemarkPresB(e,e,n)}

In their paper \cite{CorranPicantin}, Corran and Picantin showed that if we remove the quadratic relations (Relations 6 of Definition \ref{DefinitionPresentationCorranPicantin}) from the presentation of $G(e,e,n)$, we get a presentation of the complex braid group $B(e,e,n)$. They also proved that this presentation provides a Garside structure for $B(e,e,n)$. The notion of Garside structures will be developed in the next chapter. 

\end{remark}

We set the following convention.

\begin{convention}\label{ConventionDecreaseIncreaseIndex}

A decreasing-index expression of the form $\mathbf{s}_i \mathbf{s}_{i-1} \cdots \mathbf{s}_{i'}$ is the empty word when $i < i'$ and an increasing-index expression of the form $\mathbf{s}_i \mathbf{s}_{i+1} \cdots \mathbf{s}_{i'}$ is the empty word when $i > i'$. Similarly, in $G(e,e,n)$, a decreasing-index product of the form $s_i s_{i-1} \cdots s_{i'}$ is equal to $I_n$ when $i < i'$ and an increasing-index product of the form $s_i s_{i+1} \cdots s_{i'}$ is equal to $I_n$ when $i > i'$, where $I_n$ is the identity $n \times n$ matrix.

\end{convention}

\subsection{Minimal word representatives}

Recall that an element $\boldsymbol{w} \in \mathbf{X}^{*}$ is called a word over $\mathbf{X}$. We denote by $\boldsymbol{\ell}(\boldsymbol{w})$ the length over $\mathbf{X}$ of the word $\boldsymbol{w}$.

\begin{definition}\label{def.Slength.reducedword}
Let $w$ be an element of $G(e,e,n)$.
We define $\ell(w)$ to be the minimal word length $\boldsymbol{\ell}(\boldsymbol{w})$ of a word $\boldsymbol{w}$ over $\mathbf{X}$ that represents $w$. A reduced expression of $w$ is any word representative of $w$ of word length $\ell(w)$.
\end{definition}

Our aim is to represent each element of $G(e,e,n)$ by a reduced word over $\mathbf{X}$, where $\mathbf{X}$ denotes the set of the generators of the presentation of Corran and Picantin of $G(e,e,n)$. This requires the elaboration of a combinatorial technique to determine a reduced expression decomposition over $\mathbf{X}$ for an element of $G(e,e,n)$.\\

We introduce Algorithm \ref{Algo1} below (see next page) that produces a word $R\!E(w)$ over $\mathbf{X}$ for a given matrix $w$ in $G(e,e,n)$. Note that we use Convention \ref{ConventionDecreaseIncreaseIndex} in the elaboration of the algorithm. Later on, we prove that $R\!E(w)$ is a reduced expression over $\mathbf{X}$ of $w$, see Proposition \ref{PropREwRedExp}.\\

Let $w_n := w \in G(e,e,n)$. For $i$ from $n$ to $2$, the $i$-th step of Algorithm\! \ref{Algo1} transforms the block diagonal matrix $\left(
\begin{array}{c|c}
w_i & 0 \\
\hline
0 & I_{n-i}
\end{array}
\right)$ into a block diagonal matrix $\left(
\begin{array}{c|c}
w_{i-1} & 0 \\
\hline
0 & I_{n-i+1}
\end{array}
\right) \in G(e,e,n)$ with $w_1 = 1$. Actually, for $2 \leq i \leq n$, there exists a unique $c$ with $1 \leq c \leq n$ such that $w_i[i,c] \neq 0$. At each step $i$ of Algorithm\! \ref{Algo1}, if $w_i[i,c] =1$, we shift it into the diagonal position $[i,i]$ by right multiplication by transpositions of the symmetric group $S_n$. If $w_i[i,c] \neq 1$, we shift it into the first column by right multiplication by transpositions, transform it into $1$ by right multiplication by an element of $\{t_0, t_1, \cdots, t_{e-1} \}$, and then shift the $1$ obtained into the diagonal position $[i,i]$.\\ 

\begin{algorithm}[]\label{Algo1}

\SetKwInOut{Input}{Input}\SetKwInOut{Output}{Output}

\noindent\rule{12cm}{0.5pt}

\Input{$w$, a matrix in $G(e,e,n)$, with $e \geq 1$ and $n \geq 2$.}
\Output{$R\!E(w)$, a word over $\mathbf{X}$.}

\noindent\rule{12cm}{0.5pt}

\textbf{Local variables}: $w'$, $R\!E(w)$,  $i$, $U$, $c$, $k$.

\noindent\rule{12cm}{0.5pt}

\textbf{Initialisation}:
$U:=[1,\zeta_e,\zeta_{e}^2,...,\zeta_{e}^{e-1}]$, $s_2 := t_0$, $\mathbf{s}_2 := \mathbf{t}_0$,\\
$R\!E(w) := \varepsilon$: the empty word, $w' := w$.

\noindent\rule{12cm}{0.5pt}

\For{$i$ \textbf{from} $n$ \textbf{down to} $2$} {
	$c:=1$; $k:=0$; \\
	\While{$w'[i,c] = 0$}{$c :=c+1$\; 
	}
	 \textit{\#Then $w'[i,c]$ is the root of unity on the row $i$}\;
	\While{$U[k+1]\neq w'[i,c] $}{$k :=k+1$\;
	}
	\textit{\#Then $w'[i,c] = \zeta_{e}^k$.}\\

		\If{$k \neq 0$}{
		$w' := w's_{c}s_{c-1} \cdots s_{3}s_{2}t_{k}$; \textit{\#Then $w'[i,2] =1$}\;
		$R\!E(w) := \mathbf{t}_k\mathbf{s}_2\mathbf{s}_3 \cdots \mathbf{s}_c R\!E(w)$\;
		$c:=2$\;
	}	
	$w' := w's_{c+1} \cdots s_{i-1} s_{i}$; \textit{\#Then $w'[i,i] = 1$}\;
	$R\!E(w) := \mathbf{s}_i \mathbf{s}_{i-1} \cdots \mathbf{s}_{c+1} R\!E(w)$\;
}
\textbf{Return} $R\!E(w)$;

\noindent\rule{12cm}{0.5pt}

\caption{A word over $\mathbf{X}$ corresponding to an element $w \in G(e,e,n)$.}
\end{algorithm}

We provide two examples in order to better understand Algorithm \ref{Algo1}. The first one is for an element $w$ of $G(3,3,4)$ and the second example is for an element $w$ of $G(2,2,4)$, that is the Coxeter group of type $D_4$. At each step, we indicate the values of $i$, $k$, and $c$ such that $w_i[i,c] = \zeta_e^k$.

\begin{example}\label{examp algo NF}

We apply Algorithm\! \ref{Algo1} to $w := \begin{pmatrix}

0 & 0 & 0 & 1\\
0 & \zeta_3^2 & 0 & 0\\
0 & 0 & \zeta_3 & 0\\
1 & 0 & 0 & 0\\

\end{pmatrix}$ $\in G(3,3,4)$.\\
Step $1$ $(i=4, k=0, c=1)$: $w':=ws_2s_3s_4=\begin{pmatrix}

0 & 0 & 1 & 0\\
\zeta_3^2 & 0 & 0 & 0\\
0 & \boxed{\zeta_3} & 0 & 0\\
0 & 0 & 0 & \mathbf{1}\\

\end{pmatrix}$.\\
Step $2$ $(i=3, k=1, c=2)$: $w' := w's_2 = \begin{pmatrix}

0 & 0 & 1 & 0\\
0 & \zeta_3^2 & 0 & 0\\
\zeta_3 & 0 & 0 & 0\\
0 & 0 & 0 & 1\\

\end{pmatrix}$,\\
then $w' := w't_1 = \begin{pmatrix}

0 & 0 & 1 & 0\\
1 & 0 & 0 & 0\\
0 & 1 & 0 & 0\\
0 & 0 & 0 & 1\\

\end{pmatrix}$, then $w':=w's_3 = \begin{pmatrix}

0 & 1 & 0 & 0\\
\boxed{1} & 0 & 0 & 0\\
0 & 0 & \mathbf{1} & 0\\
0 & 0 & 0 & 1\\

\end{pmatrix}$.\\
Step $3$ $(i=2, k=0, c=1)$: $w' := w's_2 = I_{4}$.\\
Hence $R\!E(w) = \mathbf{s}_2 \mathbf{s}_3 \mathbf{t}_1 \mathbf{s}_2 \mathbf{s}_4 \mathbf{s}_3 \mathbf{s}_2$. Recall that $\mathbf{s_2} = \mathbf{t}_0$. Thus, $R\!E(w) = \mathbf{t}_0 \mathbf{s}_3 \mathbf{t}_1 \mathbf{t}_0 \mathbf{s}_4 \mathbf{s}_3 \mathbf{t}_0$.\\

\end{example}

\begin{example}

We apply Algorithm\! \ref{Algo1} to $w := \begin{pmatrix}

0 & 1 & 0 & 0\\
0 & 0 & 0 & -1\\
0 & 0 & 1 & 0\\
-1 & 0 & 0 & 0\\

\end{pmatrix}$ $\in G(2,2,4)$.\\
Step $1$ $(i=4, k=1, c=1)$: $w':=wt_1=\begin{pmatrix}

-1 & 0 & 0 & 0\\
0 & 0 & 0 & -1\\
0 & 0 & 1 & 0\\
0 & 1 & 0 & 0\\

\end{pmatrix}$,\\ then $w':=w' s_3s_4=\begin{pmatrix}

-1 & 0 & 0 & 0\\
0 & 0 & -1 & 0\\
0 & 1 & 0 & 0\\
0 & 0 & 0 & 1\\

\end{pmatrix}$,\\
Step $2$ $(i=3, k=0, c=2)$: $w' := w's_3 = \begin{pmatrix}

-1 & 0 & 0 & 0\\
0 & -1 & 0 & 0\\
0 & 0 & 1 & 0\\
0 & 0 & 0 & 1\\

\end{pmatrix}$.\\
Step $3$ $(i=2, k=1, c=2)$: $w' := w's_2 = \begin{pmatrix}

0 & -1 & 0 & 0\\
-1 & 0 & 0 & 0\\
0 & 0 & 1 & 0\\
0 & 0 & 0 & 1\\

\end{pmatrix}$,\\ then $w' := w' t_1 = I_4$.\\
Hence $R\!E(w) = \mathbf{t}_1 \mathbf{s}_2 \mathbf{s}_3 \mathbf{s}_4 \mathbf{s}_3 \mathbf{t}_1 = \mathbf{t}_1 \mathbf{t}_0 \mathbf{s}_3 \mathbf{s}_4 \mathbf{s}_3 \mathbf{t}_1$ \emph{(since $\mathbf{s}_2 = \mathbf{t}_0$)}.\\

\end{example}

\begin{remark}

Let $w$ be an element of $G(e,e,2)$, that is the dihedral group $I_2(e)$. Denote by $\varepsilon$ the empty word. By Algorithm \ref{Algo1} and by assuming Convention \ref{ConventionDecreaseIncreaseIndex}, $R\!E(w)$ belongs to $\{\varepsilon, \mathbf{t}_0, \mathbf{t}_1, \cdots, \mathbf{t}_{e-1}, \mathbf{t}_1\mathbf{t}_0, \mathbf{t}_2\mathbf{t}_0, \cdots, \mathbf{t}_{e-1}\mathbf{t}_0 \}$.

\end{remark}

The next lemma follows directly from Algorithm \ref{Algo1}. 

\begin{lemma}\label{LemmaBlocks}

For $2 \leq i \leq n$, suppose $w_i[i,c] \neq 0$. The block $w_{i-1}$ is obtained by
\begin{itemize}

\item removing the row $i$ and the column $c$ from $w_i$, then by

\item multiplying the first column of the new matrix by $w_i[i,c]$.\\

\end{itemize}

\end{lemma}

\begin{example}

Let $w$ be as in Example \ref{examp algo NF}, where $n = 4$. The block $w_{3}$ is obtained by removing the row number $4$ and first column from $w_4 = w$ to obtain $\begin{pmatrix}

0 & 0 & 1\\
\zeta_3^2 & 0 & 0\\
0 & \zeta_3 & 0\\

\end{pmatrix}$, then by multiplying the first column of this matrix by $1$. The same can be said for the other block $w_2$.

\end{example}

\begin{definition}\label{DefREiw}

Let $2 \leq i \leq n$. Denote by $w_i[i,c]$ the unique nonzero entry on the row $i$ with $1 \leq c \leq i$.

\begin{itemize}

\item If $w_{i}[i,c] =1$, we define $R\!E_{i}(w)$ to be the word \\ 
$\mathbf{s}_i \mathbf{s}_{i-1} \cdots \mathbf{s}_{c+1}$ (decreasing-index expression).
\item If $w_{i}[i,c] = \zeta_{e}^{k}$ with $k \neq 0$, we define $R\!E_{i}(w)$ to be the word \\
			\begin{tabular}{ll}
			$\mathbf{s}_i \cdots \mathbf{s}_3 \mathbf{t}_k$ & if $c=1$,\\
			$\mathbf{s}_i \cdots \mathbf{s}_3 \mathbf{t}_k \mathbf{t}_0$  & if $c=2$,\\
			$\mathbf{s}_i \cdots \mathbf{s}_3 \mathbf{t}_k \mathbf{t}_0 \mathbf{s}_3 \cdots \mathbf{s}_c$ & if $c \geq 3$.
			\end{tabular}

\end{itemize}

\end{definition}

Remark that for $3 \leq i \leq n$, the word $R\!E_{i}(w)$ is either the empty word (when $w_{i}[i,i] =1$, see Convention \ref{ConventionDecreaseIncreaseIndex}) or a word that contains $\mathbf{s}_{i}$ necessarily but does not contain any of $\mathbf{s}_{i+1},\mathbf{s}_{i+2}, \cdots , \mathbf{s}_n$. Remark also that for $i = 2$, by using Convention \ref{ConventionDecreaseIncreaseIndex}, we have $R\!E_2(w) \in \{\varepsilon, \mathbf{t}_0, \mathbf{t}_1, \cdots, \mathbf{t}_{e-1}, \mathbf{t}_1\mathbf{t}_0, \cdots, \mathbf{t}_{e-1}\mathbf{t}_0\}$.

\begin{lemma}

We have $R\!E(w) = R\!E_2(w) R\!E_3(w) \cdots R\!E_n(w)$.

\end{lemma}

\begin{proof}

The output $R\!E(w)$ of Algorithm \ref{Algo1} is a concatenation of the words\\ $R\!E_2(w), R\!E_3(w), \cdots$, and $R\!E_n(w)$ obtained at each step $i$ from $n$ to $2$ of \mbox{Algorithm \ref{Algo1}}.

\end{proof}

\begin{example}

If $w$ is defined as in Example \ref{examp algo NF}, we have\\
$R\!E(w) = \underset{R\!E_{2}(w)}{\underbrace{\mathbf{t}_0}} \hspace{0.2cm} \underset{R\!E_{3}(w)}{\underbrace{\mathbf{s}_3\mathbf{t}_1\mathbf{t}_0}} \hspace{0.2cm} \underset{R\!E_{4}(w)}{\underbrace{ \mathbf{s}_4\mathbf{s}_3\mathbf{t}_0}}$.

\end{example}

\begin{proposition}

The word $R\!E(w)$ given by Algorithm \ref{Algo1} is a word representative over $\mathbf{X}$ of $w \in G(e,e,n)$.

\end{proposition}

\begin{proof}

Algorithm \ref{Algo1} transforms the matrix $w$ into $I_n$ by multiplying it on the right by elements of $X$. We get $wx_1 \cdots x_r = I_n$, where $x_1$, $\cdots$, $x_r$ are elements of $X$. Hence $w = x_r^{-1} \cdots x_1^{-1} = x_r \cdots x_1$ since $x_i^2 = 1$ for all $x_i \in X$. The output $R\!E(w)$ of Algorithm \ref{Algo1} is $R\!E(w) = \mathbf{x}_r \cdots \mathbf{x}_1$. Hence it is a word representative over $\mathbf{X}$ of $w \in G(e,e,n)$.

\end{proof}

The following proposition will prepare us to prove that the output of \mbox{Algorithm \ref{Algo1}} is a reduced expression over $\mathbf{X}$ of a given element $w \in G(e,e,n)$.

\begin{proposition}\label{prop.lengthcomp}

Let $w$ be an element of $G(e,e,n)$. For all $x \in X$, we have\\
$$|\boldsymbol{\ell}(R\!E(xw)) - \boldsymbol{\ell}(R\!E(w))| = 1.$$

\end{proposition}

\begin{proof}

For $1 \leq i \leq n$, there exists a unique $c_i$ such that $w[i,c_i] \neq 0$. We denote $w[i,c_i]$ by $a_i$.\\

\underline{Case 1: Suppose $x = s_i$ for $3 \leq i \leq n$.}\\

Set $w' := s_i w$. Since the left multiplication by the matrix $x$ exchanges the rows $i-1$ and $i$ of $w$ and the other rows remain the same, by Definition \ref{DefREiw} and Lemma \ref{LemmaBlocks}, we have:\\
$R\!E_{i+1}(xw)R\!E_{i+2}(xw) \cdots R\!E_{n}(xw)=R\!E_{i+1}(w) R\!E_{i+2}(w) \cdots R\!E_{n}(w)$ and\\
$R\!E_{2}(xw)R\!E_{3}(xw) \cdots R\!E_{i-2}(xw)=R\!E_{2}(w) R\!E_{3}(w) \cdots R\!E_{i-2}(w)$.\\
Then, in order to prove our property, we should compare $\boldsymbol{\ell}_1:=\boldsymbol{\ell}(R\!E_{i-1}(w)R\!E_i(w))$ and $\boldsymbol{\ell}_2:=\boldsymbol{\ell}(R\!E_{i-1}(xw)R\!E_i(xw))$.\\

Suppose  $c_{i-1} < c_{i}$, by Lemma \ref{LemmaBlocks}, the rows $i-1$ and $i$ of the blocks $w_i$ and $w'_{i}$ are of the form:

\begin{tabular}{ccc}

\begin{tikzpicture}[scale=0.5]

\node at (0,0) {};
\node at (0,0.8) {$w_i$ :};

\end{tikzpicture}

 & & \begin{tikzpicture}[scale=0.5]

\node at (0,1) {$i$};
\node at (0,2) {$i-1$};
\node at (2,3) {$..$};
\node at (2.6,3) {$c$};
\node at (4,3) {$..$};
\node at (5,3) {$c'$};
\node at (5.7,3) {$..$};
\node at (7,3) {$i$};

\node at (3,2) {$b_{i-1}$};
\node at (5,1) {$a_i$};

\draw [-] (1,0.5) to node[auto] {} (1,2.5);
\draw [-] (1,0.5) to node[auto] {} (7.3,0.5);
\draw [-] (7.3,0.5) to node[auto] {} (7.3,2.5);
\draw [-] (1,2.5) to node[auto] {} (7.3,2.5);

\end{tikzpicture}\\

\begin{tikzpicture}[scale=0.5]

\node at (0,0) {};
\node at (0,0.8) {$w'_i$ :};

\end{tikzpicture}

 & & \begin{tikzpicture}[scale=0.5]

\node at (0,1) {$i$};
\node at (0,2) {$i-1$};
\node at (2,3) {$..$};
\node at (2.6,3) {$c$};
\node at (4,3) {$..$};
\node at (5,3) {$c'$};
\node at (5.7,3) {$..$};
\node at (7,3) {$i$};

\node at (3,1) {$b_{i-1}$};
\node at (5,2) {$a_i$};

\draw [-] (1,0.5) to node[auto] {} (1,2.5);
\draw [-] (1,0.5) to node[auto] {} (7.3,0.5);
\draw [-] (7.3,0.5) to node[auto] {} (7.3,2.5);
\draw [-] (1,2.5) to node[auto] {} (7.3,2.5);

\end{tikzpicture}

\end{tabular}

with $c < c'$ and where we write $b_{i-1}$ instead of $a_{i-1}$ since $a_{i-1}$ may change when applying Algorithm \ref{Algo1} if $c_{i-1} =1$, that is $a_{i-1}$ on the first column of $w$.\\

We will discuss different cases depending on the values of $a_i$ and $b_{i-1}$.

\begin{itemize}

\item \underline{Suppose $a_i = 1$.}

\begin{itemize}

\item \underline{If $b_{i-1} =1$,}\\
we have $R\!E_{i}(w) = \boldsymbol{s}_i \cdots \boldsymbol{s}_{c'+2} \boldsymbol{s}_{c'+1}$ and $R\!E_{i-1}(w) = \boldsymbol{s}_{i-1} \cdots \boldsymbol{s}_{c+2} \boldsymbol{s}_{c+1}$.
Furthermore, we have $R\!E_{i}(xw) = \boldsymbol{s}_i \cdots \boldsymbol{s}_{c+2} \boldsymbol{s}_{c+1}$\\
and $R\!E_{i-1}(xw) = \boldsymbol{s}_{i-1} \cdots \boldsymbol{s}_{c'+1} \boldsymbol{s}_{c'}$.\\
It follows that $\boldsymbol{\ell}_1 = ((i-1)-(c+1)+1) + (i-(c'+1)+1) = 2i-c-c'-1$ and $\boldsymbol{\ell}_2 = ((i-1)-c'+1) + (i-(c+1)+1) = 2i-c-c'$ hence $\boldsymbol{\ell}_2 = \boldsymbol{\ell}_1 +1$.\\
\item \underline{If $b_{i-1} = \zeta_e^{k}$ with $1 \leq k \leq e-1$,}\\
we have $R\!E_{i}(w) = \boldsymbol{s}_i \cdots \boldsymbol{s}_{c'+2} \boldsymbol{s}_{c'+1}$ and $R\!E_{i-1}(w) = \boldsymbol{s}_{i-1} \cdots \boldsymbol{s}_3 \boldsymbol{t}_k \boldsymbol{t}_0 \boldsymbol{s}_{3} \cdots \boldsymbol{s}_{c}$. Furthermore, we have $R\!E_{i}(xw) = \boldsymbol{s}_i \cdots \boldsymbol{s}_3 \boldsymbol{t}_k \boldsymbol{t}_0 \boldsymbol{s}_{3} \cdots \boldsymbol{s}_{c}$ and $R\!E_{i-1}(xw) = \boldsymbol{s}_{i-1} \cdots \boldsymbol{s}_{c'}$.\\
It follows that $\boldsymbol{\ell}_1 = (((i-1)-3+1) + 2 + (c-3+1)) + (i-(c'+1)+1) = 2i+c-c'-3$ and $\boldsymbol{\ell}_2 = ((i-1)-c'+1) + ((i-3+1) + 2 + (c-3+1)) = 2i+c-c'-2$ hence $\boldsymbol{\ell}_2 = \boldsymbol{\ell}_1 +1$.

\end{itemize}

It follows that

\begin{center}
if $a_i =1$, then $\boldsymbol{\ell}(R\!E(s_iw))= \boldsymbol{\ell}(R\!E(w)) +1. \hspace{1cm} (a)$
\end{center}

\item \underline{Suppose now that $a_i = \zeta_e^{k}$ with $1 \leq k \leq e-1$.}

\begin{itemize}

\item \underline{If $b_{i-1} = 1$,}\\
we have $R\!E_{i}(w) = \boldsymbol{s}_{i} \cdots \boldsymbol{s}_3 \boldsymbol{t}_k \boldsymbol{t}_0 \boldsymbol{s}_{3} \cdots \boldsymbol{s}_{c'}$ and $R\!E_{i-1}(w) =\boldsymbol{s}_{i-1} \cdots \boldsymbol{s}_{c+1}$.\\
Also, we have $R\!E_{i}(xw) = \boldsymbol{s}_i \cdots \boldsymbol{s}_{c+1}$ and\\
$R\!E_{i-1}(xw) = \boldsymbol{s}_{i-1} \cdots \boldsymbol{s}_3 \boldsymbol{t}_k \boldsymbol{t}_0 \boldsymbol{s}_{3} \cdots \boldsymbol{s}_{c'-1}$.\\
It follows that $\boldsymbol{\ell}_1 = ((i-1)-(c+1)-1) + ((i-3+1)+2+(c'-3+1)) = 2i-c+c'-5$ and $\boldsymbol{\ell}_2 = (((i-1)-3+1)+2+((c'-1)-3+1)) + (i-(c+1)-1) = 2i-c+c'-6$ hence $\boldsymbol{\ell}_2 = \boldsymbol{\ell}_1 -1$.\\

\item \underline{If $b_{i-1} = \zeta_e^{k'}$ with $1 \leq k' \leq e-1$,}\\
we have $R\!E_{i}(w) = \boldsymbol{s}_{i} \cdots \boldsymbol{s}_3 \boldsymbol{t}_k \boldsymbol{t}_0 \boldsymbol{s}_{3} \cdots  \boldsymbol{s}_{c'}$ and\\
$R\!E_{i-1}(w) = \boldsymbol{s}_{i-1} \cdots \boldsymbol{s}_3 \boldsymbol{t}_{k'} \boldsymbol{t}_0 \boldsymbol{s}_{3} \cdots \boldsymbol{s}_{c}$. \\
Also, we have $R\!E_{i}(xw) = \boldsymbol{s}_{i} \cdots \boldsymbol{s}_3 \boldsymbol{t}_{k'} \boldsymbol{t}_0 \boldsymbol{s}_{3} \cdots \boldsymbol{s}_{c}$ and\\
$R\!E_{i-1}(xw) = \boldsymbol{s}_{i-1} \cdots \boldsymbol{s}_3 \boldsymbol{t}_k \boldsymbol{t}_0 \boldsymbol{s}_{3} \cdots \boldsymbol{s}_{c'-1}$.\\
It follows that $\boldsymbol{\ell}_1 = ((i-1)-3+1) +2+(c-3+1) + (i-3+1) +2+ (c'-3+1) = 2i+c+c'-5$ and $\boldsymbol{\ell}_2 = ((i-1)-3+1)+2+((c'-1)-3+1)+(i-3+1)+2+(c-3+1) = 2i+c+c'-6$ hence $\boldsymbol{\ell}_2 = \boldsymbol{\ell}_1 -1$.

\end{itemize}

It follows that

\begin{center}
if $a_i \neq 1$, then $\boldsymbol{\ell}(R\!E(s_iw)) = \boldsymbol{\ell}(R\!E(w)) -1. \hspace{1cm} (b)$
\end{center}

\end{itemize}

Suppose, on the other hand, $c_{i-1} > c_i$. Recall that $w' = s_iw$. If $w'[i-1,c'_{i-1}]$ and $w'[i,c'_{i}]$ denote the nonzero entries of $w'$ on the rows $i-1$ and $i$, respectively, we have $w'[i-1,c'_{i-1}] = a_i$ and $w'[i,c'_{i}] = a_{i-1}$. For $w'$, we have $c'_{i-1} < c'_{i}$, in which case the preceding analysis would give:

\begin{center}

if $a_{i-1} = 1$, then $\boldsymbol{\ell}(R\!E(s_i(s_iw))) = \boldsymbol{\ell}(R\!E(s_iw)) + 1$,\\
if $a_{i-1} \neq 1$, then $\boldsymbol{\ell}(R\!E(s_i(s_iw))) = \boldsymbol{\ell}(R\!E(s_iw)) - 1$.

\end{center}

Hence, since $s_i^2 = 1$, we get the following:

\begin{center}

if $a_{i-1} = 1$, then $\boldsymbol{\ell}(R\!E(s_iw)) = \boldsymbol{\ell}(R\!E(w)) - 1. \hspace{1cm} (a')$,\\
if $a_{i-1} \neq 1$, then $\boldsymbol{\ell}(R\!E(s_iw)) = \boldsymbol{\ell}(R\!E(w)) + 1. \hspace{1cm} (b')$.

\end{center}

\underline{Case 2: Suppose $x = t_i$ for $0 \leq i \leq e-1$.}\\

Set $w' := t_iw$. By the left multiplication by $t_i$, we have that the last $n-2$ rows of $w$ and $w'$ are the same. Hence, by Definition \ref{DefREiw} and Lemma \ref{LemmaBlocks}, we have:\\
$R\!E_3(xw)R\!E_4(xw) \cdots R\!E_n(xw) = R\!E_3(w)R\!E_4(w) \cdots R\!E_n(w)$. In order to prove our property in this case, we should compare $\boldsymbol{\ell}_1 := \boldsymbol{\ell}(R\!E_2(w))$ and $\boldsymbol{\ell}_2 := \boldsymbol{\ell}(R\!E_2(xw))$.

\begin{itemize}

\item \underline{Consider the case where $c_1 < c_2$.}\\

Since $c_1 < c_2$, by Lemma \ref{LemmaBlocks}, the blocks $w_2$ and $w'_2$ are of the form:

$w_2 = \begin{pmatrix}

b_1 & 0\\
0 & a_2\\

\end{pmatrix}$ and $w'_2 = \begin{pmatrix}

0 & \zeta_e^{-i}a_2\\
\zeta_e^{i}b_{1} & 0\\ 

\end{pmatrix}$ with $b_{1}$ instead of $a_{1}$ since $a_{1}$ may change when applying Algorithm \ref{Algo1} if $c_{1} =1$.

\begin{itemize}
\item \underline{Suppose $a_2 = 1$,}\\
we have $b_1 = 1$ necessarily hence $\boldsymbol{\ell}_1 = 0$. Since $R\!E_2(xw) = \boldsymbol{t}_i$, we have $\boldsymbol{\ell}_2 = 1$. It follows that when $c_1 < c_2$,

\begin{center}
if $a_2 =1$, then $\boldsymbol{\ell}(R\!E(t_iw)) = \boldsymbol{\ell}(R\!E(w)) +1. \hspace{1cm} (c)$
\end{center}

\item \underline{Suppose $a_2 = \zeta_e^{k}$ with $1 \leq k \leq e-1$,} then $b_1 = \zeta_e^{-k}$.\\
We get $R\!E_2(w) = \boldsymbol{t}_k\boldsymbol{t}_0$. Thus, $\boldsymbol{\ell}_1 = 2$. We also get $R\!E_2(xw) = \mathbf{t}_{i-k}$. Thus, $\boldsymbol{\ell}_2 = 1$. It follows that when $c_1 < c_2$,

\begin{center}
if $a_2 \neq 1$, then $\boldsymbol{\ell}(R\!E(t_iw)) = \boldsymbol{\ell}(R\!E(w))  -1. \hspace{1cm} (d)$
\end{center}

\end{itemize}

\item \underline{Now, consider the case where $c_1 > c_2$.}\\

Since $c_1 > c_2$, by Lemma \ref{LemmaBlocks}, the blocks $w_2$ and $w'_2$ are of the form:

$w_2 = \begin{pmatrix}

0 & a_1\\
b_2 & 0\\

\end{pmatrix}$ and $w'_2 = \begin{pmatrix}

\zeta_e^{-i}b_2 & 0\\
0 & \zeta_e^{i}a_{1}\\ 

\end{pmatrix}$ with $b_{2}$ instead of $a_{2}$ since $a_{2}$ may change when applying Algorithm \ref{Algo1} if $c_{2} =1$.

\begin{itemize}

\item \underline{Suppose $a_1 \neq \zeta_e^{-i}$,} then $b_2 \neq \zeta_e^i$.\\
We have $\boldsymbol{\ell}_1 =1$ necessarily, and since $\zeta_e^{i}a_1 \neq 1$, we have $\boldsymbol{\ell}_2 = 2$. Hence when $c_1 > c_2$,

\begin{center}
if $a_1 \neq \zeta_e^{-i}$, then $\boldsymbol{\ell}(R\!E(t_iw)) = \boldsymbol{\ell}(R\!E(w)) +1. \hspace{1cm} (e)$
\end{center}

\item \underline{Suppose $a_1 = \zeta_e^{-i}$,}\\
we have $\boldsymbol{\ell}_1 = 1$ and $\boldsymbol{\ell}_2 = 0$. Hence when $c_1 > c_2$,

\begin{center}
if $a_1 = \zeta_e^{-i}$, then $\boldsymbol{\ell}(R\!E(t_iw)) = \boldsymbol{\ell}(R\!E(w)) -1. \hspace{1cm} (f)$
\end{center}

\end{itemize}

\end{itemize}

This finishes our proof.
\end{proof}

\begin{proposition}\label{PropREwRedExp}

Let $w$ be an element of $G(e,e,n)$. The word $R\!E(w)$ is a reduced expression over $\mathbf{X}$ of $w$.

\end{proposition}

\begin{proof}

We must prove that $\ell(w) = \boldsymbol{\ell}(R\!E(w))$.\\
Let $\mathbf{x}_1\mathbf{x}_2 \cdots \mathbf{x}_r$ be a reduced expression over $\mathbf{X}$ of $w$. Hence $\ell(w) = \boldsymbol{\ell}(\mathbf{x}_1\mathbf{x}_2 \cdots \mathbf{x}_r) = r$. Since $R\!E(w)$ is a word representative over $\mathbf{X}$ of $w$, we have $\boldsymbol{\ell}(R\!E(w)) \geq \boldsymbol{\ell}(\mathbf{x}_1\mathbf{x}_2 \cdots \mathbf{x}_r) = r$. We prove that $\boldsymbol{\ell}(R\!E(w)) \leq r$. Observe that we can write $w$ as $x_1x_2 \cdots x_r$ where $x_1,x_2, \cdots, x_r$ are the matrices of $G(e,e,n)$ corresponding to $\mathbf{x}_1,\mathbf{x}_2, \cdots, \mathbf{x}_r$.\\
By Proposition \ref{prop.lengthcomp}, we have: $\boldsymbol{\ell}(R\!E(w)) = \boldsymbol{\ell}(R\!E(x_1x_2 \cdots x_r)) \leq \boldsymbol{\ell}(R\!E(x_2x_3 \cdots x_r)) +1 \leq \boldsymbol{\ell}(R\!E(x_3 \cdots x_r)) +2 \leq \cdots \leq r$. Hence $\boldsymbol{\ell}(R\!E(w)) = r = \ell(w)$ and we are done.

\end{proof}

The following proposition is useful in the next chapter. Its proof is based on the proof of Proposition \ref{prop.lengthcomp}.

\begin{proposition}\label{PropLengthdecreas}

Let $w$ be an element of $G(e,e,n)$. Denote by $a_i$ the unique nonzero entry $w[i,c_i]$ on the row $i$ of $w$ where $1 \leq i, c_i \leq n$.

\begin{enumerate}

\item For $3 \leq i \leq n$, we have:
	\begin{enumerate}
		\item if $c_{i-1} < c_i$, then
				\begin{center}$\ell(s_{i}w) = \ell(w)-1$ if and only if $a_{i} \neq 1.$\end{center}
		\item if $c_{i-1} > c_i$, then
				\begin{center}$\ell(s_{i}w) = \ell(w)-1$ if and only if $a_{i-1} =1.$\end{center}
	\end{enumerate}
\item If $c_1 < c_2$, then $\forall\ 0 \leq k \leq e-1$, we have
				\begin{center}$\ell(t_k w) = \ell(w)-1$ if and only if $a_2 \neq 1.$\end{center}
\item If $c_1 > c_2$, then $\forall\ 0 \leq k \leq e-1$, we have
				\begin{center}$\ell(t_k w) = \ell(w)-1$ if and only if $a_1 = \zeta_{e}^{-k}.$\end{center}
				
\end{enumerate}

\end{proposition}

\begin{proof}

The claim $1(a)$ is deduced from $(a)$ and $(b)$, $1(b)$ is deduced from $(a')$ and $(b')$, $2$ is deduced from $(c)$ and $(d)$, and $3$ is deduced from $(e)$ and $(f)$ where $(a)$, $(b)$, $(a')$, $(b')$, $(c)$, $(d)$, $(e)$, and $(f)$ are given in the proof of Proposition \ref{prop.lengthcomp}.

\end{proof}

\begin{remark}

Proposition \ref{PropLengthdecreas} will be useful to implement in Appendix A the interval Garside monoids that we will construct in the next chapter.

\end{remark}

\section{The general case of $G(de,e,n)$}

In this section, we generalize the geodesic normal forms to all the general series of complex reflection groups $G(de,e,n)$ where $d > 1$, $e > 1$, and $n \geq 2$. Geodesic normal forms for $G(d,1,n)$ are studied in the last subsection. We start by providing the presentation of Corran-Lee-Lee \cite{CorranLeeLee} of $G(de,e,n)$.

\subsection{Presentation for $G(de,e,n)$}

Recall that the complex reflection group $G(de,e,n)$ is the group of monomial matrices whose nonzero entries are $de$-th roots of unity and their product is a $d$-th root of unity.
There exists a presentation of the complex reflection group $G(de,e,n)$ given in \cite{CorranLeeLee} for $d > 1$, $e \geq 1$, and $n \geq 2$.

\begin{definition}\label{DefCorranLeePresG(de,e,n)}

The complex reflection group $G(de,e,n)$ is defined by a presentation with set of generators: $\mathbf{X} = \{ \mathbf{z}\} \cup \{ \mathbf{t}_i \ | \ i \in \mathbb{Z}/de\mathbb{Z}\} \cup \{\mathbf{s}_3, \mathbf{s}_4, \cdots, \mathbf{s}_n \}$ and relations as follows.

\begin{enumerate}

\item $\mathbf{z} \mathbf{t}_i = \mathbf{t}_{i-e} \mathbf{z}$ for $i \in \mathbb{Z}/de\mathbb{Z}$,
\item $\mathbf{z}  \mathbf{s}_j=\mathbf{s}_j \mathbf{z}$ for $3 \leq j \leq n$,
\item $\mathbf{t}_i \mathbf{t}_{i-1} = \mathbf{t}_j \mathbf{t}_{j-1}$ for $i, j \in \mathbb{Z}/de\mathbb{Z}$,
\item $\mathbf{t}_i \mathbf{s}_3 \mathbf{t}_i = \mathbf{s}_3 \mathbf{t}_i \mathbf{s}_3$ for $i \in \mathbb{Z}/de\mathbb{Z}$,
\item $\mathbf{s}_j \mathbf{t}_i = \mathbf{t}_i \mathbf{s}_j$ for $i \in \mathbb{Z}/de\mathbb{Z}$ and $4 \leq j \leq n$,
\item $\mathbf{s}_i \mathbf{s}_{i+1} \mathbf{s}_i = \mathbf{s}_{i+1} \mathbf{s}_i \mathbf{s}_{i+1}$ for $3 \leq i \leq n-1$,
\item $\mathbf{s}_i \mathbf{s}_j = \mathbf{s}_j \mathbf{s}_i$ for $|i-j| > 1$, and
\item $\mathbf{z}^d=1$, $\mathbf{t}_i^2=1$ for $i \in \mathbb{Z}/de\mathbb{Z}$, and $\mathbf{s}_j^2=1$ for $3 \leq j \leq n$.

\end{enumerate}

\end{definition}

The generators of this presentation correspond to the following $n \times n$ matrices.
The generator $\mathbf{t}_i$ is represented by the matrix $t_i =
\begin{pmatrix}

0 & \zeta_{de}^{-i} & 0\\
\zeta_{de}^{i} & 0 & 0\\
0 & 0 & I_{n-2}\\

\end{pmatrix}$
for $i \in \mathbb{Z}/de\mathbb{Z}$, $\mathbf{z}$ by the diagonal matrix  $z = Diag(\zeta_d,1,\cdots,1)$ where 
$\zeta_d = exp(2i \pi / d)$, 
and $\mathbf{s}_j$ by the transposition matrix $s_j = (j-1,j)$ for $3 \leq j \leq n$. To avoid confusion, we use normal letters for matrices and bold letters for words over $\mathbf{X}$. Denote by $X$ the set $\{z,t_0,t_1, \cdots, t_{de-1},s_3, \cdots, s_n\}$.\\

This presentation can be described by the following diagram. The dashed circle describes Relation $3$ of Definition \ref{DefCorranLeePresG(de,e,n)}. The curved arrow below $\mathbf{z}$ describes Relation $1$. The other edges used to describe all the other relations follow the standard conventions for Coxeter groups.

\begin{figure}[H]
\begin{center}
\begin{tikzpicture}[yscale=0.8,xscale=1,rotate=30]

\draw[thick,dashed] (0,0) ellipse (2cm and 1cm);

\node[draw, shape=circle, fill=white, label=above:\begin{small}$\mathbf{z}$\end{small}] (z) at (0,0) {\begin{tiny}$d$\end{tiny}};
\node[draw, shape=circle, fill=white, label=above:\begin{small}$\mathbf{t}_0$\end{small}] (t0) at (0,-1) {\begin{tiny}$2$\end{tiny}};
\node[draw, shape=circle, fill=white, label=above:\begin{small}$\mathbf{t}_1$\end{small}] (t1) at (1,-0.8) {\begin{tiny}$2$\end{tiny}};
\node[draw, shape=circle, fill=white, label=right:\begin{small}$\mathbf{t}_2$\end{small}] (t2) at (2,0) {\begin{tiny}$2$\end{tiny}};
\node[draw, shape=circle, fill=white, label=above:$\mathbf{t}_i$] (ti) at (0,1) {\begin{tiny}$2$\end{tiny}};
\node[draw, shape=circle, fill=white, label=above:\begin{small}$\mathbf{t}_{de-1}$\end{small}] (te-1) at (-1,-0.8) {\begin{tiny}$2$\end{tiny}};

\draw[->,bend left=45] (-0.3,0.26) to (0.3,0.26);
\draw[thick,-] (0,-2) arc (-180:-90:3);

\node[draw, shape=circle, fill=white, label=below left:$\mathbf{s}_3$] (s3) at (0,-2) {\begin{tiny}$2$\end{tiny}};

\draw[thick,-] (t0) to (s3);
\draw[thick,-,bend left] (t1) to (s3);
\draw[thick,-,bend left] (t2) to (s3);
\draw[thick,-,bend left] (s3) to (te-1);

\node[draw, shape=circle, fill=white, label=below:$\mathbf{s}_4$] (s4) at (0.15,-3) {\begin{tiny}$2$\end{tiny}};
\node[draw, shape=circle, fill=white, label=below:$\mathbf{s}_{n-1}$] (sn-1) at (2.2,-4.9) {\begin{tiny}$2$\end{tiny}};
\node[draw, shape=circle, fill=white, label=right:$\mathbf{s}_{n}$] (sn) at (3,-5) {\begin{tiny}$2$\end{tiny}};

\node[fill=white] () at (1,-4.285) {$\cdots$};

\end{tikzpicture}
\end{center}
\caption{\mbox{Diagram for the presentation of Corran-Lee-Lee of $G(de,e,n)$.}}\label{FigPresG(de,e,n)CorranLeeLee}
\end{figure}

\begin{proposition}\label{PropPresB(d,1,n)}

Let $e=1$. The presentation provided in Definition \ref{DefCorranLeePresG(de,e,n)} is equivalent to the classical presentation of the complex reflection group $G(d,1,n)$ that can be described by the following diagram.

\begin{figure}[H]

\begin{center}
\begin{tikzpicture}

\node[draw, shape=circle, label=above:$\mathbf{z}$] (1) at (0,0) {$d$};
\node[draw, shape=circle,label=above:$\mathbf{s}_2$] (2) at (2,0) {$2$};
\node[draw, shape=circle,label=above:$\mathbf{s}_3$] (3) at (4,0) {$2$};
\node[draw, shape=circle,label=above:$\mathbf{s}_{n-1}$] (n-1) at (6,0) {$2$};
\node[draw,shape=circle,label=above:$\mathbf{s}_n$] (n) at (8,0) {$2$};

\draw[thick,-,double] (1) to (2);
\draw[thick,-] (2) to (3);
\draw[dashed,-,thick] (3) to (n-1);
\draw[thick,-] (n-1) to (n);

\end{tikzpicture}
\end{center}

\caption{Diagram for the presentation of $G(d,1,n)$.}\label{PresofBMRofGd1n}
\end{figure}

\end{proposition}

\begin{proof}

Let $e=1$. Relation $1$ of Definition \ref{DefCorranLeePresG(de,e,n)} becomes $\mathbf{z} \mathbf{t}_1=\mathbf{t}_0 \mathbf{z}$, that is $\mathbf{t}_1 = \mathbf{z}^{-1} \mathbf{t}_0 \mathbf{z}$. Also by Relation $3$ of  Definition \ref{DefCorranLeePresG(de,e,n)}, we have $\mathbf{t}_k=\mathbf{z}^{-k}\mathbf{t}_0\mathbf{z}^k$ for $1 \leq k \leq d-1$. If we remove $\mathbf{t}_1, \cdots, \mathbf{t}_{d-1}$ from the set of generators and replace every occurrence of $\mathbf{t}_k$ in the defining relations with $\mathbf{z}^{-k} \mathbf{t}_0 \mathbf{z}^{k}$ for $1 \leq k \leq d-1$, we recover the classical presentation of the complex reflection group $G(d,1,n)$.

\end{proof}

\begin{remark}

For $d = 2$, the presentation described by the diagram of Figure \ref{PresofBMRofGd1n} is the classical presentation of the Coxeter group of type $B_n$. 

\end{remark}

\begin{remark}\label{RemarkPresB(de,e,n)}

In \cite{CorranLeeLee}, it is shown that if we remove Relations 8 of Definition \ref{DefCorranLeePresG(de,e,n)} from the presentation of $G(de,e,n)$, we get a presentation of the complex braid group $B(de,e,n)$. It is also shown that this presentation provides a quasi-Garside structure for $B(de,e,n)$. The notion of Garside (and quasi-Garside) structures will be developed in the next chapter. 

\end{remark}

We set Convention \ref{ConventionDecreaseIncreaseIndex} in the case of the presentation of $G(de,e,n)$ given in Definition \ref{DefCorranLeePresG(de,e,n)} and in the case of the presentation of $G(d,1,n)$ described by the diagram of Figure \ref{PresofBMRofGd1n}. We also set the convention that $\mathbf{z}^0$ is the empty word. These conventions will be helpful to provide all the possible cases of the words that appear in Definitions \ref{DefREiwG(de,e,n)} and \ref{DefinitionREG(d,1,n)} below. Also they ensure that Algorithms \ref{Algo2} and \ref{Algo3} that we will introduce in the next subsections work properly.

\subsection{Minimal word representatives}

Consider the complex reflection group $G(de,e,n)$ with $d > 1$, $e > 1$, and $n \geq 2$. Our aim is to represent each element of $G(de,e,n)$ by a reduced word over $\mathbf{X}$, where $\mathbf{X}$ is the set of the generators of the presentation of Corran-Lee-Lee of $G(de,e,n)$, see Definition \ref{DefCorranLeePresG(de,e,n)}. Recall that $X$ denotes the set of the matrices in $G(de,e,n)$ that  correspond to the elements of $\mathbf{X}$.\\

We introduce Algorithm \ref{Algo2} below (see next page) that produces a word $R\!E(w)$ over $\mathbf{X}$ for a given matrix $w$ in $G(de,e,n)$. For $d=1$, this algorithm is the same as Algorithm \ref{Algo1} that corresponds to the case of $G(e,e,n)$. We will also have that the output $R\!E(w)$ of Algorithm \ref{Algo2} is a reduced word representative of $w \in G(de,e,n)$ over $\mathbf{X}$.\\

Let $w_n := w \in G(de,e,n)$. For $i$ from $n$ to $2$, the $i$-th step of Algorithm\! \ref{Algo2} transforms the block diagonal matrix $\left(
\begin{array}{c|c}
w_i & 0 \\
\hline
0 & I_{n-i}
\end{array}
\right)$ into a block diagonal matrix $\left(
\begin{array}{c|c}
w_{i-1} & 0 \\
\hline
0 & I_{n-i+1}
\end{array}
\right) \in G(de,e,n)$ in the same way as Algorithm \ref{Algo1}. We finally get $w_1 = \zeta_d^k$ for some $0 \leq k \leq d-1$, where $\zeta_d^k$ is equal to the product of the nonzero entries of $w$. By multiplying $\left(
\begin{array}{c|c}
w_{1} & 0 \\
\hline
0 & I_{n-1}
\end{array}
\right)$ on the right by $z^{-k}$, we get the identity matrix $I_n$.

\begin{algorithm}[]\label{Algo2}

\SetKwInOut{Input}{Input}\SetKwInOut{Output}{Output}

\noindent\rule{12cm}{0.5pt}

\Input{$w$, a matrix in $G(de,e,n)$ with $d > 1$, $e > 1$, and $n \geq 2$.}
\Output{$R\!E(w)$, a word over $\mathbf{X}$.}

\noindent\rule{12cm}{0.5pt}

\textbf{Local variables}: $w'$, $R\!E(w)$,  $i$, $U$, $V$, $c$, $k$.

\noindent\rule{12cm}{0.5pt}

\textbf{Initialisation}:
$U:=[1,\zeta_{de},\zeta_{de}^2,...,\zeta_{de}^{e-1}]$, $V:=[1,\zeta_d,\zeta_d^2, \cdots, \zeta_d^{d-1}]$, $s_2 := t_0$, $\mathbf{s}_2 := \mathbf{t}_0$, $R\!E(w) := \varepsilon$: the empty word, $w' := w$.

\noindent\rule{12cm}{0.5pt}

\For{$i$ \textbf{from} $n$ \textbf{down to} $2$} {
	$c:=1$; $k:=0$; \\
	\While{$w'[i,c] = 0$}{$c :=c+1$\; 
	}
	 \textit{\#Then $w'[i,c]$ is the root of unity on the row $i$}\;
	\While{$U[k+1]\neq w'[i,c] $}{$k :=k+1$\;
	}
	\textit{\#Then $w'[i,c] = \zeta_{de}^k$.}\\

		\If{$k \neq 0$}{
		$w' := w' s_{c} s_{c-1} \cdots s_{3} s_{2} t_{k}$; \textit{\#Then $w'[i,2] =1$}\;
		$R\!E(w) := \mathbf{t}_k \mathbf{s}_2 \mathbf{s}_3 \cdots \mathbf{s}_c R\!E(w)$\;
		$c:=2$\;
	}	
	$w' := w' s_{c+1} \cdots s_{i-1}s_{i}$; \textit{\#Then $w'[i,i] = 1$}\;
	$R\!E(w) := \mathbf{s}_i \mathbf{s}_{i-1} \cdots \mathbf{s}_{c+1} R\!E(w)$\;
}
$k:=0$\;
\While {$V[k+1] \neq w'[1,1]$}{$k:=k+1$\;}
\textit{\#Then $w'[1,1] = \zeta_d^k$}\;
$w':=w' z^{-k}$; \textit{\#Then $w'=I_n$}\;
\If{$k \neq 0$}{$R\!E(w) = \mathbf{z}^{k} R\!E(w)$\;}
\textbf{Return} $R\!E(w)$;

\noindent\rule{12cm}{0.5pt}

\caption{A word over $\mathbf{X}$ corresponding to a matrix $w \in G(de,e,n)$.}
\end{algorithm}

\begin{example}\label{ExampleAlgo1}

We apply Algorithm\! \ref{Algo2} to $w := \begin{pmatrix}

\zeta_9 & 0 & 0 & 0\\
0 & 0 & 1 & 0\\
0 & 0 & 0 & \zeta_9\\
0 & \zeta_9 & 0 & 0\\

\end{pmatrix}$ $\in G(9,3,4)$.\\
Step $1$ $(i=4, k=0, c=1)$: $w':= w s_2=\begin{pmatrix}

0 & \zeta_9 & 0 & 0\\
0 & 0 & 1 & 0\\
0 & 0 & 0 & \zeta_9\\
\zeta_9 & 0 & 0 & 0\\

\end{pmatrix}$, then $w':= w' t_1 = \begin{pmatrix}

\zeta_9^2 & 0 & 0 & 0\\
0 & 0 & 1 & 0\\
0 & 0 & 0 & \zeta_9\\
0 & 1 & 0 & 0\\

\end{pmatrix}$, then $w':= w' s_3 s_4 = \begin{pmatrix}

\zeta_9^2 & 0 & 0 & 0\\
0 & 1 & 0 & 0\\
0 & 0 & \boxed{\zeta_9} & 0\\
0 & 0 & 0 & \mathbf{1}\\

\end{pmatrix}$.\\
Step $2$ $(i=3, k=1, c=3)$: $w' := w' s_3 s_2 = \begin{pmatrix}

0 & \zeta_9^2 & 0 & 0\\
0 & 0 & 1 & 0\\
\zeta_9 & 0 & 0 & 0\\
0 & 0 & 0 & 1\\

\end{pmatrix}$, then $w' := w' t_1 = \begin{pmatrix}

\zeta_9^{3} & 0 & 0 & 0\\
0 & 0 & 1 & 0\\
0 & 1 & 0 & 0\\
0 & 0 & 0 & 1\\

\end{pmatrix}$, then $w':= w' s_3 = \begin{pmatrix}

\zeta_9^3 & 0 & 0 & 0\\
0 & \boxed{1} & 0 & 0\\
0 & 0 & \mathbf{1} & 0\\
0 & 0 & 0 & 1\\

\end{pmatrix}$.\\
Step $3$ $(i=2, k=0, c=2)$: $w' = \begin{pmatrix}

\boxed{\zeta_9^3} & 0 & 0 & 0\\
0 & \mathbf{1} & 0 & 0\\
0 & 0 & 1 & 0\\
0 & 0 & 0 & 1\\

\end{pmatrix} = \begin{pmatrix}

\boxed{\zeta_3} & 0 & 0 & 0\\
0 & \mathbf{1} & 0 & 0\\
0 & 0 & 1 & 0\\
0 & 0 & 0 & 1\\

\end{pmatrix}$.\\
Step $4$ $(k=1)$: $w' := w' z^{-1} = I_4$.\\
Hence $R\!E(w) = \mathbf{z} \mathbf{s}_3 \mathbf{t}_1 \mathbf{s}_2 \mathbf{s}_3 \mathbf{s}_4 \mathbf{s}_3 \mathbf{t}_1 \mathbf{s}_2 = \mathbf{z} \mathbf{s}_3 \mathbf{t}_1 \mathbf{t}_0 \mathbf{s}_3 \mathbf{s}_4 \mathbf{s}_3 \mathbf{t}_1 \mathbf{t}_0$ \emph{(since $\mathbf{s}_2 = \mathbf{t}_0$)}.

\end{example}

The next lemma follows directly from Algorithm \ref{Algo2}.

\begin{lemma}\label{LemmaBlocksG(de,e,n)}

For $2 \leq i \leq n$, suppose $w_i[i,c] \neq 0$. The block $w_{i-1}$ is obtained by
\begin{itemize}

\item removing the row $i$ and the column $c$ from $w_i$, then by

\item multiplying the first column of the new matrix by $w_i[i,c]$.

\end{itemize}

Moreover, if we denote by $a_i$ the unique nonzero entry on the row $i$ of $w$, we have $w_1 = \displaystyle\prod_{i=1}^{n} a_i = \zeta_d^k$ for $0 \leq k \leq d-1$.

\end{lemma}

\begin{definition}\label{DefREiwG(de,e,n)}

Let $1 \leq i \leq n$. Let $w_i[i,c] \neq 0$ for $1 \leq c \leq i$.

\begin{itemize}

\item If $w_1 = \zeta_d^k$ with $0 \leq k \leq d-1$, we define $R\!E_{1}(w)$ to be the word $\mathbf{z}^{k}$.	
\item If $w_{i}[i,c] =1$, we define $R\!E_{i}(w)$ to be the word\\
$\mathbf{s}_i \mathbf{s}_{i-1} \cdots \mathbf{s}_{c+1}$ (decreasing-index expression).
\item If $w_{i}[i,c] = \zeta_{de}^{k}$ with $k \neq 0$, we define $R\!E_{i}(w)$ to be the word\\
\begin{tabular}{ll}
			$\mathbf{s}_i \cdots \mathbf{s}_3 \mathbf{t}_k$ & if $c=1$,\\
			$\mathbf{s}_i \cdots \mathbf{s}_3 \mathbf{t}_k \mathbf{t}_0$  & if $c=2$,\\
			$\mathbf{s}_i \cdots \mathbf{s}_3 \mathbf{t}_k \mathbf{t}_0 \mathbf{s}_3 \cdots \mathbf{s}_c$ & if $c \geq 3$.\\
			
\end{tabular}

\end{itemize}
\end{definition}

The output $R\!E(w)$ of Algorithm \ref{Algo2} is a concatenation of the words $R\!E_1(w)$, $R\!E_2(w)$, $\cdots$, and $R\!E_n(w)$ obtained at each step $i$ from $n$ to $1$ of Algorithm \ref{Algo2}. Then we have $R\!E(w) = R\!E_1(w) R\!E_2(w)  \cdots R\!E_n(w)$.

\begin{example}

If $w$ is defined as in Example \ref{ExampleAlgo1}, we have\\
$R\!E(w) = \underset{R\!E_{1}(w)}{\underbrace{\mathbf{z}}} \hspace{0.2cm} \underset{R\!E_{3}(w)}{\underbrace{\mathbf{s}_3 \mathbf{t}_1 \mathbf{t}_0 \mathbf{s}_3}} \hspace{0.2cm} \underset{R\!E_{4}(w)}{\underbrace{ \mathbf{s}_4 \mathbf{s}_3 \mathbf{t}_1 \mathbf{t}_0}}$, where $R\!E_2(w)$ is the empty word.

\end{example}

\begin{proposition}\label{PropWordRepG(de,e,n)}

Let $w \in G(de,e,n)$. The word $R\!E(w)$ given by Algorithm \ref{Algo2} is a word representative over $\mathbf{X}$ of $w \in G(de,e,n)$.

\end{proposition}

\begin{proof}

Let $w \in G(de,e,n)$ such that the product of all the nonzero entries of $w$ is equal to $\zeta_d^k$ for some $0 \leq k \leq d-1$.  Algorithm \ref{Algo2} transforms the matrix $w$ into $I_n$ by multiplying it on the right by elements of $X$. We get $w x_1 \cdots x_{r-1}x_r = I_n$, where $x_1, \cdots, x_{r-1}$ are elements of $X \setminus \{z\}$ and $x_r = z^{-k}$. Hence $w = x_r^{-1} x_{r-1}^{-1} \cdots x_1^{-1} = z^{k} x_{r-1} \cdots x_1$ since $x_i^2 = 1$ for $1 \leq i \leq r-1$. The output $R\!E(w)$ of Algorithm \ref{Algo2} is $R\!E(w) = \mathbf{z}^k \mathbf{x}_{r-1} \cdots \mathbf{x}_1$. Hence it is a word representative over $\mathbf{X}$ of $w \in G(de,e,n)$.  

\end{proof}

The following proposition is similar to Proposition \ref{prop.lengthcomp}. It enables us to prove that the output of Algorithm \ref{Algo2} is a reduced expression over $\mathbf{X}$ of a given element $w \in G(de,e,n)$.

\begin{proposition}\label{prop.lengthcompG(de,e,n)}

Let $w$ be an element of $G(de,e,n)$. For all $x \in X$, we have $$\boldsymbol{\ell}(R\!E(xw)) \leq \boldsymbol{\ell}(R\!E(w)) + 1.$$

\end{proposition}

\begin{proof}

For $1 \leq i \leq n$, there exists a unique $c_i$ such that $w[i,c_i] \neq 0$. We denote $w[i,c_i]$ by $a_i$. We have  $\displaystyle\prod_{i=1}^{n} a_i = \zeta_d^k$ for some $0 \leq k \leq d-1$. \\

\underline{Case 1: Suppose $x = s_i$ for $3 \leq i \leq n$.}\\

This case is done in the same way as in the proof of Proposition \ref{prop.lengthcomp}. We get $(a)$, $(b)$, $(a')$, and $(b')$, see Case 1 in the proof of Proposition \ref{prop.lengthcomp}.\\

\underline{Case 2: Suppose $x = t_i$ for $0 \leq i \leq de-1$.}\\

Set $w' := t_iw$. By the left multiplication by $t_i$, we have that the last $n-2$ rows of $w$ and $w'$ are the same. Hence, by Definition \ref{DefREiwG(de,e,n)} and Lemma \ref{LemmaBlocksG(de,e,n)}, we have:\\
$R\!E_3(xw)R\!E_4(xw) \cdots R\!E_n(xw) = R\!E_3(w)R\!E_4(w) \cdots R\!E_n(w)$. In order to prove our property in this case, we should compare $\boldsymbol{\ell}_1 := \boldsymbol{\ell}(R\!E_1(w)R\!E_2(w))$ and\\ $\boldsymbol{\ell}_2 := \boldsymbol{\ell}(R\!E_1(xw)R\!E_2(xw))$.

\begin{itemize}

\item \underline{Consider the case where $c_1 < c_2$.}\\

Since $c_1 < c_2$, by Lemma \ref{LemmaBlocksG(de,e,n)}, the blocks $w_2$ and $w'_2$ are of the form:

$w_2 = \begin{pmatrix}

b_1 & 0\\
0 & a_2\\

\end{pmatrix}$ and $w'_2 = \begin{pmatrix}

0 & \zeta_{de}^{-i}a_2\\
\zeta_{de}^{i}b_{1} & 0\\ 

\end{pmatrix}$ with $b_{1}$ instead of $a_{1}$ since $a_{1}$ may change when applying Algorithm \ref{Algo2} if $c_{1} =1$.

\begin{itemize}
\item \underline{Suppose $a_2 = 1$,}\\
we have $b_1 = \zeta_d^k$ hence $\boldsymbol{\ell}_1 = k$. We also have $R\!E_2(xw) = \boldsymbol{t}_{i+ke}$ and $R\!E_1(xw) = \mathbf{z}^k$. Hence we get $\boldsymbol{\ell}_2 = k+1$. It follows that when $c_1 < c_2$,

\begin{center}
if $a_2 =1$, then $\boldsymbol{\ell}(R\!E(t_iw)) = \boldsymbol{\ell}(R\!E(w)) +1. \hspace{1cm} (c)$
\end{center}

\item \underline{Suppose $a_2 = \zeta_{de}^{k'}$ with $1 \leq k' \leq de-1$,}\\
we have $b_1 = \zeta_{de}^{ke-k'}$. We get $R\!E_2(w) = \boldsymbol{t}_{k'}\boldsymbol{t}_0$ and $R\!E_1(w) = \mathbf{z}^k$. Thus, $\boldsymbol{\ell}_1 = k+2$. We also get $R\!E_2(xw) = \boldsymbol{t}_{ke+i-k'}$ and $R\!E_1(xw) = \mathbf{z}^k$. Thus, $\boldsymbol{\ell}_2 = k+1$. It follows that when $c_1 < c_2$,

\begin{center}
if $a_2 \neq 1$, then $\boldsymbol{\ell}(R\!E(t_iw)) = \boldsymbol{\ell}(R\!E(w))  -1. \hspace{1cm} (d)$
\end{center}

\end{itemize}

\item \underline{Now, consider the case where $c_1 > c_2$.}\\

Since $c_1 > c_2$, by Lemma \ref{LemmaBlocksG(de,e,n)}, the blocks $w_2$ and $w'_2$ are of the form:

$w_2 = \begin{pmatrix}

0 & a_1\\
b_2 & 0\\

\end{pmatrix}$ and $w'_2 = \begin{pmatrix}

\zeta_{de}^{-i}b_2 & 0\\
0 & \zeta_{de}^{i}a_{1}\\ 

\end{pmatrix}$ with $b_{2}$ instead of $a_{2}$ since $a_{2}$ may change when applying Algorithm \ref{Algo2} if $c_{2} =1$.

\begin{itemize}

\item \underline{Suppose $a_1 \neq \zeta_{de}^{-i}$,}\\
we have $\boldsymbol{\ell}_1 = k+1$, and since $\zeta_{de}^{i}a_1 \neq 1$, we have $\boldsymbol{\ell}_2 = k+2$. Hence when $c_1 > c_2$,

\begin{center}
if $a_1 \neq \zeta_{de}^{-i}$, then $\boldsymbol{\ell}(R\!E(t_iw)) = \boldsymbol{\ell}(R\!E(w)) +1. \hspace{1cm} (e)$
\end{center}

\item \underline{Suppose $a_1 = \zeta_{de}^{-i}$,} we have $b_2 = \zeta_{de}^{i+ek}$.\\
We get $\boldsymbol{\ell}_1 = k+1$ and $\boldsymbol{\ell}_2 = k$. Hence when $c_1 > c_2$,

\begin{center}
if $a_1 = \zeta_{de}^{-i}$, then $\boldsymbol{\ell}(R\!E(t_iw)) = \boldsymbol{\ell}(R\!E(w)) -1. \hspace{1cm} (f)$
\end{center}

\end{itemize}

\end{itemize}

\underline{Case 3: Suppose $x = z$.}\\

Set $w' := zw$. By the left multiplication by $z$, we have that the last $n-1$ rows of $w$ and $w'$ are the same. Hence, by Definition \ref{DefREiwG(de,e,n)} and Lemma \ref{LemmaBlocksG(de,e,n)}, we have:\\
$R\!E_2(xw)R\!E_3(xw) \cdots R\!E_n(xw) = R\!E_2(w)R\!E_3(w) \cdots R\!E_n(w)$. In order to prove our property in this case, we should compare $\boldsymbol{\ell}_1 := \boldsymbol{\ell}(R\!E_1(w))$ and $\boldsymbol{\ell}_2 := \boldsymbol{\ell}(R\!E_1(xw))$.\\

We get $w_1$ is equal to $b_1$ and $w'_1 = \zeta_d b_1$ with $b_{1}$ instead of $a_{1}$ since $a_{1}$ may change when applying Algorithm \ref{Algo2} if $c_{1} =1$. We have $b_1 = \displaystyle\prod_{i=1}^{n} a_i = \zeta_d^k$ for some $0 \leq k \leq d-1$. Hence if $k \neq d-1$, we get $\boldsymbol{\ell}_1 = k$ and $\boldsymbol{\ell}_2 = k+1$ and if $k = d-1$, we get $\boldsymbol{\ell}_1 = d-1$ and $\boldsymbol{\ell}_2 = 0$. It follows that

\begin{center}
$\boldsymbol{\ell}(R\!E(zw)) \leq \boldsymbol{\ell}(R\!E(w)) + 1. \hspace{1cm} (g)$
\end{center}

\end{proof}

By using Proposition \ref{prop.lengthcompG(de,e,n)}, we apply the argument of the proof of Proposition \ref{PropREwRedExp} and deduce that $R\!E(w)$ is a reduced expression over $\mathbf{X}$ of $w \in G(de,e,n)$. Hence Algorithm \ref{Algo2} produces geodesic normal forms for $G(de,e,n)$. Proposition \ref{PropLengthdecreas} is also valid for the case of $G(de,e,n)$, where we replace $e$ by $de$. It summarizes the proof of Proposition \ref{prop.lengthcompG(de,e,n)}.

\subsection{The case of $G(d,1,n)$}

We provide similar constructions for the case of $G(d,1,n)$ for $d > 1$ and $n \geq 2$. We recall the diagram of the presentation of $G(d,1,n)$. 

\begin{figure}[H]

\begin{center}
\begin{tikzpicture}

\node[draw, shape=circle, label=above:$\mathbf{z}$] (1) at (0,0) {$d$};
\node[draw, shape=circle,label=above:$\mathbf{s}_2$] (2) at (2,0) {$2$};
\node[draw, shape=circle,label=above:$\mathbf{s}_3$] (3) at (4,0) {$2$};
\node[draw, shape=circle,label=above:$\mathbf{s}_{n-1}$] (n-1) at (6,0) {$2$};
\node[draw,shape=circle,label=above:$\mathbf{s}_n$] (n) at (8,0) {$2$};

\draw[thick,-,double] (1) to (2);
\draw[thick,-] (2) to (3);
\draw[dashed,-,thick] (3) to (n-1);
\draw[thick,-] (n-1) to (n);

\end{tikzpicture}
\end{center}

\end{figure}

The set of the generators is denoted by $\mathbf{X} = \{\mathbf{z}, \mathbf{s}_2, \cdots, \mathbf{s}_n\}$. The generator $\mathbf{z}$ corresponds to the matrix $z := Diag(\zeta_d,1,\cdots,1)$ in $G(d,1,n)$ with $\zeta_d = exp(2i\pi / d)$ and $\mathbf{s}_j$ corresponds to the transposition matrix $s_j := (j-1,j)$ for $2 \leq j \leq n$. Denote by $X = \{z,s_2,s_3, \cdots, s_n\}$ the set of these matrices.\\

We define Algorithm \ref{Algo3} below (see next page) in the same way as Algorithms \ref{Algo1} and \ref{Algo2}. It produces a word $R\!E(w)$ for each matrix $w$ of $G(d,1,n)$. We explain the steps of the algorithm. Let $w_n := w \in G(d,1,n)$. For $i$ from $n$ to $1$, the $i$-th step of the algorithm transforms the block diagonal matrix $\left(
\begin{array}{c|c}
w_i & 0 \\
\hline
0 & I_{n-i}
\end{array}
\right)$ into a block diagonal matrix $\left(
\begin{array}{c|c}
w_{i-1} & 0 \\
\hline
0 & I_{n-i+1}
\end{array}
\right) \in G(d,1,n)$. Let $w_i[i,c] \neq 0$ be the nonzero coefficient on the row $i$ of $w_i$. If $w_i[i,c] =1$, we shift it into the diagonal position $[i,i]$ by right multiplication by transpositions. If $w_i[i,c] = \zeta_d^k$ with $k \geq 1$, we shift it into position $[i,1]$ by right multiplication by transpositions, followed by a right multiplication by $z^{-k}$, then we shift the $1$ obtained in position $[i,1]$ into the diagonal position $[i,i]$ by right multiplication by transpositions. Let us illustrate these operations by the following example.\\

\begin{algorithm}[]\label{Algo3}

\SetKwInOut{Input}{Input}\SetKwInOut{Output}{Output}

\noindent\rule{12cm}{0.5pt}

\Input{$w$, a matrix in $G(d,1,n)$, with $d > 1$ and $n \geq 2$.}
\Output{$R\!E(w)$, a word over $\mathbf{X}$.}

\noindent\rule{12cm}{0.5pt}

\textbf{Local variables}: $w'$, $R\!E(w)$,  $i$, $U$, $c$, $k$.

\noindent\rule{12cm}{0.5pt}

\textbf{Initialisation}:
$U:=[1,\zeta_d,\zeta_{d}^2, \cdots ,\zeta_{d}^{d-1}]$, $R\!E(w) := \varepsilon$: the empty word, $w' := w$.

\noindent\rule{12cm}{0.5pt}

\For{$i$ \textbf{from} $n$ \textbf{down to} $1$} {
	$c:=1$; $k:=0$; \\
	\While{$w'[i,c] = 0$}{$c :=c+1$\; 
	}
	 \textit{\#Then $w'[i,c]$ is the root of unity on the row $i$}\;
	\While{$U[k+1]\neq w'[i,c] $}{$k :=k+1$\;
	}
	\textit{\#Then $w'[i,c] = \zeta_{d}^k$.}\\

		\If{$k \neq 0$}{
		$w' := w's_{c}s_{c-1} \cdots s_{3}s_{2}z^{-k}$; \textit{\#Then $w'[i,2] =1$}\;
		$R\!E(w) := \mathbf{z}^{k}\mathbf{s}_2\mathbf{s}_3 \cdots \mathbf{s}_c R\!E(w)$\;
		$c:=1$\;
	}	
	$w' := w's_{c+1} \cdots s_{i-1} s_{i}$; \textit{\#Then $w'[i,i] = 1$}\;
	$R\!E(w) := \mathbf{s}_i \mathbf{s}_{i-1} \cdots \mathbf{s}_{c+1} R\!E(w)$\;
}
\textbf{Return} $R\!E(w)$;

\noindent\rule{12cm}{0.5pt}

\caption{A word over $\mathbf{X}$ corresponding to an element $w \in G(d,1,n)$.}
\end{algorithm}

\begin{example}\label{ExampAlgoG(d,1,n)}

Let $w := \begin{pmatrix}

0 & \zeta_3 & 0\\
0 & 0 & \zeta_3^2\\
\zeta_3^2 & 0 & 0

\end{pmatrix}$ $\in G(3,1,3)$.\\
Step $1$ $(i=3, k=2, c=1)$: $w':= w z^{-2}=\begin{pmatrix}

0 & \zeta_3 & 0\\
0 & 0 & \zeta_3^2\\
1 & 0 & 0

\end{pmatrix}$, then\\ $w':= w' s_2 s_3 = \begin{pmatrix}

\zeta_3 & 0 & 0\\
0 & \boxed{\zeta_3^2} & 0\\
0 & 0 & \mathbf{1}

\end{pmatrix}$.\\
Step $2$ $(i= 2, k = 2, c=1)$: $w' := w' s_2 = \begin{pmatrix}

0 & \zeta_3 & 0\\
\zeta_3^2 & 0 & 0\\
0 & 0 & 1

\end{pmatrix}$, then $w':= w' z^{-2} = \begin{pmatrix}

0 & \zeta_3 & 0\\
1 & 0 & 0\\
0 & 0 & 1

\end{pmatrix}$, then $w' := w' s_2 = \begin{pmatrix}

\boxed{\zeta_3} & 0 & 0\\
0 & \mathbf{1} & 0\\
0 & 0 & 1

\end{pmatrix}$.\\
Step $3$ $(i=1, k=1, c=1)$: $w' := w' z^{-1} = I_3$.\\
Hence $R\!E(w) = \mathbf{z} \mathbf{s}_2 \mathbf{z}^2 \mathbf{s}_2 \mathbf{s}_3 \mathbf{s}_2 \mathbf{z}^2$.

\end{example}

The next lemma follows directly from Algorithm \ref{Algo3}. Remark the difference between the next lemma and lemmas \ref{LemmaBlocks} and \ref{LemmaBlocksG(de,e,n)}. 

\begin{lemma}\label{LemmaBlocksG(d,1,n)}

For $2 \leq i \leq n$, the block $w_i$ is obtained by removing the row $i$ and the column $c$ from $w_i$. \mbox{Moreover, $w_1$ is equal to the nonzero entry on the first row of $w$.}

\end{lemma}

\begin{definition}\label{DefinitionREG(d,1,n)}

Let $1 \leq i \leq n$. Let $w_i[i,c] \neq 0$ for $1 \leq c \leq i$.

\begin{itemize}

\item If $w_1 = \zeta_d^k$ for some $0 \leq k \leq d-1$ (this is equal to the nonzero entry on the first row of $w$), we define $R\!E_{1}(w)$ to be the word $\mathbf{z}^{k}$.	
\item If $w_{i}[i,c] =1$, we define $R\!E_{i}(w)$ to be the word\\
$\mathbf{s}_i \mathbf{s}_{i-1} \cdots \mathbf{s}_{c+1}$ (decreasing-index expression).
\item If $w_{i}[i,c] = \zeta_{d}^{k}$ with $k \neq 0$, we define $R\!E_{i}(w)$ to be the word\\
\begin{tabular}{ll}
			$\mathbf{s}_i \cdots \mathbf{s}_3 \mathbf{z}^k$ & if $c=1$,\\
			$\mathbf{s}_i \cdots \mathbf{s}_3 \mathbf{s}_2 \mathbf{z}^k \mathbf{s}_2 \mathbf{s}_3 \cdots \mathbf{s}_c$ & if $c \geq 2$.\\
			
\end{tabular}

\end{itemize}

\end{definition}

As for Algorithm \ref{Algo2}, the output of Algorithm \ref{Algo3} is equal to $R\!E_1(w) R\!E_2(w) \cdots R\!E_n(w)$. In Example \ref{ExampAlgoG(d,1,n)}, we have $R\!E(w) = \underset{R\!E_1(w)}{\underbrace{\mathbf{z}}} \hspace{0.2cm} \underset{R\!E_2(w)}{\underbrace{\mathbf{s}_2 \mathbf{z}^2 \mathbf{s}_2}} \hspace{0.2cm} \underset{R\!E_3(w)}{\underbrace{\mathbf{s}_3 \mathbf{s}_2 \mathbf{z}^2}}$.\\

The proof of the next proposition is the same as the proof of Proposition \ref{PropWordRepG(de,e,n)}.

\begin{proposition}\label{PropWordRepG(d,1,n)}

Let $w \in G(d,1,n)$. The word $R\!E(w)$ given by Algorithm \ref{Algo3} is a word representative over $\mathbf{X}$ of $w \in G(d,1,n)$.

\end{proposition}

The following proposition enables us to prove that the output of Algorithm \ref{Algo3} is a reduced expression over $\mathbf{X}$ of a given element $w \in G(d,1,n)$. It is similar to Propositions \ref{prop.lengthcomp} and \ref{prop.lengthcompG(de,e,n)}.

\begin{proposition}\label{prop.lengthcompG(d,1,n)}

Let $w$ be an element of $G(d,1,n)$. For all $x \in X$, we have $$\boldsymbol{\ell}(R\!E(xw)) \leq \boldsymbol{\ell}(R\!E(w)) + 1.$$

\end{proposition}

\begin{proof}

We have two cases. The first one is for $x = s_i$ ($2 \leq i \leq n$). This is done in the same way as Case 1 in the proof of Proposition \ref{prop.lengthcomp}. The second case is for $x = z$. This is done in the same way as Case 3 in the proof of Proposition \ref{prop.lengthcompG(de,e,n)}. We get $(a)$, $(b)$, $(a')$, $(b')$, and $(g)$ as in the proofs of Propositions \ref{prop.lengthcomp} and \ref{prop.lengthcompG(de,e,n)}.

\end{proof}

Now, we apply the argument of the proof of Proposition \ref{PropREwRedExp} and deduce that $R\!E(w)$ is a reduced expression over $\mathbf{X}$ of $w \in G(d,1,n)$. Hence Algorithm \ref{Algo3} produces geodesic normal forms for $G(d,1,n)$. We also get Proposition \ref{PropLengthdecreas} (1). When $d=2$, that is the case of the Coxeter group $G(2,1,n)$ of type $B_n$, as a direct consequence of the proof of Proposition \ref{prop.lengthcompG(d,1,n)}, we have the following.

\begin{proposition}\label{PropIsLeftDescendingG(2,1,n)}
Let $w$ be an element of $G(2,1,n)$. Denote by $a_i$ the unique nonzero entry $w[i,c_i]$ on the row $i$ of $w$ where $1 \leq i, c_i \leq n$.
\begin{enumerate}
\item We have $\ell(z w) = \ell(w) -1$ $\Longleftrightarrow$ $a_1 = -1$.
\item For $2 \leq i \leq n$, we have:
	\begin{enumerate}
		\item if $c_{i-1} < c_i$, then
				$\ell(s_{i} w) = \ell(w)-1$ $\Longleftrightarrow$ $a_{i} \neq 1.$
		\item if $c_{i-1} > c_i$, then
				$\ell(s_{i} w) = \ell(w)-1$ $\Longleftrightarrow$ $a_{i-1} =1.$
	\end{enumerate}
\end{enumerate}
\end{proposition}

\section{Elements of maximal length}

Using the geodesic normal forms of the elements of $G(de,e,n)$ defined by Algorithms \ref{Algo1}, \ref{Algo2}, and \ref{Algo3}, we will characterize in this section the elements of $G(de,e,n)$ that are of maximal length. We distinguish three cases: the first one is for the group $G(e,e,n)$ defined by the presentation of Corran-Picantin (see Definition \ref{DefinitionPresentationCorranPicantin}), the second case is for $G(de,e,n)$ with $d > 1, e > 1$, and $n \geq 2$ defined by the presentation of Corran-Lee-Lee (see Definition \ref{DefCorranLeePresG(de,e,n)}), and the third case is for $G(d,1,n)$ with $d > 1$ defined by its classical presentation (see Proposition \ref{PropPresB(d,1,n)}).

\begin{proposition}\label{PropMaxLength}

Let $e > 1$ and $n \geq 2$. The maximal length of an element of $G(e,e,n)$ is $n(n-1)$. It is realized for diagonal matrices $w$ such that $w[i,i]$ is an $e$-th root of unity different from $1$ for $2 \leq i \leq n$. A minimal word representative of such an element is of the form $$(\mathbf{t}_{k_2}\mathbf{t}_0)(\mathbf{s}_3\mathbf{t}_{k_3}\mathbf{t}_0\mathbf{s}_3)\cdots (\mathbf{s}_n \cdots \mathbf{s}_3\mathbf{t}_{k_n}\mathbf{t}_0\mathbf{s}_3 \cdots \mathbf{s}_n),$$ with $1 \leq k_2, \cdots, k_n \leq e-1$. The number of elements of this form is $(e-1)^{(n-1)}$.

\end{proposition}

\begin{proof}

By Algorithm \ref{Algo1}, an element $w$ in $G(e,e,n)$ is of maximal length when $w_i[i,i] = \zeta_e^{k}$ for $2 \leq i \leq n$ and $\zeta_e^{k} \neq 1$. By Lemma \ref{LemmaBlocks}, this condition is satisfied when $w$ is a diagonal matrix such that $w[i,i]$ is an $e$-th root of unity different from $1$ for $2 \leq i \leq n$. A minimal word representative given by Algorithm \ref{Algo1} for such an element is of the form $(\mathbf{t}_{k_2}\mathbf{t}_0)(\mathbf{s}_3\mathbf{t}_{k_3}\mathbf{t}_0\mathbf{s}_3)\cdots (\mathbf{s}_n \cdots \mathbf{s}_3\mathbf{t}_{k_n}\mathbf{t}_0\mathbf{s}_3 \cdots \mathbf{s}_n)$ with $1 \leq k_2, \cdots, k_n \leq (e-1)$ which is of length $n(n-1)$. The number of elements of this form is $(e-1)^{(n-1)}$.

\end{proof}

\noindent Denote by $\lambda$ the element $\begin{pmatrix}

(\zeta_e^{-1})^{(n-1)} & & & \\
 & \zeta_e & & \\
 & & \ddots & \\
 & & & \zeta_e\\

\end{pmatrix} \in G(e,e,n)$.

\begin{example}\label{LemmaLengthlambda}

We have $R\!E(\lambda) = (\mathbf{t}_{1}\mathbf{t}_0)(\mathbf{s}_3\mathbf{t}_{1}\mathbf{t}_0\mathbf{s}_3)\cdots (\mathbf{s}_n \cdots \mathbf{s}_3\mathbf{t}_{1}\mathbf{t}_0\mathbf{s}_3 \cdots \mathbf{s}_n)$. Hence ${\ell}(\lambda) = n(n-1)$ which is the maximal length of an element of $G(e,e,n)$.\\

\end{example}

\begin{remark}

Consider the group $G(1,1,n)$, that is the symmetric group $S_n$. There exists a unique element of maximal length in $S_n$ that is of the form $$\mathbf{t}_0 (\mathbf{s}_3\mathbf{t}_0) \cdots (\mathbf{s}_{n-1}\cdots \mathbf{s}_3\mathbf{t}_0) (\mathbf{s}_{n}\cdots \mathbf{s}_3\mathbf{t}_0).$$ This corresponds to the maximal number of steps of Algorithm \ref{Algo1}. The length of such an element is $n(n-1)/2$ which is already known for the symmetric group $S_n$, see Example 1.5.4 of \cite{GeckPfeifferBook}.

\end{remark}

\begin{remark}\label{RemUniqueMaximalG(2,2,n)}

Consider the group $G(2,2,n)$, that is the Coxeter group of type $D_n$. We have $e = 2$. Then, by Proposition \ref{PropMaxLength}, the number of elements of maximal length is equal to $(e-1)^{(n-1)} = 1$. Hence there exists a unique element of maximal length in $G(2,2,n)$. It is of the form $$(\mathbf{t}_{1}\mathbf{t}_0)(\mathbf{s}_3\mathbf{t}_{1}\mathbf{t}_0\mathbf{s}_3) \cdots (\mathbf{s}_n \cdots \mathbf{s}_3\mathbf{t}_{1}\mathbf{t}_0\mathbf{s}_3 \cdots \mathbf{s}_n).$$ The length of this element is $n(n-1)$ which is already known for Coxeter groups of type $D_n$, see Example 1.5.5 of \cite{GeckPfeifferBook}.\\

\end{remark}

The next two propositions are a direct consequence of Algorithms \ref{Algo2} and \ref{Algo3}.

\begin{proposition}\label{PropMaxLengthG(de,e,n)}

Let $d > 1, e > 1$, and $n \geq 2$. The maximal length of an element of $G(de,e,n)$ is $n(n-1) + d-1$. It is realized for diagonal matrices $w$ such that for all $2 \leq i \leq n$, we have $w[i,i] = \zeta_{de}^{k_i}$ with $1 \leq k_i \leq de-1$ and $w[1,1] = \zeta_{de}^x$ with $x + (k_2 \cdots k_n) = e(d-1)$. A minimal word representative of such an element is of the form $$\mathbf{z}^{d-1} (\mathbf{t}_{k_2}\mathbf{t}_0) (\mathbf{s}_3\mathbf{t}_{k_3}\mathbf{t}_0\mathbf{s}_3)\cdots (\mathbf{s}_n \cdots \mathbf{s}_3\mathbf{t}_{k_n}\mathbf{t}_0\mathbf{s}_3 \cdots \mathbf{s}_n),$$ with $1 \leq k_2, \cdots, k_n \leq de-1$. The number of elements that are of maximal length is then $(de-1)^{(n-1)}$.

\end{proposition}

\begin{proposition}\label{PropMaxLengthG(d,1,n)}

Let $d > 1$ and $n \geq 2$. There exists a unique element of maximal length of $G(d,1,n)$. Its minimal word representative is of the form $$\mathbf{z}^{d-1} (\mathbf{s}_2 \mathbf{z}^{d-1}\mathbf{s}_2) (\mathbf{s}_3\mathbf{s}_2\mathbf{z}^{d-1}\mathbf{s}_2\mathbf{s}_3) \cdots (\mathbf{s}_n \cdots \mathbf{s}_2 \mathbf{z}^{d-1} \mathbf{s}_2 \cdots \mathbf{s}_n).$$ Its length is then equal to $n(n+d-2)$.

\end{proposition}

\begin{remark}

When $d=2$, the group $G(2,1,n)$ is the Coxeter group of type $B_n$. By Proposition \ref{PropMaxLengthG(d,1,n)}, the longest element is of the form $$\mathbf{z} (\mathbf{s}_2 \mathbf{z}\mathbf{s}_2) (\mathbf{s}_3\mathbf{s}_2\mathbf{z}\mathbf{s}_2\mathbf{s}_3) \cdots (\mathbf{s}_n \cdots \mathbf{s}_2 \mathbf{z} \mathbf{s}_2 \cdots \mathbf{s}_n).$$ Its length is equal to $n^2$ which is already known for Coxeter groups of type $B_n$, see Example 1.4.6 of \cite{GeckPfeifferBook}.\\

\end{remark}

%% file: chap33.tex
\chapter{Interval Garside structures for $B(e,e,n)$}\label{ChapterIntervalGarside}
 
\minitoc

\bigskip

In this chapter, we are interested in Garside structures that derive from intervals. We construct intervals in the complex reflection group $G(e,e,n)$ which give rise to Garside groups. Some of these groups correspond to the complex braid group $B(e,e,n)$. For the other Garside groups that appear, we give some of their properties in order to understand these new structures.
 
\section{Generalities about Garside structures}

In his PhD thesis, defended in 1965 \cite{GarsideThesis}, and in the article that followed \cite{GarsideArt}, Garside solved the Conjugacy Problem for the classical braid group $B_n$ by introducing a submonoid $B_{n}^{+}$ of $B_n$ and an element $\Delta_n$ of $B_{n}^{+}$ that he called fundamental, and then showing that there exists a normal form for every element in $B_n$. In the beginning of the 1970's, it was realized by Brieskorn and Saito \cite{GarsideBrieskornSaito} and Deligne \cite{GarsideDeligne} that Garside's results extend to all finite-type Artin-Tits groups. At the end of the 1990's, after listing the abstract properties of $B_n^{+}$ and the fundamental element $\Delta_n$, Dehornoy and Paris \cite{GarsideDehornoyParis} defined the notion of Gaussian groups and Garside groups which leads, in ``a natural, but slowly emerging program'' as stated in \cite{GarsideBookPatrick}, to Garside theory. For a detailed study about Garside structures, we refer the reader to \cite{GarsideBookPatrick}. The aim of this section is to give the necessary preliminaries about Garside structures.

\subsection{Garside monoids and groups}

Let $M$ be a monoid. Under some assumptions about $M$, more precisely the assumptions 1 and 2 of Definition \ref{DefGarsideMonoid}, one can define a partial order relation on $M$ as follows.

\begin{definition}

Let $f,g \in M$. We say that $f$ left-divides $g$ or simply $f$ divides $g$ when there is no confusion, written $f \preceq g$, if $fg'=g$ holds for some $g' \in M$. Similarly, we say that $f$ right-divides $g$, written $f \preceq_r g$, if $g'f=g$ holds for some $g' \in M$. 

\end{definition}

We are ready to define Garside monoids and groups.

\begin{definition}\label{DefGarsideMonoid}

A Garside monoid is a pair $(M, \Delta)$, where $M$ is a monoid and

\begin{enumerate}

\item $M$ is cancellative, that is $fg=fh \Longrightarrow g=h$ and $gf=hf \Longrightarrow g=h$ for $f,g,h \in M$,

\item there exists $\lambda : M \longrightarrow \mathbb{N}$ s.t. $\lambda(fg) \geq \lambda(f) + \lambda(g)$ and $g \neq 1 \Longrightarrow \lambda(g) \neq 0$,

\item any two elements of $M$ have a gcd and an lcm for $\preceq$ and $\preceq_r$, and

\item $\Delta$ is a \emph{Garside element} of $M$, this meaning that the set of its left divisors coincides with the set of its right divisors, generate $M$, and is finite. 

\end{enumerate} 

The divisors of $\Delta$ are called the \emph{simples} of $M$.

\end{definition}

A quasi-Garside monoid is a pair $(M,\Delta)$ that satisfies the conditions of Definition \ref{DefGarsideMonoid}, except the finiteness of the number of divisors of $\Delta$.\\

Assumptions 1 and 3 of Definition \ref{DefGarsideMonoid} ensure that Ore's conditions are satisfied. Hence there exists a group of fractions of the monoid $M$ in which it embeds. This allows us to give the following definition.

\begin{definition}

A Garside group is the group of fractions of a Garside monoid.

\end{definition}

Consider a pair $(M,\Delta)$ that satisfies the conditions of Definition \ref{DefGarsideMonoid}. The pair $(M,\Delta)$ and the group of fractions of $M$ provide a Garside structure for $M$.\\

Note that one of the important aspects of a Garside structure is the existence of a normal form for all elements of the Garside group. Furthermore, many problems like the Word and Conjugacy problems can be solved in Garside groups which makes their study interesting. 

\subsection{Interval Garside structures}

Let $G$ be a finite group generated by a finite set $S$. There is a way to construct Garside structures from intervals in $G$.\\

We start by defining a partial order relation on $G$.

\begin{definition}\label{DefPartalOrder11}

Let $f, g \in G$. We say that $g$ is a divisor of $f$ or $f$ is a multiple of $g$, and write $g \preceq f$, if $f = g h$ with $h \in G$ and $\ell(f) = \ell(g) + \ell(h)$, where $\ell(f)$ is the length over $S$ of $f \in G$.

\end{definition}

\begin{definition}\label{DefGeneralIntervalMonoid}

For $w \in G$, define a monoid $M([1,w])$ by the monoid presentation with

\begin{itemize}

\item generating set $\underline{P}$ in bijection with the interval
\begin{center} $[1,w] := \{ f \in G \ | \ 1 \preceq f \preceq w \}$ and \end{center}

\item relations: $\underline{f}\ \underline{g} = \underline{h}$ if $f, g , h \in [1,w]$, $fg=h$, and $f \preceq h$, that is \mbox{$\ell(f) + \ell(g) = \ell(h)$}.

\end{itemize}

\end{definition}

\noindent Similarly, one can define the partial order relation on $G$
\begin{center}$g \preceq_r f$ if and only if $\ell(f{g}^{-1}) + \ell(g) = \ell(f)$, \end{center}
then define the interval $[1,w]_r$ and the monoid $M([1,w]_r)$. 

\begin{definition}

Let $w$ be in $G$. We say that $w$ is a balanced element of $G$ if $[1,w] = [1,w]_r$.

\end{definition}

We have the following theorem due to Michel (see Section 10 of \cite{CorseJeanMichel} for a proof).

\begin{theorem}\label{TheoremMichelGarside}

If $w \in G$ is balanced and both posets $([1,w],\preceq)$ and $([1,w]_r, \preceq_r)$ are lattices, then $(M([1,w]),\underline{w})$ is a Garside monoid with simples $\underline{[1,w]}$, where $\underline{w}$ and $\underline{[1,w]}$ are given in Definition \ref{DefGeneralIntervalMonoid}.

\end{theorem}

The previous construction gives rise to an interval structure. The interval monoid is $M([1,w])$. When $M([1,w])$ is a Garside monoid, its group of fractions exists and is denoted by $G(M([1,w]))$. We call it the interval group. We will give a classical example of this structure. It shows that Artin-Tits groups admit interval structures. For details about this example, see Chapter 9, Section 1.3 in \cite{GarsideBookPatrick}.

\begin{example}\label{ExampleArtinTitsIntervals}

Let $W$ be a finite coxeter group and $B(W)$ the Artin-Tits group associated to $W$.

\begin{center} $W = <S\ |\ s^2=1, \ \underset{m_{st}}{\underbrace{sts\cdots}}=\underset{m_{st}}{\underbrace{tst\cdots}}$ for $s, t \in S, s \neq t, m_{st} = o(st)>,$\\
$B(W) = <\widetilde{S}\ |\ \underset{m_{st}}{\underbrace{\tilde{s}\tilde{t}\tilde{s}\cdots}}=\underset{m_{st}}{\underbrace{\tilde{t}\tilde{s}\tilde{t}\cdots}}\ for\ \tilde{s}, \tilde{t} \in \widetilde{S},\ \tilde{s} \neq \tilde{t}>$.
\end{center}
Take $G=W$ and $g=w_0$ the longest element over $S$ in $W$. We have \mbox{$[1,w_0] = W$}. Construct the interval monoid $M([1,w_0])$. We have $M([1,w_0])$ is the Artin-Tits monoid $B^{+}(W)$, where $B^{+}(W)$ is the monoid defined by the same presentation as $B(W)$. Hence $B^{+}(W)$ is generated by a copy $\underline{W}$ of $W$ with $\underline{f}\ \underline{g} = \underline{h}$ if $fg=h$ and $\ell(f) + \ell(g) = \ell(h)$; $f,g$, and $h \in W$. It is also known that $w_0$ is balanced and both posets $([1,w_0],\preceq)$ and $([1,w_0]_r, \preceq_r)$ are lattices. Hence by Theorem \ref{TheoremMichelGarside}, we have the following result.

\end{example}

\begin{proposition}

$(B^{+}(W),\underline{w_0})$ is a Garside monoid with simples $\underline{W}$, where $\underline{w_0}$ and $\underline{W}$ are given in Example \ref{ExampleArtinTitsIntervals}.

\end{proposition}

Our aim is to construct interval structures for the complex braid group $B(e,e,n)$. Since the interval Garside structures depend on the definition of a length function in the corresponding complex reflection group, we will use the results and notations of Section 2.1 of Chapter \ref{ChapterNormalFormsG(de,e,n)} where geodesic normal forms are constructed for all the elements of $G(e,e,n)$.

\section{Balanced elements of maximal length}

Consider the group $G(e,e,n)$ defined by the presentation of Corran and Picantin with generating set $\mathbf{X}$, see Definition \ref{DefinitionPresentationCorranPicantin}. We always assume Convention \ref{ConventionDecreaseIncreaseIndex}. Denote by $\ell(w)$ the length over $\mathbf{X}$ of $w \in G(e,e,n)$. It is equal to the word length of the output $RE(w)$ of Algorithm \ref{Algo1}, see Proposition \ref{PropREwRedExp}.\\

The aim of this section is to prove that the only balanced elements of $G(e,e,n)$ that are of maximal length over $\mathbf{X}$ are $\lambda^k$ with $1 \leq k \leq e-1$, where $\lambda$ is the diagonal matrix such that $\lambda[i,i] = \zeta_e$ for $2 \leq i \leq n$ and $\lambda[1,1] = {(\zeta_e^{-1})}^{n-1}$, see Proposition \ref{PropMaxLength}. This is done by characterizing the intervals of the elements of maximal length.\\

We start by defining two partial order relations on $G(e,e,n)$.

\begin{definition}\label{DefPartalOrder1}

Let $w, w' \in G(e,e,n)$. We say that $w'$ is a divisor of $w$ or $w$ is a multiple of $w'$, and write $w' \preceq w$, if $w = w' w''$ with $w'' \in G(e,e,n)$ and $\ell(w) = \ell(w') + \ell(w'')$. This defines a partial order relation on $G(e,e,n)$.

\end{definition}

\noindent Similarly, we have another partial order relation on $G(e,e,n)$.

\begin{definition}\label{DefPartialOrder2}

Let $w$, $w' \in G(e,e,n)$. We say that $w'$ is a right divisor of $w$ or $w$ is a left multiple of $w'$, and write $w' \preceq_r w$, if there exists $w'' \in G(e,e,n)$ such that $w=w''w'$ and $\ell(w) = \ell(w'')+\ell(w')$.

\end{definition}

\begin{lemma}\label{LemmaCaractDivlambda}

Let $w, w' \in G(e,e,n)$ and let $\mathbf{x}_1\mathbf{x}_2 \cdots \mathbf{x}_r$ be a reduced expression over $\mathbf{X}$ of $w'$. We have $w' \preceq w$ if and only if  $\forall\ 1 \leq i \leq r, \ell(x_ix_{i-1}\cdots x_1 w)=\ell(x_{i-1}\cdots x_1 w)-1.$ 

\end{lemma}

\begin{proof}

On the one hand, we have $w'w''=w$ with $w''=x_rx_{r-1} \cdots x_1 w$ and the condition $\forall\ 1 \leq i \leq r, \ell(x_ix_{i-1}\cdots x_1w)=\ell(x_{i-1}\cdots x_1 w)-1$ implies that \mbox{$\ell(w'')=\ell(w)-r$}. So we get $\ell(w'')+\ell(w')=\ell(w)$. Hence $w' \preceq w$.

On the other hand, since $x^2=1$ for all $x \in X$, we have $\ell(xw)=\ell(w) \pm 1$ for all \mbox{$w \in G(e,e,n)$}. If there exists $i$ such that $\ell(x_i x_{i-1}\cdots x_1 w) = \ell(x_{i-1} \cdots x_1 w)+1$ with $1 \leq i \leq r$, then $\ell(w'') = \ell(x_r x_{r-1}\cdots x_1 w) > \ell(w) -r$. It follows that $\ell(w') + \ell(w'') > \ell(w)$. Hence $w' \npreceq w$.

\end{proof}

Consider the homomorphism ${}^-$: $\mathbf{X}^{*} \longrightarrow G(e,e,n): \mathbf{x} \longmapsto \overline{\mathbf{x}} := x \in X$. If $R\!E(w) = \boldsymbol{x}_1\boldsymbol{x}_2 \cdots \boldsymbol{x}_r$ with $w \in G(e,e,n)$ and $\boldsymbol{x}_1,\boldsymbol{x}_2, \cdots, \boldsymbol{x}_r \in \mathbf{X}$, then $\overline{R\!E(w)} =  x_1x_2 \cdots x_r  = w$ where $x_1, x_2, \cdots, x_r \in X$.\\

In the sequel, we fix $1 \leq k \leq e-1$ and let $w \in G(e,e,n)$.

\begin{definition}

Let $\lambda$ be the diagonal matrix of $G(e,e,n)$ such that $\lambda[i,i] = \zeta_e$ for $2 \leq i \leq n$ and $\lambda[1,1] = (\zeta_e^{-1})^{n-1}$. We define $D_k$ to be the set \begin{center} $\left\{w \in G(e,e,n)\ s.t.\ \overline{R\!E_i(w)} \preceq \overline{R\!E_i(\lambda^k)}\ for\ 2 \leq i \leq n \right\}$, \end{center}
where $R\!E_i(w)$ is given in Definition \ref{DefREiw}.

\end{definition}

\begin{proposition}\label{PropCaractD}

The set $D_k$ consists of the elements $w$ of $G(e,e,n)$ such that for all \mbox{$2 \leq i \leq n$}, $R\!E_i(w)$ can be any of the following words:

\begin{tabular}{ll}
$\mathbf{s}_i \cdots \mathbf{s}_{i'}$ & with $2 \leq i' \leq i$,\\
$\mathbf{s}_i \cdots \mathbf{s}_3 \mathbf{t}_{k'}$ & with $0 \leq k' \leq e-1$, and\\
$\mathbf{s}_i \cdots \mathbf{s}_3\mathbf{t}_{k}\mathbf{t}_0\mathbf{s}_3 \cdots \mathbf{s}_{i'}$ & with $2 \leq i' \leq i$.

\end{tabular}

\end{proposition}

\begin{proof}

We have $R\!E_i(\lambda^k) = \mathbf{s}_i \cdots \mathbf{s}_3 \mathbf{t}_k \mathbf{t}_0 \mathbf{s}_3 \cdots \mathbf{s}_i$. Let $w \in G(e,e,n)$. Note that $R\!E_i(w)$ is necessarily one of the words given in the first column of the following table. For each $R\!E_i(w)$, there exists a unique $w' \in G(e,e,n)$ with $R\!E(w')$ given in the second column, such that $\overline{R\!E_i(w)}w' = \overline{R\!E_i(\lambda^k)}$.

For $R\!E_i(w) = \mathbf{s}_i \cdots \mathbf{s}_{i'}$ with $2 \leq i' \leq i$, we get $R\!E(w') = \mathbf{s}_{i'-1} \cdots \mathbf{s}_3 \mathbf{t}_k \mathbf{t}_0 \mathbf{s}_3 \cdots \mathbf{s}_i$.

For $R\!E_i(w) = \mathbf{s}_i \cdots \mathbf{s}_3 \mathbf{t}_{k'}$ with $0 \leq k' \leq e-1$, we get $R\!E(w') = \mathbf{t}_{k'-k}\mathbf{s}_3 \cdots \mathbf{s}_i$. In this case, $\overline{R\!E_i(w)}\ \overline{R\!E(w')} = s_i \cdots s_3 t_{k'}t_{k'-k} s_3 \cdots s_i = s_i \cdots s_3 t_{k}t_{0} s_3 \cdots s_i = \overline{R\!E_i(\lambda^k)}$.

For $R\!E_i(w) = \mathbf{s}_i \cdots \mathbf{s}_3\mathbf{t}_k\mathbf{t}_0\mathbf{s}_3 \cdots \mathbf{s}_{i'}$ with $2 \leq i' \leq i$, we get $R\!E(w') = \mathbf{s}_{i'+1} \cdots \mathbf{s}_i$.

Finally, for $R\!E_i(w) = \mathbf{s}_i \cdots \mathbf{s}_3 \mathbf{t}_{k'} \mathbf{t}_0 \mathbf{s}_3 \cdots \mathbf{s}_{i'}$ with $1 \leq k' \leq e-1$, $k' \neq k$, and $2 \leq i' \leq i$, we get $R\!E(w') = \mathbf{s}_{i'} \cdots \mathbf{s}_3 \mathbf{t}_{k-k'} \mathbf{t}_{0} \mathbf{s}_3 \cdots \mathbf{s}_i$.

In the last column, we compute $\ell\left(\overline{R\!E_i(w)}\right) + \ell(\overline{R\!E(w')})$. It is equal to $\ell\left(\overline{R\!E_i(\lambda^k)}\right) = 2(i-1)$ only for the first three cases. The result follows immediately.\\

\begin{tabular}{|l|l|l|}
\hline
$R\!E_i(w)$ & $R\!E(w')$ & \\
\hline
$\mathbf{s}_i \cdots \mathbf{s}_{i'}$ with $2 \leq i' \leq i$ & $\mathbf{s}_{i'-1} \cdots \mathbf{s}_3 \mathbf{t}_k \mathbf{t}_0 \mathbf{s}_3 \cdots \mathbf{s}_i$ & $2(i-1)$\\
\hline
$\mathbf{s}_i \cdots \mathbf{s}_3 \mathbf{t}_{k'}$ with $0 \leq k' \leq e-1$ & $\mathbf{t}_{k'-k}\mathbf{s}_3 \cdots \mathbf{s}_i$ & $2(i-1)$\\
\hline
$\mathbf{s}_i \cdots \mathbf{s}_3\mathbf{t}_k \mathbf{t}_0\mathbf{s}_3 \cdots \mathbf{s}_{i'}$ with $2 \leq i' \leq i$ & $\mathbf{s}_{i'+1} \cdots \mathbf{s}_i$ & $2(i-1)$\\
\hline
$\mathbf{s}_i \cdots \mathbf{s}_3 \mathbf{t}_{k'} \mathbf{t}_0 \mathbf{s}_3 \cdots \mathbf{s}_{i'}$ with $1 \leq k' \leq e-1$, & $\mathbf{s}_{i'} \cdots \mathbf{s}_3 \mathbf{t}_{k-k'} \mathbf{t}_{0} \mathbf{s}_3 \cdots \mathbf{s}_i$ & $2(i-1)+$\\
$k' \neq k$, and $2 \leq i' \leq i$ & & $2(i'-1)$ \\
\hline

\end{tabular}

\end{proof}

The next proposition characterizes the divisors of $\lambda^k$ in $G(e,e,n)$.

\begin{proposition}\label{PropDisDivLambda}

The set $D_k$ is equal to the interval $[1,\lambda^k]$, where \begin{center}$[1,\lambda^k] = \left\{ w \in G(e,e,n)\ s.t.\ 1 \preceq w \preceq \lambda^k  \right\}$.\end{center}

\end{proposition}

\begin{proof}

Let $w \in G(e,e,n)$. We have $R\!E(w) = R\!E_2(w)R\!E_3(w) \cdots R\!E_n(w)$. Let $\textbf{w} \in \mathbf{X}^{*}$ be a word representative of $w$. Denote by $\overleftarrow{\mathbf{w}} \in \mathbf{X}^{*}$ the word obtained by reading $\textbf{w}$ from right to left. For $3 \leq i \leq n$, we denote by $\alpha_i$ the element that corresponds to \begin{small}$\overline{\overleftarrow{R\!E_{i-1}(w)}} \cdots \overline{\overleftarrow{R\!E_2(w)}}$\end{small} in $G(e,e,n)$.

Suppose that $w \in D_k$. We apply Lemma  \ref{LemmaCaractDivlambda} to prove that $w \preceq \lambda^k$. Fix $2 \leq i \leq n$. By Proposition \ref{PropCaractD}, we have three different possibilities for $R\!E_i(w)$.

First, consider the cases $R\!E_i(w) = \mathbf{s}_i \cdots  \mathbf{s}_3  \mathbf{t}_k  \mathbf{t}_0  \mathbf{s}_3 \cdots  \mathbf{s}_{i'}$ or $\mathbf{s}_i \cdots \mathbf{s}_{i'}$ with $2 \leq i' \leq i$. Hence $\overleftarrow{R\!E_i(w)} = \mathbf{s}_{i'} \cdots \mathbf{s}_3 \mathbf{t}_0 \mathbf{t}_k \mathbf{s}_3 \cdots \mathbf{s}_i$ or $\mathbf{s}_{i'} \cdots \mathbf{s}_i$, respectively.
Note that the left multiplication of the matrix $\lambda^k$ by $\alpha_i$ produces permutations only in the block consisting of the first $i-1$ rows and the first $i-1$ columns of $\lambda^k$. Since  $\lambda[i,i] = \zeta_e^k$ ($\neq 1$), by $1(a)$ of Proposition \ref{PropLengthdecreas}, the left multiplication of $\alpha_i \lambda^k$ by $s_{i'} \cdots s_i$ decreases the length maximally $-$ that is, each generator causes a decrease of length $1$. Now, by $2$ of Proposition \ref{PropLengthdecreas}, the left multiplication of $s_3 \cdots s_i \alpha_i \lambda^k$ by $t_k$ decreases the length of $1$. Note that by these left multiplications, $\lambda^k[i,i] = \zeta_e^k$ is shifted to the first row then transformed to $ \zeta_e^k  \zeta_e^{-k}=1$. Hence, by $1(b)$ of Proposition \ref{PropLengthdecreas}, the left multiplication of $t_1s_3 \cdots s_i \alpha_i \lambda^k$ by $s_{i'} \cdots s_3 t_0$ decreases the length maximally. Thus, by Lemma \ref{LemmaCaractDivlambda}, we have $w \preceq \lambda^k$.

Suppose that $R\!E_i(w) = \mathbf{s}_i \cdots \mathbf{s}_3 \mathbf{t}_{k'}$ with $0 \leq k' \leq e-1$. We have $\overleftarrow{R\!E_i(w)} = \mathbf{t}_{k'} \mathbf{s}_3 \cdots \mathbf{s}_i$.
Since $\lambda^k[i,i] = \zeta_e^k$ ($\neq 1$), by $1(a)$ of Proposition \ref{PropLengthdecreas}, the left multiplication of $\alpha_i \lambda^k$ by $s_3 \cdots s_i$ decreases the length maximally. By $2$ of Proposition \ref{PropLengthdecreas}, the left multiplication of $s_3 \cdots s_i \alpha_i \lambda^k$ by $t_{k'}$ also decreases the length of $1$. Hence, by applying Lemma \ref{LemmaCaractDivlambda}, we have $w \preceq \lambda^k$.

Conversely, suppose that $w \notin D_k$, we prove that $w \npreceq \lambda^k$.
If $R\!E(w) = \mathbf{x}_1 \cdots \mathbf{x}_r$, by Lemma \ref{LemmaCaractDivlambda}, we show that there exists $1 \leq i \leq r$ such that $\ell(x_i x_{i-1} \cdots x_1 \lambda^k) = \ell(x_{i-1} \cdots x_1 \lambda^k) +1$. Since $w \notin D$, by Proposition \ref{PropCaractD}, we may consider the first $R\!E_i(w)$ that appears in $R\!E(w) = R\!E_2(w) \cdots R\!E_n(w)$ such that $R\!E_i(w) =\\ \mathbf{s}_i \cdots \mathbf{s}_3 \mathbf{t}_{k'} \mathbf{t}_0 \mathbf{s}_3 \cdots \mathbf{s}_{i'}$ with $1 \leq k' \leq e-1$, $k' \neq k$, and $2 \leq i' \leq i$. Thus, we have $\overleftarrow{{R\!E}_i(w)} = \mathbf{s}_{i'} \cdots \mathbf{s}_3 \mathbf{t}_0\mathbf{t}_{k'} \mathbf{s}_3 \cdots \mathbf{s}_i$.
Since $\lambda^k[i,i] = \zeta_e^k$ ($\neq 1$), by $1(a)$ of Proposition \ref{PropLengthdecreas}, the left multiplication of $\alpha_i \lambda^k$ by $s_3 \cdots s_i$ decreases the length maximally. By $2$ of Proposition \ref{PropLengthdecreas}, the left multiplication of $s_3 \cdots s_i \alpha_i \lambda^k$ by $t_{k'}$ also decreases the length of $1$. Note that by these left multiplications, $\lambda^k[i,i] = \zeta_e^k$ is shifted to the first row then transformed to $\zeta_e^k \zeta_e^{-k'} = \zeta_e^{k-k'}$. Since $k \neq k'$, we have $\zeta_e^{k-k'} \neq 1$. By $3$ of Proposition \ref{PropLengthdecreas}, it follows that the left multiplication of $t_{k'} s_3 \cdots s_i \alpha_i \lambda^k$ by $t_0$ increases the length. Hence $w \npreceq \lambda^k$.

\end{proof}

We want to recognize if an element $w \in G(e,e,n)$ is in the set $D_k$ directly from its matrix form. For this purpose, we introduce nice combinatorial tools defined as follows. Fix $1 \leq k \leq e-1$. Let $w \in G(e,e,n)$.

\begin{definition}\label{DefinitionBullets}

An index $[i,c]$ is said to be a \emph{bullet} if $w[j,d] = 0$ for all $[j,d] \in \left\{ [j,d]\ s.t.\ j \leq i\ and\ d\leq c  \right\} \setminus \left\{[i,c] \right\} $. When $[i,c]$ is a bullet, $w[i,c]$ is represented by an encircled element.

\end{definition}

\begin{definition}\label{DefinitionZwZ'w}
We define two sets of matrix indices $Z(w)$ and $Z'(w)$ as follows.
\begin{itemize}

\item $Z(w) := \left\{ [j,d]\ s.t.\ j \leq i\ and\ d \leq c\ for\ some\ bullet\ [i,c] \right\}$.
\item $Z'(w)$ is the set of matrix indices not in $Z(w)$.

\end{itemize}

\end{definition}

We draw a path in the matrix $w$ that separates it into two parts such that the upper left-hand side is $Z(w)$ and the other side is $Z'(w)$. Let us illustrate this by the following example.

\begin{example}


Let $w =
\left(
\begin{BMAT}{ccccc}{ccccc}
0 & 0 & 0  & \encircle{$\zeta_3^2$} & 0 \\
0 & 0 & 0 & 0 & \zeta_3 \\
0 & 0 & \encircle{$\zeta_3$} & 0 & 0 \\
\encircle{$1$} & 0 & 0 & 0 & 0 \\
0 & \zeta_3^2 & 0 & 0 & 0
\addpath{(0,1,0)rurruurur}
\end{BMAT}
\right) \in G(3,3,5).
$ When $[i,c]$ is a bullet, $w[i,c]$ is an encircled element and the drawn path separates $Z(w)$ from $Z'(w)$.

\end{example}

\begin{remark}\label{RemBulletsinZ}

Let $[i,c]$ be one of the bullets of $w \in G(e,e,n)$. We have
\begin{center} $[i-1,c] \in Z(w)$ and $[i,c-1] \in Z(w)$. \end{center}
An index $[i,c]$ such that $w[i,c] \neq 0$ and $[i,c]$ is not a bullet does not satisfy this condition.

\end{remark}

\begin{remark}\label{RemBulletsFirstRowColumn}

The indices corresponding to nonzero entries on the first row and the first column of $w$ are always bullets. In particular, when $w[1,1] \neq 0$, we have $[1,1]$ is a bullet and it is the only bullet of $w$ (as this nonzero entry at $[1,1]$ is above, or to the left of, every entry of $w$).  

\end{remark}

The following proposition gives a nice description of the divisors of $\lambda^k$ in $G(e,e,n)$.

\begin{proposition}\label{PropZ'=1orZetae}

Let $w \in G(e,e,n)$. We have that $w \in D_k$ (i.e. $w \preceq \lambda^k$) if and only if, for all $[j,d] \in Z'(w)$, $w[j,d]$ is either $0$, $1$, or $\zeta_e^k$.

\end{proposition}

\begin{proof}

Let $w \in D_k$ and let $w[i,c] \neq 0$ for $[i,c] \in Z'(w)$. Since $w \in D_k$, by Proposition \ref{PropCaractD}, we have $R\!E_i(w) = \mathbf{s}_i \cdots \mathbf{s}_{i'}$ or $\mathbf{s}_i \cdots \mathbf{s}_3\mathbf{t}_k\mathbf{t}_0\mathbf{s}_3 \cdots \mathbf{s}_{i'}$ for $2 \leq i' \leq i$ (the case $R\!E_i(w) = \mathbf{s}_i \cdots \mathbf{s}_3\mathbf{t}_{k'}$ for $0 \leq k' \leq e-1$ appears only when $w[i,c] \neq 0$ and $[i,c]$ is a bullet). By Lemma \ref{LemmaBlocks}, we have $w[i,c] = w_i[i,d]$ for some $1 < d \leq i$. It follows that for $R\!E_i(w) = \mathbf{s}_i \cdots \mathbf{s}_{i'}$, we have $w[i,c] = 1$ and for $R\!E_i(w) = \mathbf{s}_i \cdots \mathbf{s}_3\mathbf{t}_k\mathbf{t}_0\mathbf{s}_3 \cdots \mathbf{s}_{i'}$, we have $w[i,c] = \zeta_e^k$.

Conversely, suppose that $w[j,d]$ is $0$, $1$, or $\zeta_e^k$ whenever $[j,d] \in Z'(w)$. Firstly, consider a nonzero entry $w[i,c]$ of $w$ for which $[i,c] \in Z'(w)$. From Remark \ref{RemBulletsFirstRowColumn}, we have $i \geq 2$. Once again $R\!E_i(w) = \mathbf{s}_i \cdots \mathbf{s}_3\mathbf{t}_k\mathbf{t}_0\mathbf{s}_3 \cdots \mathbf{s}_{i'}$ or $\mathbf{s}_i \cdots \mathbf{s}_{i'}$ for $2 \leq i' \leq i$. On the other hand, if $w[i,c]$ is a nonzero entry of $w$ for which $[i,c] \notin Z'(w)\ -$ that is, $[i,c]$ is a bullet of $w$, so by Lemma \ref{LemmaBlocks}, we have $w_i[i,1] = \zeta_e^{k'}$ for some $0 \leq k' \leq e-1$, for which case $R\!E_{i}(w) = \mathbf{s}_i \cdots \mathbf{s}_3 \mathbf{t}_{k'}$. Hence, by Proposition \ref{PropCaractD}, we have $w \in D_k$.

\end{proof}

\begin{example}

Let $w =
\left(
\begin{BMAT}{ccccc}{ccccc}
0 & 0 & 0  & \encircle{$1$} & 0 \\
0 & 0 & 0 & 0 & \boxed{\zeta_3} \\
0 & 0 & \encircle{$\zeta_3$} & 0 & 0 \\
\encircle{$\zeta_3$} & 0 & 0 & 0 & 0 \\
0 & \boxed{1} & 0 & 0 & 0
\addpath{(0,1,0)rurruurur}
\end{BMAT}
\right) \in G(3,3,5)$.\\ For all $[i,c] \in Z'(w)$, we have $w[i,c]$ is equal to $1$ or $\zeta_3$ (these are the boxed entries of $w$). It follows immediately that $w \preceq \lambda$ ($w \in [1,\lambda]$).

\end{example}

\begin{example}

Let $w =
\left(
\begin{BMAT}{ccccc}{ccccc}
0 & 0 & 0  & \encircle{$1$} & 0 \\
0 & 0 & 0 & 0 & \boxed{1} \\
0 & 0 & \encircle{$\zeta_3$} & 0 & 0 \\
\encircle{$1$} & 0 & 0 & 0 & 0 \\
0 & \boxed{\zeta_3^2} & 0 & 0 & 0
\addpath{(0,1,0)rurruurur}
\end{BMAT}
\right) \in G(3,3,5).
$\\ For all $[i,c] \in Z'(w)$, we have $w[i,c]$ is equal to $1$ or $\zeta_3^2$ (these are the boxed entries of $w$). It follows immediately that $w \in [1,\lambda^2]$.

\end{example}

\begin{example}

Let $w =
\left(
\begin{BMAT}{cccc}{cccc}
\encircle{$\zeta_3^2$} & 0 & 0  & 0 \\
0 & 0 & \zeta_3 & 0 \\
0 & \zeta_3 & 0 & 0 \\
0 & 0 & 0 & \boxed{\zeta_3^2} 
\addpath{(0,3,0)rurrr}
\end{BMAT}
\right) \in G(3,3,4).$\\ There exists $[i,c] \in Z'(w)$ such that $w[i,c] = \zeta_3^2$ (the boxed element in $w$). It follows immediately that $w \notin [1,\lambda]$. Moreover, there exists $[i',c'] \in Z'(w)$ such that $w[i',c'] = \zeta_3$. Hence $w \notin [1,\lambda^2]$. 

\end{example}

\begin{remark}\label{RemDivisorsLambdaG(2,2,n)}

Let $w$ be an element of the Coxeter group $G(2,2,n)$. The nonzero elements $w[i,c]$ with $[i,c] \in Z'(w)$ are always equal to $1$ or $-1$. Hence by Proposition \ref{PropZ'=1orZetae}, all the elements of $G(2,2,n)$ are left divisors of the unique element of maximal length $\lambda$ in $G(2,2,n)$, see Remark \ref{RemUniqueMaximalG(2,2,n)}.\\ For example,
Let $w =
\left(
\begin{BMAT}{ccccc}{ccccc}
0 & 0 & 0  & \encircle{$1$} & 0 \\
0 & 0 & 0 & 0 & \boxed{1} \\
0 & 0 & \encircle{$-1$} & 0 & 0 \\
\encircle{$1$} & 0 & 0 & 0 & 0 \\
0 & \boxed{-1} & 0 & 0 & 0
\addpath{(0,1,0)rurruurur}
\end{BMAT}
\right)$ be an element of $G(2,2,4)$. Since all the nonzero elements $w[i,c]$ with $[i,c] \in Z'(w)$ are equal to $1$ or $-1$, it follows immediately that $w \preceq \lambda$.

\end{remark}

\medskip

Our description of the interval $[1, \lambda^k]$ allows us to prove easily that $\lambda^k$ is balanced. Let us recall the definition of a balanced element.

\begin{definition}\label{DefBalancedElement}

A balanced element in $G(e,e,n)$ is an element $w$ such that $w' \preceq w$ holds precisely when $w \preceq_r w'$. 

\end{definition}

\noindent The next lemma is obvious.

\begin{lemma}\label{LemmaDivisorsForInduction}

Let $g$ be a balanced element and let $w, w' \in [1,g]$. If $w' \preceq w$, then $(w')^{-1}w \preceq g$.

\end{lemma}

In order to prove that $\lambda^k$ is balanced, we first check the following.

\begin{lemma}\label{LemmaforPropBalanced}

If $w \in D_k$, we have $w^{-1}\lambda^k \in D_k$ and $\lambda^k w^{-1} \in D_k$.

\end{lemma}

\begin{proof}

Let $w \in D_k$. We show that $w^{-1}\lambda^k = \overline{t_w}\lambda^k \in D_k$ and $\lambda^k w^{-1} = \lambda^k \overline{t_w} \in D_k$, where $\overline{t_w}$ is the complex conjugate of the transpose $t_w$ of the matrix $w$. We use the matrix form of an element of $D_k$. If $[i,c]$ is a bullet of $w$, then $[c,i]$ is a bullet of $\overline{t_w} = w^{-1}$ and $w^{-1}[c,i] = \overline{w[i,c]}$. Then, if $[i,c] \in Z'(w^{-1})$, we have $[c,i] \in Z'(w)$. Since $w \in D_k$, we have $w[c,i] \in \left\{ 0,1,\zeta_e^k \right\}$ whenever $[c,i] \in Z'(w)$. Then $w^{-1}[i,c] \in \left\{ 0,1,\zeta_e^{-k} \right\}$ whenever $[i,c] \in Z'(w^{-1})$. Multiplying $w^{-1}$ by $\lambda^k$, we get that $(w^{-1}\lambda^k)[i,c]$ and $(\lambda^k w^{-1})[i,c]$ are equal to $0$, $1$, or $\zeta_e^k$ whenever $[i,c] \in Z'(w^{-1}\lambda^k)$ and $Z'(\lambda^kw^{-1})$. Hence $w^{-1}\lambda^k$ and $\lambda^kw^{-1}$ belong to $D_k$.

\end{proof}

\begin{example}

We illustrate the idea of the proof of Lemma \ref{LemmaforPropBalanced} for $k=1$.\\
Consider $w \in D_1$ as follows and show that $\overline{t_w}\lambda \in D_1$:
$w =
\left(
\begin{BMAT}{ccccc}{ccccc}
0 & 0 & 0  & 0 & \encircle{$1$} \\
0 & \encircle{$1$} & 0 & 0 & 0 \\
0 & 0 & 0 & \encircle{$\zeta_3$} & 0 \\
\encircle{$\zeta_3^2$} & 0 & 0 & 0 & 0 \\
0 & 0 & \boxed{1} & 0 & 0
\addpath{(0,1,0)ruururrru}
\end{BMAT}
\right)$ $\overset{t_w}{\longrightarrow}$ $\left(
\begin{BMAT}{ccccc}{ccccc}
0 & 0 & 0  & \encircle{$\zeta_3^2$} & 0 \\
0 & \encircle{$1$} & 0 & 0 & 0 \\
0 & 0 & 0 & 0 & \boxed{1} \\
0 & 0 &  \boxed{\zeta_3} & 0 & 0 \\
\encircle{$1$} & 0 & 0 & 0 & 0
\addpath{(0,0,0)ruuururrur}
\end{BMAT}
\right)$ $\overset{\overline{t_w}}{\longrightarrow}$ $\left(
\begin{BMAT}{ccccc}{ccccc}
0 & 0 & 0  & \encircle{$\zeta_3$} & 0 \\
0 & \encircle{$1$} & 0 & 0 & 0 \\
0 & 0 & 0 & 0 & \boxed{1} \\
0 & 0 & \boxed{\zeta_3^2} & 0 & 0 \\
\encircle{$1$} & 0 & 0 & 0 & 0
\addpath{(0,0,0)ruuururrur}
\end{BMAT}
\right)$ $\overset{\overline{t_w}\lambda}{\longrightarrow}$ $\left(
\begin{BMAT}{ccccc}{ccccc}
0 & 0 & 0  & \encircle{$1$} & 0 \\
0 & \encircle{$\zeta_3$} & 0 & 0 & 0 \\
0 & 0 & 0 & 0 & \boxed{\zeta_3} \\
0 & 0 & \boxed{1} & 0 & 0 \\
\encircle{$\zeta_3$} & 0 & 0 & 0 & 0
\addpath{(0,0,0)ruuururrur}
\end{BMAT}
\right)$.

\end{example}

\medskip

\begin{proposition}\label{Propbalanced}

The element $\lambda^k$ is balanced.

\end{proposition}

\begin{proof}

Suppose that $w \preceq \lambda^k$. By Proposition \ref{PropDisDivLambda}, we have $w \in D_k$, so $\lambda^k w^{-1}$ is in $D_k$ by Lemma \ref{LemmaforPropBalanced}. Hence $\lambda^k = (\lambda^k w^{-1})w$ satisfies $\ell(\lambda^k w^{-1}) + \ell(w) = \ell(\lambda^k)$, namely $w \preceq_r \lambda^k$. 

Conversely, suppose that $w \preceq_r \lambda^k$. We have $\lambda^k = w' w$ with $w' \in G(e,e,n)$ and $\ell(w') + \ell(w) = \ell(\lambda^k)$. It follows that $w' \in D_k$, then ${w'}^{-1} \lambda^k \in D_k$ by Lemma \ref{LemmaforPropBalanced}. Since  $w = {w'}^{-1}\lambda^k$, we have $w \in D_k$, namely $w \preceq \lambda^k$.

\end{proof}

In the following theorem, we show that the elements $\lambda^k$ are the only balanced elements of maximal length of $G(e,e,n)$ for $1 \leq k \leq e-1$.


\begin{theorem}\label{PropAllBalancedofMaxLength}

The balanced elements of $G(e,e,n)$  that are of maximal length are precisely $\lambda^k$ with $1 \leq k \leq e-1$. The set $D_k$ of the divisors of each $\lambda^k$ is characterized in Propositions \ref{PropCaractD} and \ref{PropZ'=1orZetae}.

\end{theorem}

\begin{proof}

Let $w \in G(e,e,n)$ be an element of $G(e,e,n)$ of maximal length, namely by Proposition \ref{PropMaxLength}, a diagonal matrix such that for $2 \leq i \leq n$, $w[i,i]$ is an $e$-th root of unity different from $1$. Analogously to Proposition \ref{PropZ'=1orZetae}, a left divisor $w'$ of $w$ satisfies that for all $2 \leq i \leq n$, if $[i,c]$ is not a bullet of $w'$, then $w'[i,c]$ is either $0$, $1$, or $w[i,i]$.

By Proposition \ref{Propbalanced}, we already have that $\lambda^k$ is balanced for $1 \leq k \leq e-1$. Suppose that $w$ is of maximal length such that $w[i,i] \neq w[j,j]$ for $2 \leq i,j \leq n$ and $i \neq j$. Let $s_{ij}$ be the transposition matrix. We have $s_{ij}[1,1]$ is nonzero and so, by Remark \ref{RemBulletsFirstRowColumn}, $[1,1]$ is the unique bullet of $s_{ij}$. Hence if $[j,d]$ is not a bullet of $s_{ij}$, then $s_{ij}[j,d]$ is either $0$ or $1$. So, $s_{ij}$ left-divides $w$. Hence $w' := s_{ij}^{-1}w = s_{ij}w$ right-divides $w$. We also have $w'[1,1]$ is nonzero, and so $[1,1]$ is the unique bullet of $w'$. Thus, $[j,i]$ is not a bullet and $w'[j,i] = w[i,i]$ is neither $0$, $1$, nor $w[j,j]$ (which was assumed different from $w[i,i]$). So $w' \npreceq w$, and so $w$ is not balanced.

It follows that the balanced elements of $G(e,e,n)$ that are of maximal length are precisely $\lambda^k$ with $1 \leq k \leq e-1$.

\end{proof}

We are ready to study the interval structures associated with the intervals $[1,\lambda^k]$ with $1 \leq k \leq e-1$.

\section{Interval structures}

In this section, we construct the monoid $M([1,\lambda^k])$ associated to each of the intervals $[1,\lambda^k]$ constructed in the previous section with $1 \leq k \leq e-1$. By Proposition \ref{PropAllBalancedofMaxLength}, $\lambda^k$ is balanced. Hence, by Theorem \ref{TheoremMichelGarside}, in order to prove that $M([1,\lambda^k])$ is a Garside monoid, it remains to show that both posets $([1,\lambda^k],\preceq)$ and $([1,\lambda^k],\preceq_r)$ are lattices. This is stated in Corollary \ref{CorBothPosetsLattices}. The interval structures are given \mbox{in Theorem \ref{TheoremIntervalStructure}.}

\subsection{Least common multiples}

Let $1 \leq k \leq e-1$ and let $w \in [1,\lambda^k]$. For each $1 \leq i \leq n$ there exists a unique $c_i$ such that $w[i,c_i] \neq 0$. We denote $w[i,c_i]$ by $a_i$. We apply Lemma \ref{LemmaCaractDivlambda} to prove the following lemmas. The reader is invited to write the matrix form of $w$ to illustrate each step of the proof.

\begin{lemma}\label{tit0Divw}

Let $t_i \preceq w$ where $i \in \mathbb{Z}/e\mathbb{Z}$.

\begin{itemize}

\item If $c_1 < c_2$, then $t_kt_0 \preceq w$ and
\item if $c_2 < c_1$, then  $t_j \npreceq w$ for all $j$ with $j \neq i$.

\end{itemize}

Hence if $t_i \preceq w$ and $t_j \preceq w$ with $i,j \in \mathbb{Z}/e\mathbb{Z}$ and $i \neq j$, then $t_kt_0 \preceq w$.

\end{lemma}

\begin{proof}

\underline{Suppose $c_1 < c_2$.}\\
Since $a_1 = w[1,c_1]$ is nonzero , and above and to the left of $a_2 = w[2,c_2]$ (as $c_1 < c_2$), then $[2,c_2]$ is not a bullet. It belongs to $Z'(w)$. Since $w \in [1,\lambda^k]$ and $[2,c_2] \in Z'(w)$, by Proposition \ref{PropZ'=1orZetae}, $a_2 = 1$ or $\zeta_e^k$. Since $t_i \preceq w$, we have $\ell(t_iw) = \ell(w) - 1$. Hence by $2$ of \mbox{Proposition \ref{PropLengthdecreas}}, we get $a_2 \neq 1$. Hence $a_2 = \zeta_e^k$.
Again by $2$ of Proposition \ref{PropLengthdecreas}, since $a_2 \neq 1$, we have $\ell(t_kw) = \ell(w) - 1$. Let $w' := t_kw$. We have $w'[1,c_2] = \zeta_e^{-k}a_2 = \zeta_e^{-k} \zeta_e^{k} = 1$. Hence by $3$ of Proposition $\ref{PropLengthdecreas}$, $\ell(t_0w') = \ell(w') - 1$. It follows that $t_kt_0 \preceq w$.\\
\underline{Suppose $c_2 < c_1$.}\\
Since $t_i \preceq w$, we have $\ell(t_iw) = \ell(w) - 1$. Hence by $3$ of Proposition \ref{PropLengthdecreas}, we have $a_1 = \zeta_e^{-i}$. If there exists $j \in \mathbb{Z}/e\mathbb{Z}$ with $j \neq i$ such that $t_j \preceq w$, then $\ell(t_jw) = \ell(w) - 1$. Again by $3$ of Proposition \ref{PropLengthdecreas}, we have $a_1 = \zeta_e^{-j}$. Thus, $i = j$ which contradicts the hypothesis.\\
The last statement of the lemma follows immediately.

\end{proof}

\begin{lemma}\label{tis3Divw}

If $t_i \preceq w$ with $i \in \mathbb{Z}/e\mathbb{Z}$ and $s_3 \preceq w$, then $s_3t_is_3 = t_is_3t_i \preceq w$.

\end{lemma}

\begin{proof}

Set $w' := s_3w$ and $w'':=t_iw'$.\\
\underline{Suppose $c_1 < c_2$.}\\
Since $w \in [1,\lambda^k]$ and $[2,c_2] \in Z'(w)$, by Proposition \ref{PropZ'=1orZetae}, we get $a_2 = 1$ or $\zeta_e^k$. Since $t_i \preceq w$, we have $\ell(t_iw) = \ell(w) - 1$. Thus, by $2$ of Proposition \ref{PropLengthdecreas}, we get $a_2 \neq 1$. Hence $a_2 = \zeta_e^k$.\\
Suppose that $c_3 < c_2$. Since $s_3 \preceq w$, we have $\ell(s_3w) = \ell(w) - 1$. Hence by $1(b)$ of Proposition \ref{PropLengthdecreas}, $a_2 = 1$ which is not the case. So instead it must be that $c_2 < c_3$. Assume $c_2 < c_3$. Since $c_1 < c_2 < c_3$, $a_1 = w[1,c_1]$ is a nonzero entry which is above and to the left of $a_3 = w[3,c_3]$. Then $[3,c_3]$ is in $Z'(w)$. Since $w \in [1, \lambda^k]$ and $[3,c_3] \in Z'(w)$, we have $a_3 = 1$ or $\zeta_e^k$. By $1(a)$ of Proposition \ref{PropLengthdecreas}, we have $a_3 \neq 1$. Hence $a_3$ is equal to $\zeta_e^k$. Now, we prove that $s_3t_is_3 \preceq w$ by applying Lemma \ref{LemmaCaractDivlambda}. Indeed, since $a_3 \neq 1$, by $2$ of Proposition \ref{PropLengthdecreas}, we have $\ell(t_i w') = \ell(w') -1$, and since $a_2 \neq 1$, by $1(a)$ of Proposition \ref{PropLengthdecreas}, we have $\ell(s_3 w'') = \ell(w'') -1$.\\
\underline{Suppose $c_2 < c_1$.}\\
Since $\ell(t_iw) = \ell(w) -1$, by $3$ of Proposition \ref{PropLengthdecreas}, we have $a_1 = \zeta_e^{-i}$.
\begin{itemize}

\item \underline{Assume $c_2 < c_3$.}\\
Since $\ell(s_3 w) = \ell(w) -1$, by $1(a)$ of Proposition \ref{PropLengthdecreas}, we have $a_3 \neq 1$. We have $\ell(t_iw') = \ell(w') - 1$ for both cases $c_1 < c_3$ and $c_3 < c_1$. Actually, if $c_1 < c_3$, since $a_3 \neq 1$, by $2$ of Proposition \ref{PropLengthdecreas}, we have $\ell(t_iw') = \ell(w') -1$, and if $c_3 < c_1$, since $a_1 = \zeta_e^{-i}$, by $3$ of Proposition \ref{PropLengthdecreas}, $\ell(t_i w') = \ell(w') -1$. Now, since $\zeta_e^{i}a_1 = \zeta_e^{i}\zeta_e^{-i} = 1$, by $1(b)$ of Proposition \ref{PropLengthdecreas}, we have $\ell(s_3 w'') = \ell(w'') -1$.

\item \underline{Assume $c_3 < c_2$.}\\
Since $a_1 = \zeta_e^{-i}$, by $3$ of Proposition \ref{PropLengthdecreas}, $\ell(t_iw') = \ell(w') -1$. Since $\zeta_e^{i}a_1 = 1$, by $1(b)$ of Proposition \ref{PropLengthdecreas}, we have $\ell(s_3 w'') = \ell(w'') -1$.

\end{itemize}

\end{proof}

\begin{lemma}\label{tisjDivw}

If $t_i \preceq w$ with $i \in \mathbb{Z}/e\mathbb{Z}$ and $s_j \preceq w$ with $4 \leq j \leq n$, then \mbox{$t_i s_j = s_j t_i \preceq w$}.

\end{lemma}

\begin{proof}

We distinguish four different cases: case 1: $c_1 < c_2$ and $c_{j-1} < c_j$, case 2: $c_1 < c_2$ and $c_{j} < c_{j-1}$, case 3: $c_2 < c_1$ and $c_{j-1} < c_j$, and case 4: $c_2 < c_1$ and $c_{j} < c_{j-1}$. The proof is similar to the proofs of Lemmas \ref{tit0Divw} and \ref{tis3Divw} so we prove that $s_jt_i \preceq w$ only for the first case. Suppose that $c_1 < c_2$ and $c_{j-1} < c_j$. Since $t_i \preceq w$, we have $\ell(t_iw) = \ell(w) -1$. Hence, by $2$ of Proposition \ref{PropLengthdecreas}, we have $a_2 \neq 1$. Also, since $s_j \preceq w$, we have $\ell(s_jw) = \ell(w) -1$. Hence, by $1(a)$ of Proposition \ref{PropLengthdecreas}, we have $a_j \neq 1$. Set $w' := s_jw$. Since $a_2 \neq 1$, then $\ell(t_iw') = \ell(w') -1$. Hence $s_j t_i \preceq w$.

\end{proof}

\noindent The proof of the following lemma is similar to the proofs of Lemmas \ref{tis3Divw} and \ref{tisjDivw}.

\begin{lemma}\label{sis(i+1)sisjdivw}

If $s_i \preceq w$ and $s_{i+1} \preceq w$ for $3 \leq i \leq n-1$, then \mbox{$s_is_{i+1}s_i = s_{i+1}s_is_{i+1} \preceq w$}, and if $s_i \preceq w$ and $s_j \preceq w$ for $3 \leq i,j \leq n$ and $|i-j| > 1$, then $s_is_j = s_js_i \preceq w$.

\end{lemma}

The following proposition is a direct consequence of all the preceding lemmas.

\begin{proposition}\label{PropLCMofGenIn1lambdak}

Let $x, y \in X = \{t_0,t_1, \cdots, t_{e-1},s_3, \cdots, s_n \}$. The least common multiple in $([1, \lambda^k], \preceq)$ of $x$ and $y$, denoted by $x \vee y$, exists and is given by the following identities:

\begin{itemize}

\item $t_i \vee t_j = t_kt_0 = t_{i}t_{i-k} = t_{j}t_{j-k}$ for $i \neq j \in \mathbb{Z}/e\mathbb{Z}$,
\item $t_i \vee s_3 = t_is_3t_i = s_3t_is_3$ for $i \in \mathbb{Z}/e\mathbb{Z}$,
\item $t_i \vee s_j = t_is_j = s_jt_i$ for $i \in \mathbb{Z}/e\mathbb{Z}$ and $4 \leq j \leq n$,
\item $s_i \vee s_{i+1} = s_{i}s_{i+1}s_{i} = s_{i+1}s_{i}s_{i+1}$ for $3 \leq i \leq n-1$, and
\item $s_i \vee s_j = s_is_j = s_js_i$ for $3 \leq i \neq j \leq n$ and $|i-j| > 1$.

\end{itemize}

\end{proposition}

We have a similar result for $([1,\lambda^k],\preceq_r)$.

\begin{proposition}\label{PropLCMofGenRight}

Let $x, y \in X$. The least common multiple in $([1, \lambda^k], \preceq_r)$ of $x$ and $y$, denoted by $x \vee_r y$, exists and is equal to $x \vee y$.

\end{proposition}

\begin{proof}

Define a map on the generators of $G(e,e,n)$ by $\phi: t_i \longmapsto t_{-i}$ and $s_j \longmapsto s_j$ with $i \in \mathbb{Z}/e\mathbb{Z}$ and $3 \leq j \leq n$. On examination of the relations for $G(e,e,n)$, it is quickly verified that this map extends to an anti-homomorphism $G(e,e,n) \longrightarrow G(e,e,n)$. Clearly $\phi^2$ is the identity, so in particular, $\phi$ respects the length function on $G(e,e,n)$: $\ell(\phi(w)) = \ell(w)$. Suppose that $x,y \in X$ and $x \preceq_r w$ and $y \preceq_r w$, we will show that $x \vee y \preceq_r w$. Write $w = vx$ and $w=v'y$ with $v, v' \in G(e,e,n)$. Thus, $\phi(w) = \phi(x) \phi(v) = \phi(y) \phi(v')$, and since $\phi$ respects length, $\phi(x) \preceq \phi(w)$ and $\phi(y) \preceq \phi(w)$. Hence $\phi(x) \vee \phi(y) \preceq \phi(w)$.

For each pair of generators, it is straightforward to check that\\ $\phi(x) \vee \phi(y) = \phi(x \vee y)$: the only non-trivial case being when $x = t_i$ and $y = t_j$ for $i \neq j$, when we have: $$\phi(t_i) \vee \phi(t_j) = t_{-i} \vee t_{-j} = t_kt_0 = \phi(t_0t_{-k}) = \phi(t_kt_0).$$ Thus, $\phi(x \vee y) \preceq \phi(w)$, that is $\phi(w) = \phi(x \vee y)u$ for some $u \in G(e,e,n)$. Applying $\phi$ again gives $w = \phi(u)(x \vee y)$, and since $\phi$ respects length, $x \vee y \preceq_r w$.

\end{proof}

Note that Propositions \ref{PropLCMofGenIn1lambdak} and \ref{PropLCMofGenRight} are important to prove that both posets $([1,\lambda^k],\preceq)$ and $([1,\lambda^k],\preceq_r)$ are lattices. Actually, they will make possible an induction proof of Proposition \ref{PropInductionExistLCM} in the next subsection.

\subsection{The lattice property and interval structures}

We start by recalling some general properties about lattices that will be useful in our proof. Let $(S, \preceq)$ be a finite poset. 

\begin{definition}\label{DefMeetSemiLattice}

Say that $(S, \preceq)$ is a meet-semilattice if and only if $f \wedge g := gcd(f,g)$ exists for any $f, g \in S$.

\end{definition} 

Equivalently, $(S, \preceq)$ is a meet-semilattice if and only if $\bigwedge T$ exists for any finite nonempty subset $T$ of $S$.

\begin{definition}\label{DefJoinSemiLattice}

Say that $(S, \preceq)$ is a join-semilattice if and only if $f \vee g := lcm(f,g)$ exists for any $f, g \in S$.

\end{definition}

Equivalently, $(S, \preceq)$ is a join-semilattice if and only if $\bigvee T$ exists for any finite nonempty subset $T$ of $S$.

\begin{proposition}\label{PropCaractMeetSemiLattice}

Let $(S, \preceq)$ be a finite poset. $(S, \preceq)$ is a meet-semilattice if and only if for any $f, g \in S$, either $f \vee g$ exists, or $f$ and $g$ have no common multiples.

\end{proposition}

\begin{proof}

Let $f, g \in S$ and suppose that $f$ and $g$ have at least one common multiple. Let $A := \{h \in S\ |\ f \preceq h$ and $g \preceq h \}$ be the set of the common multiples of $f$ and $g$. Since $S$ is finite, $A$ is also finite. Since $(S, \preceq)$ is a meet-semilattice, $\bigwedge A$ exists and $\bigwedge A = lcm(f,g) = f \vee g$.\\
Conversely, let $f, g \in S$ and let $B := \{h \in S\ |\ h \preceq f$ and $h \preceq g \}$ be the set of all common divisors of $f$ and $g$. Since all elements of $B$ have common multiples, $\bigvee B$ exists and we have $\bigvee B = gcd(f,g) = f \wedge g$.

\end{proof}

\begin{definition}\label{DefFromSemiLatToLat}

The poset $(S, \preceq)$ is a lattice if and only if it is both a meet- and join- semilattice.

\end{definition}

The following lemma is a consequence of Proposition \ref{PropCaractMeetSemiLattice}.

\begin{lemma}\label{LemmaSemiLatticeLattice}

If $(S, \preceq)$ is a meet-semilattice such that $\bigvee S$ exists, then $(S, \preceq)$ is a lattice.\\

\end{lemma}

We will prove that $([1,\lambda^k],\preceq)$ is a meet-semilattice by applying Proposition \ref{PropCaractMeetSemiLattice}. For $1 \leq m \leq \ell(\lambda^k)$ with $\ell(\lambda^k) = n(n-1)$, we introduce 

\begin{center} $([1,\lambda^k])_m := \{w \in [1,\lambda^k]\ s.t.\ \ell(w) \leq m \}.$ \end{center}

\begin{proposition}\label{PropInductionExistLCM}

Let $0 \leq k \leq e-1$. For $1 \leq m \leq n(n-1)$ and $u,v$ in $([1,\lambda^k])_m$, either $u \vee v$ exists in $([1,\lambda^k])_m$, or $u$ and $v$ do not have common multiples in $([1,\lambda^k])_m$.

\end{proposition}

\begin{proof}

Let $1 \leq m \leq n(n-1)$. We make a proof by induction on $m$. By Proposition \ref{PropLCMofGenIn1lambdak}, our claim holds for $m=1$. Suppose $m > 1$. Assume that the claim holds for all $1 \leq m' \leq m-1$. We want to prove it for $m' = m$. The proof is illustrated in Figure \ref{fig2} below. Let $u, v$ be in $([1,\lambda^k])_m$ such that $u$ and $v$ have at least one common multiple in $([1,\lambda^k])_m$ which we denote by $w$. Write $u= xu_1$ and $v = y v_1$ such that $x, y \in X$ and $\ell(u) = \ell(u_1) + 1$, $\ell(v) = \ell(v_1) + 1$. By Proposition \ref{PropLCMofGenIn1lambdak}, $x \vee y$ exists. Since $x \preceq w$ and $y \preceq w$, $x \vee y$ divides $w$. We can write $x \vee y = xy_1 = yx_1$ with $\ell(x \vee y) = \ell(x_1) + 1 = \ell(y_1) + 1$. By Lemma \ref{LemmaDivisorsForInduction}, we have $x_1, v_1 \in [1, \lambda^k]$. Also, we have $\ell(x_1) < m$, $\ell(v_1) < m$ and $x_1, v_1$ have a common multiple in $([1,\lambda^k])_{m-1}$. Thus, by the induction hypothesis, $x_1 \vee v_1$ exists in $([1,\lambda^k])_{m-1}$. Similarly, $y_1 \vee u_1$ exists in $([1,\lambda^k])_{m-1}$. Write $x_1 \vee v_1 = v_1x_2 = x_1v_2$ with $\ell(x_1 \vee v_1) = \ell(v_1) + \ell(x_2) = \ell(v_2) + \ell(x_1)$ and write $y_1 \vee u_1 = u_1y_2 = y_1u_2$ with $\ell(y_1 \vee u_1) = \ell(y_1) + \ell(u_2) = \ell(u_1) + \ell(y_2)$. By Lemma \ref{LemmaDivisorsForInduction}, we have $u_2, v_2 \in [1, \lambda^k]$. Also, we have $\ell(u_2) < m$, $\ell(v_2) < m$ and $u_2, v_2$ have a common multiple in $([1,\lambda^k])_{m-1}$. Thus, by the induction hypothesis, $u_2 \vee v_2$ exists in $([1,\lambda^k])_{m-1}$. Write $u_2 \vee v_2 = v_2 u_3 = u_2 v_3$ with $\ell(u_2 \vee v_2) = \ell(v_2) + \ell(u_3) = \ell(u_2) + \ell(v_3)$. Since $uy_2v_3 = vx_2u_3$ is a common multiple of $u$ and $v$ that divides every common multiple $w$ of $u$ and $v$, we deduce that $u \vee v = uy_2v_3 = vx_2u_3$ and we are done.

\end{proof}

\begin{figure}[H]

\begin{small}
\begin{center}

\begin{tikzpicture}[scale=0.7]

\draw [->] (0,0) to node[auto] {$y_2$} (1,0);
\draw [->] (1,0) to node[auto] {$v_3$} (4,0);
\draw [->] (0,2) to node[left] {$u_1$} (0,0);
\draw [->] (1,2) to node[left] {$u_2$} (1,0);
\draw [->] (4,2) to node[left] {$u_3$} (4,0);
\draw [->] (0,2) to node[auto] {$y_1$} (1,2);
\draw [->] (1,2) to node[auto] {$v_2$} (4,2);
\draw [->] (0,3) to node[left] {$x$} (0,2);
\draw [->] (1,3) to node[auto] {$x_1$} (1,2);
\draw [->] (4,3) to node[left] {$x_2$} (4,2);
\draw [->] (0,3) to node[auto] {$y$} (1,3);
\draw [->] (1,3) to node[auto] {$v_1$} (4,3);
\draw [-] (4,0) -- (6,-2) node[right] {$w$};
\draw [->, dashed] (0,0) to[bend right] (6,-2);
\draw [-, dashed] (1,0) to[bend right] (6,-2);
\draw [-, dashed] (1,2) to[bend right] (6,-2);
\draw [-, dashed] (4,2) to[bend left] (6,-2);
\draw [->, dashed] (4,3) to[bend left] (6,-2);

\end{tikzpicture}

\end{center}
\end{small}

\caption{The proof of Proposition \ref{PropInductionExistLCM}.}
\label{fig2}
\end{figure}
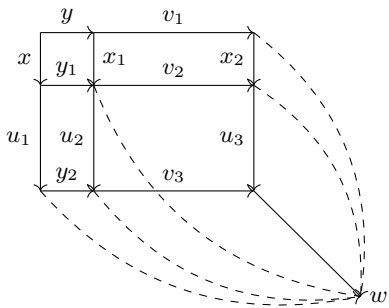

\noindent Similarly, applying Proposition \ref{PropLCMofGenRight}, we obtain the same results for $([1,\lambda^k],\preceq_r)$. We thus proved the following.

\begin{cor}\label{CorBothPosetsLattices}

Both posets $([1,\lambda^k],\preceq)$ and $([1,\lambda^k],\preceq_r)$ are lattices. 

\end{cor}

\begin{proof}

Applying Proposition \ref{PropCaractMeetSemiLattice}, $([1,\lambda^k],\preceq)$ is a meet-semilattice. Also, by definition of the interval $[1,\lambda^k]$, we have $\bigvee [1,\lambda^k] = \lambda^k$. Thus, applying Proposition \ref{PropCaractMeetSemiLattice}, $([1,\lambda^k],\preceq)$ is a lattice. The same can be done for $([1,\lambda^k],\preceq_r)$.

\end{proof}

We are ready to define the interval monoid $M([1,\lambda^k])$.

\begin{definition}\label{DefIntervalMonoid}

Let $\underline{D_k}$ be a set in bijection with $D_k = [1,\lambda^k]$ with $$[1,\lambda^k] \longrightarrow \underline{D_k}: w \longmapsto \underline{w}.$$ We define the monoid $M([1,\lambda^k])$ by the following presentation of monoid with

\begin{itemize}

\item generating set: $\underline{D_k}$ (a copy of the interval $[1,\lambda^k]$) and
\item relations: $\underline{w} = \underline{w'} \hspace{0.1cm} \underline{w''}$ whenever $w,w'$ and $w'' \in [1,\lambda^k]$, $w=w'w''$ and $\ell(w) = \ell(w') + \ell(w'')$.

\end{itemize}

\end{definition}

We have that $\lambda^k$ is balanced. Also, by Corollary \ref{CorBothPosetsLattices}, both posets $([1,\lambda^k],\preceq)$ and $([1,\lambda^k],\preceq_r)$ are lattices. Hence, by Theorem \ref{TheoremMichelGarside}, we have:

\begin{theorem}\label{TheoremIntervalStructure}

$(M([1,\lambda^k]),\underline{\lambda^k})$ is an interval Garside monoid with simples $\underline{D_k}$, where $\underline{D_k}$ is given in Definition \ref{DefIntervalMonoid}. Its group of fractions exists and is denoted by $G(M([1,\lambda^k]))$.

\end{theorem}

These interval structures have been implemented by Michel and the author using the package CHEVIE of GAP3 (see \cite{CHEVIEJMichel} and \cite{GAP3CPMichelNeaime}). We provide this implementation in Appendix A. The next section is devoted to the study of these interval structures.

\section{About the interval structures}

In this section, we provide a new presentation for the interval monoid $M([1,\lambda^k])$. Furthermore, we prove that $G(M([1,\lambda^k]))$ is isomorphic to the complex braid group $B(e,e,n)$ if and only if $k \wedge e = 1$ ($k$ and $e$ are relatively prime). When $k \wedge e \neq 1$, we describe these new structures and show some of their properties.

\subsection{Presentations}

Our first aim is to prove that the interval monoid $M([1,\lambda^k])$ is isomorphic to the monoid $B^{\oplus k}(e,e,n)$ defined as follows.

\begin{definition}\label{DefofB+keen}

For $1 \leq k \leq e-1$, we define the monoid $B^{\oplus k}(e,e,n)$ by a presentation of monoid with

\begin{itemize}

\item generating set: $\widetilde{X} = \{ \tilde{t}_{0}, \tilde{t}_{1}, \cdots, \tilde{t}_{e-1}, \tilde{s}_{3}, \cdots, \tilde{s}_{n}\}$ and
\item relations: $\left\{
\begin{array}{ll}

\tilde{s}_{i}\tilde{s}_{j}\tilde{s}_{i} = \tilde{s}_{j}\tilde{s}_{i}\tilde{s}_{j} & for\ |i-j|=1,\\
\tilde{s}_{i}\tilde{s}_{j} = \tilde{s}_{j}\tilde{s}_{i} & for\ |i-j| > 1,\\
\tilde{s}_{3}\tilde{t}_{i}\tilde{s}_{3} = \tilde{t}_{i}\tilde{s}_{3}\tilde{t}_{i} & for\ i \in \mathbb{Z}/e\mathbb{Z},\\
\tilde{s}_{j}\tilde{t}_{i} = \tilde{t}_{i}\tilde{s}_{j} & for\ i \in \mathbb{Z}/e\mathbb{Z}\ and\ 4 \leq j \leq n,\\
\tilde{t}_{i}\tilde{t}_{i-k} = \tilde{t}_{j}\tilde{t}_{j-k} & for\ i, j \in \mathbb{Z}/e\mathbb{Z}.

\end{array}
\right.$

\end{itemize}

\end{definition}

Note that the monoid $B^{\oplus 1}(e,e,n)$ is the monoid $B^{\oplus}(e,e,n)$ of Corran and Picantin, see \cite{CorranPicantin}. It is shown in \cite{CorranPicantin} that this monoid gives rise to a Garside structure for $B(e,e,n)$. The diagram corresponding to the presentation of $B^{\oplus}(e,e,n)$ is as follows. The dashed circle between $\tilde{t}_0, \tilde{t}_1, \cdots , \tilde{t}_{e-1}$ describes the relation $\tilde{t}_{i}\tilde{t}_{i-1} = \tilde{t}_{j}\tilde{t}_{j-1}$. The other edges of the diagram follow the standard conventions for Artin-Tits monoids.

\begin{figure}[H]
\begin{center}
\begin{tikzpicture}[yscale=0.8,xscale=1,rotate=30]

\draw[thick,dashed] (0,0) ellipse (2cm and 1cm);

\node[draw, shape=circle, fill=white, label=above:$\tilde{t}_0$] (t0) at (0,-1) {};
\node[draw, shape=circle, fill=white, label=above:$\tilde{t}_1$] (t1) at (1,-0.8) {};
\node[draw, shape=circle, fill=white, label=right:$\tilde{t}_2$] (t2) at (2,0) {};
\node[draw, shape=circle, fill=white, label=above:$\tilde{t}_i$] (ti) at (0,1) {};
\node[draw, shape=circle, fill=white, label=above:$\tilde{t}_{e-1}$] (te-1) at (-1,-0.8) {};

\draw[thick,-] (0,-2) arc (-180:-90:3);

\node[draw, shape=circle, fill=white, label=below left:$\tilde{s}_3$] (s3) at (0,-2) {};

\draw[thick,-] (t0) to (s3);
\draw[thick,-,bend left] (t1) to (s3);
\draw[thick,-,bend left] (t2) to (s3);
\draw[thick,-,bend left] (s3) to (te-1);

\node[draw, shape=circle, fill=white, label=below:$\tilde{s}_4$] (s4) at (0.15,-3) {};
\node[draw, shape=circle, fill=white, label=below:$\tilde{s}_{n-1}$] (sn-1) at (2.2,-4.9) {};
\node[draw, shape=circle, fill=white, label=right:$\tilde{s}_{n}$] (sn) at (3,-5) {};

\node[fill=white] () at (1,-4.285) {$\cdots$};

\end{tikzpicture}
\end{center}

\caption{Diagram for the presentation of Corran-Picantin of $B(e,e,n)$.}\label{PresofCPBeen}
\end{figure}
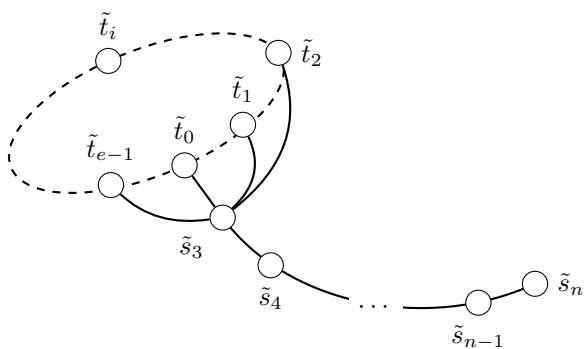

Fix $k$ such that $1 \leq k \leq e-1$. We define a diagram for the presentation of $B^{\oplus k}(e,e,n)$ in the same way as the diagram corresponding to the presentation of Corran and Picantin of $B(e,e,n)$ given in \mbox{Figure \ref{PresofCPBeen}}, with a dashed edge between $\tilde{t}_i$ and $\tilde{t}_{i-k}$ and between $\tilde{t}_j$ and $\tilde{t}_{j-k}$ for each relation of the form $\tilde{t}_{i}\tilde{t}_{i-k} = \tilde{t}_{j}\tilde{t}_{j-k}$, $i, j \in \mathbb{Z}/e\mathbb{Z}$. For example, the diagram corresponding to $B^{(2)}(8,8,2)$ is as follows.

\begin{small}
\begin{figure}[H]
\begin{center}
\begin{tikzpicture}

\node[label=$0$] (0) at (1.5,0) {$\bullet$};
\node[label=$1$] (1) at (1,1) {$\bullet$};
\node[label=$2$] (2) at (0,1.5) {$\bullet$};
\node[label=$3$] (3) at (-1,1) {$\bullet$};
\node[label=$4$] (4) at (-1.5,0) {$\bullet$};
\node[label=below:$5$] (5) at (-1,-1) {$\bullet$};
\node[label=below:$6$] (6) at (0,-1.5) {$\bullet$};
\node[label=below:$7$] (7) at (1,-1) {$\bullet$};

\draw[thick,dashed] (0) to (2);
\draw[thick,dashed] (2) to (4);
\draw[thick,dashed] (4) to (6);
\draw[thick,dashed] (6) to (0);

\draw[thick,dashed] (1) to (3);
\draw[thick,dashed] (3) to (5);
\draw[thick,dashed] (5) to (7);
\draw[thick,dashed] (7) to (1);

\end{tikzpicture}
\end{center}
\caption{Diagram for the presentation of $B^{(2)}(8,8,2)$.}\label{PresofB2882}
\end{figure}
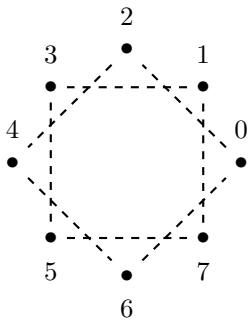
\end{small}




The following result is similar to Matsumoto's property in the case of real reflection groups, see \cite{MatsumotoTheorem}.

\begin{proposition}\label{PropMatsumoto}

There exists a map $F: [1,\lambda^k] \longrightarrow B^{\oplus k}(e,e,n)$ defined by\\ $F(w) = \tilde{x}_1 \tilde{x}_2 \cdots \tilde{x}_r$ whenever $\mathbf{x}_1 \mathbf{x}_2 \cdots \mathbf{x}_r$ is a reduced expression over $\mathbf{X}$ of $w$, where $\tilde{x}_i \in \widetilde{X}$ for $1 \leq i \leq r$.

\end{proposition}

\begin{proof}

Let $\mathbf{w}_1$ and $\mathbf{w}_2$ be in $\mathbf{X}^{*}$. We write $\mathbf{w}_1 \overset{\mathcal{B}}{\leadsto} \mathbf{w}_2$ if $\mathbf{w}_2$ is obtained from $\mathbf{w}_1$ by applying only the relations of the presentation of $B^{\oplus k}(e,e,n)$ where we replace $\tilde{t}_i$ by $\mathbf{t}_i$ and $\tilde{s}_j$ by $\mathbf{s}_j$ for all $i \in \mathbb{Z}/e\mathbb{Z}$ and $3 \leq j \leq n$.\\
Let $w$ be in $[1,\lambda^k]$ and suppose that $\mathbf{w}_1$ and $\mathbf{w}_2$ are two reduced expressions over $\mathbf{X}$ of $w$. We prove that $\mathbf{w}_1 \overset{\mathcal{B}}{\leadsto} \mathbf{w}_2$ by induction on $\boldsymbol{\ell}(\mathbf{w}_1)$.\\
The result holds vacuously for $\boldsymbol{\ell}(\mathbf{w}_1) = 0$ and $\boldsymbol{\ell}(\mathbf{w}_1) = 1$. Suppose that $\boldsymbol{\ell}(\mathbf{w}_1) > 1$. Write 
\begin{center} $\mathbf{w}_1 = \mathbf{x}_1\mathbf{w'}_1$ and $\mathbf{w}_2 = \mathbf{x}_2\mathbf{w'}_2$, with $\mathbf{x}_1, \mathbf{x}_2 \in \mathbf{X}$.\end{center}
\underline{If $\mathbf{x}_1 = \mathbf{x}_2$}, we have $x_1 w'_{1} = x_2 w'_{2}$ in $G(e,e,n)$ from which we get $w'_{1} = w'_{2}$. Then, by the induction hypothesis, we have $\mathbf{w'}_1 \overset{\mathcal{B}}{\leadsto} \mathbf{w'}_2$. Hence $\mathbf{w}_1 \overset{\mathcal{B}}{\leadsto} \mathbf{w}_2$.\\
\underline{If $\mathbf{x}_1 \neq \mathbf{x}_2$}, since $x_1 \preceq w$ and $x_2 \preceq w$, we have $x_1 \vee x_2 \preceq w$ where $x_1 \vee x_2$ is the lcm of $x_1$ and $x_2$ in $([1,\lambda^k],\preceq)$ given in Proposition \ref{PropLCMofGenIn1lambdak}. Write $w = (x_1 \vee x_2) w'$. Also, write $\mathbf{x}_1 \vee \mathbf{x}_2 = \mathbf{x}_1 \mathbf{v}_1$ and $\mathbf{x}_1 \vee \mathbf{x}_2 = \mathbf{x}_2 \mathbf{v}_2$ where we can check that $\mathbf{x}_1 \mathbf{v}_1 \overset{\mathcal{B}}{\leadsto} \mathbf{x}_2 \mathbf{v}_2$ for all possible cases for $\mathbf{x}_1$ and $\mathbf{x}_2$. All the words $\mathbf{x}_1\mathbf{w'}_1$, $\mathbf{x}_2\mathbf{w'}_2$, $\mathbf{x}_1\mathbf{v}_1\mathbf{w'}$, and $\mathbf{x}_2\mathbf{v}_2\mathbf{w'}$ represent $w$. In particular, $\mathbf{x}_1\mathbf{w'}_1$ and $\mathbf{x}_1\mathbf{v}_1\mathbf{w'}$ represent $w$. Hence $w'_1 = v_1w'$ and by the induction hypothesis, we have $\mathbf{w'}_1 \overset{\mathcal{B}}{\leadsto} \mathbf{v}_1\mathbf{w'}$. Thus, we have
\begin{center} $\mathbf{x}_1\mathbf{w'}_1 \overset{\mathcal{B}}{\leadsto} \mathbf{x}_1\mathbf{v}_1\mathbf{w'}$. \end{center}
Similarly, since $\mathbf{x}_2\mathbf{w'}_2$ and $\mathbf{x}_2\mathbf{v}_2\mathbf{w'}$ represent $w$, we get
\begin{center} $\mathbf{x}_2\mathbf{v}_2\mathbf{w'} \overset{\mathcal{B}}{\leadsto} \mathbf{x}_2\mathbf{w'}_2$. \end{center}
Since $\mathbf{x}_1 \mathbf{v}_1 \overset{\mathcal{B}}{\leadsto} \mathbf{x}_2 \mathbf{v}_2$, we have
\begin{center} $\mathbf{x}_1\mathbf{v}_1\mathbf{w'} \overset{\mathcal{B}}{\leadsto} \mathbf{x}_2\mathbf{v}_2\mathbf{w'}$. \end{center} We obtain:
\begin{center}$\mathbf{w}_1 \overset{\mathcal{B}}{\leadsto} \mathbf{x}_1\mathbf{w'}_1 \overset{\mathcal{B}}{\leadsto}  \mathbf{x}_1\mathbf{v}_1\mathbf{w'} \overset{\mathcal{B}}{\leadsto} \mathbf{x}_2\mathbf{v}_2\mathbf{w'} \overset{\mathcal{B}}{\leadsto} \mathbf{x}_2\mathbf{w'}_2 \overset{\mathcal{B}}{\leadsto} \mathbf{w}_2$.\end{center}
Hence $\mathbf{w}_1 \overset{\mathcal{B}}{\leadsto} \mathbf{w}_2$ and we are done.

\end{proof}

\begin{remark}

In the case of the Coxeter group $G(2,2,n)$ of type $D_n$, we have $[1,\lambda] = G(2,2,n)$, where $\lambda$ is the unique balanced element of maximal length of $G(2,2,n)$, see Remarks \ref{RemUniqueMaximalG(2,2,n)} and \ref{RemDivisorsLambdaG(2,2,n)}. Hence the previous result is valid for $G(2,2,n)$. This is Matsumoto's property for the case of $G(2,2,n)$, see \cite{MatsumotoTheorem}.

\end{remark}




By the following proposition, we provide an alternative presentation for the interval monoid $M([1,\lambda^k])$ given in Definition \ref{DefIntervalMonoid}.

\begin{proposition}

The monoid $B^{\oplus k}(e,e,n)$ is isomorphic to $M([1,\lambda^k])$.

\end{proposition}

\begin{proof}

Consider the map $\rho: \underline{D_k} \longrightarrow B^{\oplus k}(e,e,n): \underline{w} \longmapsto F(w)$ where $F$ is defined in Proposition \ref{PropMatsumoto}.
Let $\underline{w} = \underline{w}' \underline{w}''$ be a defining relation of $M([1,\lambda^k])$. Since $\ell(w) = \ell(w') + \ell(w'')$, a reduced expression for $w'w''$ is obtained by concatenating reduced expressions for $w'$ and $w''$. It follows that $F(w'w'') = F(w') F(w'')$. We conclude that $\rho$ has a unique extension to a monoid homomorphism $M([1,\lambda^k]) \longrightarrow B^{\oplus k}(e,e,n)$, which we denote by the same symbol.\\
Conversely, consider the map $\rho' : \widetilde{X} \longrightarrow M([1,\lambda^k]): \widetilde{x} \longmapsto \underline{x}$. In order to prove that $\rho'$ extends to a unique monoid homomorphism $B^{\oplus k}(e,e,n) \longrightarrow M([1,\lambda^k])$, we have to check that $\underline{w_1} = \underline{w_2}$ in $M([1,\lambda^k])$ for any defining relation $\widetilde{w}_1 = \widetilde{w}_2$ of $B^{\oplus k}(e,e,n)$. Given a relation $\widetilde{w}_1 = \widetilde{w}_2 = \tilde{x}_1\tilde{x}_2 \cdots \tilde{x}_r$ of $B^{\oplus k}(e,e,n)$, we have $\boldsymbol{w_1} = \boldsymbol{w_2} = \mathbf{x}_1 \mathbf{x}_2 \cdots \mathbf{x}_r$ a reduced word over $\mathbf{X}$. On the other hand, applying repeatedly the defining relations in $M([1,\lambda^k])$ yields to $\underline{w} = \underline{x_1}\ \underline{x_2} \cdots \underline{x_r}$ if $\boldsymbol{w} = \mathbf{x}_1\mathbf{x}_2 \cdots \mathbf{x}_r$ is a reduced expression over $\mathbf{X}$. Thus, we can conclude that $\underline{w_1} = \underline{w_2}$, as desired.\\
Hence we have defined two homomorphisms $\rho: \underline{D_k} \longrightarrow B^{\oplus k}(e,e,n)$ and $\rho': \widetilde{X} \longrightarrow M([1,\lambda^k])$ such that $\rho \circ \rho' = id_{B^{\oplus k}(e,e,n)}$ and $\rho' \circ \rho = id_{M([1,\lambda^k])}$. It follows that $B^{\oplus k}(e,e,n)$ is isomorphic to $M([1,\lambda^k])$.

\end{proof}

Since $B^{\oplus k}(e,e,n)$ is isomorphic to $M([1,\lambda^k])$, we deduce that $B^{\oplus k}(e,e,n)$ is a Garside monoid and we denote by $B^{(k)}(e,e,n)$ its group of fractions.

\subsection{Identifying $B(e,e,n)$}

Now, we want to check which of the monoids $B^{\oplus k}(e,e,n)$ are isomorphic to $B^{\oplus}(e,e,n)$. Assume there exists an isomorphism $\phi : B^{\oplus k}(e,e,n) \longrightarrow  B^{\oplus}(e,e,n)$ for a given $k$ with $1 \leq k \leq e-1$. We start with the following lemma.

\begin{lemma}\label{LemmaIsomMonoids}

The isomorphism $\phi$ fixes $\tilde{s}_3, \tilde{s}_4, \cdots, \tilde{s}_n$ and permutes the $\tilde{t}_i$ where \mbox{$i \in \mathbb{Z}/e\mathbb{Z}$}.  

\end{lemma}

\begin{proof}

Let $f$ be in $\widetilde{X}^{*}$. We have $\boldsymbol{\ell}(f) \leq \boldsymbol{\ell}(\phi(f))$. Thus, we have $\boldsymbol{\ell}(\tilde{x}) \leq \boldsymbol{\ell}(\phi(\tilde{x}))$ for $\tilde{x} \in \widetilde{X}$. Also, $\boldsymbol{\ell}(\phi(\tilde{x})) \leq \boldsymbol{\ell}(\phi^{-1}(\phi(\tilde{x}))) = \boldsymbol{\ell}(x)$. Hence $\boldsymbol{\ell}(\tilde{x}) = \boldsymbol{\ell}(\phi(\tilde{x})) = 1$. It follows that $\phi$ maps generators to generators, that is $\phi(\tilde{x}) \in \{\tilde{t}_0, \tilde{t}_1, \cdots, \tilde{t}_{e-1},\tilde{s}_3, \cdots, \tilde{s}_n\}$.

Furthermore, the only generator of  $B^{\oplus k}(e,e,n)$ that commutes with all other generators except for one of them is $\tilde{s}_{n}$. On the other hand, $\tilde{s}_n$ is the only generator of $B^{\oplus}(e,e,n)$ that satisfies the latter property. Hence $\phi(\tilde{s}_{n}) = \tilde{s}_n$. Next, $\tilde{s}_{n-1}$ is the only generator of $B^{\oplus k}(e,e,n)$ that does not commute with $\tilde{s}_{n}$. The only generator of $B^{\oplus}(e,e,n)$ that does not commute with $\tilde{s}_n$ is also $\tilde{s}_{n-1}$. Hence $\phi(\tilde{s}_{n-1}) = \tilde{s}_{n-1}$. Next, the only generator of  $B^{\oplus k}(e,e,n)$ different from $\tilde{s}_{n}$ and that does not commute with $\tilde{s}_{n-1}$ is $\tilde{s}_{n-2}$. And so on, we get $\phi(\tilde{s}_{j}) = \tilde{s}_j$ for $3 \leq j \leq n$. It remains that $\phi(\{ \tilde{t}_{i}\ |\ 0 \leq i \leq e-1 \}) = \{ \tilde{t}_i\ |\ 0 \leq i \leq e-1 \}$.

\end{proof}

\begin{proposition}\label{PropIsomwithMonoidCP}

The monoids $B^{\oplus k}(e,e,n)$ and $B^{\oplus}(e,e,n)$ are isomorphic if and only if $k \wedge e = 1$.

\end{proposition}

\begin{proof}

Assume there exists an isomorphism $\phi$ between the monoids $B^{\oplus k}(e,e,n)$ and $B^{\oplus}(e,e,n)$. By Lemma \ref{LemmaIsomMonoids}, we have $\phi(\tilde{s}_{j}) = \tilde{s}_j$ for $3 \leq j \leq n$ and \mbox{$\phi(\{ \tilde{t}_{i}\ |\ 0 \leq i \leq e-1 \})$} $= \{ \widetilde{t}_i\ |\ 0 \leq i \leq e-1 \}$. Write $\tilde{t}_{a_i}$ for the generator of $B^{\oplus k}(e,e,n)$ for which $\phi(\tilde{t}_{a_i}) = \tilde{t}_i$. The existence of the isomorphism implies $\tilde{t}_{a_{i+1}}\tilde{t}_{a_i} = \tilde{t}_{a_1}\tilde{t}_{a_0}$ for each $i \in \mathbb{Z}/e\mathbb{Z}$. But the only relations of this type in $B^{\oplus k}(e,e,n)$ are $\tilde{t}_{a_{i}+k}\tilde{t}_{a_i} = \tilde{t}_{a_0 + k}\tilde{t}_{a_0}$. Thus, $a_1 = a_0+k$, $a_2 = a_0+2k$, $\cdots$, $a_{e-1} = a_0 + (e-1)k$. Since $\phi$ is bijective on the $\tilde{t}_i$-type generators, $\left\{ a_0 + ik\ |\ i \in \mathbb{Z}/e\mathbb{Z} \right\} = \mathbb{Z}/e\mathbb{Z}$, so $k \wedge e = 1$.

Conversely, let $1 \leq k \leq e-1$ such that $k \wedge e = 1$. We define a map $\phi : B^{\oplus}(e,e,n) \longrightarrow B^{\oplus k}(e,e,n)$ where $\phi(\tilde{s}_j) = \tilde{s}_{j}$ for $3 \leq j \leq n$ and $\phi(\tilde{t}_i) = \phi(\tilde{t}_{ik})$, that is $\phi(\tilde{t}_0) = \tilde{t}_{0}$, $\phi(\tilde{t}_1) = \tilde{t}_{k}$, $\phi(\tilde{t}_2) = \tilde{t}_{2k}$, $\cdots$, $\phi(\tilde{t}_{e-1}) = \widetilde{t}_{(e-1)k}$. The map $\phi$ is a well-defined monoid homomorphism, which is both surjective (as it sends a generator of $B^{\oplus}(e,e,n)$ to a generator of $B^{\oplus k}(e,e,n)$) and injective (as it is bijective on the relations). Hence $\phi$ defines an isomorphism of monoids.

\end{proof}

When $k \wedge e = 1$, since $B^{\oplus k}(e,e,n)$ is isomorphic to $B^{\oplus}(e,e,n)$, we have the following.

\begin{cor}\label{CorIsomGrFractBraidGroup}

$B^{(k)}(e,e,n)$ is isomorphic to the complex braid group $B(e,e,n)$ for $k \wedge e = 1$.

\end{cor}

The reason that the proof of Proposition \ref{PropIsomwithMonoidCP} fails in the case $k \wedge e \neq 1$ is that we have more than one connected component in $\Gamma_k$ that link $\tilde{t}_0$, $\tilde{t}_1$, $\cdots$, and $\tilde{t}_{e-1}$ together, as we can see in Figure \ref{PresofB2882}. Actually, it is easy to check that the number of connected components that link $\tilde{t}_0$, $\tilde{t}_1$, $\cdots$, and $\tilde{t}_{e-1}$ together is the number of cosets of the subgroup of $\mathbb{Z}/e\mathbb{Z}$ generated by the class of $k$, that is equal to $k \wedge e$, and each of these cosets have $e'=e/k \wedge e$ elements. This will be useful in the next subsection.

\subsection{New Garside structures}

When $k \wedge e \neq 1$, we describe $B^{(k)}(e,e,n)$ as an amalgamated product of $k \wedge e$ copies of the complex braid group $B(e',e',n)$ with $e'=e/e \wedge k$, over a common subgroup which is the Artin-Tits group $B(2,1,n-1)$. This allows us to compute the center of $B^{(k)}(e,e,n)$. Finally, using the Garside structure of $B^{(k)}(e,e,n)$, we compute its first and second integral homology groups using the Dehornoy-Lafont complex \cite{DehLafHomologyofGaussianGroups} and the method used in \cite{CalMarHomologyComputations}.\\

By an adaptation of the results of Crisp \cite{CrispInjectivemaps} as in Lemma 5.2 of \cite{CalMarHomologyComputations}, we have the following embedding. Let $B:=B(2,1,n-1)$ be the Artin-Tits group defined by the following diagram presentation.

\begin{figure}[H]
\begin{center}
\begin{tikzpicture}

\node[draw, shape=circle, label=above:$q_1$] (1) at (0,0) {};
\node[draw, shape=circle, label=above:$q_2$] (2) at (1.5,0) {};
\node[draw, shape=circle, label=above:$q_3$] (3) at (3,0) {};
\node[draw, shape=circle, label=above:$q_{n-2}$] (n-2) at (7,0) {};
\node[draw, shape=circle, label=above:$q_{n-1}$] (n-1) at (8.5,0) {};

\draw[double,thick,-] (1) to (2);
\draw[thick,-] (2) to (3);
\draw[thick,-] (n-2) to (n-1);
\draw[thick,dashed] (3) to (n-2);

\end{tikzpicture}
\end{center}
\end{figure}

\begin{proposition}

The group $B$ injects in $B^{(k)}(e,e,n)$.

\end{proposition}

\begin{proof}

Define a monoid homomorphism $\phi : B^{+} \longrightarrow B^{\oplus k}(e,e,n) : q_{1} \longmapsto \tilde{t}_i\tilde{t}_{i-k}$, \mbox{$q_2 \longmapsto \tilde{s}_3$}, $\cdots$, $q_{n-1} \longmapsto \tilde{s}_n$. It is easy to check that for all $x, y \in \{q_1, q_2, \cdots, q_{n-1}\}$, we have $lcm(\phi(x), \phi(y))=\phi(lcm(x,y))$. Hence by applying Lemma 5.2 of \cite{CalMarHomologyComputations}, \mbox{$B(2,1,n-1)$} injects in $B^{(k)}(e,e,n)$.

\end{proof}

We construct $B^{(k)}(e,e,n)$ as follows:

\begin{proposition}\label{PropBkeenAmalgamatedProduct}

Let $B(1) := B(e',e',n)$ where $e'=e/e \wedge k$. Inductively, define $B(i+1)$ to be the amalgamated product $B(i+1) := B(i) *_{B} B(e',e',n)$ over $B = B(2,1,n-1)$. Then we have $B^{(k)}(e,e,n)$ is isomorphic to $B(e \wedge k)$.

\end{proposition}

\begin{proof}

Due to the presentation of $B^{(k)}(e,e,n)$ given in Definition \ref{DefofB+keen} and to the presentation of the amalgamated products (see Section 4.2 of \cite{CombinatorialGroupTheory}), one can deduce that $B(e \wedge k)$ is isomorphic to $B^{(k)}(e,e,n)$.

\end{proof}

\begin{figure}[H]
\begin{center}
\begin{tikzpicture}

\node[draw, shape=rectangle, label=below:{$B(e',e',n)$}] (1) at (0,0) {};
\node[draw, shape=rectangle, label=below:{$B(e',e',n)$}] (2) at (2,0) {};
\node[draw, shape=rectangle, label=below:{$B(e',e',n)$}] (3) at (-1,1) {};
\node[draw, shape=rectangle, label=right:{$B(2)$}] (4) at (1,1) {};
\node[draw, shape=rectangle, label=right:{$B(3)$}] (5) at (0,2) {};
\node[draw, shape=rectangle, label=left:{$B(e',e',n)$}] (6) at (-3,3) {};
\node[draw, shape=rectangle] (7) at (-1,3) {};
\node[draw, shape=rectangle, label=right:{$B(e \wedge k) = B^{(k)}(e,e,n)$}] (8) at (-2,4) {};

\draw[thick,-] (1) to (4);
\draw[thick,-] (2) to (4);
\draw[thick,-] (5) to (3);
\draw[thick,-] (5) to (4);
\draw[thick,dashed,-] (7) to (5);
\draw[thick,dashed,-] (-1.5,1.5) to (-2.5,2.5);
\draw[thick,-] (8) to (6);
\draw[thick,-] (8) to (7);

\end{tikzpicture}
\end{center}
\caption{The construction of $B^{(k)}(e,e,n)$.}
\end{figure}
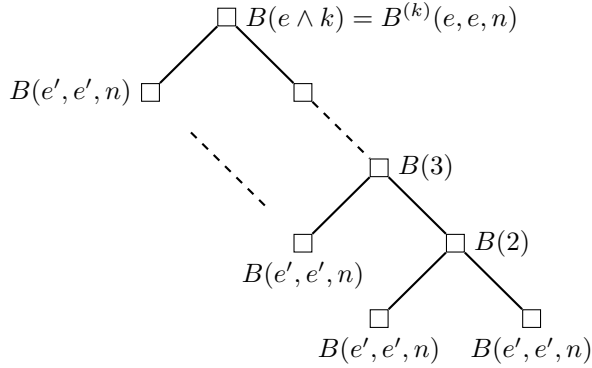

\begin{example}

Consider the case of $B^{(2)}(6,6,3)$. It is an amalgamated product of $k \wedge e = 2$ copies of $B(e',e',3)$ with $e' = e/(k \wedge e) = 3$ over the Artin-Tits group $B(2,1,2)$. Consider the following diagram of this amalgamation.\\
The presentation of $B(3,3,3) * B(3,3,3)$ over $B(2,1,2)$ is as follows:
\begin{itemize}

\item the generators are the union of the generators of the two copies of $B(3,3,3)$, 
\item the relations are the union of the relations of the two copies of $B(3,3,3)$ with the additional relations $\tilde{s}_3 = \tilde{s}'_3$ and $\tilde{t}_2\tilde{t}_0 = \tilde{t}_3\tilde{t}_1$

\end{itemize}
This is exactly the presentation of $B^{(2)}(6,6,3)$ given in Definition \ref{DefofB+keen}.

\begin{small}
\begin{center}

\begin{tabular}{lll}

\begin{tikzpicture}
\node[draw, shape=circle, label=above:$\tilde{t}_3\tilde{t}_1$] (1) at (0,0) {};
\node[draw, shape=circle, label=above:$\tilde{s}_3$] (2) at (1.5,0) {};
\draw[double,thick,-] (1) to (2);
\draw[thick,->] (0.75,-0.8) to (0.75,-1.4);
\end{tikzpicture} & \begin{tikzpicture}
\node[] () at (0,0) {};
\node[] () at (0,1.4) {$\overset{\sim}{\longrightarrow}$};
\end{tikzpicture}& \begin{tikzpicture}
\node[draw, shape=circle, label=above:$\tilde{t}_2\tilde{t}_0$] (1) at (0,0) {};
\node[draw, shape=circle, label=above:$\tilde{s}'_3$] (2) at (1.5,0) {};
\draw[double,thick,-] (1) to (2);
\draw[thick,->] (0.75,-0.8) to (0.75,-1.4);
\end{tikzpicture}\\

\begin{tikzpicture}[yscale=0.8,xscale=1]

\draw[thick,dashed] (0,0) ellipse (1cm and 0.5cm);

\node[draw, shape=circle, fill=white, label=above:$\tilde{t}_1$] (t1) at (0,-0.5) {};
\node[draw, shape=circle, fill=white, label=above:$\tilde{t}_3$] (t3) at (1,0) {};
\node[draw, shape=circle, fill=white, label=above:$\tilde{t}_5$] (t5) at (-1,0) {};

\node[draw, shape=circle, fill=white, label=below right:$\tilde{s}_3$] (s3) at (0,-1.5) {};

\draw[thick,-] (t1) to (s3);
\draw[thick,-, bend left] (t3) to (s3);
\draw[thick,-,bend right] (t5) to (s3);

\end{tikzpicture} & & \begin{tikzpicture}[yscale=0.8,xscale=1]

\draw[thick,dashed] (0,0) ellipse (1cm and 0.5cm);

\node[draw, shape=circle, fill=white, label=above:$\tilde{t}_0$] (t0) at (0,-0.5) {};
\node[draw, shape=circle, fill=white, label=above:$\tilde{t}_2$] (t2) at (1,0) {};
\node[draw, shape=circle, fill=white, label=above:$\tilde{t}_4$] (t4) at (-1,0) {};

\node[draw, shape=circle, fill=white, label=below right:$\tilde{s}'_3$] (s3) at (0,-1.5) {};

\draw[thick,-] (t0) to (s3);
\draw[thick,-, bend left] (t2) to (s3);
\draw[thick,-,bend right] (t4) to (s3);

\end{tikzpicture}\\

\end{tabular}

\begin{tikzpicture}

\node[] () at (0,0) {};
\draw[thick,->] (0,0) to (1,-0.5);
\draw[thick,->] (3.5,0) to (2.5,-0.5);

\end{tikzpicture}

$B^{(2)}(6,6,3)$

\end{center}
\end{small}

\end{example}

\begin{proposition}

The center of $B^{(k)}(e,e,n)$ is infinite cyclic isomorphic to $\mathbb{Z}$.

\end{proposition}

\begin{proof}

By Corollary 4.5 of \cite{CombinatorialGroupTheory} that computes the center of an amalgamated product, the center of $B^{(k)}(e,e,n)$ is the intersection of the centers of $B$ and $B(e',e',n)$. Since the center of $B$ and $B(e',e',n)$ is infinite cyclic \cite{BMR}, the center of $B^{(k)}(e,e,n)$ is infinite cyclic isomorphic to $\mathbb{Z}$.

\end{proof}

Since the center of $B(e,e,n)$ is also isomorphic to $\mathbb{Z}$ (see \cite{BMR}), in order to distinguish $B^{(k)}(e,e,n)$ from the braid groups $B(e,e,n)$, we compute its first and second integral homology groups. We recall the Dehornoy-Lafont complex and follow the method in \cite{CalMarHomologyComputations} where the second integral homology group of $B(e,e,n)$ is computed.\\

We order the elements of $\widetilde{X}$ by considering $\tilde{s}_{n} < \tilde{s}_{n-1} < \cdots < \tilde{s}_{3} < \tilde{t}_{0} < \tilde{t}_{1} < \cdots < \tilde{t}_{e-1}$. For $f \in B^{\oplus k}(e,e,n)$, denote by $d(f)$ the least element in $\widetilde{X}$ which divides $f$ on the right. An $r$-cell is an $r$-tuple $[x_1, \cdots, x_r]$ of elements in $\widetilde{X}$ such that $x_1 < x_2 < \cdots < x_r$ and $x_i = d(lcm(x_i,x_{i+1}, \cdots, x_{r}))$. The set $C_r$ of $r$-chains is the free $\mathbb{Z} B^{\oplus k}(e,e,n)$-module with basis $\widetilde{X}_r$, the set of all $r$-cells with the convention $\widetilde{X}_0 = \{[\emptyset]\}$. We provide the definition of the differential $\partial_r : C_r \longrightarrow C_{r-1}$.

\begin{definition}\label{DefintionDifferentialHomology}

Let $[\alpha,A]$ be an $(r+1)$-cell, with $\alpha \in \widetilde{X}$ and $A$ an $r$-cell. Denote $\alpha_{/A}$ the unique element of $B^{\oplus k}(e,e,n)$ such that $(\alpha_{/A})lcm(A) = lcm(\alpha,A)$. Define the differential $\partial_r : C_r \longrightarrow C_{r-1}$ recursively through two $\mathbb{Z}$-module homomorphisms $s_r: C_r \longrightarrow C_{r+1}$ and $u_r: C_r \longrightarrow C_r$ as follows.

\begin{center}

$\partial_{r+1}[\alpha,A] = \alpha_{/A}[A] - u_r(\alpha_{/A}[A])$,

\end{center}
with $u_{r+1} = s_r \circ \partial_{r+1}$ where $u_0(f[\emptyset]) = [\emptyset]$, for all $f \in B^{\oplus k}(e,e,n)$, and\\
$s_r([\emptyset])=0$, $s_r(x[A]) = 0$ if $\alpha:=d(x lcm(A))$ coincides with the first coefficient in $A$, and otherwise
$s_r(x[A]) = y[\alpha,A] + s_{r}(y u_{r}(\alpha_{/A}[A]))$ with $x = y \alpha_{/A}$.\\

\end{definition}

We provide the result of the computation of $\partial_1$, $\partial_2$, and $\partial_3$ for all $1$, $2$, and $3$-cells, respectively.
For all $x \in \widetilde{X}$, we have
\begin{center}

$\partial_1[x] = (x-1)[\emptyset]$,

\end{center}
for all $1 \leq i \leq e-1$,
\begin{center}

$\partial_2[\tilde{t}_0,\tilde{t}_i] = \tilde{t}_{i+k}[\tilde{t}_i] - \tilde{t}_k[\tilde{t}_0] - [\tilde{t}_k] + [\tilde{t}_{i+k}]$,

\end{center}
for $x,y \in \widetilde{X}$ with $xyx=yxy$,
\begin{center}

$\partial_2[x,y] = (yx+1-x)[y] + (y-xy-1)[x]$, and

\end{center}
for $x,y \in \widetilde{X}$ with $xy=yx$,
\begin{center}

$\partial_2[x,y] = (x-1)[y] - (y-1)[x]$.

\end{center}
For $j \neq -k$ mod $e$, we have:
\begin{center}

$\partial_3[\tilde{s}_3,\tilde{t}_0,\tilde{t}_j] = (\tilde{s}_3\tilde{t}_k\tilde{t}_0\tilde{s}_3 - \tilde{t}_k\tilde{t}_0\tilde{s}_3 + \tilde{t}_{j+2k}\tilde{s}_3)[\tilde{t}_0,\tilde{t}_j] - \tilde{t}_{j+2k}\tilde{s}_3\tilde{t}_{j+k}[\tilde{s}_3,\tilde{t}_j] + (\tilde{t}_{j+2k} - \tilde{s}_3\tilde{t}_{j+2k})[\tilde{s}_3,\tilde{t}_{j+k}] + (\tilde{s}_3 - \tilde{t}_{j+2k}\tilde{s}_3 - 1)[\tilde{t}_0, \tilde{t}_{j+k}] + (\tilde{s}_3 \tilde{t}_{2k} - \tilde{t}_{2k})[\tilde{s}_3,\tilde{t}_k] + (\tilde{t}_{2k}\tilde{s}_3 + 1 - \tilde{s}_3)[\tilde{t}_0,\tilde{t}_k] + [\tilde{s}_3,\tilde{t}_{j+2k}] + \tilde{t}_{2k}\tilde{s}_3\tilde{t}_k[\tilde{s}_3,\tilde{t}_0] - [\tilde{s}_3, \tilde{t}_{2k}]$ and

\end{center}

\begin{center}

$\partial_3[\tilde{s}_3,\tilde{t}_0,\tilde{t}_{-k}] = (\tilde{s}_3\tilde{t}_k\tilde{t}_0\tilde{s}_3 - \tilde{t}_k\tilde{t}_0\tilde{s}_3 + \tilde{t}_k\tilde{s}_3)[\tilde{t}_0, \tilde{t}_{-k}] - \tilde{t}_k\tilde{s}_3\tilde{t}_0[\tilde{s}_3,\tilde{t}_{-k}] + (1-\tilde{t}_{2k}+\tilde{s}_3\tilde{t}_{2k})[\tilde{s}_3,\tilde{t}_k] + (1+\tilde{t}_{2k}\tilde{s}_3 - \tilde{s}_3)[\tilde{t}_0,\tilde{t}_k] + (\tilde{t}_k - \tilde{s}_3\tilde{t}_k + \tilde{t}_{2k}\tilde{s}_3\tilde{t}_k)[\tilde{s}_3,\tilde{t}_0] - [\tilde{s}_3,\tilde{t}_{2k}]$.

\end{center}
Also, for $1 \leq i \leq e-1$ and $4 \leq j \leq n$, we have:
\begin{center}

$\partial_3[\tilde{s}_j,\tilde{t}_0,\tilde{t}_i] = (\tilde{s}_j -1)[\tilde{t}_0,\tilde{t}_i] - \tilde{t}_{i+k}[\tilde{s}_j, \tilde{t}_i] + \tilde{t}_k[\tilde{s}_j,\tilde{t}_0] - [\tilde{s}_j,\tilde{t}_{i+k}] + [\tilde{s}_j,\tilde{t}_k]$,

\end{center}
for $x,y,z \in \widetilde{X}$ with $xyx=yxy$, $xz=zx$, and $yzy=zyz$,
\begin{center}

$\partial_3[x,y,z] = (z + xyz - yz -1)[x,y] - [x,z] + (xz-z-x+1-yxz)y[x,z] + (x-1-yx+zyx)[y,z]$,

\end{center}
for $x,y,z \in \widetilde{X}$ with $xyx=yxy$, $xz=zx$, and $yz=zy$,
\begin{center}

$\partial_3[x,y,z] = (1-x+yx)[y,z] + (y-1-xy)[x,z] + (z-1)[x,y]$,

\end{center}
for $x,y,z \in \widetilde{X}$ with $xy=yx$, $xz=zx$, and $yzy=zyz$,
\begin{center}

$\partial_3[x,y,z] = (1+yz-z)[x,y] + (y-1-zy)[x,z] + (x-1)[y,z]$, and

\end{center}
for $x,y,z \in \widetilde{X}$ with $xy=yx$, $xz=zx$, and $yz=zy$,
\begin{center}

$\partial_3[x,y,z] = (1-y)[x,z] + (z-1)[x,y] +(x-1)[y,z]$.

\end{center}

Let $d_r = \partial_r \otimes_{\mathbb{Z} B^{\oplus k}(e,e,n)} \mathbb{Z} : C_r  \otimes_{\mathbb{Z} B^{\oplus k}(e,e,n)} \mathbb{Z} \longrightarrow C_{r-1}  \otimes_{\mathbb{Z} B^{\oplus k}(e,e,n)} \mathbb{Z}$ be the differential with trivial coefficients. For example, for $d_2$, we have:
for all $1 \leq i \leq e-1$,
\begin{center}

$d_2[\tilde{t}_0,\tilde{t}_i] = [\tilde{t}_i] - [\tilde{t}_0] - [\tilde{t}_k] + [\tilde{t}_{i+k}]$,

\end{center}
for $x,y \in \widetilde{X}$ with $xyx=yxy$,
\begin{center}

$d_2[x,y] = [y] - [x]$, and

\end{center}
for $x,y \in \widetilde{X}$ with $xy=yx$,
\begin{center}

$d_2[x,y] = 0$.

\end{center}
The same can be done for $d_3$.

\begin{definition}

Define the integral homology group of order $r$ to be $$H_r(B^{(k)}(e,e,n),\mathbb{Z}) = ker(d_r) / Im(d_{r+1}).$$

\end{definition}

We are ready to compute the first and second integral homology groups of $B^{(k)}(e,e,n)$. Using the presentation of $B^{(k)}(e,e,n)$ given in Definition \ref{DefofB+keen}, one can check that the abelianization of $B^{(k)}(e,e,n)$ is isomorphic to $\mathbb{Z}$. Since $H_1(B^{(k)}(e,e,n),\mathbb{Z})$ is isomorphic to the abelianization of $B^{(k)}(e,e,n)$, we deduce that $H_1(B^{(k)}(e,e,n),\mathbb{Z})$ is isomorphic to $\mathbb{Z}$. Since $H_1(B(e,e,n),\mathbb{Z})$ is also isomorphic to $\mathbb{Z}$ (see \cite{BMR}), the first integral homology group does not give any additional information whether these groups are isomorphic to some $B(e,e,n)$ or not.\\

Recall that by \cite{CalMarHomologyComputations}, we have

\begin{itemize}

\item $H_2(B(e,e,3),\mathbb{Z}) \simeq \mathbb{Z}/e\mathbb{Z}$ where $e \geq 2$,
\item $H_2(B(e,e,4),\mathbb{Z}) \simeq \mathbb{Z}/e\mathbb{Z} \times \mathbb{Z}/2\mathbb{Z}$ when $e$ is odd and $H_2(B(e,e,4),\mathbb{Z}) \simeq \mathbb{Z}/e\mathbb{Z} \times (\mathbb{Z}/2\mathbb{Z})^2$ when $e$ is even, and
\item $H_2(B(e,e,n),\mathbb{Z}) \simeq \mathbb{Z}/e\mathbb{Z} \times \mathbb{Z}/2\mathbb{Z}$ when $n \geq 5$ and $e \geq 2$.

\end{itemize}

In order to compute $H_2(B^{(k)}(e,e,n),\mathbb{Z})$, we follow exactly the proof of \mbox{Theorem 6.4} in \cite{CalMarHomologyComputations} and use the same notations. We only point out the part of our proof that is different from \cite{CalMarHomologyComputations}. Define $v_i = [\tilde{t}_0,\tilde{t}_i] + [\tilde{s}_3,\tilde{t}_0] + [\tilde{s}_3,\tilde{t}_k] - [\tilde{s}_3,\tilde{t}_i] - [\tilde{s}_3, \tilde{t}_{i+k}]$ where $1 \leq i \leq e-1$. As in \cite{CalMarHomologyComputations}, we also have $H_2(B^{(k)}(e,e,n),\mathbb{Z}) = (K_1/d_3(C_1)) \oplus (K_2/d_3(C_2))$. We have $d_3[\tilde{s}_3,\tilde{t}_0, \tilde{t}_j] = v_j - v_{j+k} + v_k$ if $j \neq -k$ and $d_3[\tilde{s}_3,\tilde{t}_0, \tilde{t}_{-k}] =v_{-k} + v_k$. Denote $u_i = [\tilde{s}_3, \tilde{t}_0,\tilde{t}_i]$ for $1 \leq i \leq e-1$. We define a basis of $C_2$ as follows. For each coset of the subgroup of $\mathbb{Z}/e\mathbb{Z}$ generated by the class of $k$, say $\{\tilde{t}_x,\tilde{t}_{x+k}, \cdots, \tilde{t}_{x-k}\}$ such that $1 \leq x \leq e-1$, we define $w_{x+ik} = u_{x+ik} + u_{x+(i+1)k} + \cdots + u_{x-k}$ for $0 \leq i \leq e-1$, and when $x=0$, we define $w_{ik} = u_{ik} + u_{(i+1)k} + \cdots + u_{-k}$ for $1 \leq i \leq e-1$. Written on the $\mathbb{Z}$-basis $(w_{k}, w_{2k}, \cdots , w_{-k}, w_{1}, w_{1+k}, \cdots, w_{1-k}, \cdots, w_{x}, w_{x+k}, \cdots, w_{x-k}, \cdots)$ and $(v_k, v_{2k}, \cdots, v_{-k}, v_1, v_{1+k}, \cdots, v_{1-k}, \cdots, v_x, v_{x+k}, \cdots, v_{x-k}, \cdots )$, $d_3$ is in triangular form with $(e \wedge k) -1$ diagonal coefficients that are zero, all other diagonal coefficients are equal to $1$ except one of them that is equal to $e' = e / (e \wedge k)$. In this case, we have $H_2(B^{(k)}(e,e,3),\mathbb{Z}) = \mathbb{Z}^{(e \wedge k)-1} \times \mathbb{Z}/e'\mathbb{Z}$. The rest of the proof is essentially similar to the proof of Theorem 6.4 in  \cite{CalMarHomologyComputations}.

When $n=4$, we get $[\tilde{s}_4,\tilde{t}_i] + [\tilde{s}_4,\tilde{t}_{i+k}] \equiv [\tilde{s}_4,\tilde{t}_{k}] + [\tilde{s}_4,\tilde{t}_{0}]$ for every $i$ and since $2[\tilde{s}_4,\tilde{t}_i] \equiv 0$ for every $i$, we get $K_2/d_3(C_2) \simeq (\mathbb{Z}/2\mathbb{Z})^c$, where $c$ is the number of $[\tilde{s}_4,\tilde{t}_j]$ that appear in the solution of the congruence relations $[\tilde{s}_4,\tilde{t}_i] + [\tilde{s}_4,\tilde{t}_{i+k}] \equiv [\tilde{s}_4,\tilde{t}_{k}] + [\tilde{s}_4,\tilde{t}_{0}]$ and $2[\tilde{s}_4,\tilde{t}_i] \equiv 0$ in $\mathbb{Z}/e\mathbb{Z}$. We summarize our proof by providing the second integral homology group of $B^{(k)}(e,e,n)$ in the following proposition.

\begin{proposition}\label{PropH2ofBkeen}

Let $e \geq 2$, $n \geq 3$, $1 \leq k \leq e-1$, and $e' = e/ e \wedge k$.
\begin{itemize}
\item If $n=3$, we have $H_2(B^{(k)}(e,e,3),\mathbb{Z}) \simeq \mathbb{Z}^{(e \wedge k)-1} \times \mathbb{Z}/e'\mathbb{Z}$.
\item If $n=4$, we have $H_2(B^{(k)}(e,e,4),\mathbb{Z}) \simeq \mathbb{Z}^{(e \wedge k)-1} \times \mathbb{Z}/e'\mathbb{Z} \times (\mathbb{Z}/2\mathbb{Z})^{c}$, where $c$ is defined in the previous paragraph.
\item If $n \geq 5$, we have $H_2(B^{(k)}(e,e,n),\mathbb{Z}) \simeq \mathbb{Z}^{(e \wedge k)-1} \times \mathbb{Z}/e'\mathbb{Z} \times \mathbb{Z}/2\mathbb{Z}$.

\end{itemize}

\end{proposition}

Comparing this result with $H_2(B(e,e,n),\mathbb{Z})$, one can check that if $k \wedge e \neq 1$, $B^{(k)}(e,e,n)$ is not isomorphic to a complex braid group of type $B(e,e,n)$. Thus, we conclude by the following theorem.

\medskip

\begin{theorem}\label{TheoremBkIsomBiff}

$B^{(k)}(e,e,n)$ is isomorphic to $B(e,e,n)$ if and only if \mbox{$k\wedge e=1$.}

\end{theorem}

\medskip

In Appendix A, we provide a general code to compute the homology groups of order $r \geq 2$ of Garside structures by using the Dehornoy-Lafont complexes. In particular, we apply the code for the Garside monoids $B^{\oplus k}(e,e,n)$ that have been already implemented (see \cite{CHEVIEJMichel} and \cite{GAP3CPMichelNeaime}). This enables us to compute the second integral homology groups for $B^{(k)}(e,e,n)$ and to check the results of Proposition \ref{PropH2ofBkeen}.

\newpage~
\thispagestyle{empty}

%% file: chap44.tex
\chapter{Hecke algebras for $G(e,e,n)$ and $G(d,1,n)$}\label{ChapterHeckeAlgebras}
 
\minitoc

\bigskip

In this chapter, we define the Hecke algebra associated to each complex reflection group of type $G(e,e,n)$ and $G(d,1,n)$ as a quotient of the group algebra $R_0B$ by some polynomial relations, where $R_0$ is a polynomial ring and $B$ is the corresponding complex braid group. We define a basis for this Hecke algebra and get a new proof of Theorem \ref{PropositionBMRFreenessTheorem} for the general series of complex reflection groups of type $G(e,e,n)$ and $G(d,1,n)$.

\section{Presentations for the Hecke algebras}

We start by recalling the presentation of the complex braid group $B(e,e,n)$ for $e \geq 1$ and $n \geq 2$, see Definition \ref{DefofB+keen} for $k = 1$. Note that we remove the tilde character from each generator symbol of the presentation in order to simplify the notations.

\begin{definition}\label{DefPresB(de,e,n)CorranLeeLee}

Fix $e \geq 1$, and $n \geq 2$. The group $B(e,e,n)$ is defined by a presentation with set of generators: $\{t_0,t_1, \cdots, t_{e-1}, s_3, s_4, \cdots, s_n \}$ and relations as follows.

\begin{enumerate}

\item $t_i t_{i-1} = t_j t_{j-1}$ for $i, j \in \mathbb{Z}/e\mathbb{Z}$,
\item $t_i s_3 t_i = s_3 t_i s_3$ for $i \in \mathbb{Z}/e\mathbb{Z}$,
\item $s_j t_i = t_i s_j$ for $i \in \mathbb{Z}/e\mathbb{Z}$ and $4 \leq j \leq n$,
\item $s_i s_{i+1} s_i = s_{i+1} s_i s_{i+1}$ for $3 \leq i \leq n-1$,
\item $s_i s_j = s_j s_i$ for $|i-j| > 1$.

\end{enumerate}

\end{definition}

Recall that this presentation can be described by the following diagram. The dashed circle describes Relation $1$ of Definition \ref{DefPresB(de,e,n)CorranLeeLee}. The other edges used to describe all the other relations follow the standard conventions for Artin-Tits diagrams.

\begin{figure}[H]
\begin{center}
\begin{tikzpicture}[yscale=0.8,xscale=1,rotate=30]

\draw[thick,dashed] (0,0) ellipse (2cm and 1cm);

\node[draw, shape=circle, fill=white, label=above:\begin{small}$t_0$\end{small}] (t0) at (0,-1) {};
\node[draw, shape=circle, fill=white, label=above:\begin{small}$t_1$\end{small}] (t1) at (1,-0.8) {};
\node[draw, shape=circle, fill=white, label=right:\begin{small}$t_2$\end{small}] (t2) at (2,0) {};
\node[draw, shape=circle, fill=white, label=above:$t_i$] (ti) at (0,1) {};
\node[draw, shape=circle, fill=white, label=above:\begin{small}$t_{e-1}$\end{small}] (te-1) at (-1,-0.8) {};

\draw[thick,-] (0,-2) arc (-180:-90:3);

\node[draw, shape=circle, fill=white, label=below left:$s_3$] (s3) at (0,-2) {};

\draw[thick,-] (t0) to (s3);
\draw[thick,-,bend left] (t1) to (s3);
\draw[thick,-,bend left] (t2) to (s3);
\draw[thick,-,bend left] (s3) to (te-1);

\node[draw, shape=circle, fill=white, label=below:$s_4$] (s4) at (0.15,-3) {};
\node[draw, shape=circle, fill=white, label=below:$s_{n-1}$] (sn-1) at (2.2,-4.9) {};
\node[draw, shape=circle, fill=white, label=right:$s_{n}$] (sn) at (3,-5) {};

\node[fill=white] () at (1,-4.285) {$\cdots$};

\end{tikzpicture}
\end{center}
\caption{\mbox{Diagram for the presentation of Corran-Picantin of $B(e,e,n)$.}}
\end{figure}

Now, we recall the presentation of the braid group $B(d,1,n)$ for $d > 1$ and $n \geq 2$.

\begin{proposition}\label{PropPresB(d,1,n)}

Fix $d > 1$ and $n \geq 2$. The complex braid group $B(d,1,n)$ is defined by a presentation with generators $z, s_2, s_3$, $\cdots$, and $s_n$ and relations as follows.
\begin{enumerate}

\item $zs_2zs_2 = s_2zs_2z$,
\item $zs_j = s_jz$ for $2 \leq j \leq n$, and
\item $s_is_{i+1}s_i = s_{i+1}s_is_{i+1}$ for $2 \leq i \leq n-1$ and $s_is_j = s_js_i$ for $|i-j| > 1$.
\end{enumerate}

This presentation can be described by the following diagram.
\begin{figure}[H]

\begin{center}
\begin{tikzpicture}

\node[draw, shape=circle, label=above:$z$] (1) at (0,0) {};
\node[draw, shape=circle,label=above:$s_2$] (2) at (2,0) {};
\node[draw, shape=circle,label=above:$s_3$] (3) at (4,0) {};
\node[draw, shape=circle,label=above:$s_{n-1}$] (n-1) at (6,0) {};
\node[draw,shape=circle,label=above:$s_n$] (n) at (8,0) {};

\draw[thick,-,double] (1) to (2);
\draw[thick,-] (2) to (3);
\draw[dashed,-,thick] (3) to (n-1);
\draw[thick,-] (n-1) to (n);

\end{tikzpicture}
\end{center}

\caption{Diagram for the presentation of $B(d,1,n)$.}\label{DiagPresB(d,1,n)}
\end{figure}

\end{proposition}

\begin{remark}\label{RemPresG(de,e,n)}

As mentioned in Chapter \ref{ChapterNormalFormsG(de,e,n)}, if we add the relations $t_0^2=1$, $\cdots$, $t_{e-1}^2=1$, $s_3^2=1$, $\cdots$, $s_n^2=1$ to the presentation of $B(e,e,n)$, we get a presentation of the complex reflection groups $G(e,e,n)$. Also, if we add the relations $z^d = 1$, $s_2^2 = 1$, $\cdots$, $s_n^2 = 1$ to the presentation of $B(d,1,n)$, we get a presentation of the complex reflection group $G(d,1,n)$.

\end{remark}

We are ready to give the definitions of the Hecke algebras.

\begin{definition}\label{DefPresH(de,e,n)}

Let $e \geq 1$ and $n \geq 2$. We exclude the case (n=2, e even), see Remark \ref{RemarkTwoConjugHecke}. Let $R_0 = \mathbb{Z}[a]$. We define a unitary associative Hecke algebra $H(e,e,n)$ as the quotient of the group algebra $R_0(B(e,e,n))$ by the following relations.

\begin{enumerate}

\item $t_i^2 - at_i - 1 = 0$ for $0 \leq i \leq e-1$,
\item $s_j^2 - as_j - 1 = 0$ for $3 \leq j \leq n$.

\end{enumerate}
Then, a presentation of $H(e,e,n)$ is obtained by adding these relations to the braid relations of the presentation of $B(e,e,n)$ given in Definition \ref{DefPresB(de,e,n)CorranLeeLee}.

\end{definition}

\begin{definition}\label{DefPresH(d,1,n)}

Let $d > 1$ and $n \geq 2$. Let $R_0 = \mathbb{Z}[a,b_1,b_2, \cdots, b_{d-1}]$. We define a unitary associative Hecke algebra $H(d,1,n)$ as the quotient of the group algebra $R_0(B(d,1,n))$ by the following relations.

\begin{enumerate}

\item $z^d -b_1 z^{d-1} - b_2 z^{d-2} - \cdots - b_{d-1} z -1 = 0$,
\item $s_j^2 - as_j - 1 = 0$ for $2 \leq j \leq n$.

\end{enumerate}
Then, a presentation of $H(d,1,n)$ is obtained by adding these relations to the braid relations of the presentation of $B(d,1,n)$ given in Definition \ref{PropPresB(d,1,n)}.

\end{definition}

In the previous definitions, we use the polynomial ring $R_0$, instead of the usual Laurent polynomial ring $R$ that was introduced originally in Definition \ref{DefinitionHeckeAlgebra}. Actually, by Proposition 2.3 $(ii)$ of \cite{MarinG20G21}, Conjecture \ref{ConjectureOfLibertyBMR} (for $W = G(e,e,n)$) is equivalent to the fact that $H(e,e,n)$ of Definition \ref{DefPresH(de,e,n)} is a free $R_0$-module of rank $|G(e,e,n)|$. The same can be said for $G(d,1,n)$. Then, by proving that $H(e,e,n)$ and $H(d,1,n)$ are free $R_0$-modules of rank $|G(e,e,n)|$ and $|G(d,1,n)|$, respectively, we get a proof of Theorem \ref{PropositionBMRFreenessTheorem} for the case of $G(e,e,n)$ and $G(d,1,n)$.

\section{Basis for the case of $G(e,e,n)$}

The Hecke algebra $H(e,e,n)$ is described in Definition \ref{DefPresH(de,e,n)} by a presentation with generating set $\{ t_0, t_1, \cdots, t_{e-1}, s_3, \cdots, s_n \}$. We will replace $t_0$ by $s_2$ in some cases to simplify notations. We use the geodesic normal form of $G(e,e,n)$ defined in Section 2.1 of Chapter 2 in order to construct a basis for $H(e,e,n)$ that is different from the one defined by Ariki in \cite{ArikiHecke}. We introduce the following subsets of $H(e,e,n)$.

\begin{center}
\begin{tabular}{lllll}
$\Lambda_2 =$ & $\{$ & $1$,\\
& & $t_k$ & for $0 \leq k \leq e-1$, & \\
 & & $t_kt_0$ & for $1 \leq k \leq e-1$ & $\}$,
\end{tabular}
\end{center}
and for $3 \leq i \leq n$,
\begin{center}
\begin{tabular}{lllll}
$\Lambda_i =$ & $\{$ & $1$, & & \\
 & & $s_i \cdots s_{i'}$ & for $3 \leq i' \leq i$, & \\ 
 & & $s_i \cdots s_{3}t_k$ & for $0 \leq k \leq e-1$, & \\
 & & $s_i \cdots s_{3}t_ks_2 \cdots s_{i'}$ & for $1 \leq k \leq e-1$ and $2 \leq i' \leq i$ & $\}.$\\
\end{tabular}
\end{center}

\noindent Define $\Lambda = \Lambda_2 \cdots \Lambda_n$ to be the set of the products $a_2 \cdots a_n$, where $a_2 \in \Lambda_2, \cdots, a_n \in \Lambda_n$. Recall that $R_0 = \mathbb{Z}[a]$. The aim of this section is to prove the following theorem.

\begin{theorem}\label{TheoremNewBasis}

The set $\Lambda$ provides an $R_0$-basis of the Hecke algebra $H(e,e,n)$.

\end{theorem}

In order to prove this theorem, it is shown in Proposition 2.3 $(i)$ of \cite{MarinG20G21} that it is enough to find a spanning set of $H(e,e,n)$ over $R_0$ of $|G(e,e,n)|$ elements. This is a general fact about Hecke algebras associated to complex reflection groups. We have $|\Lambda_2| = 2e$, $|\Lambda_3| = 3e$, $\cdots$, and $|\Lambda_n| = ne$ by the definition of $\Lambda_2$, $\cdots$, and $\Lambda_n$. Thus, $|\Lambda|$ is equal to $e^{n-1}n!$ that is the order of $G(e,e,n)$. If we manage to prove that $\Lambda$ is a spanning set of $H(e,e,n)$ over $R_0$, then we get Theorem \ref{TheoremNewBasis}. Denote by Span$(S)$ the sub-$R_0$-module of $H(e,e,n)$ generated by $S$.\\

We prove Theorem \ref{TheoremNewBasis} by induction on $n \geq 2$. Propositions \ref{PropInductionHypothesis} and \ref{PropCasen=3} correspond to cases $n=2$ and $n=3$, respectively. Suppose that $\Lambda_2 \cdots \Lambda_{n-1}$ is an $R_0$-basis of $H(e,e,n-1)$. As mentioned above, in order to prove that $\Lambda = \Lambda_2 \cdots \Lambda_{n}$ is an $R_0$-basis of $H(e,e,n)$, it is enough to show that it is an $R_0$-generating set of $H(e,e,n)$, that is $\Lambda$ stable under left multiplication by $t_0, \cdots , t_{e-1}, s_3, \cdots$, and $s_n$. Since $\Lambda_2 \cdots \Lambda_{n-1}$ is an $R_0$-basis of $H(e,e,n-1)$, the set $\Lambda_2 \cdots \Lambda_n$ is stable under left multiplication by $t_0, \cdots , t_{e-1},s_3, \cdots$, and $s_{n-1}$. We prove that it is stable under left multiplication by $s_n$, that is $s_n(a_2 \cdots a_n) = a_2 \cdots a_{n-2}s_n(a_{n-1}a_n)$ belongs to Span$(\Lambda)$ for $a_2 \in \Lambda_2$, $\cdots$, and $a_n \in \Lambda_n$. We will check now all the different possibilities for $a_{n-1} \in \Lambda_{n-1}$ and $a_n \in \Lambda_n$.\\

Let us consider the case $n > 3$. If $a_{n-1} =1$ and $a_n \in \Lambda_n$, it is obvious that $s_n(a_{n-1}a_n)$ belongs to Span$(\Lambda_{n-1}\Lambda_n)$. Also, if $a_{n-1} \in \Lambda_{n-1}$ and $a_n = 1$, it is obvious that $s_n(a_{n-1}a_n)$ belongs to Span$(\Lambda_{n-1}\Lambda_n)$. If $a_{n-1} = s_{n-1} \cdots s_{i}$ for $2 \leq i \leq n-1$, we distinguish $3$ different cases for $a_n$ that belongs to $\Lambda_n$. This is done in Lemmas \ref{Lemma2}, \ref{Lemma3}, and \ref{Lemma4} below. If $a_{n-1} = s_{n-1} \cdots s_{3}t_k$ for $0 \leq k \leq e-1$, we also distinguish $3$ different cases for $a_n \in \Lambda_n$. This is done in Lemmas \ref{Lemma5}, \ref{Lemma6}, and \ref{Lemma7}. Finally, if $a_{n-1} = s_{n-1} \cdots s_{3}t_ks_2 \cdots s_{i}$ for $1 \leq k \leq e-1$ and $2 \leq i \leq n-1$, we also check the $3$ different possibilities for $a_n \in \Lambda_n$. This is done in Lemmas \ref{Lemma8}, \ref{Lemma9}, and \ref{Lemma10}.\\

In order to prove all this, we first need to establish the following two preliminary lemmas.

\begin{lemma}\label{LemmaTLTKInVect}

For $i, j \in \mathbb{Z}/e\mathbb{Z}$, we have $t_j t_i \in$ Span$(\Lambda_2)$.

\end{lemma}

\begin{proof}

If $i=j$, then we have $t_i^2 = at_i + 1 \in$ Span$(\Lambda_2)$. Suppose $i \neq j$. We have\\
$t_i t_{i-1} = t_j t_{j-1}$. If we multiply by $t_j$ on the left and by $t_{i-1}$ on the right, we get\\
$t_j t_i t_{i-1}^2 = t_j^{2} t_{j-1}t_{i-1}$. Using the quadratic relations, we have\\
$t_jt_i(a t_{i-1} +1) = (a t_j + 1)t_{j-1}t_{i-1}$, that is\\
$a t_j t_i t_{i-1} + t_jt_i = a t_j t_{j-1} t_{i-1} + t_{j-1}t_{i-1}$.\\
Replacing $t_i t_{i-1}$ by $t_jt_{j-1}$ and  $t_jt_{j-1}$ by $t_i t_{i-1}$, we get\\
$at_j^2 t_{j-1} + t_jt_i = at_it_{i-1}^2 + t_{j-1}t_{i-1}$. Using the quadratic relations, we have\\
$a(at_j +1) t_{j-1} + t_jt_i = at_i(at_{i-1}+1) + t_{j-1}t_{i-1}$, that is\\
$a^2 t_1t_0 + at_{j-1} + t_jt_i = a^2t_1t_0 + a t_i + t_{j-1}t_{i-1}$. Simplifying this relation, we get
\begin{center} $t_jt_i = t_{j-1}t_{i-1} + a(t_i - t_{j-1})$.\end{center}
Now, we apply the same operations to compute $t_{j-1}t_{i-1}$ and so on until we arrive to a term of the form $t_kt_0$ for some $k \in \mathbb{Z}/e\mathbb{Z}$. Thus, if $i \neq j$, then $t_jt_i$ belongs to Span$(\Lambda_2 \setminus \{1\})$.

\end{proof}

\begin{lemma}\label{LemmaTkT0}

For $1 \leq k \leq e-1$, we have $t_kt_0 \in R_0(t_1t_0) + R_0(t_1t_0)^2 + \cdots + R_0(t_1t_0)^k + R_0t_1 + R_0t_2 + \cdots + R_0t_{k-1}$.

\end{lemma}

\begin{proof}

We prove the property by induction on $k$. The property is clearly satisfied for $k = 1$. Let $k \geq 2$. Suppose $t_{k-1}t_0 \in R_0(t_1t_0) + R_0(t_1t_0)^2 + \cdots + R_0(t_1t_0)^{k-1} + R_0t_1 + R_0t_2 + \cdots + R_0t_{k-2}$. We have that $t_{k+1}t_k = t_kt_{k-1}$. Multiplying by $t_{k+1}$ on the left and by $t_0$ on the right, we get $t_{k+1}^2t_kt_0 =t_{k+1} t_kt_{k-1}t_0$. Using the quadratic relations and replacing $t_{k+1} t_k$ with $t_1t_0$, we get $(at_{k+1} +1)t_kt_0 =t_1 t_0t_{k-1}t_0$. After simplifying this relation, we have $t_kt_0 = t_1t_0(t_{k-1}t_0) - a(t_1t_0)t_0$. Using the induction hypothesis, we replace $t_{k-1}t_0$ by its value and we get $t_kt_0 \in t_1t_0(R_0(t_1t_0) + R_0(t_1t_0)^2 + \cdots + R_0(t_1t_0)^{k-1} + R_0t_1 + R_0t_2 + \cdots + R_0t_{k-2}) + R_0(t_1t_0)t_0$. This is equal to $R_0(t_1t_0)^2 + R_0(t_1t_0)^3 + \cdots + R_0(t_1t_0)^{k} + R_0(t_1t_0)t_1 + R_0(t_1t_0)t_2 + \cdots + R_0(t_1t_0)t_{k-2} + R_0(t_1t_0)t_0$. Now $(t_1t_0)t_m$ is equal to $t_{m+1}t_mt_m \in R_0t_1t_0 + R_0t_{m+1}$ for $1 \leq m \leq k-2$ and $(t_1t_0)t_0 \in R_0(t_1t_0) + R_0t_1$. It follows that $t_kt_0 \in R_0(t_1t_0) + R_0(t_1t_0)^2 + \cdots + R_0(t_1t_0)^k + R_0t_1 + R_0t_2 + \cdots + R_0t_{k-1}$.

\end{proof}

\begin{proposition}\label{PropInductionHypothesis}

Let $x = t_l$ with $l \in \mathbb{Z}/e\mathbb{Z}$. For all $a_2 \in \Lambda_2$, we have $xa_2$ belongs to Span$(\Lambda_2)$.

\end{proposition}

\begin{proof}
This is an immediate consequence of Lemma \ref{LemmaTLTKInVect}.
\end{proof}

\begin{remark}\label{RemarkTwoConjugHecke}

We excluded the case $n = 2$ and $e$ even in Definition \ref{DefPresH(de,e,n)}. In this case, there are two conjugacy classes of reflections. One can define $H(e,e,2)$ for $e$ even in the same way as Definition \ref{DefPresH(de,e,n)} with $R_0 = \mathbb{Z}[a_1,a_2]$ and two types of polynomial relations $t_i^2 - a_1 t_i - 1 = 0$, $t_i$'s corresponding to the first conjugacy class and $t_j^2 - a_2 t_j - 1 = 0$, $t_j$'s corresponding to the second conjugacy class. Similarly to Proposition \ref{PropInductionHypothesis}, one shows that $\Lambda_2$ is stable under left multiplication by $t_i$ for $i \in \mathbb{Z}/e\mathbb{Z}$. Then, by Proposition (2.3) (i) of \cite{MarinG20G21}, $\Lambda_2$ is an $R_0$-basis of $H(e,e,2)$ for $e$ even.

\end{remark}

\begin{proposition}\label{PropCasen=3}

For all $a_2 \in \Lambda_2$ and $a_3 \in \Lambda_3$, the element $s_3(a_2a_3)$ belongs to Span$(\Lambda_2\Lambda_3)$.

\end{proposition}

\begin{proof}

The case where $a_1 \in \Lambda_1$ and $a_2 = 1$ is obvious. The case where $a_1 = 1$ and $a_2 \in \Lambda_2$ is also obvious.

\emph{Case 1}. Suppose $a_2 = t_k$ for $0 \leq k \leq e-1$ and $a_3 = s_3$.\\
We have $s_3t_ks_3 = t_k\underline{s_3t_k} \in$ Span$(\Lambda_2\Lambda_3)$.

\emph{Case 2}. Suppose $a_2 = t_k$ for $0 \leq k \leq e-1$ and $a_3 = s_3t_l$ for $0 \leq l \leq e-1$.\\
We have $\underline{s_3t_ks_3}t_l = t_ks_3t_kt_l$. After replacing $t_kt_l$ by its decomposition over $\Lambda_2$ (see Lemma \ref{LemmaTLTKInVect}), we directly have $s_3t_ks_3t_l \in$ Span$(\Lambda_2\Lambda_3)$.

\emph{Case 3}. Suppose $a_2 = t_k$ for $0 \leq k \leq e-1$ and $a_3 = s_3t_lt_0$ or $a_3 = s_3t_lt_0s_3$ for $1 \leq l \leq e-1$. We have $\underline{s_3t_ks_3}t_lt_0s_3 = t_ks_3t_kt_lt_0s_3$. By replacing $t_kt_l$ by its value (see Lemma \ref{LemmaTLTKInVect}), we obviously have $s_3t_ks_3t_lt_0$ and $s_3t_ks_3t_lt_0s_3$ belong to Span$(\Lambda_2\Lambda_3)$.

\emph{Case 4}. Suppose $a_2 = t_kt_0$ for $1 \leq k \leq e-1$ and $a_3 = s_3$. We have $s_3(a_2a_3) = s_3t_kt_0s_3 \in$ Span$(\Lambda_2\Lambda_3)$.

\emph{Case 5}. Suppose $a_2 = t_kt_0$ with $1 \leq k \leq e-1$ and $a_3 = s_3t_l$ with $0 \leq l \leq e-1$.\\
We have $s_3(a_2a_3) = s_3t_kt_0s_3t_l$. Recall that by Lemma \ref{LemmaTkT0}, we have $t_kt_0 \in R_0(t_1t_0) + R_0(t_1t_0)^2 + \cdots + R_0(t_1t_0)^k + R_0t_1 + R_0t_2 + \cdots + R_0t_{k-1}$. Replacing $t_kt_0$ by its value, we have to deal with the following two terms:\\
$s_3 t_x s_3t_l$ with $1 \leq x \leq k-1$ and
$s_3 (t_1t_0)^x s_3t_l$ with $1 \leq x \leq k$.\\
The first term is done in Case 2. For the second term, we decrease the power of $(t_1t_0)$ and use $t_1t_0 = t_{l+1}t_l$ to get
$s_3 (t_1t_0)^{x-1}t_{l+1}\underline{t_l s_3t_l}$. We apply a braid relation and then get
$s_3 (t_1t_0)^{x-1}t_{l+1}s_3t_ls_3$. Again, we decrease the power of $(t_1t_0)$ and use $t_1t_0 = t_{l+2}t_{l+1}$. We get $s_3 (t_1t_0)^{x-2}t_{l+2}t_{l+1}^2s_3t_ls_3 \in R_0 s_3(t_1t_0)^{x-2}t_{l+2}t_{l+1}s_3t_ls_3 + R_0s_3 (t_1t_0)^{x-2}t_{l+2}s_3t_ls_3$. We continue by decreasing the power of $(t_1t_0)$ and we get in the next step that $s_3(a_2 a_3)$ belongs to
$R_0 s_3(t_1t_0)^{x-3}t_{l+1}s_3t_{l}s_3^3$
$+ R_0 s_3(t_1t_0)^{x-3}t_{l+2}s_3t_{l}s_3^2$\\
$+ R_0 s_3 (t_1t_0)^{x-3}t_{l+1}s_3t_{l}s_3^2$
$+ R_0 s_3 (t_1t_0)^{x-3}t_{l+3}s_3t_ls_3$.
Inductively, we arrive to terms of the form
$s_3 t_1t_0t_{x'}s_3t_l(s_3)^{x''}$ ($0 \leq x' \leq e-1$ and $x'' \in \mathbb{N}$). Replace $t_1t_0$ by $t_{x'+1}t_{x'}$, we get $s_3 t_{x'+1}(t_{x'})^2s_3t_l(s_3)^{x''}$ which belongs to
$R_0 s_3 t_{x'+1}t_{x'}s_3t_l(s_3)^{x''}$
$+ R_0 s_3 t_{x'+1}s_3t_l(s_3)^{x''}$. Replacing $t_{x'+1}t_{x'}$ by $t_{l+1}t_l$ and applying a braid relation in the first term, we get
$R_0 s_3 t_{l+1}s_3t_l(s_3)^{x''+1}$
$+ R_0 s_3 t_{x'+1}s_3t_l(s_3)^{x''}$.
Since $s_3^2=as_3+1$, it remains to deal with these $2$ terms:
\begin{center}$s_3 t_xs_3t_l$ and $s_3 t_xs_3t_ls_3$, for some $0 \leq x \leq e-1$.\end{center} It is easily checked that they belong to Span$(\Lambda_2 \Lambda_3)$.

\emph{Case 6}. Suppose $a_2 = t_kt_0$ and $a_3 = s_3t_lt_0$ or $a_3 = s_3t_lt_0s_3$ ($1 \leq k,l \leq e-1$).\\
By Case 5, we get two terms of the form $s_3 t_xs_3t_l$ and $s_3 t_xs_3t_ls_3$. Multiplying them on the right by $t_0$ then by $s_3$, we get that $s_3(a_2a_3)$ obviously belongs to Span$(\Lambda_2 \Lambda_3)$.

\end{proof}

In the sequel, we will indicate by (1) the operation that shifts the underlined letters to the left, by (2) the operation that applies braid relations, and by (3) the one that applies quadratic relations. The following lemma is useful in the proofs of Lemmas \ref{Lemma7}, \ref{Lemma9},  and \ref{Lemma10}. Denote by $S_{n-1}^{*}$ the set of the words over $\{t_0, \cdots, t_{e-1}, s_3, \cdots, s_{n-1} \}$.\\

\begin{lemma}\label{LmScholie1}

Let $3 \leq i \leq n$. We have $s_n \cdots s_4s_3^2 s_4 \cdots s_i$ belongs to Span$(S_{n-1}^{*}\Lambda_n)$. 

\end{lemma}

\begin{proof}

If $i=3$, we have $s_n \cdots s_4s_3^2 \in R_0 s_n \cdots s_4s_3 + R_0 s_n \cdots s_4$ (the last term is equal to $1$ if $n=3$). If $i = 4$, we have $s_n \cdots s_4 s_3^2s_4 \in R_0 s_n \cdots s_4s_3s_4 + R_0 s_n \cdots s_4^2 \in R_0 s_3 s_n \cdots s_3 + R_0 s_n \cdots s_4 + R_0 s_n \cdots s_5$. The last term is equal to $1$ if $n = 4$. Let $i \geq 5$.
We have $s_n \cdots s_4s_3^2 s_4 \cdots s_i$ belongs to $R_0 s_n \cdots s_4 s_3 s_4 \cdots s_i + R_0 s_n \cdots s_5 s_4^2 s_5 \cdots s_i$. We apply the quadratic relation $s_4^2 = as_4+1$ to the second term and get\\
$R_0 s_n \cdots s_5 s_4 s_5 \cdots s_i + R_0 s_n \cdots s_6 s_5^2 s_6 \cdots s_i$. And so on, we apply quadratic relations. We get terms of the form $s_n \cdots s_{k+1}s_ks_{k+1} \cdots s_i$ with $k+1 \leq i$ and a term of the form $s_n \cdots s_{i+1}s_is_{i-1}^{2}s_i$.\\
We have $s_n \cdots s_{i+1}s_is_{i-1}^{2}s_i$ belongs to\\
$R_0 s_n \cdots s_{i+1}s_is_{i-1}s_i + R_0 s_n \cdots s_{i+1}s_i + R_0 s_n \cdots s_{i+1} \subseteq \\ R_0 s_{i-1} s_n \cdots s_{i-1} + R_0 s_n \cdots s_{i+1}s_i + R_0 s_n \cdots s_{i+1}$ (the last term is equal to $1$ if $i = n$). Hence it belongs to Span$(S_{n-1}^{*}\Lambda_n)$.\\
The other terms are of the form $s_n \cdots s_{k+1}s_ks_{k+1} \cdots s_i$ ($k+1 \leq i$). We have
$\begin{array}{ll}
s_n \cdots \underline{s_{k+1}s_ks_{k+1}} \cdots s_i & \stackrel{(2)}{=}\\
s_n \cdots s_{k+2}\underline{s_k} s_{k+1}s_k \underline{s_{k+2}} \cdots s_{i} & \stackrel{(1)}{=}\\
s_k s_n \cdots \underline{s_{k+2} s_{k+1}s_{k+2}}s_k s_{k+3} \cdots s_{i} & \stackrel{(2)}{=}\\
s_k s_n \cdots \underline{s_{k+1}} s_{k+2}s_{k+1}s_k s_{k+3} \cdots s_{i} & \stackrel{(1)}{=}
\end{array}$\\
$\begin{array}{l}
s_k s_{k+1} s_n \cdots s_{k+2}s_{k+1}s_k \underline{s_{k+3}} \cdots \underline{s_{i-1}} s_{i}.
\end{array}$\\
We apply the same operations to $\underline{s_{k+3}}$, $\cdots$, $\underline{s_{i-1}}$ and get\\
$\begin{array}{ll}
s_k s_{k+1} \cdots s_{i-2} s_n \cdots s_{k} \underline{s_{i}} & \stackrel{(1)}{=}\\
s_k s_{k+1} \cdots s_{i-2} s_n \cdots \underline{s_i s_{i-1}s_i} s_{i-2} \cdots s_k & \stackrel{(2)}{=}\\
s_k s_{k+1} \cdots s_{i-2} s_n \cdots s_{i+1}\underline{s_{i-1}} s_{i}s_{i-1} s_{i-2} \cdots s_k & \stackrel{(1)}{=}\\
s_ks_{k+1} \cdots s_{i-1} s_ns_{n-1} \cdots s_{k}, &
\end{array}$\\
which belongs to Span$(S_{n-1}^{*}\Lambda_n)$.

\end{proof}

\begin{lemma}\label{Lemma2}

If $a_{n-1} = s_{n-1}s_{n-2} \cdots s_i$ with $3 \leq i \leq n-1$ and $a_n = s_n s_{n-1} \cdots s_{i'}$ with $3 \leq i' \leq n$, then $s_n (a_{n-1}a_n)$ belongs to Span$(\Lambda_{n-1}\Lambda_n)$.

\end{lemma}

\begin{proof}

\emph{Suppose $i < i'$.} We have $s_n(a_{n-1}a_n)$ is equal to\\
$\begin{array}{ll}
s_ns_{n-1}s_{n-2} \cdots s_i\underline{s_n}s_{n-1} \cdots s_{i'} & \stackrel{(1)}{=}\\
\underline{s_ns_{n-1}s_n}s_{n-2} \cdots s_{i}s_{n-1} \cdots s_{i'} & \stackrel{(2)}{=}\\
s_{n-1}s_ns_{n-1}s_{n-2} \cdots s_{i}\underline{s_{n-1}} \cdots s_{i'} & \stackrel{(1)}{=}\\
s_{n-1}s_n\underline{s_{n-1}s_{n-2}s_{n-1}} \cdots s_{i}s_{n-2} \cdots s_{i'} & \stackrel{(2)}{=}\\
s_{n-1}s_n\underline{s_{n-2}}s_{n-1}s_{n-2} \cdots s_{i}s_{n-2} \cdots s_{i'} & \stackrel{(1)}{=} 
\end{array}$\\
$\begin{array}{l}
s_{n-1}s_{n-2}s_ns_{n-1}s_{n-2} \cdots s_{i}\underline{s_{n-2}} \cdots \underline{s_{i'}}.
\end{array}$\\
We apply the same operations to the underlined letters $\underline{s_{n-2}} \cdots \underline{s_{i'}}$ in order to get
$s_{n-1}s_{n-2}\cdots s_{i'-1}s_ns_{n-1} \cdots s_i$ which belongs to Span$(\Lambda_{n-1}\Lambda_n)$.\\

\emph{Suppose $i \geq i'$.} We have $s_n(a_{n-1}a_n)$ is equal to $s_ns_{n-1} \cdots s_i\underline{s_n}s_{n-1} \cdots s_{i'}$. We apply the operations (1) and (2) and get $s_{n-1}s_{n}s_{n-1} \cdots s_{i}\underline{s_{n-1}} \cdots s_{i'}$. Then we apply the same operations to $\underline{s_{n-1}}$ and get $s_{n-1}s_{n-2}s_ns_{n-1} \cdots s_i \underline{s_{n-2} \cdots s_{i'}}$.
Since $i \geq i'$, we write $s_{i+2}s_{i+1}s_i$ in $\underline{s_{n-2} \cdots s_{i'}}$ and get $s_{n-1}s_{n-2}s_ns_{n-1} \cdots s_i \underline{s_{n-2}} \cdots \underline{s_{i+2}}s_{i+1}s_i \cdots s_{i'}$. Similarly, we apply the same operations to $\underline{s_{n-2}}$, $\cdots$, $\underline{s_{i+2}}$ and get\\
$\begin{array}{ll}
s_{n-1} \cdots s_{i+1}s_n \cdots \underline{s_{i+1}s_is_{i+1}}s_{i}\cdots s_{i'} & \stackrel{(2)}{=}\\
s_{n-1} \cdots s_{i+1}s_n \cdots s_{i+2}\underline{s_{i}}s_{i+1}s_{i}^2s_{i-1} \cdots s_{i'} & \stackrel{(1)}{=}\\
s_{n-1} \cdots s_{i}s_n \cdots s_{i+2}s_{i+1}\underline{s_{i}^2}s_{i-1} \cdots s_{i'} & \stackrel{(3)}{=}
\end{array}$\\
$\begin{array}{l} a s_{n-1} \cdots s_{i}s_n \cdots s_{i+1}s_is_{i-1}\cdots s_{i'} +  s_{n-1} \cdots s_{i}s_n \cdots s_{i+2}s_{i+1}\underline{s_{i-1}}\cdots \underline{s_{i'}}.\end{array}$\\
The first term belongs to Span$(\Lambda_{n-1}\Lambda_n)$. For the second term, the underlined letters commute with $s_n \cdots s_{i+2}s_{i+1}$ hence they are shifted to the left. We thus get $s_n(a_{n-1}a_n)$ is equal to
$a s_{n-1} \cdots s_{i}s_n \cdots s_{i'} +  s_{n-1} \cdots s_{i'} s_n \cdots s_{i+1}$ which belongs to Span$(\Lambda_{n-1}\Lambda_n)$.

\end{proof}

\begin{lemma}\label{Lemma3}

If $a_{n-1} = s_{n-1} \cdots s_i$ with $3 \leq i \leq n-1$ and $a_n = s_n \cdots s_{3}t_{k}$ with $0 \leq k \leq e-1$, then $s_n (a_{n-1}a_n)$ belongs to Span$(\Lambda_{n-1}\Lambda_n)$.

\end{lemma}

\begin{proof}

This corresponds to the case $i'=3$ in the proof of Lemma \ref{Lemma2} with a right multiplication by $t_k$ for $0 \leq k \leq e-1$. Since $i \geq 3$, by the case \mbox{$i \geq i'$} of Lemma \ref{Lemma2}, we have $s_n(a_{n-1}a_n) = a s_{n-1} \cdots s_{i}s_n \cdots s_{3}t_k +  s_{n-1} \cdots s_{3} s_n \cdots s_{i+1} t_k$. In the second term, $t_k$ commutes with $s_n \cdots s_{i+1}$ hence it is shifted to the left. We get $s_n(a_{n-1}a_n) = a s_{n-1} \cdots s_{i}s_n \cdots s_{3}t_k +  s_{n-1} \cdots s_{3}t_k s_n \cdots s_{i+1}$ which belongs to Span$(\Lambda_{n-1}\Lambda_n)$.

\end{proof}

\begin{lemma}\label{Lemma4}

If $a_{n-1} = s_{n-1} \cdots s_i$ with $3 \leq i \leq n-1$ and $a_n = s_n \cdots s_{3}t_{k}s_2s_3 \cdots s_{i'}$ with $1 \leq k \leq e-1$ and $2 \leq i' \leq n$, then $s_n (a_{n-1}a_n)$ belongs to Span$(\Lambda_{n-1}\Lambda_n)$.

\end{lemma}

\begin{proof}

According to Lemma \ref{Lemma3}, we have\\ $s_n(a_{n-1}a_n) = a s_{n-1} \cdots s_{i}s_n \cdots s_{3}t_ks_2 \cdots s_{i'} +  s_{n-1} \cdots s_{3}t_k s_n \cdots s_{i+1} s_2 \cdots s_{i'}$.
The first term is an element of Span$(\Lambda_{n-1}\Lambda_n)$. We check that the second term also belongs to Span$(\Lambda_{n-1}\Lambda_n)$. Actually,

\emph{If $i' < i$}, the second term is equal to
$s_{n-1} \cdots s_{3}t_k s_n \cdots s_{i+1} \underline{s_2} \cdots \underline{s_{i'}}$.\\ The underlined letters commute with $s_n \cdots s_{i+1}$ and are shifted to the left. We get $s_{n-1} \cdots s_{3}t_ks_2 \cdots s_{i'}s_n \cdots s_{i+1} \in$ Span$(\Lambda_{n-1}\Lambda_n)$.

\emph{If $i' \geq i$}, we write $s_{i-1}s_is_{i+1}$ in $s_2 \cdots s_{i'}$ and get $s_{n-1} \cdots s_{3}t_k s_n \cdots s_{i+1} (s_2 \cdots s_{i'}) = $\\
$\begin{array}{ll}
s_{n-1} \cdots s_{3}t_k s_n \cdots s_{i+1} (\underline{s_2} \cdots \underline{s_{i-1}}s_is_{i+1} \cdots s_{i'}) & \stackrel{(1)}{=}\\
s_{n-1} \cdots s_{3}t_k s_2 \cdots s_{i-1} s_n \cdots \underline{s_{i+1} (s_is_{i+1}} \cdots s_{i'}) & \stackrel{(2)}{=}\\
s_{n-1} \cdots s_{3}t_k s_2 \cdots s_{i-1} s_n \cdots \underline{s_{i}} s_{i+1} s_{i} (s_{i+2} \cdots s_{i'}) & \stackrel{(1)}{=}\\
s_{n-1} \cdots s_{3}t_k s_2 \cdots s_{i-1}s_i s_n \cdots s_{i+1} s_{i} (\underline{s_{i+2}} \cdots s_{i'}) & \stackrel{(1)}{=}\\
s_{n-1} \cdots s_{3}t_k s_2 \cdots s_{i-1}s_i s_n \cdots \underline{s_{i+2}s_{i+1} s_{i+2}} s_{i}(s_{i+3} \cdots s_{i'}) & \stackrel{(2)}{=}\\
s_{n-1} \cdots s_{3}t_k s_2 \cdots s_{i-1}s_i s_n \cdots \underline{s_{i+1}}s_{i+2} s_{i+1} s_{i}(s_{i+3} \cdots s_{i'}) & \stackrel{(1)}{=}
\end{array}$\\
$\begin{array}{l}
s_{n-1} \cdots s_{3}t_k s_2 \cdots s_i s_{i+1} s_n \cdots s_{i+2} s_{i+1} s_{i}(\underline{s_{i+3}} \cdots \underline{s_{i'}}).
\end{array}$\\
We apply the same operations to the underlined letters $\underline{s_{i+3}}$, $\cdots$, $\underline{s_{i'}}$. We finally get $s_{n-1} \cdots s_{3}t_ks_2 \cdots s_{i'-1}s_n \cdots s_i \in$ Span$(\Lambda_{n-1}\Lambda_n)$.

\end{proof}

\begin{lemma}\label{Lemma5}

If $a_{n-1} = s_{n-1} \cdots s_3t_k$ with $0 \leq k \leq e-1$ and $a_n = s_n \cdots s_{i}$ with $3 \leq i \leq n$, then $s_n (a_{n-1}a_n)$ belongs to Span$(\Lambda_{n-1}\Lambda_n)$.

\end{lemma}

\begin{proof}

We have $s_n(a_{n-1}a_n)$ is equal to\\
$\begin{array}{ll}
s_ns_{n-1} \cdots s_3t_k \underline{s_n} \cdots s_i & \stackrel{(1)}{=}\\
\underline{s_{n}s_{n-1}s_n} \cdots s_3t_ks_{n-1} \cdots s_i  & \stackrel{(2)}{=}\\
s_{n-1}s_{n}s_{n-1} \cdots s_3t_k \underline{s_{n-1}} \cdots s_i & \stackrel{(1)}{=}\\
s_{n-1}s_{n}\underline{s_{n-1}s_{n-2}s_{n-1}} \cdots s_3t_k s_{n-2} \cdots s_i & \stackrel{(2)}{=}\\
s_{n-1}s_{n} \underline{s_{n-2}}s_{n-1}s_{n-2} \cdots s_3t_k s_{n-2} \cdots s_i & \stackrel{(1)}{=}\\
\end{array}$\\
$\begin{array}{l}
s_{n-1} s_{n-2} s_{n}s_{n-1} \cdots s_3t_k \underline{s_{n-2}} \cdots \underline{s_i}.
\end{array}$ \mbox{Now we apply the same operations for $\underline{s_{n-2}}$, $\cdots$, $\underline{s_i}$.}

\emph{If $i=3$}, we get $s_{n-1} \cdots s_3 s_{n}s_{n-1} \cdots \underline{s_3t_ks_3}$. Next, we apply a braid relation to get $s_{n-1} \cdots s_3 s_{n}s_{n-1} \cdots \underline{t_k}s_3t_k$, then we shift $\underline{t_k}$ to the left and we finally get $s_{n-1} \cdots s_3 t_k s_{n}s_{n-1} \cdots s_3t_k$ which belongs to Span$(\Lambda_{n-1}\Lambda_n)$.\\
\emph{If $i > 3$}, we directly get 
$ s_{n-1} \cdots s_is_{i-1} s_{n} \cdots s_3t_k$ which also belongs to Span$(\Lambda_{n-1}\Lambda_n)$. 

\end{proof}

\begin{lemma}\label{Lemma6}

If $a_{n-1} = s_{n-1} \cdots s_3t_k$ with $0 \leq k \leq e-1$ and $a_n = s_n \cdots s_{3}t_l$ with $0 \leq l \leq e-1$, then $s_n (a_{n-1}a_n)$ belongs to Span$(\Lambda_{n-1}\Lambda_n)$.

\end{lemma}

\begin{proof}

By Lemma \ref{Lemma5}, one can write
$s_n(a_{n-1}a_n) = s_{n-1} \cdots s_3 t_k s_{n}s_{n-1} \cdots s_3t_kt_l$.
By Lemma \ref{LemmaTLTKInVect} where we compute $t_kt_l$, we directly deduce that $s_n(a_{n-1}a_n)$ belongs to Span$(\Lambda_{n-1}\Lambda_n)$.

\end{proof}

\begin{lemma}\label{Lemma7}

If $a_{n-1} = s_{n-1} \cdots s_3t_k$ with $0 \leq k \leq e-1$ and $a_n = s_n \cdots s_{3}t_l s_2 \cdots s_i$ with $2 \leq i \leq n$ and $1 \leq l \leq e-1$, then $s_n (a_{n-1}a_n)$ belongs to Span$(S_{n-1}^{*} \Lambda_n)$.

\end{lemma}

\begin{proof}

By the previous lemma, we have $s_n(a_{n-1}a_n) = s_{n-1}\cdots s_3t_ks_n \cdots s_3 t_k t_l (s_2 \cdots s_i)$.
By Lemma \ref{LemmaTLTKInVect}, the case $i = 2$ is obvious. Suppose $i \geq 3$.
After replacing $t_kt_l$ by its value given in Lemma \ref{LemmaTLTKInVect}, we have two different terms in $s_n(a_{n-1}a_n)$ of the form $s_{n-1}\cdots s_3t_ks_n \cdots s_3 t_x (s_2 \cdots s_i)$ with $0 \leq x \leq e-1$ and of the form $s_{n-1}\cdots s_3t_ks_n \cdots s_3 t_x (s_3 \cdots s_i)$ with $0 \leq x \leq e-1$.

For terms of the form $s_{n-1}\cdots s_3t_ks_n \cdots s_3 t_x (s_3 \cdots s_i)$ with $0 \leq x \leq e-1$, we have\\
$\begin{array}{ll}
s_{n-1} \cdots s_3 t_k s_n \cdots \underline{s_3 t_x (s_3} \cdots s_i) & \stackrel{(2)}{=}\\
s_{n-1} \cdots s_3 t_k s_n \cdots \underline{t_x} s_3 t_x(s_4 \cdots s_i) & \stackrel{(1)}{=}\\
s_{n-1} \cdots s_3 t_kt_x s_n \cdots s_3 t_x(\underline{s_4} \cdots s_i) & \stackrel{(1)}{=}\\
s_{n-1} \cdots s_3 t_kt_x s_n \cdots \underline{s_4 s_3 s_4}t_x (s_5\cdots s_i) & \stackrel{(2)}{=}\\
s_{n-1} \cdots s_3 t_kt_x s_n \cdots \underline{s_3} s_4 s_3 t_x (s_5\cdots s_i) & \stackrel{(1)}{=}\\
\end{array}$\\
$\begin{array}{l}
s_{n-1} \cdots s_3 t_kt_x s_3 s_n \cdots s_3 t_x (\underline{s_5} \cdots \underline{s_{i}}).
\end{array}$\\
We apply the same operations for the underlined letters to get\\ $s_{n-1} \cdots s_3 t_kt_xs_3 \cdots s_{i-1} \underline{s_n \cdots s_3t_x}$ which belongs to Span$(S_{n-1}^{*} \Lambda_n)$.

Consider the terms of the form $s_{n-1}\cdots s_3t_ks_n \cdots s_3 t_x (s_2 \cdots s_i)$ with $0 \leq x \leq e-1$.\\ If $x \neq 0$, they belong to Span$(\Lambda_{n-1}\Lambda_n)$.\\ If $x =0$, we have $s_{n-1}\cdots s_3t_ks_n \cdots s_3 t_0 (s_2s_3 \cdots s_i) \in R_0 s_{n-1}\cdots s_3t_ks_n \cdots s_3 t_0 s_3 \cdots s_i + R_0 s_{n-1}\cdots s_3t_ks_n \cdots s_4s_3^2s_4 \cdots s_i$. The first term correspond to the previous case (with $x=0$) and then belongs to Span$(S_{n-1}^{*} \Lambda_n)$. By Lemma \ref{LmScholie1}, the second term also belongs to Span$(S_{n-1}^{*} \Lambda_n)$.

\end{proof}

\begin{lemma}\label{Lemma8}

If $a_{n-1} = s_{n-1} \cdots s_3t_ks_2 \cdots s_i$ with $2 \leq i \leq n-1$, $1 \leq k \leq e-1$ and $a_n = s_n \cdots s_{i'}$ with $3 \leq i' \leq n$, then $s_n (a_{n-1}a_n)$ belongs to Span$(\Lambda_{n-1}\Lambda_n)$.

\end{lemma}

\begin{proof}

\emph{Suppose $i < i'$}. We have $s_n(a_{n-1}a_n)$ is equal to\\
$\begin{array}{ll}
s_n \cdots s_3t_ks_2 \cdots s_i \underline{s_n} \cdots s_{i'} & \stackrel{(1)}{=}\\
\underline{s_n s_{n-1}s_n} \cdots s_3t_ks_2 \cdots s_i s_{n-1} \cdots s_{i'} & \stackrel{(2)}{=}\\
s_{n-1} s_{n}s_{n-1} \cdots s_3t_ks_2 \cdots s_i \underline{s_{n-1}} \cdots s_{i'} & \stackrel{(1)}{=}\\
s_{n-1} s_{n}\underline{s_{n-1}s_{n-2}s_{n-1}} \cdots s_3t_ks_2 \cdots s_i s_{n-2} \cdots s_{i'} & \stackrel{(2)}{=}\\
s_{n-1} s_{n} \underline{s_{n-2}} s_{n-1}s_{n-2} \cdots s_3t_ks_2 \cdots s_i s_{n-2} \cdots s_{i'} & \stackrel{(1)}{=}\\
\end{array}$\\
$\begin{array}{l}
s_{n-1} s_{n-2} s_{n} s_{n-1}s_{n-2} \cdots s_3t_ks_2 \cdots s_i \underline{s_{n-2}} \cdots \underline{s_{i'+1}}s_{i'}.
\end{array}$\\
We apply the same operations to the underlined letters and\\
we get $s_{n-1} \cdots s_{i'}s_n\cdots s_3t_ks_2 \cdots s_i s_{i'}$.\\
\emph{If $i'=i+1$}, we directly have $s_n(a_{n-1}a_n) \in$ Span$(\Lambda_{n-1}\Lambda_n)$.\\
\emph{If $i' > i+1$}, then we write $s_{i'+1}s_{i'}s_{i'-1}$ in the underlined word of\\ $s_{n-1} \cdots s_{i'}\underline{s_n\cdots s_3}t_ks_2 \cdots s_i s_{i'}$ and get\\
$\begin{array}{ll}
s_{n-1} \cdots s_{i'}s_n \cdots s_{i'+1}s_{i'}s_{i'-1}\cdots s_3t_ks_2 \cdots s_i \underline{s_{i'}} & \stackrel{(1)}{=}\\
s_{n-1} \cdots s_{i'}s_n \cdots s_{i'+1}\underline{s_{i'}s_{i'-1}s_{i'}} \cdots s_3t_ks_2 \cdots s_i  & \stackrel{(2)}{=}\\
s_{n-1} \cdots s_{i'}s_n \cdots s_{i'+1}\underline{s_{i'-1}}s_{i'}s_{i'-1} \cdots s_3t_ks_2 \cdots s_i & \stackrel{(1)}{=}\\
\end{array}$\\
$\begin{array}{l}
s_{n-1} \cdots s_{i'-1}s_n \cdots s_3t_ks_2 \cdots s_i
\end{array}$ which belongs to Span$(\Lambda_{n-1}\Lambda_n)$.

\emph{Suppose $i \geq i'$}. We have $s_n(a_{n-1}a_n)$ is equal to\\
$s_n \cdots s_3t_ks_2 \cdots s_i \underline{s_n} \cdots s_{i'}$. We shift $\underline{s_n}$ to the left and apply a braid relation to get
$s_{n-1} s_{n} s_{n-1} \cdots s_3t_ks_2 \cdots s_i \underline{s_{n-1} \cdots s_{i'}}$. Write $s_{i+2}s_{i+1}$ in the underlined word to get
$s_{n-1} s_{n} s_{n-1} \cdots s_3t_ks_2 \cdots s_i \underline{s_{n-1}} \cdots \underline{s_{i+2}}s_{i+1} \cdots s_{i'}$. We apply the same operations to the underlined letters to get\\
$\begin{array}{ll}
s_{n-1}\cdots s_{i+1}s_n \cdots s_3t_ks_2 \cdots \underline{s_i s_{i+1} s_i} s_{i-1} \cdots s_{i'} & \stackrel{(2)}{=}\\
s_{n-1}\cdots s_{i+1}s_n \cdots s_3t_ks_2 \cdots \underline{s_{i+1}} s_{i} s_{i+1} s_{i-1} \cdots s_{i'} & \stackrel{(1)}{=}\\
\end{array}$\\
$\begin{array}{l}
s_{n-1}\cdots s_{i}s_n \cdots s_3t_ks_2 \cdots s_{i-1}s_{i} s_{i+1}s_{i-1} \cdots s_{i'}
\end{array}$ (The details of the computation is left to the reader). Then we have\\
$\begin{array}{ll}
s_{n-1}\cdots s_{i}s_n \cdots s_3t_ks_2 \cdots s_{i-1}s_{i}s_{i+1}\underline{s_{i-1}} \cdots s_{i'} & \stackrel{(1)}{=}\\
s_{n-1}\cdots s_{i}s_n \cdots s_3t_ks_2 \cdots \underline{s_{i-1}s_{i}s_{i-1}}s_{i+1}s_{i-2} \cdots s_{i'} & \stackrel{(2)}{=}\\
s_{n-1}\cdots s_{i}s_n \cdots s_3t_ks_2 \cdots \underline{s_{i}}s_{i-1}s_{i}s_{i+1}s_{i-2} \cdots s_{i'} & \stackrel{(1)}{=}\\
\end{array}$\\
$\begin{array}{l}
s_{n-1}\cdots s_{i-1}s_n \cdots s_3t_ks_2 \cdots s_{i-2}s_{i-1}s_{i}s_{i+1}\underline{s_{i-2}} \cdots \underline{s_{i'+1}}s_{i'}.
\end{array}$ We apply the same operations to the underlined letters and we finally get\\
$s_{n-1}\cdots s_{i'+1}s_n \cdots s_3t_k\underline{s_2s_3 \cdots s_is_{i+1}}s_{i'}$. We write $s_{i'-1}s_{i'}s_{i'+1}$ in the underlined word and get $s_{n-1}\cdots s_{i'+1}s_n \cdots s_3t_ks_2s_3 \cdots s_{i'-1}s_{i'}s_{i'+1}\cdots s_is_{i+1} \underline{s_{i'}}$. We shift $\underline{s_{i'}}$ to the left. We get $s_{n-1}\cdots s_{i'+1}s_n \cdots s_3t_ks_2s_3 \cdots s_{i'-1}\underline{s_{i'}s_{i'+1}s_{i'}}\cdots s_is_{i+1}$. We apply a braid relation and get $s_{n-1}\cdots s_{i'+1}s_n \cdots s_3t_ks_2s_3 \cdots s_{i'-1} \underline{s_{i'+1}}s_{i'}s_{i'+1}\cdots s_is_{i+1}$. Now we shift $\underline{s_{i'+1}}$ to the left and get\\ 
$\begin{array}{ll}
s_{n-1}\cdots s_{i'+1}s_n \cdots \underline{s_{i'+1}s_{i'}s_{i'+1}}s_{i'-1}\cdots s_3t_ks_2 \cdots s_{i+1} & \stackrel{(2)}{=}\\
s_{n-1}\cdots s_{i'+1}s_n \cdots \underline{s_{i'}}s_{i'+1}s_{i'}s_{i'-1}\cdots s_3t_ks_2 \cdots s_{i+1} & \stackrel{(1)}{=}\\
s_{n-1}\cdots s_{i'+1}s_{i'}s_n \cdots s_{i'+1}s_{i'}s_{i'-1}\cdots s_3t_ks_2 \cdots s_{i+1} & = \\
\end{array}$\\
$\begin{array}{l}
s_{n-1} \cdots s_{i'}s_n \cdots s_3t_ks_2 \cdots s_{i+1},
\end{array}$ which belongs to Span$(\Lambda_{n-1}\Lambda_n)$.

\end{proof}

\begin{lemma}\label{Lemma9}

If $a_{n-1} = s_{n-1} \cdots s_3t_ks_2 \cdots s_i$ with $2 \leq i \leq n-1$, \mbox{$1 \leq k \leq e-1$} and $a_n = s_n s_{n-1} \cdots s_{3}t_l$ with $0 \leq l \leq e-1$, then $s_n (a_{n-1}a_n)$ belongs to Span$(S_{n-1}^{*} \Lambda_n)$.

\end{lemma}

\begin{proof}

By the final result of the computations in Lemma \ref{Lemma8}, we have $s_n(a_{n-1}a_n)$ is equal to
$s_{n-1}\cdots s_3 s_n \cdots s_3t_ks_2s_3 \cdots s_{i+1}\underline{t_l}$. We shift $\underline{t_l}$ to the left and get\\
$s_{n-1}\cdots s_3 s_n \cdots s_3t_ks_2s_3t_l \cdots s_{i+1}$. By Case 5 of Proposition \ref{PropCasen=3}, we have to deal with the following two terms:
\begin{itemize}

\item $s_{n-1}\cdots s_3s_n \cdots s_3 t_x s_3t_ls_4 \cdots s_{i+1}$ and
\item $s_{n-1}\cdots s_3s_n \cdots s_3 t_x s_3 t_l s_3 s_4 \cdots s_{i+1}$ with $1 \leq x,l \leq e-1$.
\end{itemize}
The first term is of the form\\
$\begin{array}{ll}
s_{n-1}\cdots s_3s_n \cdots \underline{s_3 t_x s_3}t_ls_4 \cdots s_{i+1} & \stackrel{(2)}{=}\\
s_{n-1}\cdots s_3s_n \cdots s_4\underline{t_x} s_3 t_x t_ls_4 \cdots s_{i+1} & \stackrel{(1)}{=}\\
s_{n-1}\cdots s_3t_xs_n \cdots s_3 t_x t_l\underline{s_4} \cdots s_{i+1} & \stackrel{(1)}{=}\\
s_{n-1}\cdots s_3t_xs_n \cdots \underline{s_4s_3s_4} t_x t_ls_5 \cdots s_{i+1} & \stackrel{(2)}{=}\\
s_{n-1}\cdots s_3t_xs_n \cdots \underline{s_3}s_4s_3 t_x t_ls_5 \cdots s_{i+1} & \stackrel{(1)}{=}\\
\end{array}$\\
$\begin{array}{l}
s_{n-1}\cdots s_3t_xs_3s_n \cdots s_4s_3 t_x t_l\underline{s_5} \cdots \underline{s_{i+2}}s_{i+1}.
\end{array}$ We apply the same operations to the underlined letters, we get $s_{n-1}\cdots s_3t_xs_3\cdots s_{i-1} s_n \cdots s_4s_3 t_x t_l \underline{s_{i+1}}$. Finally, we shift $\underline{s_{i+1}}$ to the left and get $s_{n-1}\cdots s_3t_xs_3\cdots s_{i} s_n \cdots s_4s_3 t_x t_l$.
Since $2 \leq i \leq n-1$ and by the computation of $t_xt_l$ in Lemma \ref{LemmaTLTKInVect}, the lemma is satisfied for this case.\\
The second term is equal to\\
$\begin{array}{ll}
s_{n-1}\cdots s_3s_n \cdots \underline{s_3 t_x s_3} t_l s_3 s_4 \cdots s_{i+1} & \stackrel{(2)}{=}\\
s_{n-1}\cdots s_3s_n \cdots \underline{t_x} s_3 t_x t_l s_3 s_4 \cdots s_{i+1} & \stackrel{(1)}{=}\\
\end{array}$\\
$\begin{array}{l}
s_{n-1}\cdots s_3t_xs_n \cdots  s_3 t_x t_l s_3 s_4 \cdots s_{i+1}.
\end{array}$ We replace $t_xt_l$ by its value given in Lemma \ref{LemmaTLTKInVect}, we get terms of the three following forms:
\begin{itemize}

\item $s_{n-1}\cdots s_3t_xs_n \cdots  s_3 t_m t_0 s_3 s_4 \cdots s_{i+1}$ with $1 \leq m \leq e-1$,
\item $s_{n-1}\cdots s_3t_xs_n \cdots  s_3 t_m s_3 s_4 \cdots s_{i+1}$ with $0 \leq m \leq e-1$,
\item $s_{n-1}\cdots s_3t_xs_n \cdots  s_4 s_3^2 s_4 \cdots s_{i+1}$.

\end{itemize}
The first term belongs to Span$(S_{n-1}^{*} \Lambda_n)$. The third term also belongs to Span$(S_{n-1}^{*} \Lambda_n)$. This is done by using the computation in the proof of Lemma \ref{LmScholie1}. For the second term, we have\\
$\begin{array}{ll}
s_{n-1}\cdots s_3t_xs_n \cdots s_4\underline{s_3t_ms_3}s_4 \cdots s_{i+1} & \stackrel{(2)}{=}\\
s_{n-1}\cdots s_3t_xs_n \cdots \underline{s_4 t_m}s_3 t_m\underline{s_4} \cdots s_{i+1} & \stackrel{(1)}{=}\\
s_{n-1}\cdots s_3t_xt_ms_n \cdots \underline{s_4s_3s_4}t_m s_5 \cdots s_{i+1} & \stackrel{(2)}{=}\\
s_{n-1}\cdots s_3t_xt_ms_n \cdots \underline{s_3}s_4s_3t_m s_5 \cdots s_{i+1} & \stackrel{(1)}{=}\\
\end{array}$\\
$\begin{array}{l}
s_{n-1}\cdots s_3t_xt_ms_3 s_n \cdots s_4s_3t_m \underline{s_5} \cdots \underline{s_{i+2}} s_{i+1}.
\end{array}$ We apply the same operations to the underlined letters and get $s_{n-1}\cdots s_3t_xt_ms_3\cdots s_{i-1} s_n \cdots s_3t_m s_{i+1}$. Now we shift $\underline{s_{i+1}}$ to the left and finally get $s_{n-1}\cdots s_3t_xt_ms_3\cdots s_{i} \underline{s_n \cdots s_3t_m}$ with $2 \leq i \leq n-1$ which belongs to Span$(S_{n-1}^{*} \Lambda_n)$.\\
Note that for $i = 2$, we get terms that are equal to the itemized terms (given at the beginning of this proof) after replacing $s_4 \cdots s_{i+1}$ by $1$.

\end{proof}

\begin{lemma}\label{Lemma10}

If $a_{n-1} = s_{n-1} \cdots s_3t_ks_2 \cdots s_i$ with $2 \leq i \leq n-1$, \mbox{$1 \leq k \leq e-1$} and $a_n = s_n \cdots s_{3}t_ls_2 \cdots s_{i'}$ with $1 \leq l \leq e-1$ and $2 \leq i' \leq n$, then $s_n (a_{n-1}a_n)$ belongs to Span$(S_{n-1}^{*} \Lambda_n)$.

\end{lemma}

\begin{proof}

According to the computation in the proof of Lemma \ref{Lemma9}, we get the following possible terms. They appear in the proof of Lemma \ref{Lemma9} in the following order.
\begin{enumerate}

\item $s_n \cdots s_3t_xt_l$ with $0 \leq x,l \leq e-1$,
\item $s_n \cdots s_3 t_mt_0 s_3 \cdots s_{i+1}$ with $1 \leq m \leq e-1$,
\item $s_n \cdots s_3 t_m$ with $0 \leq m \leq e-1$,
\item $s_n \cdots s_4s_3^2 s_4 \cdots s_{i+1}$

\end{enumerate}

We show that the product on the right by $s_2 \cdots s_{i'}$ of each of the previous terms belongs to Span$(S_{n-1}^{*}\Lambda_n)$.

\emph{Case 1}. Consider the first term $s_n \cdots s_3t_xt_l (s_2 \cdots s_{i'})$ with $0 \leq x,l \leq e-1$.
We replace $t_xt_l$ by its decomposition given in Lemma \ref{LemmaTLTKInVect}, we get these terms
\begin{itemize}

\item $s_n \cdots s_3 t_mt_0 s_3 \cdots s_{i'}$ with $1 \leq m \leq e-1$,
\item $s_n \cdots s_3 t_m s_3 \cdots s_{i'}$ with $0 \leq m \leq e-1$,
\item $s_n \cdots s_4 s_3^2 s_4 \cdots s_{i'}$.

\end{itemize}
The first term belongs to Span$(S_{n-1}^{*} \Lambda_n)$. The third one is done in Lemma \ref{LmScholie1}. For the second one, we have\\
$\begin{array}{ll}
s_n \cdots \underline{s_3t_ms_3} \cdots s_{i'} & \stackrel{(2)}{=}\\
s_n \cdots \underline{t_m}s_3t_m\underline{s_4} \cdots s_{i'} & \stackrel{(1)}{=}\\
t_ms_n \cdots \underline{s_4s_3s_4}t_ms_5 \cdots s_{i'} & \stackrel{(2)}{=}\\
t_ms_n \cdots \underline{s_3}s_4s_3 t_ms_5 \cdots s_{i'} & \stackrel{(1)}{=}\\
\end{array}$\\
$\begin{array}{l}
t_ms_3s_n \cdots s_4s_3 t_m\underline{s_5} \cdots \underline{s_{i'-1}}s_{i'}.
\end{array}$ We apply the same operations to $\underline{s_5}$, $\cdots$, $\underline{s_{i'-1}}$ and get
$t_ms_3 \cdots s_{i'-2}s_n \cdots s_3 t_m \underline{s_{i'}}$. We shift $\underline{s_{i'}}$ to the left and finally get\\
$t_ms_3 \cdots s_{i'-1} \underline{s_n \cdots s_3 t_m}$ which belongs to Span$(S_{n-1}^{*} \Lambda_n)$.

We now consider \emph{Case 3} because we use the computation we made in Case 1. In this case, the term is of the form $s_n \cdots s_3t_m(s_2 \cdots s_{i'})$ with $0 \leq m \leq e-1$. If $m \neq 0$, then it belongs to Span$(S_{n-1}^{*} \Lambda_n)$. If $m = 0$, we get two terms $s_n \cdots s_3 t_0 s_3 \cdots s_{i'}$ and $s_n \cdots s_4 s_3^2 s_4 \cdots s_{i'}$. The first term is done in Case 1. The second term is done in Lemma \ref{LmScholie1}.

Consider \emph{Case 4}. We replace $s_n \cdots s_4 s_3^2s_4 \cdots s_{i+1}$ by its decomposition given by the computation in the proof of Lemma \ref{LmScholie1}. We multiply each term of the decomposition by $s_2 \cdots s_{i'}$ on the right and we prove that it belongs to Span$(S_{n-1}^{*} \Lambda_n)$ in the same way as the proof of Lemma \ref{Lemma4}.

Finally, it remains to show that the term corresponding to \emph{Case 2} belongs to Span$(S_{n-1}^{*} \Lambda_n)$. It is of the form $s_n \cdots s_3 t_mt_0 s_3 \cdots s_{i+1} (s_2 \cdots s_{i'})$ with $1 \leq m \leq e-1$.

\emph{Suppose $i' \leq i$}. We have\\
$\begin{array}{ll}
s_n \cdots s_3t_ms_2s_3 \cdots s_{i+1}\underline{s_2} \cdots s_{i'} & \stackrel{(1)}{=}\\
s_n \cdots s_3t_m\underline{s_2s_3s_2}s_4 \cdots s_{i+1}s_3 \cdots s_{i'} & \stackrel{(2)}{=}\\
s_n \cdots \underline{s_3t_m s_3}s_2s_3 s_4 \cdots s_{i+1}s_3 \cdots s_{i'} & \stackrel{(2)}{=}\\
s_n \cdots \underline{t_m}s_3t_m s_2s_3 s_4 \cdots s_{i+1} s_3 \cdots s_{i'} & \stackrel{(1)}{=}\\
t_m s_n \cdots s_3t_m s_2s_3 s_4 \cdots s_{i+1}\underline{s_3} \cdots s_{i'} & \stackrel{(1)}{=}\\
\end{array}$\\
$\begin{array}{ll}
t_m s_n \cdots s_3t_m s_2\underline{s_3 s_4s_3}s_5 \cdots s_{i+1}s_4 \cdots s_{i'} & \stackrel{(2)}{=}\\
t_m s_n \cdots s_3t_m s_2 \underline{s_4}s_3 s_4s_5 \cdots s_{i+1}s_4 \cdots s_{i'} & \stackrel{(1)}{=}\\
t_m s_n \cdots \underline{s_4s_3s_4}t_m s_2s_3 s_4s_5 \cdots s_{i+1}s_4 \cdots s_{i'} & \stackrel{(2)}{=}\\
t_m s_n \cdots \underline{s_3}s_4s_3t_m s_2s_3 s_4s_5 \cdots s_{i+1}s_4 \cdots s_{i'} & \stackrel{(1)}{=}\\
\end{array}$\\
$\begin{array}{l}
t_m s_3 s_n \cdots s_4s_3t_m s_2s_3 s_4s_5 \cdots s_{i+1}\underline{s_4} \cdots \underline{s_{i'-1}} s_{i'}.
\end{array}$ We apply the same operations to $\underline{s_4}$, $\cdots$, $\underline{s_{i'-1}}$ to get
$t_ms_3 \cdots s_{i'-1}s_n \cdots s_3 t_m s_2 s_3 \cdots s_i s_{i+1} \underline{s_{i'}}$. We shift $\underline{s_{i'}}$ to the left and finally get
$t_ms_3 \cdots s_{i'}s_n \cdots s_3 t_m s_2 s_3 \cdots s_i s_{i+1}$ which satisfies the property of the lemma since $i' \leq i \leq n-1$.

\emph{Suppose $i' > i$}. As previously, we have\\
$\begin{array}{ll}
s_n \cdots s_3t_ms_2s_3 \cdots s_{i+1}\underline{s_2} \cdots s_{i'} & \stackrel{(1)}{=}\\
t_m s_n \cdots s_3t_ms_2s_3 \cdots s_{i+1}\underline{s_3} \cdots s_{i'} & \stackrel{(1)}{=}\\
t_ms_3s_n \cdots s_3t_ms_2s_3 \cdots s_{i+1}s_4 \cdots s_{i'}. & \\
\end{array}$\\ Now we write $s_is_{i+1}$ in $s_4 \cdots s_{i'}$ and get $t_ms_3s_n \cdots s_3t_ms_2s_3 \cdots s_{i+1}\underline{s_4} \cdots \underline{s_i}s_{i+1} \cdots s_{i'}$.\\
We apply the same operations to $\underline{s_4}$, $\cdots$, $\underline{s_i}$ to get\\
$t_ms_3 \cdots s_i s_n \cdots s_3t_ms_2s_3 \cdots s_{i}s_{i+1}^2s_{i+2} \cdots s_{i'}$. Applying a quadratic relation, we finally get
$a t_ms_3 \cdots s_i s_n \cdots s_3t_m s_2 s_3 \cdots s_{i'}
+ t_m s_3 \cdots s_i s_n \cdots s_3t_ms_2s_3 \cdots s_{i}s_{i+2} \cdots s_{i'}$.\\
The first term satisfies the property of the lemma. For the second term, we write $s_{i+2}s_{i+1}$ in $\underline{s_n \cdots s_3}$ of 
$t_m s_3 \cdots s_i \underline{s_n \cdots s_3}t_ms_2s_3 \cdots s_{i} s_{i+2} \cdots s_{i'}$ and get\\
$\begin{array}{ll}
t_m s_3 \cdots s_i s_n \cdots s_{i+2}s_{i+1} \cdots s_3t_ms_2s_3 \cdots s_{i} \underline{s_{i+2}} \cdots s_{i'} & \stackrel{(1)}{=}\\
t_m s_3 \cdots s_i s_n \cdots \underline{s_{i+2}s_{i+1}s_{i+2}} \cdots s_3t_ms_2s_3 \cdots s_{i} s_{i+3} \cdots s_{i'} & \stackrel{(2)}{=}\\
t_m s_3 \cdots s_i s_n \cdots \underline{s_{i+1}}s_{i+2}s_{i+1} \cdots s_3t_ms_2s_3 \cdots s_{i} s_{i+3} \cdots s_{i'} & \stackrel{(1)}{=}\\
\end{array}$\\
$\begin{array}{l}
t_m s_3 \cdots s_{i+1} s_n \cdots s_3t_ms_2s_3 \cdots s_{i} \underline{s_{i+3}} \cdots \underline{s_{i'}}.
\end{array}$ We apply the same operations to $\underline{s_{i+3}}$, $\cdots$, $\underline{s_{i'}}$ and finally get
$t_m s_3 \cdots s_{i'-1} s_n \cdots s_3t_ms_2s_3 \cdots s_{i}$. Since $i'-1 \leq n-1$, this term belongs to Span$(S_{n-1}^{*} \Lambda_n)$.

\end{proof}

As a consequence of Theorem \ref{TheoremNewBasis}, by Proposition 2.3 $(ii)$ of \cite{MarinG20G21}, we get another proof of the BMR freeness conjecture (see Theorem \ref{PropositionBMRFreenessTheorem}) for the complex reflection groups $G(e,e,n)$.

\begin{remark}

Our basis never coincides with the Ariki basis for the Hecke algebra associated with $G(e,e,n)$. For example, consider the element $t_1t_0.\ t_0$ which belongs to Ariki's basis. In our basis, it is equal to the linear combination $a t_1t_0 + t_1$, where $t_1t_0$ and $t_1$ are two distinct elements of our basis.

\end{remark}

\newpage

\section{The case of $G(d,1,n)$}

Let $d > 1$ and $n \geq 2$. Let $R_0 = \mathbb{Z}[a,b_1,b_2, \cdots, b_{d-1}]$. Recall that the Hecke algebra $H(d,1,n)$ is the unitary associative $R_0$-algebra generated by the elements $z, s_2, s_3, \cdots, s_n$ with the following relations:
\begin{enumerate}

\item $zs_2zs_2 = s_2zs_2z$,
\item $zs_j = s_jz$ for $2 \leq j \leq n$,
\item $s_is_{i+1}s_i = s_{i+1}s_is_{i+1}$ for $2 \leq i \leq n-1$ and $s_is_j = s_js_i$ for $|i-j| > 1$,
\item $z^d -b_1 z^{d-1} - b_2 z^{d-2} - \cdots - b_{d-1} z -1 = 0$ and $s_j^2 - as_j - 1 = 0$ for $2 \leq j \leq n$.

\end{enumerate}

We use the geodesic normal form defined in Section 2.2 of Chapter \ref{ChapterNormalFormsG(de,e,n)} for all the elements of $G(d,1,n)$ in order to construct a basis for $H(d,1,n)$ that is different from the one defined by Ariki and Koike in \cite{ArikiKoikeHecke}. We introduce the following subsets of $H(d,1,n)$.
\begin{center}
\begin{tabular}{lllll}
$\Lambda_1 =$ & $\{$ & $z^k$ & for $0 \leq k \leq d-1$ & $\}$,
\end{tabular}
\end{center}
and for $2 \leq i \leq n$,
\begin{center}
\begin{tabular}{lllll}
$\Lambda_i =$ & $\{$ & $1$, & & \\
 & & $s_i \cdots s_{i'}$ & for $2 \leq i' \leq i$, & \\ 
 & & $s_i \cdots s_{2}z^k$ & for $1 \leq k \leq d-1$, & \\
 & & $s_i \cdots s_{2}z^ks_2 \cdots s_{i'}$ & for $1 \leq k \leq d-1$ and $2 \leq i' \leq i$ & $\}.$\\
\end{tabular}
\end{center}

\noindent Define $\Lambda = \Lambda_1\Lambda_2 \cdots \Lambda_n$ to be the set of the products $a_1a_2 \cdots a_n$, where $a_1 \in \Lambda_1, \cdots,\\ a_n \in \Lambda_n$. We prove the following theorem.

\begin{theorem}\label{TheoremNewBasisH(d,1,n)}

The set $\Lambda$ provides an $R_0$-basis of the Hecke algebra $H(d,1,n)$.

\end{theorem}

We have $|\Lambda_1| = d$, $|\Lambda_i| = id$ for $2 \leq i \leq n$. Then $|\Lambda|$ is equal to $d^n n!$ that is the order of $G(d,1,n)$. Hence by Proposition 2.3 $(i)$ of \cite{MarinG20G21}, it is sufficient to show that $\Lambda$ is an $R_0$-generating set of $H(d,1,n)$. This is proved by induction on $n$ in much the same way as Theorem \ref{TheoremNewBasis}. We provide some preliminary lemmas that are useful in the proof of the theorem.

\begin{lemma}\label{Lemma(s2ts2)^k}

For $1 \leq k \leq d-1$, the element $(s_2zs_2)^k$ belongs to $\sum\limits_{\substack{\lambda_1 \in \Lambda_1,\\ \lambda_2 \in \Lambda'_2}}R_0 \lambda_1 \lambda_2$, where $\Lambda'_2 = \{1, s_2, s_2z, s_2z^2, \cdots, s_2z^{k-1}, s_2zs_2, s_2z^2s_2, \cdots, s_2z^ks_2 \}$.

\end{lemma}

\begin{proof}

The property is clear for $k=1$. Let $k=2$. We have
$(s_2zs_2)^2 = s_2zs_2^2zs_2$. We apply a quadratic relation and get
$a s_2zs_2zs_2 + s_2z^2s_2$. Applying a braid relation, one gets
$a zs_2zs_2^2 + s_2z^2s_2$. Using a quadratic relation, this is equal to
$a^2 zs_2zs_2 + a zs_2z + s_2z^2s_2$, where each term is of the form $\lambda_1\lambda_2$ with $\lambda_1 \in \Lambda_1$ and $\lambda_2 \in \{1,s_2,s_2z,s_2zs_2,s_2z^2s_2\}$.

Let $k \geq 3$. Suppose the property is satisfied for $(s_2zs_2)^3$, $\cdots$, and $(s_2zs_2)^{k-1}$. We have $(s_2zs_2)^k = (s_2zs_2)^{k-1}(s_2zs_2)$. By the induction hypothesis, the terms that appear in the decomposition of $(s_2zs_2)^{k-1}$ are of the following forms.
\begin{itemize}

\item $z^c$ with $0 \leq c \leq d-1$,
\item $z^cs_2z^{c'}$ with $0 \leq c \leq d-1$ and $0 \leq c' \leq k-2$,
\item $z^cs_2z^{c'}s_2$ with $0 \leq c \leq d-1$ and $1 \leq c' \leq k-1$.

\end{itemize}
Multiplying these terms by $s_2zs_2$ on the right, we get the following $3$ cases.

\emph{Case 1}: A term of the form $z^c s_2zs_2$ with $0 \leq c \leq d-1$. It is of the form $\lambda_1\lambda_2$ with $\lambda_1 \in \Lambda_1$ and $\lambda_2 \in \Lambda'_2$.

\emph{Case 2}: A term of the form $z^cs_2z^{c'}s_2zs_2$ with $0 \leq c \leq d-1$ and $0 \leq c' \leq k-2$. We shift $z^{c'}$ to the right by applying braid relations and get
$z^cs_2^2zs_2z^{c'}$. Applying a quadratic relation, this is equal to
$a z^cs_2zs_2\underline{z^{c'}} +  z^{c+1}s_2z^{c'}$. Now we shift $\underline{z^{c'}}$ to the left by applying braid relations and get
$a z^{c+c'}s_2zs_2 +  z^{c+1}s_2z^{c'}$. Each term is of the form $\lambda_1\lambda_2$ with $\lambda_1 \in \Lambda_1$ and $\lambda_2 \in \Lambda'_2$.

\emph{Case 3}: A term of the form $z^cs_2z^{c'}s_2^2zs_2$ with $0 \leq c \leq d-1$ and $1 \leq c' \leq k-1$. By applying a quadratic relation, we have $z^cs_2z^{c'}s_2^2zs_2 = a z^cs_2z^{c'}s_2zs_2 + z^cs_2z^{c'+1}s_2$. The first term is the same as in the previous case. Then both terms are of the form $\lambda_1\lambda_2$ with $\lambda_1 \in \Lambda_1$ and $\lambda_2 \in \Lambda'_2$.

\end{proof}

\begin{lemma}\label{Lemma(s2ts2)^ks2}

For $1 \leq k \leq d-1$, the element $(s_2zs_2)^ks_2$ belongs to $R_0(s_2zs_2)^k + R_0z(s_2zs_2)^{k-1} + \cdots + R_0z^{k-1}(s_2zs_2) + R_0s_2z^k$.

\end{lemma}

\begin{proof}

For $k = 1$, we have $(s_2zs_2)s_2 = s_2zs_2^2 = a s_2zs_2 + s_2z$. Then the property is satisfied for $k = 1$. Let $k \geq 2$. Suppose that the property is satisfied for $(s_2zs_2)^{k-1}$. We have $(s_2zs_2)^k s_2 = (s_2zs_2)(s_2zs_2)^{k-1} s_2$. By the induction hypothesis, it belongs to $R_0(s_2zs_2)(s_2zs_2)^{k-1} + R_0(s_2zs_2)z(s_2zs_2)^{k-2} + \cdots + R_0(s_2zs_2)z^{k-2}(s_2zs_2) + R_0(s_2zs_2)s_2z^{k-1}$. Then it belongs to $R_0(s_2zs_2)^{k} + R_0z(s_2zs_2)^{k-1} + \cdots + R_0z^{k-2}(s_2zs_2)^2\\ + R_0z^{k-1}(s_2zs_2) + R_0s_2z^k$. It follows that for all $1 \leq k \leq d-1$, the element $(s_2zs_2)^ks_2$ belongs to $R_0(s_2zs_2)^k + R_0z(s_2zs_2)^{k-1} + \cdots + R_0z^{k-1}(s_2zs_2) + R_0s_2z^k$.

\end{proof}

\begin{lemma}\label{Lemmas2t^ks2In(s2ts2)^k}

For $1 \leq k \leq d-1$, the element $s_2z^ks_2$ belongs to $\sum\limits_{\substack{\lambda_1 \in \Lambda_1,\\ \lambda_2 \in \Lambda''_2}}R_0 \lambda_1 \lambda_2$, where $\Lambda''_2 = \{1,s_2,s_2z,s_2z^2, \cdots, s_2z^{k-1}, s_2zs_2, (s_2zs_2)^2, \cdots, (s_2zs_2)^{k} \}$.

\end{lemma}

\begin{proof}

The lemma is satisfied for $k=1$. For $k=2$, we have
$s_2z^2s_2 = s_2zs_2^{-1}s_2zs_2$. Using that $s_2^{-1} = s_2 - a$, we get
$s_2zs_2^2zs_2 - a \underline{s_2zs_2z}s_2$. Now we apply a braid relation then a quadratic relation and get
$(s_2zs_2)^2 - a zs_2zs_2^2 = (s_2zs_2)^2 - a^2 zs_2zs_2 - a zs_2z$ which satisfies the property we are proving.

Suppose the property is satisfied for $s_2z^{k-1}s_2$. We have
$s_2z^ks_2 = s_2z^{k-1}s_2^{-1}s_2zs_2 = s_2z^{k-1}s_2s_2zs_2 - a s_2z^{k-1}s_2zs_2$ by replacing $s_2^{-1}$ by $s_2 - a$.
For the second term, we shift $z^{k-1}$ to the right and get
$s_2^2zs_2z^{k-1}$. We apply a quadratic relation to get
$a s_2zs_2\underline{z^{k-1}} + zs_2z^{k-1}$ then we shift $\underline{z^{k-1}}$ to the left and finally get
$a z^{k-1}s_2zs_2 + zs_2z^{k-1}$, where each term is of the form $\lambda_1 \lambda_2$ with $\lambda_1 \in \Lambda_1$ and $\lambda_2 \in \Lambda''_2$.\\
For the first term $s_2z^{k-1}s_2s_2zs_2$, by the induction hypothesis, the terms that appear in the decomposition of $s_2z^{k-1}s_2$ are of the following forms.
\begin{itemize}

\item $z^c$ and $z^cs_2$ with $0 \leq c \leq d-1$,
\item $z^c(s_2zs_2)^{c'}$ with $0 \leq c \leq d-1$ and $1 \leq c' \leq k-1$,
\item $z^cs_2z^{c'}$ with $0 \leq c \leq d-1$ and $1 \leq c' \leq k-2$.

\end{itemize}
Multiplying these terms by $s_2zs_2$ on the right, we get the following $3$ cases.

\emph{Case 1}. We have $z^c s_2zs_2$ and $z^c\underline{s_2s_2}zs_2 = a z^cs_2zs_2 + z^{c+1}s_2$, where each term in both expressions is of the form $\lambda_1 \lambda_2$ with $\lambda_1 \in \Lambda_1$ and $\lambda_2 \in \Lambda''_2$.

\emph{Case 2}. The term $z^c(s_2zs_2)^{c'}s_2zs_2 = z^c(s_2zs_2)^{c'+1}$ is of the form $\lambda_1 \lambda_2$ with $\lambda_1 \in \Lambda_1$ and $\lambda_2 \in \Lambda''_2$ since $1 \leq c' \leq k-1$.

\emph{Case 3}. We have $z^cs_2z^{c'}s_2zs_2 =
z^cs_2^2zs_2z^{c'} = a z^cs_2zs_2z^{c'} + z^{c+1}s_2z^{c'}$. The first term is equal to $z^{c+c'}s_2zs_2$ and the second term is equal to $z^{c+1}s_2z^{c'}$ with $1 \leq c' \leq k-2$. Both are of the form $\lambda_1 \lambda_2$ with $\lambda_1 \in \Lambda_1$ and $\lambda_2 \in \Lambda''_2$.

\end{proof}

The following proposition ensures that the case $n=2$ of Theorem \ref{TheoremNewBasisH(d,1,n)} works properly.

\begin{proposition}\label{PropInductionHypothesisG(d,1,n)}

For all $a_1 \in \Lambda_1$ and $a_2 \in \Lambda_2$, the elements $za_1a_2$ and $s_2 a_1a_2$ belong to \emph{Span}$(\Lambda_1\Lambda_2)$.

\end{proposition}

\begin{proof}

It is readily checked that $za_1a_2$ belongs to Span$(\Lambda_1\Lambda_2)$. Note that when the power of $z$ exceeds $d-1$, we use Relation 1 of Definition \ref{DefPresH(d,1,n)}.

It is easily checked that if $a_1 \in \Lambda_1$ and $a_2 = 1$, the element $s_2 a_1a_2$ belongs to Span$(\Lambda_1\Lambda_2)$. Also, when $a_1 = 1$ and $a_2 \in \Lambda_2$, we have that $s_2 a_1a_2$ belongs to Span$(\Lambda_1\Lambda_2)$.

Suppose $a_1 = z^k$ with $1 \leq k \leq d-1$ and $a_2 = s_2$. We have $s_2 a_1 a_2$ is equal to $s_2 z^k s_2$. Hence it belongs to Span$(\Lambda_1\Lambda_2)$.

Suppose $a_1 = z^k$ with $1 \leq k \leq d-1$ and $a_2 = s_2 z^{k'}$ with $1 \leq k' \leq d-1$. We have $s_2 a_1 a_2 = s_2 z^k s_2 z^{k'}$. We replace $s_2 z^k s_2$ by its decomposition given in Lemma \ref{Lemmas2t^ks2In(s2ts2)^k}, then we use the result of Lemma \ref{Lemma(s2ts2)^k} to directly deduce that $s_2 z^k s_2 z^{k'}$ belongs to Span$(\Lambda_1\Lambda_2)$.

Finally, suppose $a_1 = z^k$ with $1 \leq k \leq d-1$ and $a_2 = s_2 z^{k'}s_2$ with $1 \leq k' \leq d-1$.\\
We have $s_2 a_1 a_2$ is equal to $s_2 z^k s_2 z^{k'} s_2$. We replace $s_2 z^k s_2$ by its decomposition given in Lemma \ref{Lemmas2t^ks2In(s2ts2)^k}. Then by the results of Lemmas \ref{Lemma(s2ts2)^ks2} and \ref{Lemma(s2ts2)^k}, we deduce that $s_2 z^k s_2 z^{k'} s_2$ belongs to Span$(\Lambda_1\Lambda_2)$.

\end{proof}

In order to prove Theorem \ref{TheoremNewBasisH(d,1,n)}, we introduce Lemmas \ref{LmScholie11} to \ref{Lemma1010} below that are similar to Lemmas \ref{LmScholie1} to \ref{Lemma10}. Along with Proposition \ref{PropInductionHypothesisG(d,1,n)}, they provide an inductive proof of Theorem \ref{TheoremNewBasisH(d,1,n)} that is similar to the proof of Theorem \ref{TheoremNewBasis}. Therefore, by Proposition 2.3 $(ii)$ of \cite{MarinG20G21}, we get another proof of Theorem \ref{PropositionBMRFreenessTheorem} for the complex reflection groups $G(d,1,n)$. Let $n \geq 3$. Denote by $S_{n-1}^{*}$ the set of the words over $\{z,s_2,s_3, \cdots, s_{n-1}\}$. The rest of this section is devoted to the proof of Lemmas \ref{LmScholie11} to \ref{Lemma1010}.\\

\begin{lemma}\label{LmScholie11}

Let $2 \leq i \leq n$. We have $s_n \cdots s_3 s_2^2s_3 \cdots s_i$ belongs to \emph{Span}$(S_{n-1}^{*}\Lambda_n)$.

\end{lemma}

\begin{proof}

The proof is the same as the proof of Lemma \ref{LmScholie1}. If $i = 2$, we have that $s_n \cdots s_3 s_2^2$ belongs to $R_0s_n \cdots s_3s_2 + R_0 s_n \cdots s_3$. We continue as in the proof of Lemma \ref{LmScholie1}, we get two terms of the form (see line 7 of the proof of Lemma \ref{LmScholie1}):\\ $s_n \cdots s_{k+1}s_ks_{k+1} \cdots s_i$ with $k+1 \leq i$ and $s_n \cdots s_{i+1}s_is_{i-1}^2s_i$.

We also get that the first term $s_n \cdots s_{i+1}s_is_{i-1}^2s_i$ belongs to $R_0 s_{i-1}s_n \cdots s_{i-1} + R_0 s_n \cdots s_{i+1}s_i + R_0 s_n \cdots s_{i+1}$ (the last term is equal to $1$ if $i = n$). Hence the element $s_n \cdots s_{i+1}s_is_{i-1}^2s_i$ belongs to Span$(S_{n-1}^{*}\Lambda_n)$. 

As in the proof of Lemma \ref{LmScholie1}, the element $s_n \cdots s_{k+1}s_ks_{k+1} \cdots s_i$ is equal to $s_ks_{k+1} \cdots s_{i-1}s_n s_{n-1} \cdots s_k$. Hence it belongs to Span$(S_{n-1}^{*}\Lambda_n)$.

\end{proof}

\begin{lemma}\label{Lemma22}

If $a_{n-1} = s_{n-1}s_{n-2} \cdots s_i$ with $2 \leq i \leq n-1$ and $a_n = s_n \cdots s_{i'}$ with $2 \leq i' \leq n$, then $s_n(a_{n-1}a_n)$ belongs to \emph{Span}$(\Lambda_{n-1}\Lambda_n)$.

\end{lemma}

\begin{proof}

The proof is the same as for Lemma \ref{Lemma2}.

If $i < i'$, then we get $s_n(a_{n-1}a_n) = s_{n-1}s_{n-2} \cdots s_{i'-1}s_n \cdots s_i$ which belongs to Span$(\Lambda_{n-1}\Lambda_n)$.

If $i \geq i'$, then we get $s_n(a_{n-1}a_n) = a s_{n-1}\cdots s_i s_n \cdots s_{i'} + s_n \cdots s_{i'} s_n \cdots s_{i+1}$ which also belongs to Span$(\Lambda_{n-1}\Lambda_n)$.

\end{proof}

\begin{lemma}\label{Lemma33}

If $a_{n-1} = s_{n-1}\cdots s_{i}$ with $2 \leq i \leq n-1$ and $a_n = s_n \cdots s_2 z^k$ with $0 \leq k \leq d-1$, then $s_n(a_{n-1}a_n)$ belongs to \emph{Span}$(\Lambda_{n-1}\Lambda_n)$.

\end{lemma}

\begin{proof}

This case correspond to the case $i'=2$ in the proof of Lemma \ref{Lemma22} with a right multiplication by $z^k$ for $0 \leq k \leq d-1$. Since $i \geq 2$, by the case $i \geq i'$ in the proof of Lemma \ref{Lemma22}, we get, using the same technique as in the proof of Lemma \ref{Lemma3}, that $s_n(a_{n-1}a_n)$ is equal to $a s_{n-1} \cdots s_i s_n \cdots s_2 z^k + s_{n-1}\cdots s_2 z^k s_n \cdots s_{i+1}$ which belongs to Span$(\Lambda_{n-1}\Lambda_n)$.

\end{proof}

\begin{lemma}\label{Lemma44}

If $a_{n-1} = s_{n-1} \cdots s_i$ with $2 \leq i \leq n-1$ and $a_n = s_n \cdots s_2 z^k s_2 \cdots s_{i'}$ with $1 \leq k \leq d-1$ and $2 \leq i' \leq n$, then $s_n(a_{n-1}a_n)$ belongs to \emph{Span}$(\Lambda_{n-1}\Lambda_n)$.

\end{lemma}

\begin{proof}

The proof is exactly the same as the proof of Lemma \ref{Lemma4}. According to the computation in the proof of Lemma \ref{Lemma33}, we have\\
$s_n(a_{n-1}a_n) = a s_{n-1} \cdots s_i s_n \cdots s_2 z^k (s_2 \cdots s_{i'}) + s_{n-1} \cdots s_2 z^k s_n \cdots s_{i+1}(s_2 \cdots s_{i'})$.\\
The first term is an element of Span$(\Lambda_{n-1}\Lambda_n)$. For the second term, we follow exactly the proof (starting from line 3) of Lemma \ref{Lemma4} :\\
If $i' < i$, we get $s_{n-1} \cdots s_2 z^k s_2 \cdots s_{i'} s_n \cdots s_{i+1}$ which belongs to Span$(\Lambda_{n-1}\Lambda_n)$.\\
If $i' \geq i$, we get $s_{n-1} \cdots s_2 z^k s_2 \cdots s_{i'-1} s_n \cdots s_i$ which also belongs to Span$(\Lambda_{n-1}\Lambda_n)$.

\end{proof}

\begin{lemma}\label{Lemma55}

If $a_{n-1} = s_{n-1} \cdots s_2 z^k$ with $1 \leq k \leq d-1$ and $a_n = s_n \cdots s_i$ with $2 \leq i \leq n$, then $s_n(a_{n-1}a_n)$ belongs to \emph{Span}$(\Lambda_{n-1}\Lambda_n)$.

\end{lemma}

\begin{proof}

We follow exactly the proof of Lemma \ref{Lemma5} and we get the following terms.

If $i = 2$, we get $s_{n-1} \cdots s_2 s_n \cdots s_2 z^k s_2$ (see line 8 of the proof of Lemma \ref{Lemma5}) which belongs to Span$(\Lambda_{n-1}\Lambda_n)$.

If $i > 2$, we get $s_{n-1} \cdots s_{i-1} s_n \cdots s_2 z^k$ (see the last line of the proof of Lemma \ref{Lemma5}) which also belongs to Span$(\Lambda_{n-1}\Lambda_n)$.

\end{proof}

\begin{lemma}\label{Lemma66}

If $a_{n-1} = s_{n-1} \cdots s_2z^k$ with $1 \leq k \leq d-1$ and $a_n = s_n \cdots s_2 z^l$ with $1 \leq l \leq d-1$, then $s_n(a_{n-1}a_n)$ belongs to \emph{Span}$(\Lambda_{n-1}\Lambda_n)$.

\end{lemma}

\begin{proof}

By the case $i = 2$ in the proof of Lemma \ref{Lemma55}, we get $s_n(a_{n-1}a_n) = s_{n-1} \cdots s_2 s_n \cdots s_3s_2z^ks_2z^l$. If we replace $s_2z^ks_2z^l$ by its decomposition over $\Lambda_1\Lambda_2$ (this is the case $n=2$ of Theorem \ref{TheoremNewBasisH(d,1,n)}, see Proposition \ref{PropInductionHypothesisG(d,1,n)}), we get the three following terms:
\begin{itemize}
\item $s_{n-1} \cdots s_2 s_n \cdots s_3 z^c$ with $0 \leq c \leq d-1$,
\item $s_{n-1} \cdots s_2 s_n \cdots s_3 z^cs_2 z^{c'}$ with $0 \leq c \leq d-1$ and $0 \leq c' \leq d-1$,
\item $s_{n-1} \cdots s_2 s_n \cdots s_3 z^cs_2 z^{c'}s_2$ with $0 \leq c \leq d-1$ and $1 \leq c' \leq d-1$.
\end{itemize}
The first term is equal to $s_{n-1} \cdots s_2 z^c s_n \cdots s_3$ which belongs to Span$(\Lambda_{n-1}\Lambda_n)$.\\
The second term is equal to $s_{n-1} \cdots s_2 z^c s_n \cdots s_3s_2z^{c'}$ which belongs to Span$(\Lambda_{n-1}\Lambda_n)$.\\
Finally, the third term is equal to $s_{n-1} \cdots s_2 z^c s_n \cdots s_2 z^{c'}s_2$ which also belongs to Span$(\Lambda_{n-1}\Lambda_n)$.

\end{proof}

\begin{lemma}\label{Lemma77}

If $a_{n-1} = s_{n-1} \cdots s_2 z^k$ with $1 \leq k \leq d-1$ and $a_n = s_n \cdots s_2 z^l s_2 \cdots s_i$ with $2 \leq i \leq n$ and $1 \leq l \leq d-1$, then $s_n(a_{n-1}a_n)$ belongs to \emph{Span}$(S_{n-1}^{*}\Lambda_n)$.

\end{lemma}

\begin{proof}

According to the proof of the previous lemma, we have to deal with the following three terms:
\begin{itemize}
\item $s_{n-1} \cdots s_2 z^c s_n \cdots s_3 (s_2 \cdots s_i)$ for $0 \leq c \leq d-1$,
\item $s_{n-1} \cdots s_2 z^c s_n \cdots s_3s_2z^{c'}(s_2 \cdots s_i)$ for $0 \leq c \leq d-1$ and $0 \leq c' \leq d-1$,
\item $s_{n-1} \cdots s_2 z^c s_n \cdots s_2 z^{c'}s_2(s_2 \cdots s_{i})$ for $0 \leq c \leq d-1$ and $1 \leq c' \leq d-1$.
\end{itemize}

For the first case, we have $s_n \cdots \underline{s_3s_2s_3}s_4 \cdots s_i = s_n \cdots \underline{s_2} s_3 s_2 \underline{s_4} \cdots s_i$. We shift $\underline{s_2}$ and $\underline{s_4}$ to the left in the previous expression and get $s_2s_n \cdots \underline{s_4s_3s_4}s_2 s_5 \cdots s_i = s_2s_n \cdots \underline{s_3}s_4s_3s_2s_5 \cdots s_i$. We shift $\underline{s_3}$ to the left in the previous expression and get $s_2s_3 s_n \cdots s_2 \underline{s_5} \cdots \underline{s_i}$. We apply the same operations to $\underline{s_5}$, $\cdots$, $\underline{s_i}$ and get\\ $s_2 s_3 \cdots s_{i-1}s_n \cdots s_2$. Hence the term of the first case can be written as follows $s_{n-1} \cdots s_2 z^c s_2 \cdots s_{i-1}s_n \cdots s_2$. Since $i-1 \leq n-1$, it belongs to Span$(S_{n-1}^{*}\Lambda_n)$.

For the second case, we have a term of the form $s_{n-1} \cdots s_2 z^c s_n \cdots s_3s_2z^{c'}(s_2 \cdots s_i)$ for $0 \leq c \leq d-1$ and $0 \leq c' \leq d-1$. If $c' \neq 0$, this term belongs to Span$(S_{n-1}^{*}\Lambda_n)$ and if $c' = 0$, by the computation in the proof of Lemma \ref{LmScholie11}, it also belongs to Span$(S_{n-1}^{*}\Lambda_n)$.

For the third case, we have $s_{n-1} \cdots s_2 z^c s_n \cdots s_2 z^{c'}s_2^2s_3 \cdots s_{i} =$\\
$a s_{n-1} \cdots s_2 z^c s_n \cdots s_2 z^{c'}s_2 \cdots s_{i} + s_{n-1} \cdots s_2 z^c s_n \cdots s_2 z^{c'}s_3 \cdots s_{i}$. The first term is an element of Span$(S_{n-1}^{*}\Lambda_n)$. For the second term, we have $s_{n-1} \cdots s_2 z^c s_n \cdots s_2 \underline{z^{c'}} s_3 \cdots s_{i}\\ = s_{n-1} \cdots s_2 z^c s_n \cdots s_2 s_3 \cdots s_{i} z^{c'}$, where $ s_n \cdots s_2 s_3 \cdots s_{i}$ is already computed in the first case. It is equal to $s_2 \cdots s_{i-1}s_n \cdots s_2$. Hence the term is of the form\\ $s_{n-1} \cdots s_2 z^c s_2 \cdots s_{i-1} s_n \cdots s_2 z^{c'}$ which belongs to Span$(S_{n-1}^{*}\Lambda_n)$.

\end{proof}

\begin{lemma}\label{Lemma88}

If $a_{n-1} = s_{n-1} \cdots s_2 z^k s_2 \cdots s_i$ with $2 \leq i \leq n-1$, $1 \leq k \leq d-1$ and $a_n = s_n \cdots s_{i'}$ with $2 \leq i' \leq n$, then $s_n(a_{n-1}a_n)$ belongs to \emph{Span}$(\Lambda_{n-1}\Lambda_n)$.

\end{lemma}

\begin{proof}

We follow exactly the proof of Lemma \ref{Lemma8}. 

\textit{Suppose $i < i'$}. If $i' = i+1$, we get $s_n(a_{n-1}a_n) = s_{n-1} \cdots s_{i+1}s_n \cdots s_2 z^k s_2 \cdots s_{i+1}$ which belongs to Span$(\Lambda_{n-1}\Lambda_n)$, see lines 9 and 10 of the proof of Lemma \ref{Lemma8}. If $i' > i+1$, we get $s_{n-1} \cdots s_{i'-1}s_n \cdots s_2 z^k s_2 \cdots s_i$ which belongs to Span$(\Lambda_{n-1}\Lambda_n)$, see line 16 of the proof of Lemma \ref{Lemma8}.

\textit{Suppose $i \geq i'$}. We get $s_n(a_{n-1}a_n) = s_{n-1} \cdots s_{i'}s_n \cdots s_2z^ks_2 \cdots s_{i+1}$, see the last line of the proof of Lemma \ref{Lemma8}. Hence $s_n(a_{n-1}a_n)$ belongs to Span$(\Lambda_{n-1}\Lambda_n)$.

\end{proof}

\begin{lemma}\label{Lemma99}

If $a_{n-1} = s_{n-1} \cdots s_2z^ks_2 \cdots s_i$ with $2 \leq i \leq n-1$, $1 \leq k \leq d-1$ and $a_n = s_n \cdots s_2 z^l$ with $0 \leq l \leq d-1$, then $s_n(a_{n-1}a_n)$ belongs to \emph{Span}$(S_{n-1}^{*}\Lambda_n)$.

\end{lemma}

\begin{proof}

According to the computations in the proof of the previous lemma, we have $s_n(a_{n-1}a_n) = s_{n-1} \cdots s_2 s_n \cdots s_2 z^k s_2 \cdots s_{i+1} z^l$. We shift $z^l$ to the left and get\\ $s_n(a_{n-1}a_n) = s_{n-1} \cdots s_2 s_n \cdots s_3(s_2 z^k s_2)z^l s_3 \cdots s_{i+1}$. If we replace $(s_2z^ks_2)z^l$ by its decomposition over $\Lambda_1\Lambda_2$ (this is the case $n=2$ of Theorem \ref{TheoremNewBasisH(d,1,n)}, see Proposition \ref{PropInductionHypothesisG(d,1,n)}), we get terms of the three following forms:
\begin{itemize}
\item $s_{n-1} \cdots s_2 s_n \cdots s_3 z^c s_3 \cdots s_{i+1}$ with $0 \leq c \leq d-1$,
\item $s_{n-1} \cdots s_2 s_n \cdots s_3 z^c s_2 z^{c'}s_3 \cdots s_{i+1}$ with $0 \leq c \leq d-1$ and $0 \leq c' \leq d-1$,
\item $s_{n-1} \cdots s_2 s_n \cdots s_3 z^cs_2z^{c'}s_2s_3 \cdots s_{i+1}$ with $0 \leq c \leq d-1$ and $1 \leq c' \leq d-1$.
\end{itemize}

The first term is equal to $s_{n-1} \cdots s_2 z^c s_n \cdots s_4s_3^2s_4 \cdots s_{i+1}$ with $2 \leq i \leq n-1$. Thus, by the proof of Lemma \ref{LmScholie11}, it belongs to Span$(S_{n-1}^{*}\Lambda_n)$.

The second term is equal to $s_{n-1} \cdots s_2 z^c s_n \cdots s_3s_2s_3 \cdots s_{i+1}z^{c'}$. We have\\ $s_n \cdots s_3s_2s_3 \cdots s_{i+1}$ is equal to $s_2 \cdots s_i s_n \cdots s_2$. This is done in the first case of the proof of Lemma \ref{Lemma77}. Hence we get $s_{n-1} \cdots s_2 z^c s_2 \cdots s_i s_n \cdots s_2 z^{c'}$ which belongs to Span$(S_{n-1}^{*}\Lambda_n)$.

The third term is equal to $s_{n-1} \cdots s_2 z^c s_n \cdots s_2z^{c'}s_2 \cdots s_{i+1}$. Hence it belongs to Span$(S_{n-1}^{*}\Lambda_n)$.

\end{proof}

\begin{lemma}\label{Lemma1010}

If $a_{n-1} = s_{n-1} \cdots s_2 z^k s_2 \cdots s_i$ with $1 \leq k \leq d-1$, $2 \leq i \leq n-1$ and $a_n = s_n \cdots s_2 z^ls_2 \cdots s_{i'}$ with $1 \leq l \leq d-1$, $2 \leq i' \leq n$, then $s_n(a_{n-1}a_n)$ belongs to \emph{Span}$(S_{n-1}^{*}\Lambda_n)$.

\end{lemma}

\begin{proof}

According to the proof of the previous lemma, we have to prove that the following three terms belong to Span$(S_{n-1}^{*}\Lambda_n)$:
\begin{itemize}
\item $s_{n-1} \cdots s_2 z^cs_n \cdots s_4s_3^2s_4 \cdots s_{i+1}(s_2 \cdots s_{i'})$ for $0 \leq c \leq d-1$,
\item $s_{n-1} \cdots s_2 z^cs_2 \cdots s_i s_n \cdots s_2 z^{c'}(s_2 \cdots s_{i'})$ for $0 \leq c \leq d-1$ and $0 \leq c' \leq d-1$,
\item $s_{n-1} \cdots s_2 z^c s_n \cdots s_2 z^{c'} s_2 \cdots s_{i+1}(s_2 \cdots s_{i'})$ for $0 \leq c \leq d-1$ and $1 \leq c' \leq d-1$.
\end{itemize}

The first case is similar to Case 4 in the proof of Lemma \ref{Lemma10}.

For the second term, if $c' \neq 0$, it belongs to Span$(S_{n-1}^{*}\Lambda_n)$ and if $c' = 0$, by the computation in the proof of Lemma \ref{LmScholie11}, it also belongs to Span$(S_{n-1}^{*}\Lambda_n)$.

For the third term, we apply the same technique as in Case 2 of the proof of Lemma \ref{Lemma10}. We get that $s_n \cdots s_2 z^{c'}s_2 \cdots s_{i+1} (s_2 \cdots s_{i'})$ is equal to:\\
If $i' \leq i$, it is equal to $s_2s_3 \cdots s_{i'}s_n \cdots s_2 z^{c'}s_2 \cdots s_{i+1}$.\\
If $i' > i$, it is equal to $s_2s_3 \cdots s_{i'-1}s_n \cdots s_2 z^{c'}s_2 \cdots s_i$.\\
It follows that $s_{n-1} \cdots s_2 z^c s_n \cdots s_2 z^{c'} s_2 \cdots s_{i+1}(s_2 \cdots s_{i'})$ for $0 \leq c \leq d-1$ and $1 \leq c' \leq d-1$ belongs to Span$(S_{n-1}^{*}\Lambda_n)$.
 
\end{proof}

\begin{remark}

We remark that for every $d$ and $n$ at least equal to $2$, our basis never coincides with the Ariki-Koike basis as illustrated by the following example. Consider the element $s_2zs_2^2 = s_2zs_2.s_2$ which belongs to the Ariki-Koike basis. In our basis, it is equal to the linear combination $as_2zs_2 + s_2z$, where $s_2zs_2$ and $s_2z$ are two distinct elements of our basis.

\end{remark}

%% file: chap55.tex
\chapter{Toward Krammer's representations and BMW algebras for $B(e,e,n)$}\label{ChapterKrammerRepresentations}
 
\minitoc

\bigskip

In this chapter, we construct irreducible representations of some complex braid groups $B(e,e,n)$ by defining a BMW algebra for $B(e,e,n)$ and by studying its properties. We call these representations the Krammer's representations for $B(e,e,n)$. We provide Conjecture \ref{ConjectureRho3and4faithful} about the faithfulness of these representations. We also provide Conjecture \ref{ConjectureStructureBMW} about the structure and the dimension of the BMW algebra.

\section{Motivations and preliminaries}

Both Bigelow \cite{BigelowBraidGroupsLinear} and Krammer \cite{KrammerGroupB4,KrammerBraidGroupsLinear} proved that the classical braid group is linear, that is there exists a faithful linear representation of finite dimension of the classical braid group $B_n$. We recall that $B_n$ is defined by a presentation with generators $\{s_1, s_2, \cdots, s_{n-1} \}$ and relations $s_is_{i+1}s_i = s_{i+1}s_is_{i+1}$ for $1 \leq i \leq n-2$ and $s_is_j = s_js_i$ for $|i-j| > 1$. It can be described by the following diagram:

\begin{center}
\begin{tikzpicture}

\node[draw, shape=circle, label=above:$s_1$] (1) at (0,0) {};
\node[draw, shape=circle,label=above:$s_2$] (2) at (1,0) {};
\node[draw,shape=circle,label=above:$s_{n-2}$] (n-1) at (4,0) {};
\node[draw,shape=circle,label=above:$s_{n-1}$] (n) at (5,0) {};

\draw[thick,-] (1) to (2);
\draw[dashed,-,thick] (2) to (n-1);
\draw[thick,-] (n-1) to (n);

\end{tikzpicture}
\end{center}

The representation of Krammer $\rho: B_{n} \longrightarrow GL(V)$ is defined on an $\mathbb{R}(q,t)$-vector space $V$ with basis $\{x_s \ |\ s \in \mathcal{R} \}$ indexed on the set of reflections $(i,j)$ of the symmetric group $S_n$ with $1 \leq i < j \leq n-1$. Its dimension is then $\#\mathcal{R} = \frac{n(n-1)}{2}$. We denote $x_{(i,j)}$ by $x_{i,j}$. The representation is defined as follows.\\

\begin{tabular}{ll}

 $s_{k}x_{k,k+1} = tq^{2}x_{k,k+1}$, &  \\
 $s_{k}x_{i,k} = (1-q)x_{i,k} + qx_{i,k+1}$, & $i<k$,\\
 $s_{k}x_{i,k+1} = x_{i,k} + tq^{(k-i+1)}(q-1)x_{k, k+1}$, & $i<k$,\\
 $s_{k}x_{k,j} = tq(q-1)x_{k,k+1} + qx_{k+1,j}$, & $k+1<j$,\\
 $s_{k}x_{k+1,j} = x_{k,j} + (1-q)x_{k+1,j}$, & $k+1 < j$,\\
 $s_{k}x_{i,j} = x_{i,j},$ & $i < j < k$ or $k+1 < i < j$, and\\
 $s_{k}x_{i,j} = x_{i,j} + tq^{(k-i)}(q-1)^{2}x_{k,k+1}$, & $i < k < k+1 < j$.\\

\end{tabular}

\bigskip

The faithfulness criterion used by Krammer can be stated for a Garside proup. It provides necessary conditions to prove that a linear representation of a Garside group is faithful. Let $M$ be a Garside monoid and $G(M)$ its group of fractions. Denote by $\Delta$ and $P$ the Garside element and the set of simples of $M$, respectively. Define $\alpha(x)$ to be the gcd of $x$ and $\Delta$ for $x \in M$. Let $\rho : G(M) \longrightarrow GL(V)$ be a linear representation of finite dimension of $G(M)$ and let $(C_x)_{x \in P}$ be a family of subsets of $V$ indexed by the set of simples $P$. If the $C_x$ are nonempty and (pairwise) disjoint and $x C_y \subset C_{\alpha(xy)}$ for all $x \in M$ and $y \in P$, then the representation $\rho$ is faithful.\\

Krammer's representation as well as the faithfulness proof have been generalized to all Artin-Tits groups of spherical type by a work of \cite{DigneLinearity}, \cite{CohenWalesLinearity}, and \cite{ParisArtinMonoidInjectGroup}. Note also that a simple faithfulness proof was given in \cite{HeePreuveSimpleFidelite}. Marin generalized in \cite{MarinLinearity} this representation to all the $2$-reflection groups. His representation is defined analytically over the field of (formal) Laurent series by the monodromy of some differential forms and has dimension the number of reflections in the complex reflection group. It was conjectured in \cite{MarinLinearity} that this representation is faithful. It has also been generalized by Chen in \cite{ChenFlatConnectionsBrauerAlgebras} to arbitrary reflection groups.\\

In type ADE of Coxeter groups, the generalized Krammer's representations can be constructed via BMW (Birman-Murakami-Wenzl) algebras. For more details, see Section 1.4.3 of Chapter 1. We recall the definition of the BMW algebra for the case ADE as it appears in \cite{CohenGijsbersWalesBMW}. 

\begin{definition}\label{DefinitionBMWADE}

Let $W$ be a Coxeter group of type $A_n$, $D_n$ for any $n$, or $E_n$ for $n = 6, 7, 8$. The BMW algebra associated to $W$ is the $\mathbb{Q}(l,x)$-algebra with identity, with the generating set: $\{ S_1, S_2, \cdots, S_n \} \cup \{ F_1, F_2, \cdots, F_n \}$, and the defining relations are the braid relations along with:

\begin{enumerate}

\item $mF_i = l(S_i^2 + m S_i -1)$ with $m = \frac{l-l^{-1}}{1-x}$ for all $i$,
\item $S_iF_i = F_iS_i = l^{-1}F_i$ for all $i$, and
\item $F_iS_jF_i = l F_i$ for all $i,j$ when $s_is_js_i = s_js_is_j$.

\end{enumerate}

\end{definition}

For the proof of the following proposition, see Propositions 2.1 and 2.3 of \cite{CohenGijsbersWalesBMW}.

\begin{proposition}\label{PropEi^2=xEi}

We have $S_i$ is invertible with $S_i^{-1} = S_i + m -mF_i$. We also have $F_i^{2} = x F_i$ and $S_jS_iF_j = F_iS_jS_i = F_iF_j$ when $s_is_js_i = s_js_is_j$.

\end{proposition}

If $l = 1$, then the BMW algebra of Definition \ref{DefinitionBMWADE} degenerates to the following Brauer algebra, see \cite{Brauer} and \cite{CohenFrenkWales}. In this case, $m = \frac{l-l^{-1}}{1-x} = 0$ and Relation $1$ of Definition \ref{DefinitionBMWADE} degenerates to $S_i^2 = 1$. Hence we have $S_i^{-1} = S_i$. By Proposition \ref{PropEi^2=xEi}, we still have the relation $F_i^{2} = x F_i$ in the Brauer algebra.

\begin{definition}\label{DefinitionBrauerBMW}

Let $W$ be a Coxeter group of type $A_n$, $D_n$ for all $n$, or $E_n$ for $n = 6, 7, 8$. The Brauer algebra associated to $W$ is the $\mathbb{Q}(x)$-algebra with identity, with the generating set: $\{ S_1, S_2, \cdots, S_n \} \cup \{ F_1, F_2, \cdots, F_n \}$, and the defining relations are the braid relations along with:

\begin{enumerate}

\item $S_i^2 = 1$ for all $i$,
\item $F_i^2 = x F_i$ for all $i$,
\item $S_iF_i = F_iS_i = F_i$ for all $i$, 
\item $F_iS_jF_i = F_i$ for all $i,j$ when $s_is_js_i = s_js_is_j$, and
\item $S_jS_iF_j = F_iS_jS_i = F_iF_j$ when $s_is_js_i = s_js_is_j$.

\end{enumerate}

\end{definition}

Chen defined in \cite{ChenFlatConnectionsBrauerAlgebras} a Brauer algebra for all complex reflection groups that generalizes earlier works in \cite{Brauer} and \cite{CohenFrenkWales}. Completing his results in \cite{ChenOldBMW}, Chen also defined in \cite{ChenBMW2017} a BMW algebra for the dihedral groups $I_2(e)$ for all $e \geq 2$ based on which he defined a BMW algebra for any Coxeter group that degenerates to the Brauer algebra introduced in \cite{ChenFlatConnectionsBrauerAlgebras}. It is also shown in \cite{ChenBMW2017} the existence of a representation for each Artin-Tits group associated to $I_2(e)$. The representation has dimension $e$ and is explicitly defined over $\mathbb{Q}(\alpha,\beta)$, where $\alpha$ and $\beta$ depend on the parameters of the BMW algebra. It is also conjectured in \cite{ChenBMW2017} that this representation is isomorphic to the monodromy representation constructed by Marin in \cite{MarinLinearity}.\\

Attempting to make a similar approach in order to explicitly construct faithful irreducible representations for the complex braid groups $B(e,e,n)$, we define a BMW algebra for $B(e,e,n)$ that we denote by BMW$(e,e,n)$, see Definitions \ref{DefinitionBMWB(e,e,n)eOdd} and \ref{DefinitionBMWB(e,e,n)eEven}. We show that it is a deformation of an algebra that we denote by Br$(e,e,n)$ and we call the Brauer algebra, see Definitions \ref{DefinitionBr(e,e,n)eodd} and \ref{DefinitionBr(e,e,n)eeven}. We also show in Proposition \ref{PropositionChenIsomBr(e,e,3)} that Br$(e,e,n)$ is isomorphic to the Brauer-Chen algebra defined by Chen in  \cite{ChenFlatConnectionsBrauerAlgebras} for $n=3$ and $e$ odd. By using the BMW algebra BMW$(e,e,n)$, we explicitly define what we call the Krammer's representations for $B(3,3,3)$ and $B(4,4,3)$ by using the package GBNP \cite{GBNP} of GAP4 without being able to go further since the computations become very heavy for $e > 5$. In the last section, we provide some conjectural properties about BMW$(e,e,n)$ and the Krammer's representations, see Conjectures \ref{ConjectureRho3and4faithful} and \ref{ConjectureStructureBMW}.

\section{BMW and Brauer algebras for type $(e,e,n)$}

We are inspired by the monoid $B^{\oplus}(e,e,n)$ of Corran and Picantin in order to construct a BMW algebra for the complex braid groups $B(e,e,n)$. The presentation of $B^{\oplus}(e,e,n)$ is given in Definition \ref{DefofB+keen} for $k=1$ and described by the diagram of Figure \ref{PresofCPBeen}. It consists of attaching the dual presentation of the dihedral group $I_2(e)$ with generating set $\{ \tilde{t}_0,\tilde{t}_1, \cdots, \tilde{t}_{e-1} \}$ to the classical presentation of the braid group of type $A_{n-1}$ with set of generators $\{ \tilde{t}_i, \tilde{s}_3, \tilde{s}_4, \cdots, \tilde{s}_n \}$ for all $0 \leq i \leq e-1$. Inspired by the definition of Chen of the BMW algebra for the dihedral groups in \cite{ChenBMW2017} and the definition of the BMW algebra for the type ADE of Coxeter groups in \cite{CohenGijsbersWalesBMW}, we define a BMW algebra for $B(e,e,n)$ as follows. We distinguish the cases when $e$ is odd and even and we suppose $n \geq 3$.

\begin{definition}\label{DefinitionBMWB(e,e,n)eOdd}

Suppose $e$ odd. We define the BMW algebra associated to $B(e,e,n)$ to be the $\mathbb{Q}(l,x)$-algebra with identity, with the generating set: $$\{T_i\ |\ i \in \mathbb{Z}/e\mathbb{Z} \} \cup \{S_3, S_4, \cdots, S_n \} \cup \{E_i\ |\ i \in \mathbb{Z}/e\mathbb{Z}\} \cup \{ F_3, F_4, \cdots, F_n\},$$ and the defining relations are the relations of BMW of type $A_{n-1}$ for\\ $\{ T_i, S_3, S_4, \cdots, S_n \} \cup \{E_i,F_3, \cdots , F_n\}$ with $0 \leq i \leq e-1$ as described in Definition \ref{DefinitionBMWADE} along with the dihedral BMW relations that can be described as follows:

\begin{enumerate}

\item $T_i = T_{i-1}T_{i-2}T_{i-1}^{-1}$ and $E_i = T_{i-1}E_{i-2}T_{i-1}^{-1}$ for all $i \in \mathbb{Z}/e\mathbb{Z}$, $i \neq 0,1$,
\item $\underset{e}{\underbrace{T_1T_0 \cdots T_1}}=\underset{e}{\underbrace{T_0T_1 \cdots T_0}}$,
\item $m E_i = l (T_i^2 + mT_i -1)$ for $i = 0, 1$, where $m = \frac{l-l^{-1}}{1-x}$,
\item $T_iE_i = E_iT_i = l^{-1} E_i$ for $i = 0, 1$, 
\item $E_1\underset{k}{\underbrace{T_0T_1 \cdots T_0}}E_1 = l E_1$, where $1 \leq k \leq e-2$, $k$ odd,
\item $E_0\underset{k}{\underbrace{T_1T_0 \cdots T_1}}E_0 = l E_0$, where $1 \leq k \leq e-2$, $k$ odd,
\item $\underset{e-1}{\underbrace{T_1T_0 \cdots T_0}}E_1 = E_0 \underset{e-1}{\underbrace{T_1T_0 \cdots T_0}}$,
\item $\underset{e-1}{\underbrace{T_0T_1 \cdots T_1}}E_0 = E_1\underset{e-1}{\underbrace{T_0T_1 \cdots T_1}}$.
\end{enumerate}

\end{definition}

\begin{definition}\label{DefinitionBMWB(e,e,n)eEven}

Suppose $e$ even. Let $m = v-v^{-1}$ and $x$ be such that $m = \frac{l-l^{-1}}{1-x}$. We define the BMW algebra associated to $B(e,e,n)$ to be the $\mathbb{Q}(l,v)$-algebra with identity, with the generating set: $$\{T_i\ |\ i \in \mathbb{Z}/e\mathbb{Z} \} \cup \{S_3, S_4, \cdots, S_n \} \cup \{E_i\ |\ i \in \mathbb{Z}/e\mathbb{Z}\} \cup \{ F_3, F_4, \cdots, F_n\},$$ and the defining relations are the relations of BMW of type $A_{n-1}$ for\\ $\{ T_i, S_3, S_4, \cdots, S_n \} \cup \{E_i,F_3, \cdots , F_n\}$ with $0 \leq i \leq e-1$ as described in Definition \ref{DefinitionBMWADE} along with the dihedral BMW relations that can be described as follows:

\begin{enumerate}

\item $T_i = T_{i-1}T_{i-2}T_{i-1}^{-1}$ and $E_i = T_{i-1}E_{i-2}T_{i-1}^{-1}$ for all $i \in \mathbb{Z}/e\mathbb{Z}$, $i \neq 0,1$,
\item $\underset{e}{\underbrace{T_1T_0 \cdots T_0}}=\underset{e}{\underbrace{T_0T_1 \cdots T_1}}$,
\item $m E_i = l (T_i^2 + mT_i -1)$ for $i = 0, 1$,
\item $T_iE_i = E_iT_i = l^{-1} E_i$ for $i = 0, 1$,
\item $E_1\underset{i}{\underbrace{T_0T_1 \cdots T_0}}E_1 = (v^{-1} + l) E_1$ for $i = 4k+ 1 < e/2$ and $i = 4k+3 < e/2$,
\item $E_0\underset{i}{\underbrace{T_1T_0 \cdots T_1}}E_0 = (v^{-1} + l) E_0$ for $i = 4k+ 1 < e/2$ and $i = 4k+3 < e/2$,
\item $\underset{e-1}{\underbrace{T_1T_0 \cdots T_1}}E_0 = E_0 \underset{e-1}{\underbrace{T_1T_0 \cdots T_1}} = v^{-1} E_0$,
\item $\underset{e-1}{\underbrace{T_0T_1 \cdots T_0}}E_1 = E_1\underset{e-1}{\underbrace{T_0T_1 \cdots T_0}} = v^{-1} E_1$,
\item $E_0AE_1 = E_1AE_0 = 0$, where $A$ is any square-free word over $\{T_0,T_1\}$ of length at most $e-1$.
\end{enumerate}

\end{definition}

\begin{remark}

Additional relations for the BMW$(e,e,n)$ algebras can be found in the same way as in Lemmas 2.1 and 5.1 of \cite{ChenBMW2017} for the dihedral groups. Also we can get additional relations similar to those of Propositions 2.5 and 2.8 of \cite{CohenGijsbersWalesBMW} that correspond to the BMW algebra of type $A_n$.

\end{remark}

\begin{remark}

There exists a natural quotient map BMW$(e,e,n) \longrightarrow$ H$(e,e,n)$ by sending $T_i$ to $t_i$, $S_j$ to $s_j$, $E_i$ to $0$, and $F_j$ to $0$, where $\{ t_0, t_1, \cdots , t_{e-1}, s_3, s_4, \cdots , s_n  \}$ is the set of generators of H$(e,e,n)$, see Definition \ref{DefPresH(d,1,n)}.
 
\end{remark}

In Definition \ref{DefinitionBMWB(e,e,n)eOdd}, when $l = 1$, BMW$(e,e,n)$ degenerates to the algebra Br$(e,e,n)$ that we provide in Definition \ref{DefinitionBr(e,e,n)eodd} below and that we call the Brauer algebra of type $(e,e,n)$. In this case, we have $m = 0$ and Relation $3$ of Definition \ref{DefinitionBMWB(e,e,n)eOdd} degenerates to $T_i^2 = 1$. Hence we have $T_i^{-1} = T_i$ for all $i$. We still have $E_i^{2} = x E_i$ and $F_j^{2} = x F_j$ for all $i$ and $j$, see Proposition \ref{PropEi^2=xEi}. In Proposition \ref{PropositionChenIsomBr(e,e,3)}, we prove that Br$(e,e,n)$ coincides with the Brauer algebra defined by Chen in \cite{ChenFlatConnectionsBrauerAlgebras} for $n = 3$ and $e$ odd. 

\begin{definition}\label{DefinitionBr(e,e,n)eodd}

Suppose $e$ odd. We define the Brauer algebra \emph{Br}$(e,e,n)$ to be the $\mathbb{Q}(x)$-algebra with identity, with the generating set: $$\{T_i\ |\ i \in \mathbb{Z}/e\mathbb{Z} \} \cup \{S_3, S_4, \cdots, S_n \} \cup \{E_i\ |\ i \in \mathbb{Z}/e\mathbb{Z}\} \cup \{ F_3, F_4, \cdots, F_n\},$$ and the defining relations are the relations of the Brauer algebra of type $A_{n-1}$ for\\ $\{ T_i, S_3, S_4, \cdots, S_n \} \cup \{E_i,F_3, \cdots , F_n\}$ with $0 \leq i \leq e-1$ as described in Definition \ref{DefinitionBrauerBMW} along with the normalized dihedral Brauer relations that can be described as follows:

\begin{enumerate}

\item $T_i = T_{i-1}T_{i-2}T_{i-1}$ and $\underset{e}{\underbrace{T_1T_0 \cdots T_1}}=\underset{e}{\underbrace{T_0T_1 \cdots T_0}}$ for $i \in \mathbb{Z}/e\mathbb{Z}$,
\item $E_i = T_{i-1}E_{i-2}T_{i-1}$ for all $i \in \mathbb{Z}/e\mathbb{Z}$,
\item $T_i^2 = 1$ for $i = 0, 1$, 
\item $E_i^2 = x E_i$ for $i = 0,1$,
\item $T_iE_i = E_iT_i = E_i$ for $i = 0, 1$, 
\item $E_1\underset{k}{\underbrace{T_0T_1 \cdots T_0}}E_1 = E_1$, where $1 \leq k \leq e-2$, $k$ odd,
\item $E_0\underset{k}{\underbrace{T_1T_0 \cdots T_1}}E_0 = E_0$, where $1 \leq k \leq e-2$, $k$ odd,
\item $\underset{e-1}{\underbrace{T_1T_0 \cdots T_0}}E_1 = E_0 \underset{e-1}{\underbrace{T_1T_0 \cdots T_0}}$,
\item $\underset{e-1}{\underbrace{T_0T_1 \cdots T_1}}E_0 = E_1\underset{e-1}{\underbrace{T_0T_1 \cdots T_1}}$.
\end{enumerate}

\end{definition}

In Definition \ref{DefinitionBMWB(e,e,n)eEven}, when $l = 1$ and $v = 1$, BMW$(e,e,n)$ degenerates to the algebra Br$(e,e,n)$ that we provide in Definition \ref{DefinitionBr(e,e,n)eeven} below and that we call the Brauer algebra of type $(e,e,n)$ for $e$ even.

\begin{definition}\label{DefinitionBr(e,e,n)eeven}

Suppose $e$ even. We define the Brauer algebra \emph{Br}$(e,e,n)$ to be the $\mathbb{Q}(x)$-algebra with identity, with the generating set: $$\{T_i\ |\ i \in \mathbb{Z}/e\mathbb{Z} \} \cup \{S_3, S_4, \cdots, S_n \} \cup \{E_i\ |\ i \in \mathbb{Z}/e\mathbb{Z}\} \cup \{ F_3, F_4, \cdots, F_n\},$$ and the defining relations are the relations of the Brauer algebra of type $A_{n-1}$ for\\ $\{ T_i, S_3, S_4, \cdots, S_n \} \cup \{E_i,F_3, \cdots , F_n\}$ with $0 \leq i \leq e-1$ as described in Definition \ref{DefinitionBrauerBMW} along with the normalized dihedral Brauer relations that can be described as follows:

\begin{enumerate}

\item $T_i = T_{i-1}T_{i-2}T_{i-1}$ and $\underset{e}{\underbrace{T_1T_0 \cdots T_0}}=\underset{e}{\underbrace{T_0T_1 \cdots T_1}}$ for $i \in \mathbb{Z}/e\mathbb{Z}$,
\item $E_i = T_{i-1}E_{i-2}T_{i-1}$ for all $i \in \mathbb{Z}/e\mathbb{Z}$,
\item $T_i^2 = 1$ for $i = 0, 1$,
\item $E_i^2 = x E_i$ for $i = 0,1$,
\item $T_iE_i = E_iT_i = E_i$ for $i = 0, 1$,
\item $E_1\underset{i}{\underbrace{T_0T_1 \cdots T_0}}E_1 = E_1$, where $i = 4k+ 1 < e/2$ or $i = 4k+3 < e/2$,
\item $E_0\underset{i}{\underbrace{T_1T_0 \cdots T_1}}E_0 = E_0$, where $i = 4k+ 1 < e/2$ or $i = 4k+3 < e/2$,
\item $\underset{e-1}{\underbrace{T_1T_0 \cdots T_1}}E_0 = E_0 \underset{e-1}{\underbrace{T_1T_0 \cdots T_1}} = E_0$,
\item $\underset{e-1}{\underbrace{T_0T_1 \cdots T_0}}E_1 = E_1\underset{e-1}{\underbrace{T_0T_1 \cdots T_0}} = E_1$,
\item $E_0AE_1 = E_1AE_0 = 0$, where $A$ is any square-free word over $\{T_0,T_1\}$ of length at most $e-1$.

\end{enumerate}

\end{definition}

\begin{remark}

In the definition of \emph{Br}$(e,e,n)$, one can replace Relations 6 and 7 of Definitions \ref{DefinitionBr(e,e,n)eodd} and \ref{DefinitionBr(e,e,n)eeven} by $E_1\underset{i}{\underbrace{T_0T_1 \cdots T_0}}E_1 = \mu E_1$ and $E_0\underset{i}{\underbrace{T_1T_0 \cdots T_1}}E_0 = \mu E_0$ for a given parameter $\mu$. In our definitions, we set $\mu = 1$ in the dihedral relations that we call the normalized dihedral Brauer relations.

\end{remark}

\begin{proposition}\label{PropGroupAlgebraInjectsInBrauer}

The group algebra $\mathbb{Q}(x)(G(e,e,n))$ injects in the Brauer algebra \emph{Br}$(e,e,n)$.

\end{proposition}

\begin{proof}

Define a map $r$ from the set of generators (of the presentation of Corran and Picantin) of the group algebra $\mathbb{Q}(x)(G(e,e,n))$ to the set of generators of Br$(e,e,n)$ by $r(t_i) = T_i$ and $r(s_j) = S_j$ for all $i \in \mathbb{Z}/e\mathbb{Z}$ and $3 \leq j \leq n$. Also define a map $r'$ from the set of generators of Br$(e,e,n)$ to the set of generators of $\mathbb{Q}(x)(G(e,e,n))$ by $r'(T_i) = t_i$, $r'(S_j) = s_j$, $r'(E_i) = 0$, and $r'(F_j) = 0$ for all $i \in \mathbb{Z}/e\mathbb{Z}$ and $3 \leq j \leq n$. On examination of the relations of $\mathbb{Q}(x)(G(e,e,n))$ and Br$(e,e,n)$, it is readily checked that $r$ extends to a morphism from $\mathbb{Q}(x)(G(e,e,n))$ to Br$(e,e,n)$ and $r'$ extends to a morphism from Br$(e,e,n)$ to $\mathbb{Q}(x)(G(e,e,n))$. We have $r' \circ r = id$, where $id$ is the identity morphism on $\mathbb{Q}(x)(G(e,e,n))$. Hence $r$ is injective.

\end{proof}

In the following two propositions, we provide additional relations that hold in Br$(e,e,n)$.

\begin{proposition}\label{PropAdditionalRelations1}
When $e$ is odd, we have 
\begin{equation}
\begin{aligned}
E_1E_0E_1 = E_1,\\
E_0E_1E_0 = E_0.
\end{aligned}
\end{equation}

\end{proposition}

\begin{proof}

Suppose that $e$ is odd. By Relation 8 of Definition \ref{DefinitionBr(e,e,n)eodd}, we have  $\underset{e-1}{\underbrace{T_1T_0 \cdots T_0}}E_1 = E_0 \underset{e-1}{\underbrace{T_1T_0 \cdots T_0}}$. Since $T_0^2 = 1$ and $T_1^2 = 1$, this implies that $E_1 = \underset{e-1}{\underbrace{T_0T_1 \cdots T_1}} E_0 \underset{e-1}{\underbrace{T_1T_0 \cdots T_0}}$. Hence $E_1E_0 = \underset{e-1}{\underbrace{T_0T_1 \cdots T_1}} E_0 \underset{e-1}{\underbrace{T_1T_0 \cdots T_0}}E_0 = \underset{e-1}{\underbrace{T_0T_1 \cdots T_1}} E_0 \underset{e-2}{\underbrace{T_1T_0 \cdots T_1}}E_0$ with $e-2$ odd. By Relation 7 of Definition \ref{DefinitionBr(e,e,n)eodd}, we get $E_1E_0 = \underset{e-1}{\underbrace{T_0T_1 \cdots T_1}} E_0$. It follows that $E_1E_0E_1 =  \underset{e-1}{\underbrace{T_0T_1 \cdots T_1}} E_0 \underset{e-1}{\underbrace{T_0T_1 \cdots T_1}} E_0 \underset{e-1}{\underbrace{T_1T_0 \cdots T_0}}\\ = \underset{e-1}{\underbrace{T_0T_1 \cdots T_1}} E_0 \underset{e-2}{\underbrace{T_1T_0 \cdots T_1}} E_0 \underset{e-1}{\underbrace{T_1T_0 \cdots T_0}} =\\ \underset{e-1}{\underbrace{T_0T_1 \cdots T_1}} E_0 \underset{e-1}{\underbrace{T_1T_0 \cdots T_0}} = E_1$. Similarly, one can prove $E_0E_1E_0 = E_0$.

\end{proof}

\begin{remark}

By Relation 10 of Definition \ref{DefinitionBMWB(e,e,n)eEven}, we have $E_1E_0E_1 = E_0E_1E_0 = 0$ when $e$ is even. That's why in the previous proposition, we only consider the case $e$ odd.

\end{remark}

\begin{proposition}\label{PropAdditionalRelations2}
For $e \geq 3$ and for $0 \leq k \leq e-1$, we have 
\begin{equation}
\begin{aligned}
T_kS_3E_k = F_3T_kS_3 = F_3E_k,\\
S_3T_kF_3 = E_kS_3T_k = E_kF_3,
\end{aligned}
\end{equation}

\begin{equation}
\begin{aligned}
E_kF_3E_k = E_k,\\
F_3E_kF_3 = F_3,
\end{aligned}
\end{equation}

\begin{equation}
\begin{aligned}
T_kF_3E_k = S_3E_k,\\
S_3E_kF_3 = T_kF_3,
\end{aligned}
\end{equation}

\begin{equation}
\begin{aligned}
E_kF_3T_k = E_kS_3,\\
F_3E_kS_3 = F_3T_k,
\end{aligned}
\end{equation}

\begin{equation}\label{LastEquation}
\begin{aligned}
T_kF_3T_k = S_3E_kS_3.
\end{aligned}
\end{equation}

\end{proposition}

\begin{proof}

\emph{Equation (5.2):} We have $F_3 = T_kS_3E_kS_3^{-1}T_k^{-1} = T_kS_3E_kS_3T_k$. Thus, $F_3 E_k =  T_kS_3E_kS_3T_kE_k = T_kS_3E_kS_3E_k = T_kS_3E_k$.

\emph{Equation (5.3):} By Equation (5.2), we have $E_kF_3 = E_kS_3T_k$. Hence $E_kF_3E_k = E_kS_3T_kE_k = E_kS_3E_k = E_k$. Similarly, one can prove $F_3E_kF_3 = F_3$.

\emph{Equation (5.4):} We have $T_kF_3E_k = T_kT_kS_3E_k = T_k^2 S_3E_k = S_3E_k$. Similarly, one can prove $S_3E_kF_3 = T_kF_3$.

\emph{Equation (5.5):} We have $E_kF_3T_k = E_kS_3T_kT_k = E_kS_3$. Similarly, we prove $F_3E_kS_3 = F_3T_k$.

\emph{Equation (5.6):} We have $F_3 = T_kS_3E_kS_3T_k$ by Equation (5.2). Hence $T_kF_3T_k = T_k^2 S_3E_kS_3T_k^2 = S_3E_kS_3$.

\end{proof}

From now on until the end of the next section, we set the convention that an expression of the form $\underset{-1}{\underbrace{T_1T_0 \cdots T_1}}$ is equal to $T_0$. Then, an expression of the form $\underset{-1}{\underbrace{T_1T_0 \cdots T_1}}\ E_0\ \underset{-1}{\underbrace{T_1T_0 \cdots T_1}}$ is equal to $E_0$.

In Br$(e,e,3)$, for $0 \leq k \leq e-1$, one can express $T_k$ in terms of $T_0$ and $T_1$ and $E_k$ in terms of $T_0, T_1, E_0$, and $E_1$. We have the following.

\begin{lemma}\label{LemmaTkEk}

For $0 \leq k \leq e-1$, we have
\begin{itemize}
\item $T_k = \underset{2k-1}{\underbrace{T_1T_0 \cdots T_1}} = \left\{
    \begin{array}{ll}
     \underset{k-1}{\underbrace{T_1T_0 \cdots T_0}}T_1\underset{k-1}{\underbrace{T_0T_1 \cdots T_1}}  & if\ k\ is\ odd, \\
     \underset{k-1}{\underbrace{T_1T_0 \cdots T_1}}T_0\underset{k-1}{\underbrace{T_1T_0 \cdots T_1}}  & if\ k\ is\ even.
    \end{array}
\right.$
\item $E_k = \left\{
    \begin{array}{ll}
     \underset{k-1}{\underbrace{T_1T_0 \cdots T_0}}E_1\underset{k-1}{\underbrace{T_0T_1 \cdots T_1}}  & if\ k\ is\ odd, \\
     \underset{k-1}{\underbrace{T_1T_0 \cdots T_1}}E_0\underset{k-1}{\underbrace{T_1T_0 \cdots T_1}}  & if\ k\ is\ even.
    \end{array}
\right.$
\end{itemize}

\end{lemma}

\begin{proof}
For $k = 0$, we have $\underset{-1}{\underbrace{T_1T_0 \cdots T_1}} = T_0$ by the convention. Let $1 \leq k \leq e-1$. For $k = 1$, we have $T_1 = \underset{2k-1}{\underbrace{T_1T_0 \cdots T_1}}$. From Relation 1 of Definitions \ref{DefinitionBr(e,e,n)eodd} and \ref{DefinitionBr(e,e,n)eeven}, we get $T_2 = T_1 T_0 T_1$. We also have $T_3 = T_2 T_1 T_2$. After replacing $T_2$ by $T_1T_0T_1$, we get $T_3 = T_1T_0T_1 T_1 T_1T_0T_1$. Using that $T_1^2 = 1$, we get $T_3 = T_1T_0T_1T_0T_1$. Let $k > 3$. Inductively, we get $T_{k+1} = T_kt_{k-1}T_k = \underset{2k-1}{\underbrace{T_1T_0 \cdots T_1}}\ \underset{2(k-1)-1}{\underbrace{T_1T_0 \cdots T_1}}\ \underset{2k-1}{\underbrace{T_1T_0 \cdots T_1}} = T_1T_0 \underset{2k-1}{\underbrace{T_1T_0 \cdots T_1}} = \underset{2k+1}{\underbrace{T_1T_0 \cdots T_1}}$. Hence for all $1 \leq k \leq e-1$, we have $T_k = \underset{2k-1}{\underbrace{T_1T_0 \cdots T_1}}$. It is clear that this expression is equal to $\underset{k-1}{\underbrace{T_1T_0 \cdots T_0}}T_1\underset{k-1}{\underbrace{T_0T_1 \cdots T_1}}$ if k is odd and $\underset{k-1}{\underbrace{T_1T_0 \cdots T_1}}T_0\underset{k-1}{\underbrace{T_1T_0 \cdots T_1}}$ if k is even.\\

Let us prove the expression for $E_k$ for $0 \leq k \leq e-1$. For $k = 0$, it corresponds to the convention we have set. The expression is obvious for $k = 1$. From Relation 2 of Definitions \ref{DefinitionBr(e,e,n)eodd} and \ref{DefinitionBr(e,e,n)eeven}, we have $E_2 = T_1E_0T_1$. We also have $E_3 = T_2 E_1 T_2$. After replacing $T_2$ by $T_1T_0T_1$, we get $E_3 = T_1T_0T_1 E_1 T_1T_0T_1$. Using the fact that $T_1E_1 = E_1T_1 = E_1$, we finally get $E_3 = T_1T_0 E_1 T_0T_1$. Let $k > 3$ and $k$ odd. Inductively, we get $E_{k+1} = T_kE_{k-1}T_k =$\\
$\begin{array}{l}
\underset{2k-1}{\underbrace{T_1T_0 \cdots T_1}}\ \underset{k-2}{\underbrace{T_1T_0 \cdots T_1}}\ E_0\ \underset{k-2}{\underbrace{T_1T_0 \cdots T_1}}\ \underset{2k-1}{\underbrace{T_1T_0 \cdots T_1}} =\\
\underset{k+1}{\underbrace{T_1T_0 \cdots T_0}}\ E_0\ \underset{k+1}{\underbrace{T_0T_1 \cdots T_1}} =\\
\underset{k}{\underbrace{T_1T_0 \cdots T_1}}\ E_0\ \underset{k}{\underbrace{T_1T_0 \cdots T_1}}.
\end{array}$\\
Similarly, if $k$ is even, we get $E_{k+1} = \underset{k}{\underbrace{T_1T_0 \cdots T_0}}\ E_1\ \underset{k}{\underbrace{T_0T_1 \cdots T_1}}$.

\end{proof}

In the following lemma, we extend the result of the previous lemma to all $k \in \mathbb{N}$.

\begin{lemma}\label{LemmaExtendTkEk}

Let $k \in \mathbb{N}$. In \emph{Br}$(e,e,3)$, we have
\begin{itemize}
\item $T_k = \underset{2k-1}{\underbrace{T_1T_0 \cdots T_1}}$ and
\item $E_k = \left\{
    \begin{array}{ll}
     \underset{k-1}{\underbrace{T_1T_0 \cdots T_0}}E_1\underset{k-1}{\underbrace{T_0T_1 \cdots T_1}}  & if\ k\ is\ odd, \\
     \underset{k-1}{\underbrace{T_1T_0 \cdots T_1}}E_0\underset{k-1}{\underbrace{T_1T_0 \cdots T_1}}  & if\ k\ is\ even.
    \end{array}
\right.$
\end{itemize}

\end{lemma}

\begin{proof}

Let $k \in \mathbb{N}$. We have $k = k'+xe$ for $0 \leq k' \leq e-1$ and $x \in \mathbb{N}$. Let $k' \neq 0$.\\
In Br$(e,e,3)$, we have $T_k = T_{k'+xe} = T_{k'}$. We show that $T_{k'} = \underset{2k-1}{\underbrace{T_1T_0 \cdots T_1}}$. Actually, $\underset{2k-1}{\underbrace{T_1T_0 \cdots T_1}} = \underset{2(k'+xe)-1}{\underbrace{T_1T_0 \cdots T_1}} = \underset{2k'-1}{\underbrace{T_1T_0 \cdots T_1}}\ \underset{2xe}{\underbrace{T_0T_1 \cdots T_0}}$. It is clear that $\underset{2xe}{\underbrace{T_0T_1 \cdots T_0}}$ is equal to $1$. Hence $\underset{2k-1}{\underbrace{T_1T_0 \cdots T_1}} = \underset{2k'-1}{\underbrace{T_1T_0 \cdots T_1}} = T_{k'}$. It follows that for all $k \in \mathbb{N}$, we have $T_k = \underset{2k-1}{\underbrace{T_1T_0 \cdots T_1}}$. For $k' = 0$, it is easy to check that $T_k = \underset{2xe-1}{\underbrace{T_1T_0 \cdots T_1}} = T_0$.

Now, we prove the second item of the lemma. Assume that $e$ is odd. We have to distinguish two different cases: $k$ odd or $k$ even. We provide the proof for the case $k$ odd. The proof when $k$ is even is similar and left to the reader. Suppose $k$ is odd with $k = k' +xe$ for $0 \leq k' \leq e-1$ and $x \in \mathbb{N}$. We prove that $E_k = E_{k'} = \underset{k-1}{\underbrace{T_1T_0 \cdots T_0}}E_1\underset{k-1}{\underbrace{T_0T_1 \cdots T_1}}$.

First consider $x$ even. Then $k'$ is odd. We have
$$\begin{array}{l}
\underset{k-1}{\underbrace{T_1T_0 \cdots T_0}}E_1\underset{k-1}{\underbrace{T_0T_1 \cdots T_1}} = 
\underset{k'+xe-1}{\underbrace{T_1T_0 \cdots T_0}}E_1\underset{k'+xe-1}{\underbrace{T_0T_1 \cdots T_1}} = \\
\underset{k'-1}{\underbrace{T_1T_0 \cdots T_0}}\ \underset{xe}{\underbrace{T_1T_0 \cdots T_0}}\ E_1\ \underset{xe}{\underbrace{T_0T_1 \cdots T_1}}\ \underset{k'-1}{\underbrace{T_0T_1 \cdots T_1}}.
\end{array}$$
By Relation 8 of Definition \ref{DefinitionBr(e,e,n)eodd}, we have $\underset{e}{\underbrace{T_0T_1 \cdots T_0}}\ E_1 = T_0E_0\ \underset{e-1}{\underbrace{T_1T_0 \cdots T_0}}$. Replacing $T_0E_0$ by $E_0$, we get $\underset{e}{\underbrace{T_0T_1 \cdots T_0}}\ E_1 = E_0\ \underset{e-1}{\underbrace{T_1T_0 \cdots T_0}}$. Similarly, we have $\underset{e}{\underbrace{T_1T_0 \cdots T_1}}\ E_0 = T_1\ \underset{e-1}{\underbrace{T_0T_1 \cdots T_1}}\ E_0 = T_1E_1\ \underset{e-1}{\underbrace{T_0T_1 \cdots T_1}} = E_1\ \underset{e-1}{\underbrace{T_0T_1 \cdots T_1}}$. Thus, we get $\underset{2e}{\underbrace{T_1T_0 \cdots T_0}}\ E_1 = \underset{e}{\underbrace{T_1T_0 \cdots T_1}}\ \underset{e}{\underbrace{T_0T_1 \cdots T_0}}\ E_1 = \underset{e}{\underbrace{T_1T_0 \cdots T_1}}\ E_0\ \underset{e-1}{\underbrace{T_1T_0 \cdots T_0}} = E_1\ \underset{e-1}{\underbrace{T_0T_1 \cdots T_1}}\ \underset{e-1}{\underbrace{T_1T_0 \cdots T_0}} = E_1$. Since $x$ is even, it follows that $\underset{xe}{\underbrace{T_1T_0 \cdots T_0}}\ E_1 = E_1$. Similarly, one can check that $E_1\ \underset{xe}{\underbrace{T_0T_1 \cdots T_1}}= E_1$. It follows that
$E_k = \underset{k-1}{\underbrace{T_1T_0 \cdots T_0}}E_1\underset{k-1}{\underbrace{T_0T_1 \cdots T_1}} = \underset{k'-1}{\underbrace{T_1T_0 \cdots T_0}}\ E_1\ \underset{k'-1}{\underbrace{T_0T_1 \cdots T_1}} = E_{k'}$.

Consider $x$ odd. Then $k'$ is even. Using the same technique, we also get $E_k = \underset{k-1}{\underbrace{T_1T_0 \cdots T_0}}E_1\underset{k-1}{\underbrace{T_0T_1 \cdots T_1}} = E_{k'}$.\\

Assume that $e$ is even. As for the case $e$ odd, we have to distinguish two different cases: $k$ odd or $k$ even. We provide the proof for the case $k$ odd. The proof when $k$ is even is similar and left to the reader. Suppose $k$ is odd with $k = k' +xe$ for $0 \leq k' \leq e-1$ and $x \in \mathbb{N}$. We prove that $E_k = E_{k'} = \underset{k-1}{\underbrace{T_1T_0 \cdots T_0}}E_1\underset{k-1}{\underbrace{T_0T_1 \cdots T_1}}$. Since $e$ is even and $k$ is assumed to be odd, we have $k'$ is odd for all $x \in \mathbb{N}$. We have
$$\begin{array}{l}
\underset{k-1}{\underbrace{T_1T_0 \cdots T_0}}E_1\underset{k-1}{\underbrace{T_0T_1 \cdots T_1}} = 
\underset{k'+xe-1}{\underbrace{T_1T_0 \cdots T_0}}E_1\underset{k'+xe-1}{\underbrace{T_0T_1 \cdots T_1}} = \\
\underset{k'-1}{\underbrace{T_1T_0 \cdots T_0}}\ \underset{xe}{\underbrace{T_1T_0 \cdots T_0}}\ E_1\ \underset{xe}{\underbrace{T_0T_1 \cdots T_1}}\ \underset{k'-1}{\underbrace{T_0T_1 \cdots T_1}}.
\end{array}$$
We have $\underset{e}{\underbrace{T_1T_0 \cdots T_0}}\ E_1 = \underset{e}{\underbrace{T_0T_1 \cdots T_1}}\ E_1 = \underset{e-1}{\underbrace{T_0T_1 \cdots T_0}}\ E_1$ that is equal to $E_1$ by Relation 9 of Definition \ref{DefinitionBr(e,e,n)eeven}. Similarly, one gets $\underset{xe}{\underbrace{T_1T_0 \cdots T_0}}\ E_1 = E_1$ and $E_1\ \underset{xe}{\underbrace{T_0T_1 \cdots T_1}} = E_1$. It follows that $E_k = \underset{k-1}{\underbrace{T_1T_0 \cdots T_0}}E_1\underset{k-1}{\underbrace{T_0T_1 \cdots T_1}}  = \underset{k'-1}{\underbrace{T_1T_0 \cdots T_0}}\ E_1\ \underset{k'-1}{\underbrace{T_0T_1 \cdots T_1}} = E_{k'}$.

\end{proof}

In our study of the BMW and Brauer algebras, we restrict from now on to the case $n = 3$ so that the computations (by hand or by using a computer) are feasible.\\

\section{Isomorphism with the Brauer-Chen algebra}

Recall that Chen defined in \cite{ChenFlatConnectionsBrauerAlgebras} a Brauer algebra for all complex reflection groups and found a canonical presentation for this algebra in the case of Coxeter groups and complex reflection groups of type $G(d,1,n)$. For the symmetric group, it coincides with the classical Brauer algebra, see \cite{Brauer}. For simply laced ($m_{st} \in \{1,2,3\}$ in Definition \ref{DefinitionCoxeterGroups}) and finite Coxeter groups, it coincides with the definition of the simply laced Brauer algebras described by Cohen, Frenk, and Wales in \cite{CohenFrenkWales}. For the case of $G(d,1,n)$, the cyclotomic Brauer algebra introduced by Häring-Oldenburg in \cite{HaringOldenburg} appears as a component of Chen's algebra.\\

We start by recalling in Definition \ref{DefinitionChenBrauer(e,e,3)} the Brauer-Chen algebra $\mathcal{B}(e,e,3)$ associated to the complex reflection group $G(e,e,3)$ for all $e \geq 1$. Next, we prove that our definition of the Brauer algebra Br$(e,e,n)$, see Definition \ref{DefinitionBr(e,e,n)eodd}, coincides with the definition of the Brauer-Chen algebra when $n = 3$ and $e$ odd. Thus, it provides a canonical presentation for $\mathcal{B}(e,e,3)$ when $e$ is odd. At the end of this section, we ask the question whether Br$(e,e,n)$ is isomorphic to the Brauer-Chen algebra $\mathcal{B}(e,e,n)$ for all $e$ and $n$.\\

Denote by $s_{i,j;k}$ the reflection of $G(e,e,3)$ of hyperplane $H_{i,j;k}$: $z_i = \zeta_e^{-k} z_j$ for $1 \leq i < j \leq n$ and $0 \leq k \leq e-1$. We denote $s_{i,j;k}$ by $s$ and $s_{i',j';k'}$ by $s'$ when there is no confusion. The corresponding hyperplanes can also be denoted by $H_s$ and $H_{s'}$, respectively. Let $\mathcal{R}$ be the set of reflections of $G(e,e,3)$.

\begin{definition}\label{DefinitionNonCrossingEdge}

Let $s_{i,j;k}$ and $s_{i',j';k'}$ be two reflections of $G(e,e,3)$ and $L$ the codimension $2$ edge  $H_{i,j;k} \cap H_{i',j';k'}$. We say that $L$ is a crossing edge if\\ $\{ s \in \mathcal{R}\ |\ L \subset H_s \} = \{ s_{i,j;k},s_{i',j';k'} \}$. Otherwise, $L$ is called a noncrossing edge.

\end{definition}

\begin{proposition}\label{PropositionNoCrossingEdges}

Let $s_{i,j;k}$ and $s_{i',j';k'}$ be two different reflections, we have\\ \mbox{$H_{i,j;k} \cap H_{i',j';k'}$} is a noncrossing edge. Hence there is no crossing edges in $G(e,e,3)$.

\end{proposition}

\begin{proof}

Let $s_{i,j;k}$ and $s_{i',j';k'}$ be two different reflections. If $i = i'$ and $j = j'$, it is easy to check that $H_{i,j;k} \cap H_{i',j';k'}$ is a noncrossing edge. Suppose $i \neq i'$.\\
\emph{Case 1}: $i = 1$ and $i' = 2$. Note that $i'$ cannot be equal to $3$ since $i' < j'$. We have two possibilities depending on the values of $j$ and $j'$:
\begin{itemize}
\item $s_{1,2;k}$ and $s_{2,3;k'}$,
\item $s_{1,3;k}$ and $s_{2,3;k'}$, for all $0 \leq k,k' \leq e-1$.
\end{itemize}
For the first case, we have $H_{1,2;k} \cap H_{1,3;k'} \subset H_{2,3;k'-k}$. Thus, $H_{1,2;k} \cap H_{1,3;k'}$ is a noncrossing edge. For the second case, we have $H_{1,3;k} \cap H_{2,3;k'} \subset H_{1,2;k-k'}$. Thus, $H_{1,3;k} \cap H_{2,3;k'}$ is a noncrossing edge.\\
\emph{Case 2}: $i = 2$ and $i' = 1$. This is done in \emph{Case 1}.\\
Since $i < j$, we note that $i \neq 3$ and then all the cases are done. Hence there is no crossing edges in $G(e,e,3)$. 

\end{proof}

Let $s,s'$ be two different reflections of $G(e,e,3)$. We define $$R(s,s') = \{ s'' \in \mathcal{R}\ |\ s'' s s'' = s' \}$$ to be the set of reflections that conjugate $s$ into $s'$. We are able to provide the definition of Chen of the Brauer algebra that we denote by $\mathcal{B}(e,e,3)$,\mbox{ see Definition 1.1 of \cite{ChenFlatConnectionsBrauerAlgebras}.}

\begin{definition}\label{DefinitionChenBrauer(e,e,3)}

The Brauer-Chen algebra $\mathcal{B}(e,e,3)$ is the unitary associative $\mathbb{Q}(x)$-algebra generated by $\{T_w\ |\ w \in G(e,e,3)\} \cup \{e_{i,j;k}\ |\ 1 \leq i < j \leq n, 0 \leq k \leq e-1\}$ with the relations:
\begin{enumerate}
\item $T_{w}T_{w'} = T_{w''}$ if $ww'=w''$ for $w$, $w'$, $w'' \in G(e,e,3)$,
\item $T_{s} e_{i,j;k} = e_{i,j;k} T_{s} = e_{i,j;k}$, where $s = s_{i,j;k}$,
\item ${(e_{i,j;k})}^{2} = x e_{i,j;k}$,
\item $T_{w}e_{i,j;k} = e_{i',j';k'} T_{w}$ if $w s w^{-1} = s'$, where $s = s_{i,j;k}$, $s' = s_{i',j';k'}$, $\{s , s'\} \subset \mathcal{R}$,
\item \mbox{$e_{i,j;k} e_{i',j';k'} = T_r e_{i',j';k'} = e_{i,j;k} T_r$,}\linebreak
for $r \in R(s,s')$, where $s = s_{i,j;k}$, $s' = s_{i',j';k'}$, and $\{s , s'\} \subset \mathcal{R}$,
\item $e_{i,j;k}e_{i',j';k'} = 0$ when $R(s,s') = \emptyset$ with $s = s_{i,j;k}$ and $s' = s_{i',j';k'}$.

\end{enumerate}

\noindent When there is no confusion,  one can replace $T_w$ by $w$ for $w$ in $G(e,e,3)$.

\end{definition}

From now on, we restrict to the case when $e$ is odd. Our aim is to prove that the Brauer algebra Br$(e,e,3)$ is isomorphic to the Brauer-Chen algebra $\mathcal{B}(e,e,3)$ for $e$ odd.\\

We start by explicitly writing all the cases of Relations $4$, $5$, and $6$ of Definition \ref{DefinitionChenBrauer(e,e,3)}. For this purpose, we distinguish 6 cases depending on the choice of $s$ and $s'$ in $\mathcal{R}$ and determine for each case all the elements $w \in G(e,e,3)$ such that $w s w^{-1} = s'$.\\

Every element $w \in G(e,e,3)$ will be written as a product $x_1x_2 \cdots x_r$, $x_i \in \{t_0,t_1, \cdots, t_{e-1},s_3\}$ with $R\!E(w) = \mathbf{x}_1\mathbf{x}_2 \cdots \mathbf{x}_r$, where $R\!E(w)$ is the reduced word representative for $w$ over the generating set $\mathbf{X}$ of the presentation of Corran and Picantin of $G(e,e,3)$ introduced and studied in Chapter 2.\\

Note that every time we have $w s w^{-1} = s'$ for $s,s' \in \mathcal{R}$ and $w \in G(e,e,3)$, we get $w^{-1} s' w = s$. Hence every time we have a relation $T_{w}e_{i,j;k} = e_{i',j';k'} T_{w}$, we also have $T_{w^{-1}} e_{i',j';k'} = e_{i,j;k}T_{w^{-1}}$. If $w:=r \in \mathcal{R}$, every time we have a relation $e_{i,j;k} e_{i',j';k'} = T_r e_{i',j';k'} = e_{i,j;k} T_r$, we also have $e_{i',j';k'} e_{i,j;k} = T_r e_{i,j;k} = e_{i',j';k'} T_r$.

\subsection*{Case 1: $s = s_{1,2;k}$ and $s' = s_{1,2;k'}$}

Let $w = \begin{pmatrix}
c_1 & c_2 & c_3\\
c_4 & c_5 & c_6\\
c_7 & c_8 & c_9
\end{pmatrix}$ be an element of $G(e,e,3)$. We determine the solutions of the equation $w s_{1,2;k} w^{-1} = s_{1,2;k'}$. It is easy to check that we get $c_9 \neq 0$. Thus, we necessarily have $c_3 = c_6 = c_7 = c_8 = 0$ since $w$ is monomial. Set $c_9 = \zeta_e^{l}$ with $0 \leq l \leq e-1$. We have two different cases.\\

The first case is when $c_2$ is nonzero and $c_4 = \zeta_e^{k+k'} c_2$ for which case $w = \begin{pmatrix}
0 & c_2 & 0\\
\zeta_e^{k+k'}c_2 & 0 & 0\\
0 & 0 & \zeta_e^l
\end{pmatrix}$ with the product of the nonzero entries being $1$, that is\\ $c_2^2 = \zeta_e^{-(k+k'+l)}$. We get two cases depending on the parity of $k+k'+l$.\\

\noindent \underline{If $k+k'+l$ is even}, then $c_2 = \zeta_e^{-(k+k'+l)/2}$ and $w$ is equal to\\ $w = \begin{pmatrix}
0 & \zeta_e^{-(k+k'+l)/2} & 0\\
\zeta_e^{(k+k'-l)/2} & 0 & 0\\
0 & 0 & \zeta_e^l
\end{pmatrix}$.\\

\begin{itemize}

\item If $l = 0$, then $w \in \mathcal{R}$ and $w = t_{\frac{k+k'}{2}}$. Relation 5 of Definition \ref{DefinitionChenBrauer(e,e,3)} is:

$$e_{1,2;k}e_{1,2;k'} = t_{\frac{k+k'}{2}} e_{1,2;k'} = e_{1,2;k} t_{\frac{k+k'}{2}}.$$

We refer to this case as \underline{Case (1.1)}.

\item If $l \neq 0$, then $w \notin \mathcal{R}$ and $w = t_xs_3t_lt_0s_3$ with $x = \frac{k+k'+l}{2}$. Relation 4 of Definition \ref{DefinitionChenBrauer(e,e,3)} is:

\begin{center} $t_xs_3t_lt_0s_3e_{1,2;k} = e_{1,2;k'} t_xs_3t_lt_0s_3$, with $x = \frac{k+k'+l}{2}$. \end{center}

We refer to this case as \underline{Case (1.2)}.

\end{itemize}

\noindent \underline{If $k+k'+l$ is odd}, then $k+k'+l+e$ is even and we have $c_2^2 = \zeta_e^{-(k+k'+l+e)}$. Hence $c_2 = \zeta_e^{-\frac{k+k'+l+e}{2}}$ and $w = \begin{pmatrix}
0 & \zeta_e^{-\frac{k+k'+l+e}{2}} & 0\\
\zeta_e^{\frac{k+k'-l-e}{2}} & 0 & 0\\
0 & 0 & \zeta_e^l
\end{pmatrix}$.

\begin{itemize}

\item If $l = 0$, then $w \in \mathcal{R}$ and $w = t_x$ with $x = \frac{k+k'+e}{2}$. Relation 5 of Definition \ref{DefinitionChenBrauer(e,e,3)} is:

\begin{center} $e_{1,2;k}e_{1,2;k'} = t_{x} e_{1,2;k'} = e_{1,2;k} t_{x}$, with $x = \frac{k+k'+e}{2}$. \end{center}

We refer to this case as \underline{Case (1.3)}.

\item If $l \neq 0$, then $w \notin \mathcal{R}$ and $w = t_xs_3t_lt_0s_3$ with $x = \frac{k+k'+l-e}{2}$. Relation 4 of Definition \ref{DefinitionChenBrauer(e,e,3)} is:

\begin{center} $t_xs_3t_lt_0s_3e_{1,2;k} = e_{1,2;k'} t_xs_3t_lt_0s_3$, with $x = \frac{k+k'+l-e}{2}$. \end{center}

We refer to this case as \underline{Case (1.4)}.

\end{itemize}

The second case is when $c_1$ is nonzero and $c_5 = \zeta_e^{k'-k}c_1$ for which case $w = \begin{pmatrix}
c_1 & 0 & 0\\
0 & \zeta_e^{k'-k}c_1 & 0\\
0 & 0 & \zeta_e^{l}
\end{pmatrix}$ with $c_1^2 = \zeta_e^{k-k'-l}$. Similarly, we get two cases depending on the parity of $k-k'-l$.\\

\noindent \underline{If $k-k'-l$ is even}, then $c_1 = \zeta_e^{\frac{k-k'-l}{2}}$ and $w = \begin{pmatrix}
\zeta_e^{\frac{k-k'-l}{2}} & 0 & 0\\
0 & \zeta_e^{\frac{k'-k-l}{2}} & 0\\
0 & 0 & \zeta_e^l
\end{pmatrix}$. We have $w = t_xt_0s_3t_lt_0s_3$ with $x = \frac{k'-k-l}{2}$. Relation 4 of Definition \ref{DefinitionChenBrauer(e,e,3)} is:

\begin{center} $t_xt_0s_3t_lt_0s_3e_{1,2;k} = e_{1,2;k'} t_xt_0s_3t_lt_0s_3$, with $x = \frac{k'-k-l}{2}$. \end{center}

We refer to this case as \underline{Case (1.5)}.   

\noindent \underline{If $k-k'-l$ is odd}, then $k-k'-l+e$ is even and $c_1^2 = \zeta_e^{k-k'-l+e}$. Hence $c_1 = \zeta_e^{\frac{k-k'-l+e}{2}}$ and $w = \begin{pmatrix}
\zeta_e^{\frac{k-k'-l+e}{2}} & 0 & 0\\
0 & \zeta_e^{\frac{k'-k-l+e}{2}} & 0\\
0 & 0 & \zeta_e^l
\end{pmatrix}$. We have $w = t_xt_0s_3t_lt_0s_3$ with $x = \frac{k'-k-l+e}{2}$. Relation 4 of Definition \ref{DefinitionChenBrauer(e,e,3)} is:

\begin{center} $t_xt_0s_3t_lt_0s_3e_{1,2;k} = e_{1,2;k'} t_xt_0s_3t_lt_0s_3$, with $x = \frac{k'-k-l+e}{2}$. \end{center}

We refer to this case as \underline{Case (1.6)}. 

\subsection*{Case 2: $s = s_{1,2;k}$ and $s' = s_{2,3;k'}$}

We have $s = t_k$ and $s' = t_0t_{k'}s_3t_{k'}t_0$ with $0 \leq k,k' \leq e-1$. 
Let $w = \begin{pmatrix}
c_1 & c_2 & c_3\\
c_4 & c_5 & c_6\\
c_7 & c_8 & c_9
\end{pmatrix}$ be an element of $G(e,e,3)$. We determine the solutions of the equation $w t_k w^{-1} = s_{2,3;k'}$. It is easy to check that we get $c_3 \neq 0$. Thus, we necessarily have $c_1 = c_2 = c_6 = c_9 = 0$ since $w$ is monomial. Set $c_3 = \zeta_e^{l}$ with $0 \leq l \leq e-1$. We have two different cases.\\

The first case is when $c_4$ is nonzero and $c_8 = \zeta_e^{k'-k}c_4$ with $c_4^2 = \zeta_e^{k-k'-l}$.\\

\noindent \underline{If $k-k'-l$ is even}, we have $c_4 = \zeta_e^{\frac{k-k'-l}{2}}$ and in this case we have\\ $w = \begin{pmatrix}
0 & 0 & \zeta_e^l\\
\zeta_e^{\frac{k-k'-l}{2}} & 0 & 0\\
0 & \zeta_e^{\frac{k'-k-l}{2}} & 0\\
\end{pmatrix}$ with $w = t_{-l}s_3t_{\frac{k'-k-l}{2}}t_0$. Relation 4 of Definition \ref{DefinitionChenBrauer(e,e,3)} is:

\begin{center} $t_{-l}s_3t_{\frac{k'-k-l}{2}}t_0 e_{1,2;k} = e_{2,3;k'} t_{-l}s_3t_{\frac{k'-k-l}{2}}t_0$. \end{center}

We refer to this case as \underline{Case (2.1)}.

\noindent \underline{If $k-k'-l$ is odd}, we have $c_4 = \zeta_e^{\frac{k-k'-l+e}{2}}$ and $w = \begin{pmatrix}
0 & 0 & \zeta_e^l\\
\zeta_e^{\frac{k-k'-l+e}{2}} & 0 & 0\\
0 & \zeta_e^{\frac{k'-k-l+e}{2}} & 0
\end{pmatrix}$ with $w = t_{-l}s_3t_{\frac{k'-k-l+e}{2}}t_0$. Relation 4 of Definition \ref{DefinitionChenBrauer(e,e,3)} is:

\begin{center} $t_{-l}s_3t_{\frac{k'-k-l+e}{2}}t_0 e_{1,2;k} = e_{2,3;k'} t_{-l}s_3t_{\frac{k'-k-l+e}{2}}t_0$. \end{center}

We refer to this case as \underline{Case (2.2)}.\\

\mbox{The second case is when $c_5$ is nonzero and $c_7 = \zeta_e^{k+k'} c_5$. Then we have $c_5^2 = \zeta_e^{-(k+k'+l)}$.}\\

\noindent \underline{If $k+k'+l$ is even}, we get $c_5 = \zeta_e^{\frac{-(k+k'+l)}{2}}$ and $c_7 = \zeta_e^{\frac{k+k'-l}{2}}$.\\ We have $w = \begin{pmatrix}
0 & 0 & \zeta_e^l\\
0 & \zeta_e^{\frac{-(k+k'+l)}{2}} & 0\\
\zeta_e^{\frac{(k+k'-l)}{2}} & 0 & 0
\end{pmatrix}$ with $w = t_{-l}s_3t_{\frac{k+k'-l}{2}}$.

\begin{itemize}

\item If $l \equiv -k-k'$ modulo $e$, then $w \in \mathcal{R}$ and $w = t_{k+k'}s_3t_{k+k'}$. Relation 5 of Definition \ref{DefinitionChenBrauer(e,e,3)} is:

\begin{center} $e_{1,2;k}e_{2,3;k'} = t_{k+k'}s_3t_{k+k'} e_{2,3;k'} = e_{1,2;k} t_{k+k'}s_3t_{k+k'}$. \end{center}

We refer to this case as \underline{Case (2.3)}.

\item If $l \not\equiv -k-k'$ modulo $e$, then $w \notin \mathcal{R}$. Relation 4 of Definition \ref{DefinitionChenBrauer(e,e,3)} is:

\begin{center} $t_{-l}s_3t_{\frac{k+k'-l}{2}} e_{1,2;k} = e_{2,3;k'} t_{-l}s_3t_{\frac{k+k'-l}{2}}$. \end{center}

We refer to this case as \underline{Case (2.4)}.

\end{itemize}

\noindent \underline{If $k+k'+l$ is odd}, we have $c_5 = \zeta_e^{\frac{-k-k'-l+e}{2}}$ and $c_7 = \zeta_e^{\frac{k+k'-l+e}{2}}$ for which case\\ $w = \begin{pmatrix}
0 & 0 & \zeta_e^l\\
0 & \zeta_e^{\frac{-k-k'-l+e}{2}} & 0\\
\zeta_e^{\frac{k+k'-l+e}{2}} & 0 & 0
\end{pmatrix}$ with $w = t_{-l}s_3t_{\frac{k+k'-l+e}{2}}$. Relation 4 of Definition \ref{DefinitionChenBrauer(e,e,3)} is:

\begin{center} $t_{-l}s_3t_{\frac{k+k'-l+e}{2}} e_{1,2;k} = e_{2,3;k'} t_{-l}s_3t_{\frac{k+k'-l+e}{2}}$. \end{center}

We refer to this case as \underline{Case (2.5)}.
Note that if $l \equiv -k-k'$ modulo $e$ in this case, we get $w \in \mathcal{R}$ and the corresponding relation in $\mathcal{B}(e,e,3)$ is the same as in Case (2.3).

\subsection*{Case 3: $s = s_{1,2;k}$ and $s' = s_{1,3;k'}$}

We have $s = t_k$ and $s' = s_{1,3;k'} = t_{k'}s_3t_{k'}$. Let $w = \begin{pmatrix}
c_1 & c_2 & c_3\\
c_4 & c_5 & c_6\\
c_7 & c_8 & c_9
\end{pmatrix}$ be an element of $G(e,e,3)$. We determine the solutions of the equation $w s_{1,2;k} w^{-1} = s_{1,3;k'}$. We get $c_6$ is nonzero. Set $c_6 = \zeta_e^l$ with $0 \leq l \leq e-1$. It follows that $c_4 = c_5  = c_3 = c_9 = 0$ since $w$ is monomial. We get two cases.\\

The first one is when $c_2$ is nonzero and $c_7 = \zeta_e^{k+k'} c_2$ for which case we get $w = \begin{pmatrix}
0 & c_2 & 0\\
0 & 0 & \zeta_e^l\\
\zeta_e^{k+k'}c_2 & 0 & 0
\end{pmatrix}$ with $c_2^2 = \zeta_e^{-(k+k'+l)}$.\\

\noindent \underline{If $k+k'+l$ is even}, then $c_2 = \zeta_e^{\frac{-(k+k'+l)}{2}}$. We get $w = \begin{pmatrix}
0 & \zeta_e^{\frac{-(k+k'+l)}{2}} & 0\\
0 & 0 & \zeta_e^l\\
\zeta_e^{\frac{k+k'-l}{2}} & 0 & 0
\end{pmatrix}$ with $w = t_lt_0s_3t_{\frac{k+k'-l}{2}}$. Relation 4 of Definition \ref{DefinitionChenBrauer(e,e,3)} is:

$$t_lt_0s_3t_{\frac{k+k'-l}{2}} e_{1,2;k} = e_{1,3;k'} t_lt_0s_3t_{\frac{k+k'-l}{2}}.$$

We refer to this case as \underline{Case (3.1)}.

\noindent \underline{If $k+k'+l$ is odd}, then $k+k'+l+e$ is even and $c_2 = \zeta_e^{\frac{-k-k'-l+e}{2}}$. We get $w =  \begin{pmatrix}
0 & \zeta_e^{\frac{-k-k'-l+e}{2}} & 0\\
0 & 0 & \zeta_e^l\\
\zeta_e^{\frac{k+k'-l+e}{2}} & 0 & 0
\end{pmatrix}$ with $w = t_lt_0s_3t_{\frac{k+k'-l+e}{2}}$. \mbox{Relation 4 of Definition \ref{DefinitionChenBrauer(e,e,3)} is:}

$$t_lt_0s_3t_{\frac{k+k'-l+e}{2}} e_{1,2;k} = e_{1,3;k'} t_lt_0s_3t_{\frac{k+k'-l+e}{2}}.$$

We refer to this case as \underline{Case (3.2)}.\\

The second case is when $c_1$ is generic and $c_8 = \zeta_e^{k'-k} c_1$ for which case $w = \begin{pmatrix}
c_1 & 0 & 0\\
0 & 0 & \zeta_e^l\\
0 & \zeta_e^{k'-k}c_1 & 0
\end{pmatrix}$ with $c_1^2 = \zeta_e^{k-k'-l}$.\\

\noindent \underline{If $k-k'-l$ is even}, then $c_1 = \zeta_e^{\frac{k-k'-l}{2}}$ and $c_8 = \zeta_e^{\frac{k'-k-l}{2}}$. We get\\ $w = \begin{pmatrix}
\zeta_e^{\frac{k-k'-l}{2}} & 0 & 0\\
0 & 0 & \zeta_e^l\\
0 &  \zeta_e^{\frac{k'-k-l}{2}} & 0
\end{pmatrix}$.

\begin{itemize}

\item If $l \equiv k-k'$ modulo $e$, then $w \in \mathcal{R}$ and $w = t_{k-k'}t_0s_3t_{k'-k}t_0$. Relation 5 of Definition \ref{DefinitionChenBrauer(e,e,3)} is: 

$$e_{1,2;k}e_{1,3;k'} =e_{1,2;k} t_{k-k'}t_0s_3t_{k'-k}t_0 = t_{k-k'}t_0 s_3 t_{k'-k}t_0 e_{1,3;k'}.$$

We refer to this case as Case (3.3).

\item  If $l \not\equiv k-k'$ modulo $e$, then $w \notin \mathcal{R}$. We have $w = t_lt_0s_3t_{\frac{k'-k-l}{2}}$. Relation 4 of Definition \ref{DefinitionChenBrauer(e,e,3)} is:

$$t_lt_0s_3t_{\frac{k'-k-l}{2}}t_0 e_{1,2;k} = e_{1,3;k'} t_lt_0s_3t_{\frac{k'-k-l}{2}}t_0.$$

We refer to this case as \underline{Case (3.4)}.

\end{itemize}

\noindent \underline{If $k-k'-l$ is odd}, then $k-k'-l+e$ is even and $c_1 = \zeta_e^{\frac{k-k'-l+e}{2}}$. We get $w = \begin{pmatrix}
\zeta_e^{\frac{k-k'-l+e}{2}} & 0 & 0\\
0 & 0 & \zeta_e^l\\
0 &  \zeta_e^{\frac{k'-k-l+e}{2}} & 0
\end{pmatrix}$ with $w = t_lt_0s_3t_{\frac{k'-k-l+e}{2}}t_0$. \mbox{Relation 4 of Definition \ref{DefinitionChenBrauer(e,e,3)} is:}

$$t_lt_0s_3t_{\frac{k'-k-l+e}{2}}t_0 e_{1,2;k} = e_{1,3;k'} t_lt_0s_3t_{\frac{k'-k-l+e}{2}}t_0.$$

We refer to this case as \underline{Case (3.5)}. Note that if $l \equiv k-k'$ modulo $e$ in this case, we get $w \in \mathcal{R}$ and the corresponding relation in $\mathcal{B}(e,e,3)$ is the same as in Case (3.3).

\subsection*{Case 4: $s = s_{2,3;k}$ and $s' = s_{2,3;k'}$}

We have $s = s_{2,3;k} = t_0t_{k}s_3t_kt_0$ and $s' = s_{2,3;k'} = t_0t_{k'}s_3t_{k'}t_0$ with $0 \leq k,k' \leq e-1$ and $k \neq k'$. Let $w = \begin{pmatrix}
c_1 & c_2 & c_3\\
c_4 & c_5 & c_6\\
c_7 & c_8 & c_9
\end{pmatrix} \in G(e,e,3)$. We determine the solutions of the equation $w s = s' w$. We get $c_1$ is nonzero. Set $c_1 = \zeta_e^l$ with $0 \leq l \leq e-1$. Then we have $c_2 = c_3 = c_4 = c_7 = 0$ since $w$ is monomial. We get two different cases.\\

The first case is when $c_6$ is nonzero and $c_8 = \zeta_e^{k+k'}c_6$ for which case $w = \begin{pmatrix}
\zeta_e^l & 0 & 0\\
0 & 0 & c_6\\
0 & \zeta_e^{k+k'}c_6 & 0
\end{pmatrix}$. We also have $c_6^2 = \zeta_e^{-k-k'-l}$. We get two cases depending on the parity of $k+k'+l$.\\

\noindent \underline{If $k+k'+l$ is even}, then $c_6 = \zeta_e^{\frac{-k-k'-l}{2}}$ and $w = \begin{pmatrix}
\zeta_e^l & 0 & 0\\
0 & 0 &  \zeta_e^{\frac{-k-k'-l}{2}}\\
0 & \zeta_e^{\frac{k+k'-l}{2}} & 0
\end{pmatrix}$.

\begin{itemize}

\item If $l = 0$, then $w \in \mathcal{R}$ and $w = t_{\frac{-k-k'}{2}}t_0s_3t_{\frac{k+k'}{2}}t_0$. \mbox{Relation 5 of Definition \ref{DefinitionChenBrauer(e,e,3)} is:}

\begin{center} $e_{2,3;k}e_{2,3;k'} = t_{\frac{-k-k'}{2}}t_0s_3t_{\frac{k+k'}{2}}t_0 e_{2,3;k'} = e_{2,3;k} t_{\frac{-k-k'}{2}}t_0s_3t_{\frac{k+k'}{2}}t_0$. \end{center}

We refer to this case as \underline{Case (4.1)}.

\item If $l \neq 0$, then $w = t_{\frac{-k-k'-l}{2}}t_0s_3t_{\frac{k+k'-l}{2}}t_0 \notin \mathcal{R}$. \mbox{Relation 4 of Definition \ref{DefinitionChenBrauer(e,e,3)} is:}

$$ t_{\frac{-k-k'-l}{2}}t_0s_3t_{\frac{k+k'-l}{2}}t_0 e_{2,3;k} = e_{2,3;k'} t_{\frac{-k-k'-l}{2}}t_0s_3t_{\frac{k+k'-l}{2}}t_0.$$

We refer to this case as \underline{Case (4.2)}.

\end{itemize}

\noindent \underline{If $k+k'+l$ is odd}, then $c_6 = \zeta_e^{\frac{-k-k'-l+e}{2}}$ and $w = \begin{pmatrix}
\zeta_e^l & 0 & 0\\
0 & 0 &  \zeta_e^{\frac{-k-k'-l+e}{2}}\\
0 & \zeta_e^{\frac{k+k'-l+e}{2}} & 0
\end{pmatrix}$.

\begin{itemize}

\item If $l = 0$, then $w \in \mathcal{R}$ and  $w = t_{\frac{-k-k'+e}{2}}t_0s_3t_{\frac{k+k'+e}{2}}t_0$. \mbox{Relation 5 of Definition \ref{DefinitionChenBrauer(e,e,3)} is:}

\begin{center} $e_{2,3;k}e_{2,3;k'} =t_{\frac{-k-k'+e}{2}}t_0s_3t_{\frac{k+k'+e}{2}}t_0 e_{2,3;k'} = e_{2,3;k} t_{\frac{-k-k'+e}{2}}t_0s_3t_{\frac{k+k'+e}{2}}t_0$. \end{center}

We refer to this case as \underline{Case (4.3)}.

\item If $l \neq 0$, we have $w = t_{\frac{-k-k'-l+e}{2}}t_0s_3t_{\frac{k+k'-l+e}{2}}t_0 \notin \mathcal{R}$. \mbox{Relation 4 of Definition \ref{DefinitionChenBrauer(e,e,3)} is:}

$$t_{\frac{-k-k'-l+e}{2}}t_0s_3t_{\frac{k+k'-l+e}{2}}t_0 e_{2,3;k} = e_{2,3;k'} t_{\frac{-k-k'-l+e}{2}}t_0s_3t_{\frac{k+k'-l+e}{2}}t_0.$$

We refer to this case as \underline{Case (4.4)}.

\end{itemize}

The  second case is when $c_5$ is nonzero and $c_9 = \zeta_e^{k'-k} c_5$ for which case $w = \begin{pmatrix}
\zeta_e^l & 0 & 0\\
0 & c_5 & 0\\
0 & 0 & \zeta_e^{k'-k}c_5
\end{pmatrix}$. We also have $c_5^2 = \zeta_e^{k-k'-l}$.

\noindent \underline{If $k-k'-l$ is even}, then $c_5 = \zeta_e^{\frac{k-k'-l}{2}}$ and $w = \begin{pmatrix}
\zeta_e^l & 0 & 0\\
0 & \zeta_e^{\frac{k-k'-l}{2}} & 0\\
0 & 0 & \zeta_e^{\frac{k'-k-l}{2}}
\end{pmatrix}$ with $w = t_{\frac{k-k'-l}{2}}t_0s_3t_{\frac{k'-k-l}{2}}t_0s_3$. \mbox{Relation 4 of Definition \ref{DefinitionChenBrauer(e,e,3)} is:}

$$t_{\frac{k-k'-l}{2}}t_0s_3t_{\frac{k'-k-l}{2}}t_0s_3 e_{2,3;k} = e_{2,3;k'} t_{\frac{k-k'-l}{2}}t_0s_3t_{\frac{k'-k-l}{2}}t_0s_3.$$

We refer to this case as \underline{Case (4.5)}.

\noindent \underline{If $k-k'-l$ is odd}, then $c_5 = \zeta_e^{\frac{k-k'-l+e}{2}}$ and $w = \begin{pmatrix}
\zeta_e^l & 0 & 0\\
0 & \zeta_e^{\frac{k-k'-l+e}{2}} & 0\\
0 & 0 & \zeta_e^{\frac{k'-k-l+e}{2}}
\end{pmatrix}$ with $w = t_{\frac{k-k'-l+e}{2}}t_0s_3t_{\frac{k'-k-l+e}{2}}t_0s_3$. \mbox{Relation 4 of Definition \ref{DefinitionChenBrauer(e,e,3)} is:}

$$t_{\frac{k-k'-l+e}{2}}t_0s_3t_{\frac{k'-k-l+e}{2}}t_0s_3 e_{2,3;k} = e_{2,3;k'} t_{\frac{k-k'-l+e}{2}}t_0s_3t_{\frac{k'-k-l+e}{2}}t_0s_3.$$

We refer to this case as \underline{Case (4.6)}.

\subsection*{Case 5: $s = s_{2,3;k}$ and $s' = s_{1,3;k'}$}

We have $s = s_{2,3;k} = t_0t_{k}s_3t_kt_0$ and $s' = s_{1,3;k'} = t_{k'}s_3t_{k'}$ with $0 \leq k,k' \leq e-1$. Let $w =  \begin{pmatrix}
c_1 & c_2 & c_3\\
c_4 & c_5 & c_6\\
c_7 & c_8 & c_9
\end{pmatrix} \in G(e,e,3)$. We determine the solutions of the equation $w s = s' w$. We get $c_4$ is generic nonzero. Set $c_4 = \zeta_e^l$ with $0 \leq l \leq e-1$. Then we have $c_1 = c_7 = c_5 = c_6 = 0$ since $w$ is monomial. We get two different cases.\\

The first case is when $c_3$ is nonzero and $c_8 = \zeta_e^{k+k'}c_3$ for which case $w =  \begin{pmatrix}
0 & 0 & c_3\\
\zeta_e^l & 0 & 0\\
0 & \zeta_e^{k+k'}c_3 & 0
\end{pmatrix}$ with $c_3^2 = \zeta_e^{-k-k'-l}$.\\

\noindent \underline{If $k+k'+l$ is even}, then $c_3 = \zeta_e^{\frac{-k-k'-l}{2}}$ and $w = \begin{pmatrix}
0 & 0 & \zeta_e^{\frac{-k-k'-l}{2}}\\
\zeta_e^l & 0 & 0\\
0 & \zeta_e^{\frac{k+k'-l}{2}} & 0
\end{pmatrix}$ with $w = t_{\frac{k+k'+l}{2}}s_3t_{\frac{k+k'-l}{2}}t_0$. \mbox{Relation 4 of Definition \ref{DefinitionChenBrauer(e,e,3)} is:}

$$ t_{\frac{k+k'+l}{2}}s_3t_{\frac{k+k'-l}{2}}t_0 e_{2,3;k} = e_{1,3;k'}  t_{\frac{k+k'+l}{2}}s_3t_{\frac{k+k'-l}{2}}t_0.$$

We refer to this case as \underline{Case (5.1)}.

\noindent \underline{If $k+k'+l$ is odd}, then $c_3 = \zeta_e^{\frac{-k-k'-l+e}{2}}$ and $w = \begin{pmatrix}
0 & 0 & \zeta_e^{\frac{-k-k'-l+e}{2}}\\
\zeta_e^l & 0 & 0\\
0 & \zeta_e^{\frac{k+k'-l+e}{2}} & 0
\end{pmatrix}$ with $w = t_{\frac{k+k'+l+e}{2}}s_3t_{\frac{k+k'-l+e}{2}}t_0$. \mbox{Relation 4 of Definition \ref{DefinitionChenBrauer(e,e,3)} is:}

$$t_{\frac{k+k'+l+e}{2}}s_3t_{\frac{k+k'-l+e}{2}}t_0 e_{2,3;k} = e_{1,3;k'} t_{\frac{k+k'+l+e}{2}}s_3t_{\frac{k+k'-l+e}{2}}t_0.$$

We refer to this case as \underline{Case (5.2)}.\\

The second case is when $c_2$ is nonzero and $c_9 = \zeta_e^{k'-k}c_2$ for which case $w =  \begin{pmatrix}
0 & c_2 & 0\\
\zeta_e^l & 0 & 0\\
0 & 0 & \zeta_e^{k'-k}c_2
\end{pmatrix}$ and $c_2^2 = \zeta_e^{k-k'-l}$.

\noindent \underline{If $k-k'-l$ is even}, then $c_2 = \zeta_e^{\frac{k-k'-l}{2}}$ and $w = \begin{pmatrix}
0 & \zeta_e^{\frac{k-k'-l}{2}} & 0\\
\zeta_e^l & 0 & 0\\
0 & 0 & \zeta_e^{\frac{k'-k-l}{2}}
\end{pmatrix}$.

\begin{itemize}

\item If $l \equiv k'-k$ modulo $e$, we get $w = t_{k'-k} \in \mathcal{R}$. \mbox{Relation 5 of Definition \ref{DefinitionChenBrauer(e,e,3)} is:}

\begin{center} $e_{2,3;k}e_{1,3;k'} = t_{k'-k}  e_{1,3;k'} = e_{2,3;k} t_{k'-k}$. \end{center}

We refer to this case as \underline{Case (5.3)}.

\item If $l \not\equiv k'-k$ modulo $e$, we get $w = t_{\frac{k'-k+l}{2}}s_3t_{\frac{k'-k-l}{2}}t_0s_3 \notin \mathcal{R}$. Relation 4 of Definition \ref{DefinitionChenBrauer(e,e,3)} is:

$$t_{\frac{k'-k+l}{2}}s_3t_{\frac{k'-k-l}{2}}t_0s_3 e_{2,3;k} = e_{1,3;k'} t_{\frac{k'-k+l}{2}}s_3t_{\frac{k'-k-l}{2}}t_0s_3.$$

We refer to this case as \underline{Case (5.4)}. Note that if $l \equiv k'-k$ modulo $e$ in this case, we get $w \in \mathcal{R}$ and the corresponding relation in $\mathcal{B}(e,e,3)$ is the same as in Case (5.3).

\end{itemize}

\noindent \underline{If $k-k'-l$ is odd}, then $c_2 = \zeta_e^{\frac{k-k'-l+e}{2}}$ and $w = \begin{pmatrix}
0 & \zeta_e^{\frac{k-k'-l+e}{2}} & 0\\
\zeta_e^l & 0 & 0\\
0 & 0 & \zeta_e^{\frac{k'-k-l+e}{2}}
\end{pmatrix}$ with $w = t_{\frac{k'-k+l+e}{2}}s_3t_{\frac{k'-k-l+e}{2}}t_0s_3 \notin \mathcal{R}$. Relation 4 of Definition \ref{DefinitionChenBrauer(e,e,3)} is:

$$t_{\frac{k'-k+l+e}{2}}s_3t_{\frac{k'-k-l+e}{2}}t_0s_3 e_{2,3;k} = e_{1,3;k'} t_{\frac{k'-k+l+e}{2}}s_3t_{\frac{k'-k-l+e}{2}}t_0s_3.$$

We refer to this case as \underline{Case (5.5)}.

\subsection*{Case 6: $s = s_{1,3;k}$ and $s' = s_{1,3;k'}$}

We have $s = s_{1,3;k} = t_ks_3t_k$ and $s' = s_{1,3;k'} = t_{k'}s_3t_{k'}$ with $0 \leq k,k' \leq e-1$ and $k \neq k'$. We determine the solutions of the equation $w s = s' w$ for $w = \begin{pmatrix}
c_1 & c_2 & c_3\\
c_4 & c_5 & c_6\\
c_7 & c_8 & c_9
\end{pmatrix} \in G(e,e,3)$. We get $c_5$ is nonzero. Set $c_5 = \zeta_e^l$ with $0 \leq l \leq e-1$. Then we have $c_2 = c_8 = c_4 = c_6 = 0$ since $w$ is monomial. We get two different cases.\\

The first case is when $c_3$ is nonzero and $c_7 = \zeta_e^{k+k'}c_3$ for which case\\ $w = \begin{pmatrix}
0 & 0 & c_3\\
0 & \zeta_e^l & 0\\
\zeta_e^{k+k'}c_3 & 0 & 0
\end{pmatrix}$ with $c_3^2 = \zeta_e^{-k-k'-l}$.\\

\noindent \underline{If $k+k'+l$ is even}, then $c_3 = \zeta_e^{\frac{-k-k'-l}{2}}$ and $w = \begin{pmatrix}
0 & 0 & \zeta_e^{\frac{-k-k'-l}{2}}\\
0 & \zeta_e^l & 0\\
\zeta_e^{\frac{k+k'-l}{2}} & 0 & 0
\end{pmatrix}$.

\begin{itemize}

\item If $l = 0$, then $w \in \mathcal{R}$ with $w = t_{\frac{k+k'}{2}}s_3t_{\frac{k+k'}{2}}$.  \mbox{Relation 5 of Definition \ref{DefinitionChenBrauer(e,e,3)} is:}

\begin{center} $e_{1,3;k}e_{1,3;k'} = t_{\frac{k+k'}{2}}s_3t_{\frac{k+k'}{2}} e_{1,3;k'} = e_{2,3;k} t_{\frac{k+k'}{2}}s_3t_{\frac{k+k'}{2}}$. \end{center}

We refer to this case as \underline{Case (6.1)}.

\item If $l \neq 0$, then $w \notin \mathcal{R}$ and $w = t_{\frac{k+k'+l}{2}}s_3t_{\frac{k+k'-l}{2}}$. Relation 4 of Definition \ref{DefinitionChenBrauer(e,e,3)} is:

$$t_{\frac{k+k'+l}{2}}s_3t_{\frac{k+k'-l}{2}} e_{1,3;k} = e_{1,3;k'} t_{\frac{k+k'+l}{2}}s_3t_{\frac{k+k'-l}{2}}.$$

We refer to this case as \underline{Case (6.2)}.

\end{itemize}

\noindent \underline{If $k+k'+l$ is odd}, then $c_3 = \zeta_e^{\frac{-k-k'-l+e}{2}}$ and $w = \begin{pmatrix}
0 & 0 & \zeta_e^{\frac{-k-k'-l+e}{2}}\\
0 & \zeta_e^l & 0\\
\zeta_e^{\frac{k+k'-l+e}{2}} & 0 & 0
\end{pmatrix}$ with $w = t_{\frac{k+k'+l+e}{2}}s_3t_{\frac{k+k'-l+e}{2}}$.

\begin{itemize}

\item If $l = 0$, then $w \in \mathcal{R}$ with $w = t_{\frac{k+k'+e}{2}}s_3t_{\frac{k+k'+e}{2}}$. \mbox{Relation 5 of Definition \ref{DefinitionChenBrauer(e,e,3)} is:}

\begin{center} $e_{1,3;k}e_{1,3;k'} =t_{\frac{k+k'+e}{2}}s_3t_{\frac{k+k'+e}{2}} e_{1,3;k'} = e_{2,3;k} t_{\frac{k+k'+e}{2}}s_3t_{\frac{k+k'+e}{2}}$. \end{center}

We refer to this case as \underline{Case (6.3)}.

\item If $l \neq 0$, then $w \notin \mathcal{R}$ with $w = t_{\frac{k+k'+l+e}{2}}s_3t_{\frac{k+k'-l+e}{2}}$. \mbox{Relation 4 of Definition \ref{DefinitionChenBrauer(e,e,3)} is:}

$$t_{\frac{k+k'+l+e}{2}}s_3t_{\frac{k+k'-l+e}{2}} e_{1,3;k} = e_{1,3;k'} t_{\frac{k+k'+l+e}{2}}s_3t_{\frac{k+k'-l+e}{2}}.$$

We refer to this case as \underline{Case (6.4)}.

\end{itemize}

The second case is when $c_1$ is nonzero and $c_9 = \zeta_e^{k'-k}c_1$ for which case $w = \begin{pmatrix}
c_1 & 0 & 0\\
0 & \zeta_e^l & 0\\
0 & 0 & \zeta_e^{k'-k}c_1
\end{pmatrix}$ with $c_1^2 = \zeta_e^{k-k'-l}$.

\noindent \underline{If $k-k'-l$ is even}, then $c_1 = \zeta_e^{\frac{k-k'-l}{2}}$ and $w = \begin{pmatrix}
\zeta_e^{\frac{k-k'-l}{2}} & 0 & 0\\
0 & \zeta_e^l & 0\\
0 & 0 & \zeta_e^{\frac{k'-k-l}{2}}
\end{pmatrix}$ with $w = t_lt_0s_3t_{\frac{k'-k-l}{2}}t_0s_3$. Relation 4 of Definition \ref{DefinitionChenBrauer(e,e,3)} is:

$$t_lt_0s_3t_{\frac{k'-k-l}{2}}t_0s_3 e_{1,3;k} = e_{1,3;k'} t_lt_0s_3t_{\frac{k'-k-l}{2}}t_0s_3.$$

We refer to this case as \underline{Case (6.5)}.

\noindent \underline{If $k-k'-l$ is odd}, then $c_1 = \zeta_e^{\frac{k-k'-l+e}{2}}$ and $w = \begin{pmatrix}
\zeta_e^{\frac{k-k'-l+e}{2}} & 0 & 0\\
0 & \zeta_e^l & 0\\
0 & 0 & \zeta_e^{\frac{k'-k-l+e}{2}}
\end{pmatrix}$ with $w = t_lt_0s_3t_{\frac{k'-k-l+e}{2}}t_0s_3$. Relation 4 of Definition \ref{DefinitionChenBrauer(e,e,3)} is:

$$t_lt_0s_3t_{\frac{k'-k-l+e}{2}}t_0s_3 e_{1,3;k} = e_{1,3;k'} t_lt_0s_3t_{\frac{k'-k-l+e}{2}}t_0s_3.$$

We refer to this case as \underline{Case (6.6)}.\\

\smallskip

\begin{remark}\label{RemarkEeven}

From the six cases above, one can check that Relation 6 of Definition \ref{DefinitionChenBrauer(e,e,3)} does not occur in the case of $\mathcal{B}(e,e,3)$ when $e$ is odd.

\end{remark}

The following lemmas are useful in the proof of Proposition \ref{PropositionIsomChenPsi} below.\\

\begin{lemma}\label{LemmaEkBrauerInduction}

The following relations hold in $\mathcal{B}(e,e,3)$.
\begin{enumerate}
\item Let $k \in \mathbb{N}$. We have $t_k = \underset{2k-1}{\underbrace{t_1t_0 \cdots t_1}}$.
\item For $2 \leq k \leq e-1$, we have $e_{1,2;k} = t_{k-1} e_{1,2;k-2} t_{k-1}$.
\item For $0 \leq k \leq e-1$, we have\\
$e_{1,2;k} = \left\{
    \begin{array}{ll}
     \underset{k-1}{\underbrace{t_1t_0 \cdots t_0}}\ e_{1,2;1}\ \underset{k-1}{\underbrace{t_0t_1 \cdots t_1}}  & if\ k\ is\ odd, \\
     \underset{k-1}{\underbrace{t_1t_0 \cdots t_1}}\ e_{1,2;0}\ \underset{k-1}{\underbrace{t_1t_0 \cdots t_1}}  & if\ k\ is\ even.
    \end{array}
\right.$
\end{enumerate}

\end{lemma}

\begin{proof}

The first item of the lemma is obvious. It is done in the same way as the proof of the first item of Lemma \ref{LemmaExtendTkEk}.

Since $t_k = t_{k-1}t_{k-2}t_{k-1}$ for $2 \leq k \leq e-1$, by Relation 4 of Definition \ref{DefinitionChenBrauer(e,e,3)}, we have $t_{k-1} e_{1,2;k-2} = e_{1,2;k} t_{k-1}$, that is $e_{1,2;k} = t_{k-1} e_{1,2;k-2} t_{k-1}$. Hence we get the second relation of the lemma.

Using the previous relation, one can prove the last statement of the lemma by induction in the same way as the proof of the second item of Lemma \ref{LemmaTkEk}. 

\end{proof}

\begin{lemma}\label{LemmaRelBrauerChenJet1}

Let $0 \leq k \leq e-1$. The following relations hold in $\mathcal{B}(e,e,3)$ for $e$ odd.
\begin{enumerate}

\item $e_{1,2;k}e_{2,3;0} = s_3t_ke_{2,3;0} = e_{1,2;k} s_3t_k$ and\\
$e_{2,3;0}e_{1,2;k} = t_ks_3e_{1,2;k} = e_{2,3;0} t_ks_3$.

\item $e_{1,2;k}s_3e_{1,2;k} = e_{1,2;k}$ and\\
$e_{2,3;0}t_ke_{2,3;0} = e_{2,3;0}$.

\end{enumerate}

\end{lemma}

\begin{proof}

From Case (2.3) above, for $k' = 0$, we have $e_{1,2;k}e_{2,3;0} = t_ks_3t_ke_{2,3;0} = e_{1,2;k} t_ks_3t_k$. Using Relation 2 of Definition \ref{DefinitionChenBrauer(e,e,3)}, we get $e_{1,2;k}e_{2,3;0} = s_3t_ke_{2,3;0} = e_{1,2;k} s_3t_k$. We also have $(t_ks_3t_k)s_3(t_ks_3t_k)^{-1} = t_k$. Hence, from Relations 5 and 2 of Definition \ref{DefinitionChenBrauer(e,e,3)}, we also get $e_{2,3;0}e_{1,2;k} = t_ks_3e_{1,2;k} = e_{2,3;0} t_ks_3$. Moreover, multiplying $s_3t_ke_{2,3;0} = e_{1,2;k}s_3t_k$ on the right by $e_{1,2;k}$, we get $s_3t_ke_{2,3;0}e_{1,2;k} = e_{1,2;k}s_3t_ke_{1,2;k}$. Using Relation 2 of Definition \ref{DefinitionChenBrauer(e,e,3)}, we get $s_3t_ke_{2,3;0}e_{1,2;k} = e_{1,2;k}s_3e_{1,2;k}$. Now, replace $e_{2,3;0}e_{1,2;k}$ by $t_ks_3e_{1,2;k}$, we get $s_3t_kt_ks_3e_{1,2;k} = e_{1,2;k}s_3e_{1,2;k}$, that is $e_{1,2;k}s_3e_{1,2;k} = e_{1,2;k}$. Similarly, we also get $e_{2,3;0}t_ke_{2,3;0} = e_{2,3;0}$.

\end{proof}

\begin{lemma}\label{LemmaRelBrauerChenJet2}

The following relations hold in $\mathcal{B}(e,e,3)$ for $e$ odd.
\begin{enumerate}

\item Let $1 \leq k \leq e-2$ and $k$ odd, we have:\\
$e_{1,2;1}\ \underset{k}{\underbrace{t_0t_1 \cdots t_0}}\ e_{1,2;1} = e_{1,2;1}$ and \\
$e_{1,2;0}\ \underset{k}{\underbrace{t_1t_0 \cdots t_1}}\ e_{1,2;0} = e_{1,2;0}$.

\item $\underset{e-1}{\underbrace{t_1t_0 \cdots t_0}}\ e_{1,2;1} = e_{1,2;0}\ \underset{e-1}{\underbrace{t_1t_0 \cdots t_0}}$ and\\
$\underset{e-1}{\underbrace{t_0t_1 \cdots t_1}}\ e_{1,2;0} = e_{1,2;1}\ \underset{e-1}{\underbrace{t_0t_1 \cdots t_1}}$.

\end{enumerate}

\end{lemma}

\begin{proof}

Consider Case (1.1) above. In this case, we have $w = t_{\frac{k+k'}{2}} \in \mathcal{R}$, where $k+k'$ is even. Suppose $k$ and $k'$ are both even. We have $t_{\frac{k+k'}{2}} = \underset{k+k'-1}{\underbrace{t_1t_0 \cdots t_1}}$ (see the first item of Lemma \ref{LemmaEkBrauerInduction}). The first equation of Relation 5 of Definition \ref{DefinitionChenBrauer(e,e,3)} becomes\\
$\underset{k-1}{\underbrace{t_1t_0 \cdots t_1}}\ e_{1,2;0}\ \underset{k-1}{\underbrace{t_1t_0 \cdots t_1}}\ \underset{k'-1}{\underbrace{t_1t_0 \cdots t_1}}\ e_{1,2;0}\ \underset{k'-1}{\underbrace{t_1t_0 \cdots t_1}} = \underset{k+k'-1}{\underbrace{t_1t_0 \cdots t_1}}\ \underset{k'-1}{\underbrace{t_1t_0 \cdots t_1}}\ e_{1,2;0}\ \underset{k'-1}{\underbrace{t_1t_0 \cdots t_1}}.$ After simplification, we get\\ 
$e_{1,2;0}\ \underset{k-1}{\underbrace{t_1t_0 \cdots t_1}}\ \underset{k'-1}{\underbrace{t_1t_0 \cdots t_1}}\ e_{1,2;0} = \underset{k'}{\underbrace{t_0t_1 \cdots t_1}}\ \underset{k'-1}{\underbrace{t_1t_0 \cdots t_1}}\ e_{1,2;0} = t_0 e_{1,2;0} = e_{1,2;0}$.\\
Similarly, the second equation of Relation 5 of Definition \ref{DefinitionChenBrauer(e,e,3)} gives\\
$e_{1,2;0}\ \underset{k-1}{\underbrace{t_1t_0 \cdots t_1}}\ \underset{k'-1}{\underbrace{t_1t_0 \cdots t_1}}\ e_{1,2;0} = e_{1,2;0}$.
\begin{itemize}

\item If $k < k'$, then we get $e_{1,2;0}\ \underset{k'-k}{\underbrace{t_0t_1 \cdots t_1}}\ e_{1,2;0} = e_{1,2;0}$, that is $e_{1,2;0}\ \underset{k'-k-1}{\underbrace{t_1t_0 \cdots t_1}}\ e_{1,2;0} = e_{1,2;0}$, where $k'-k-1$ is odd.

\item If $k > k'$, then we get $e_{1,2;0}\ \underset{k-k'-1}{\underbrace{t_1t_0 \cdots t_1}}\ e_{1,2;0} = e_{1,2;0}$, where $k'-k-1$ is odd.
\end{itemize}
\noindent The case when $k$ and $k'$ are both odd is done in the same way and we get\\
$e_{1,2;1}\ \underset{k'-k-1}{\underbrace{t_0t_1 \cdots t_0}}\ e_{1,2;1} = e_{1,2;1}$ if $k < k'$ and $e_{1,2;1}\ \underset{k-k'-1}{\underbrace{t_0t_1 \cdots t_0}}\ e_{1,2;1} = e_{1,2;1}$ if $k > k'$, where $k'-k-1$ is odd. In conclusion, the first statement of the lemma is satisfied.

Since $\underset{e-1}{\underbrace{t_0t_1 \cdots t_1}}\ t_0\ \underset{e-1}{\underbrace{t_1t_0 \cdots t_0}} = t_1$, then by Relation 4 of Definition \ref{DefinitionChenBrauer(e,e,3)}, we have $\underset{e-1}{\underbrace{t_0t_1 \cdots t_1}}\ e_{1,2;0} = e_{1,2;1}\ \underset{e-1}{\underbrace{t_0t_1 \cdots t_1}}$. Also, since $\underset{e-1}{\underbrace{t_1t_0 \cdots t_0}}\ t_1\ \underset{e-1}{\underbrace{t_0t_1 \cdots t_1}} = t_0$, by Relation 4 of Definition \ref{DefinitionChenBrauer(e,e,3)}, we also have $\underset{e-1}{\underbrace{t_1t_0 \cdots t_0}}\ e_{1,2;1} = e_{1,2;0}\ \underset{e-1}{\underbrace{t_1t_0 \cdots t_0}}$.

\end{proof}

In the remaining part of this section, our goal is to prove that the algebra Br$(e,e,3)$ is isomorphic to $\mathcal{B}(e,e,3)$ for all $e \geq 3$ and $e$ odd.\\

Consider the map $\psi$ on the set of generators defined by $\psi(T_0) = t_0$, $\psi(T_1) = t_1$, $\psi(S_3) = s_3$, $\psi(E_0) = e_{1,2;0}$, $\psi(E_1) = e_{1,2;1}$, and $\psi(F_3) = e_{2,3;0}$.

\begin{proposition}\label{PropositionIsomChenPsi}

For $e \geq 3$ and $e$ odd, the map $\psi$ extends to a morphism from \emph{Br}$(e,e,3)$ to $\mathcal{B}(e,e,3)$.

\end{proposition}

\begin{proof}

In order to prove that $\psi$ extends to a morphism from Br$(e,e,3)$ to $\mathcal{B}(e,e,3)$ for $e$ odd, we should show that $\psi$ preserves all the relations of Definition \ref{DefinitionBr(e,e,n)eodd}. We have two types of relations: the relations that correspond to the Brauer algebra of type $A_2$ for $\{T_k,S_3\} \cup \{E_k,F_3\}$ for $0 \leq k \leq e-1$ (see Definition \ref{DefinitionBrauerBMW}) and Relations 1 to 9 of \mbox{Definition \ref{DefinitionChenBrauer(e,e,3)}.} 

Relation 1 of Definition \ref{DefinitionBrauerBMW} gives $t_k^2 = 1$ and $s_3^2 = 1$ that are particular relations of the first item of Definition \ref{DefinitionChenBrauer(e,e,3)}. 

Relation 2 of Definition \ref{DefinitionBrauerBMW} gives $e_{1,2;k}^2 = x e_{1,2;k}$ and $e_{2,3;0}^2 = x e_{2,3;0}$ that are particular relations of the third item of Definition \ref{DefinitionChenBrauer(e,e,3)}.

Relation 3 of Definition \ref{DefinitionBrauerBMW} gives $t_k e_{1,2;k} = e_{1,2;k} t_k = e_{1,2;k}$ and $s_3 e_{2,3;0} = e_{2,3;0}s_3 = e_{2,3;0}$ that are particular relations of the second item of Definition \ref{DefinitionChenBrauer(e,e,3)}.

Relation 4 of Definition \ref{DefinitionBrauerBMW} gives $e_{1,2;k} s_3 e_{1,2;k} = e_{1,2;k}$ and $e_{2,3;0} t_k e_{2,3;0} = e_{2,3;0}$. By the second item of Lemma \ref{LemmaRelBrauerChenJet1}, these relations hold in $\mathcal{B}(e,e,3)$.

Relation 5 of Definition \ref{DefinitionBrauerBMW} gives $e_{1,2;k}e_{2,3;0} = s_3t_ke_{2,3;0} = e_{1,2;k} s_3t_k$ and $e_{2,3;0}e_{1,2;k} = t_ks_3e_{1,2;k} = e_{2,3;0} t_ks_3$. By the first item of Lemma \ref{LemmaRelBrauerChenJet1}, these relations hold in $\mathcal{B}(e,e,3)$.

Relation 1 of Definition \ref{DefinitionBr(e,e,n)eodd} gives $t_k = t_{k-1}t_{k-2}t_{k-1}$ and $\underset{e-1}{\underbrace{t_1t_0 \cdots t_0}} = \underset{e-1}{\underbrace{t_0t_1 \cdots t_1}}$ that are particular relations of the first item of Definition \ref{DefinitionChenBrauer(e,e,3)}.

Relation 2 of Definition \ref{DefinitionBr(e,e,n)eodd} gives $e_{1,2;k} = t_{k-1} e_{1,2;k-2} t_{k-1}$. By the second item of Lemma \ref{LemmaEkBrauerInduction}, these relations hold in $\mathcal{B}(e,e,3)$.

Relations 3, 4, and 5 of Definition \ref{DefinitionBr(e,e,n)eodd} are already checked. Relations 6 and 7 of Definition \ref{DefinitionBr(e,e,n)eodd} give $e_{1,2;1}\ \underset{k}{\underbrace{t_0t_1 \cdots t_0}}\ e_{1,2;1} = e_{1,2;1}$ and $e_{1,2;0}\ \underset{k}{\underbrace{t_1t_0 \cdots t_1}}\ e_{1,2;0} = e_{1,2;0}$. By the first item of Lemma \ref{LemmaRelBrauerChenJet2}, these relations hold in $\mathcal{B}(e,e,3)$. 

Relations 8 and 9 of Definition \ref{DefinitionBr(e,e,n)eodd} give $\underset{e-1}{\underbrace{t_1t_0 \cdots t_0}}\ e_{1,2;1} = e_{1,2;0}\ \underset{e-1}{\underbrace{t_1t_0 \cdots t_0}}$ and $\underset{e-1}{\underbrace{t_0t_1 \cdots t_1}}\ e_{1,2;0} = e_{1,2;1}\ \underset{e-1}{\underbrace{t_0t_1 \cdots t_1}}$. By the second item of Lemma \ref{LemmaRelBrauerChenJet2}, these relations hold in $\mathcal{B}(e,e,3)$.

It follows that $\psi$ extends to a morphism from Br$(e,e,3)$ to $\mathcal{B}(e,e,3)$.

\end{proof}

Lemmas \ref{LemmaRelBrJet1}, \ref{LemmaRelBrJet2}, and \ref{LemmaRelBrJet3} below will be useful in the proof of Proposition \ref{PropositionIsomChenPhi} below.

\begin{lemma}\label{LemmaRelBrJet1}

Let $0 \leq k,k' \leq e-1$ with $k \neq k'$ and $k + k'$ even, the following relations hold in \emph{Br}$(e,e,3)$ for $e$ odd.
\begin{enumerate}
\item $E_kE_{k'} = T_{\frac{k+k'}{2}}E_{k'} = E_k T_{\frac{k+k'}{2}}$.
\item $T_kF_3T_k\ T_{k'}F_3T_{k'} = T_{\frac{k+k'}{2}}S_3T_{\frac{k+k'}{2}}\ T_{k'}F_3T_{k'} = T_kF_3T_k\ T_{\frac{k+k'}{2}}S_3T_{\frac{k+k'}{2}}$.
\end{enumerate}

\end{lemma}

\begin{proof}

Suppose $k$ and $k'$ are both even. We have $T_{\frac{k+k'}{2}} E_{k'} =\\ \underset{k+k'-1}{\underbrace{T_1T_0 \cdots T_1}}\ \underset{k'-1}{\underbrace{T_1T_0 \cdots T_1}}\ E_0\ \underset{k'-1}{\underbrace{T_1T_0 \cdots T_1}} = \underset{k}{\underbrace{T_1T_0 \cdots T_0}}\ E_0\ \underset{k'-1}{\underbrace{T_1T_0 \cdots T_1}} =\\ \underset{k-1}{\underbrace{T_1T_0 \cdots T_1}}\ E_0\ \underset{k'-1}{\underbrace{T_1T_0 \cdots T_1}}$.\\
Also we have $E_kT_{\frac{k+k'}{2}} = \underset{k-1}{\underbrace{T_1T_0 \cdots T_1}}\ E_0\ \underset{k-1}{\underbrace{T_1T_0 \cdots T_1}}\ \underset{k+k'-1}{\underbrace{T_1T_0 \cdots T_1}} =\\ \underset{k-1}{\underbrace{T_1T_0 \cdots T_1}}\ E_0\ \underset{k'}{\underbrace{T_0T_1 \cdots T_1}} = \underset{k-1}{\underbrace{T_1T_0 \cdots T_1}}\ E_0\ \underset{k'-1}{\underbrace{T_1T_0 \cdots T_1}}$ that is equal to $T_{\frac{k+k'}{2}}E_{k'}$.\\
It remains to show that $E_kE_{k'} = T_{\frac{k+k'}{2}}E_{k'}$. We have\\
$E_kE_{k'} = \underset{k-1}{\underbrace{T_1T_0 \cdots T_1}}\ E_0\ \underset{k-1}{\underbrace{T_1T_0 \cdots T_1}}\ \underset{k'-1}{\underbrace{T_1T_0 \cdots T_1}}\ E_0\ \underset{k'-1}{\underbrace{T_1T_0 \cdots T_1}}$.

\begin{itemize}

\item If $k > k'$, then we have $E_kE_{k'} = \underset{k-1}{\underbrace{T_1T_0 \cdots T_1}}\ E_0\ \underset{k-k'}{\underbrace{T_1T_0 \cdots T_0}}\ E_0\ \underset{k'-1}{\underbrace{T_1T_0 \cdots T_1}} =\\ \underset{k-1}{\underbrace{T_1T_0 \cdots T_1}}\ E_0\ \underset{k-k'-1}{\underbrace{T_1T_0 \cdots T_1}}\ E_0\ \underset{k'-1}{\underbrace{T_1T_0 \cdots T_1}}$. Since $k-k'-1$ is odd, by Relation 7 of Definition \ref{DefinitionBr(e,e,n)eodd}, we replace $E_0\ \underset{k-k'-1}{\underbrace{T_1T_0 \cdots T_1}}\ E_0$ by $E_0$ and get\\ $\underset{k-1}{\underbrace{T_1T_0 \cdots T_1}}\ E_0\ \underset{k'-1}{\underbrace{T_1T_0 \cdots T_1}}$ which is equal to $T_{\frac{k+k'}{2}}E_{k'}$.
\item If $k < k'$, then $E_kE_{k'} = \underset{k-1}{\underbrace{T_1T_0 \cdots T_1}}\ E_0\ \underset{k'-k}{\underbrace{T_0T_1 \cdots T_1}}\ E_0\ \underset{k'-1}{\underbrace{T_1T_0 \cdots T_1}} =\\ \underset{k-1}{\underbrace{T_1T_0 \cdots T_1}}\ E_0\ \underset{k'-k-1}{\underbrace{T_1T_0 \cdots T_1}}\ E_0\ \underset{k'-1}{\underbrace{T_1T_0 \cdots T_1}} = \underset{k-1}{\underbrace{T_1T_0 \cdots T_1}}\ E_0\ \underset{k'-1}{\underbrace{T_1T_0 \cdots T_1}} = T_{\frac{k+k'}{2}}E_{k'}$.

\end{itemize}

\noindent Hence we get Relation 1.

For the second relation, we have $T_kF_3T_kT_{k'}F_3T_{k'} = S_3E_kS_3S_3E_{k'}S_3 = S_3E_kE_{k'}S_3$ and $T_{\frac{k+k'}{2}}S_3T_{\frac{k+k'}{2}}T_{k'}F_3T_{k'} = S_3T_{\frac{k+k'}{2}}S_3S_3E_{k'}S_3 = S_3T_{\frac{k+k'}{2}}E_{k'}S_3$. Similarly, we have $T_kF_3T_kT_{\frac{k+k'}{2}}S_3T_{\frac{k+k'}{2}} = S_3E_kT_{\frac{k+k'}{2}}S_3$. By Relation 1, we get that Relation 2 holds in Br$(e,e,3)$.

The case when $k$ and $k'$ are both odd is similar and left to the reader.

\end{proof}

\begin{lemma}\label{LemmaRelBrJet2}

Let $0 \leq k,k' \leq e-1$ with $k \neq k'$, the following relations hold in \emph{Br}$(e,e,3)$ for $e$ odd.
\begin{enumerate}
\item $T_0T_{k}F_3T_kT_0\ T_{k'}F_3T_{k'} = T_{k'-k}\ T_{k'}F_3T_{k'} = T_0T_{k}F_3T_kT_0\ T_{k'-k}$.
\item $E_k\ T_0T_{k'} F_3T_{k'}T_0 = T_{k+k'} S_3 T_{k+k'}\ T_0T_{k'}F_3T_{k'}T_0 = E_k\ T_{k+k'}S_3T_{k+k'}$.
\end{enumerate}

\end{lemma}

\begin{proof}

We have $T_kT_0T_{k'} = \underset{2(k+k')-1}{\underbrace{T_1T_0 \cdots T_1}} = T_{k+k'}$. Then $T_0T_{k}F_3T_kT_0T_{k'}F_3T_{k'} =\\ T_0T_{k}F_3T_{k+k'}F_3T_{k'} = T_0T_{k}F_3T_{k'}$. Since $T_kT_0T_{k'-k} = T_{k'}$, we have $T_0T_{k}F_3T_{k'} = T_0T_{k}F_3T_kT_0T_{k'-k}$. Also, since $T_{k'-k}T_{k'} = T_0T_k$, we get $T_{k'-k}T_{k'}F_3T_{k'} = T_0T_kF_3T_{k'}$. Hence we get Relation 1.\\

Let us now prove the first identity of Relation 2, that is $E_k\ T_0T_{k'} F_3 =\\ T_{k+k'} S_3 T_{k+k'}\ T_0T_{k'}F_3$. We have $T_{k+k'}T_0T_{k'} = T_{k+2k'}$. Also, we have $S_3T_{k+2k'}F_3 = E_{k+2k'}F_3$. We compute $T_{k+k'}E_{k+2k'}$.

If $k$ is even, then $\underset{2(k+k')-1}{\underbrace{T_1T_0 \cdots T_1}}\ \underset{k+2k'-1}{\underbrace{T_1T_0 \cdots T_1}} \ E_0\ \underset{k+2k'-1}{\underbrace{T_1T_0 \cdots T_1}} =\\ \underset{k}{\underbrace{T_1T_0 \cdots T_0}}\ E_0\ \underset{k-1}{\underbrace{T_1T_0 \cdots T_1}}\ \underset{2k'}{\underbrace{T_0T_1 \cdots T_1}} =\\ \underset{k-1}{\underbrace{T_1T_0 \cdots T_1}}\ E_0\ \underset{k-1}{\underbrace{T_1T_0 \cdots T_1}}\ T_0\ \underset{2k'-1}{\underbrace{T_1T_0 \cdots T_1}} = E_kT_0T_{k'}$.

The case when $k$ is odd is similar and left to the reader. Hence we get the first identity of Relation 2.

For $k$ even, the second identity is\\
$T_{k+k'}S_3T_{k+k'}\ E_k\ T_{k+k'}S_3T_{k+k'} = T_0T_{k'}F_3T_{k'}T_0$. It is easy to check that $T_{k+k'} E_k T_{k+k'} = E_{k+2k'}$. Hence $T_{k+k'}S_3T_{k+k'}\ E_k\ T_{k+k'}S_3T_{k+k'} = T_{k+k'}S_3\ E_{k+2k'}\ S_3T_{k+k'} =\\ T_{k+k'}T_{k+2k'}\ F_3\ T_{k+2k'}T_{k+k'}$. It is easy to check that $T_{k+k'}T_{k+2k'} = T_0T_{k'}$ and we are done. The case when $k$ is odd is done in the same way and is left to the reader.
\end{proof}

\begin{lemma}\label{LemmaRelBrJet3}

Let $0 \leq k,k',l \leq e-1$ with $k \neq k'$ and $k-k'-l$ even, the following relations hold in \emph{Br}$(e,e,3)$ for $e$ odd.
\begin{enumerate}
\item $T_{k'}T_{l}T_{0}S_3T_{\frac{k'-k-l}{2}}T_0\ E_k\ T_0T_{\frac{k'-k-l}{2}}S_3T_{0}T_lT_{k'} = F_3$.
\item \mbox{$T_{k'}T_0T_{\frac{k-k'-l}{2}}T_0(T_0T_1)^kS_3T_{\frac{k'-k-l}{2}}T_0\ F_3\ T_0T_{\frac{k'-k-l}{2}}S_3(T_1T_0)^kT_0T_{\frac{k-k'-l}{2}}T_0T_{k'} = F_3$.}
\end{enumerate}

\end{lemma}

\begin{proof}

Let us start by proving the first identity. Suppose $k$ is even. 
\begin{itemize}
\item If $k' \geq k+l$, we have $T_{\frac{k'-k-l}{2}}T_0\ E_k\ T_0T_{\frac{k'-k-l}{2}} =\\ \underset{k'-k-l-1}{\underbrace{T_1T_0\cdots T_1}}\ T_0\ \underset{k-1}{\underbrace{T_1T_0\cdots T_1}}\ E_0\ \underset{k-1}{\underbrace{T_1T_0\cdots T_1}}\ T_0\ \underset{k'-k-l-1}{\underbrace{T_1T_0\cdots T_1}} =\\ \underset{k'-l-1}{\underbrace{T_1T_0\cdots T_1}}\ E_0\ \underset{k'-l-1}{\underbrace{T_1T_0\cdots T_1}} = E_{k'-l}$.\\
We should prove that $T_{k'}T_{l}T_{0}S_3\ E_{k'-l}\ S_3 T_{0}T_lT_{k'} = F_3$, that is\\ $T_{k'}T_{l}T_{0}T_{k'-l}\ F_3\ T_{k'-l} T_{0}T_lT_{k'} = F_3$. It is easily checked that $T_{k'}T_{l}T_0T_{k'-l} = 1$ and the relation we want to prove follows immediately.

\item If $k' < k+l$, let $x = k'-k-l+2e$. We have  $T_{\frac{k'-k-l}{2}}T_0\ E_k\ T_0T_{\frac{k'-k-l}{2}} =$\\
$T_{x/2}T_0\ E_k\ T_0T_{x/2} =
 \underset{x-1}{\underbrace{T_1T_0\cdots T_1}}\ T_0\ \underset{k-1}{\underbrace{T_1T_0\cdots T_1}}\ E_0\ \underset{k-1}{\underbrace{T_1T_0\cdots T_1}}\ T_0\ \underset{x-1}{\underbrace{T_1T_0\cdots T_1}} $. After simplification, this is equal to $\underset{k'-l+2e-1}{\underbrace{T_1T_0\cdots T_1}}\ E_0\ \underset{k'-l+2e-1}{\underbrace{T_1T_0\cdots T_1}} = E_{k'-l+2e}$.\\
We should prove that $T_{k'}T_{l}T_{0}S_3\ E_{k'-l+2e}\ S_3 T_{0}T_lT_{k'} = F_3$. This is done as in the previous case.

\end{itemize}

Note that the case when $k$ is odd is similar and is left to the reader. Hence we get Relation 1 of the lemma.\\

For the second relation, we have $S_3T_{\frac{k'-k-l}{2}}T_0\ F_3\ T_0T_{\frac{k'-k-l}{2}}S_3 =\\ S_3T_{\frac{k'-k-l}{2}}S_3\ E_0\ S_3T_{\frac{k'-k-l}{2}}S_3 = T_{\frac{k'-k-l}{2}}S_3T_{\frac{k'-k-l}{2}}\ E_0\ T_{\frac{k'-k-l}{2}}S_3 T_{\frac{k'-k-l}{2}}$.

\emph{Suppose $k' \geq k+l$}. We have $T_{\frac{k'-k-l}{2}}\ E_0\ T_{\frac{k'-k-l}{2}}$ is equal to $E_{k'-k-l}$. We get\\
$T_{\frac{k'-k-l}{2}}S_3\ E_{k'-k-l}\ S_3T_{\frac{k'-k-l}{2}} = T_{\frac{k'-k-l}{2}}T_{k'-k-l}\ F_3\ T_{k'-k-l}T_{\frac{k'-k-l}{2}} =\\ \underset{k'-k-l}{\underbrace{T_0T_1 \cdots T_1}}\ F_3\ \underset{k'-k-l}{\underbrace{T_1T_0 \cdots T_0}}$. On the one hand, we have\\
$T_{k'}T_{0}T_{\frac{k-k'-l}{2}}T_0(T_0T_1)^k\ \underset{k'-k-l}{\underbrace{T_0T_1 \cdots T_1}} = T_{k'}T_{0}T_{\frac{k-k'-l}{2}}T_0\ \underset{2k}{\underbrace{T_0T_1 \cdots T_1}}\ \underset{k'-k-l}{\underbrace{T_0T_1 \cdots T_1}} = \\ T_{k'}T_{0}T_{\frac{k-k'-l}{2}}\ \underset{2k-1}{\underbrace{T_1T_0 \cdots T_1}}\ \underset{k'-k-l}{\underbrace{T_0T_1 \cdots T_1}} = T_{k'}T_{0}T_{\frac{k-k'-l}{2}}\ \underset{k'+k-l-1}{\underbrace{T_1T_0 \cdots T_1}} = \\ T_{k'}T_{0}\ \underset{k-k'-l-1}{\underbrace{T_1T_0 \cdots T_1}}\ \underset{k'+k-l-1}{\underbrace{T_1T_0 \cdots T_1}} = T_{k'}T_{0}\ \underset{2k'}{\underbrace{T_0T_1 \cdots T_1}} = T_{k'}\ \underset{2k'-1}{\underbrace{T_1T_0 \cdots T_1}} = T_{k'}T_{k'} = 1$.\\
On the other hand, we also have $\underset{k'-k-l}{\underbrace{T_1T_0 \cdots T_0}}\ (T_1T_0)^kT_0T_{\frac{k-k'-l}{2}}T_0T_{k'} = 1$ since it is the inverse of $T_{k'}T_{0}T_{\frac{k-k'-l}{2}}T_0(T_0T_1)^k\ \underset{k'-k-l}{\underbrace{T_0T_1 \cdots T_1}}$.\\ The relation we want to prove follows immediately.

\emph{Suppose $k' < k+l$}. Let $x = k'-k-l+2e$. We have $T_{\frac{k'-k-l}{2}}\ E_0\ T_{\frac{k'-k-l}{2}} = T_{x/2}\ E_0\ T_{x/2} = \underset{x-1}{\underbrace{T_1T_0 \cdots T_1}}\ E_0\ \underset{x-1}{\underbrace{T_1T_0 \cdots T_1}} = E_{x}$. We get $T_{x/2}T_{x}\ F_3\ T_{x}T_{x/2}$, where $T_{x/2}T_{x} = \underset{x-1}{\underbrace{T_1T_0 \cdots T_1}}\ \underset{2x-1}{\underbrace{T_1T_0 \cdots T_1}} = \underset{x}{\underbrace{T_0T_1 \cdots T_0}}$\mbox{ with $x = k'-k-l+2e$.} Similarly to the previous case, we show that $T_{k'}T_{0}T_{\frac{k-k'-l}{2}}T_0(T_0T_1)^k\ \underset{x}{\underbrace{T_0T_1 \cdots T_1}} = 1$.\\

Hence the second relation of the lemma is satisfied in Br$(e,e,3)$. 

\end{proof}

Consider the map $\phi$ on the set of generators defined by $\phi(t_k) = T_k$, $\phi(s_{2,3;k}) = T_0T_{k}S_3T_kT_0$, $\phi(s_{1,3;k}) = T_kS_3T_k$, $\phi(e_{1,2;k}) = E_k$, $\phi(e_{2,3;k}) = T_0T_{k}F_3T_kT_0$, and $\phi(e_{1,3;k}) = T_kF_3T_k$.

\begin{proposition}\label{PropositionIsomChenPhi}

For $e \geq 3$ and $e$ odd, the map $\phi$ extends to a morphism from $\mathcal{B}(e,e,3)$ to \emph{Br}$(e,e,3)$.

\end{proposition}

\begin{proof}

Mapping Relations 1, 2, and 3 of Definition \ref{DefinitionChenBrauer(e,e,3)} by $\phi$, we get relations that are clearly satisfied in Br$(e,e,3)$. Relation 6 of Definition \ref{DefinitionChenBrauer(e,e,3)} does not occur when $e$ is odd, see Remark \ref{RemarkEeven}. For Relations 4 and 5 of Definition \ref{DefinitionChenBrauer(e,e,3)}, they have been explicitly written in the 6 cases above. We prove that the image by $\phi$ of the relation we get in each case is satisfied in Br$(e,e,3)$.\\

Consider Case (1.1). In the first paragraph of the proof of Lemma \ref{LemmaRelBrauerChenJet2}, we showed that if $k$ and $k'$ are both even, we get the following relations. \begin{itemize}

\item If $k < k'$, then we get $e_{1,2;0}\ \underset{k'-k-1}{\underbrace{t_1t_0 \cdots t_1}}\ e_{1,2;0} = e_{1,2;0}$, where $k'-k-1$ is odd. The corresponding relation in Br$(e,e,3)$ is $E_0\ \underset{k'-k-1}{\underbrace{T_1T_0 \cdots T_1}}\ E_0 = E_0$ which is a case of Relation 7 of Definition \ref{DefinitionBr(e,e,n)eodd}.

\item If $k > k'$, then we get $e_{1,2;0}\ \underset{k-k'-1}{\underbrace{t_1t_0 \cdots t_1}}\ e_{1,2;0} = e_{1,2;0}$, where $k-k'-1$ is odd. The corresponding relation in Br$(e,e,3)$ is a case of Relation 7 of Definition \ref{DefinitionBr(e,e,n)eodd}.

\end{itemize}

\noindent The case when $k$ and $k'$ are both odd is done in the same way.\\

We now consider Case (1.3) because it is similar to the previous case. We have $l = 0$ and $w = t_{\frac{k+k'+e}{2}} \in \mathcal{R}$ with $k + k'$ odd. Suppose $k$ even and $k'$ odd. We have $t_{\frac{k+k'+e}{2}} = \underset{k+k'+e-1}{\underbrace{t_1t_0 \cdots t_1}}$. Relation 5 of Definition \ref{DefinitionChenBrauer(e,e,3)} is\\ $\underset{k-1}{\underbrace{t_1t_0 \cdots t_1}}\ e_{1,2;0}\ \underset{k-1}{\underbrace{t_1t_0 \cdots t_1}}\ \underset{k'-1}{\underbrace{t_1t_0 \cdots t_0}}\ e_{1,2;1}\ \underset{k'-1}{\underbrace{t_0t_1 \cdots t_1}} =\\ \underset{k+k'+e-1}{\underbrace{t_1t_0 \cdots t_1}}\ \underset{k'-1}{\underbrace{t_1t_0 \cdots t_0}}\ e_{1,2;1}\ \underset{k'-1}{\underbrace{t_0t_1 \cdots t_1}} = \underset{k+e}{\underbrace{t_1t_0 \cdots t_1}}\ e_{1,2;1}\ \underset{k'-1}{\underbrace{t_0t_1 \cdots t_1}} =\\ \underset{k-1}{\underbrace{t_1t_0 \cdots t_1}}\ \underset{e+1}{\underbrace{t_0t_1 \cdots t_1}}\ e_{1,2;1}\ \underset{k'-1}{\underbrace{t_0t_1 \cdots t_1}}$. After simplification by $\underset{k-1}{\underbrace{t_1t_0 \cdots t_1}}$ on the left and by $\underset{k'-1}{\underbrace{t_0t_1 \cdots t_1}}$ on the right, we get $e_{1,2;0}\ \underset{k-1}{\underbrace{t_1t_0 \cdots t_1}}\ \underset{k'-1}{\underbrace{t_1t_0 \cdots t_0}}\ e_{1,2;1} = \underset{e}{\underbrace{t_0t_1 \cdots t_0}}\ e_{1,2;1} = \underset{e-1}{\underbrace{t_1t_0 \cdots t_0}}\ e_{1,2;1}$. Multiplying by $\underset{e-1}{\underbrace{t_0t_1 \cdots t_1}}$ on the left, we get\\ $\underset{e-1}{\underbrace{t_0t_1 \cdots t_1}}\ e_{1,2;0}\ \underset{k-1}{\underbrace{t_1t_0 \cdots t_1}}\ \underset{k'-1}{\underbrace{t_1t_0 \cdots t_0}}\ e_{1,2;1} = e_{1,2;1}$.

\begin{itemize}
\item If $k < k'$, we get $\underset{e-1}{\underbrace{t_0t_1 \cdots t_1}}\ e_{1,2;0}\ \underset{k-1}{\underbrace{t_1t_0 \cdots t_1}}\ \underset{k'-1}{\underbrace{t_1t_0 \cdots t_0}}\ e_{1,2;1} =$\\ $\underset{e-1}{\underbrace{t_0t_1 \cdots t_1}}\ e_{1,2;0}\ \underset{k'-k}{\underbrace{t_0t_1 \cdots t_0}}\ e_{1,2;1} = \underset{e-1}{\underbrace{t_0t_1 \cdots t_1}}\ e_{1,2;0}\ \underset{k'-k-1}{\underbrace{t_1t_0 \cdots t_0}}\ e_{1,2;1}$ that is equal to $e_{1,2;1}\ \underset{e-1}{\underbrace{t_0t_1 \cdots t_1}}\ \underset{k'-k-1}{\underbrace{t_1t_0 \cdots t_0}}\ e_{1,2;1}$ by the second item of Lemma \ref{LemmaRelBrauerChenJet2}.\\ Since $0 \leq k'-k \leq e-1$, it is equal to $e_{1,2;1}\ \underset{e-k'+k}{\underbrace{t_0t_1 \cdots t_1}}\ e_{1,2;1} = e_{1,2;1}\ \underset{e-k'+k-1}{\underbrace{t_0t_1 \cdots t_0}}\ e_{1,2;1}$. Hence in $\mathcal{B}(e,e,3)$, we have $e_{1,2;1}\ \underset{e-k'+k-1}{\underbrace{t_0t_1 \cdots t_0}}\ e_{1,2;1} = e_{1,2;1}$.\\ Since $0 \leq e-k'+k-1 \leq e-1$ and $e-k'+k-1$ is odd, by Relation 6 of Definition \ref{DefinitionBr(e,e,n)eodd}, it is clear that the image of this relation by $\phi$ is satisfied in Br$(e,e,3)$.

\item If $k > k'$, we get $e_{1,2;1}\ \underset{e-1}{\underbrace{t_0t_1 \cdots t_1}}\ \underset{k-k'}{\underbrace{t_1t_0 \cdots t_1}}\ e_{1,2;1} = e_{1,2;1}$, that is\\ $e_{1,2;1}\ \underset{k'-k+e-1}{\underbrace{t_0t_1 \cdots t_0}}\ e_{1,2;1} = e_{1,2;1}$. This is similar to the previous case.

\end{itemize}

\noindent The case when $k$ is odd and $k'$ is even is done in the same way.\\

Consider Case (1.2). We have $w = t_xs_3t_lt_0s_3$ with $k+k'+l$ even and $x = \frac{k+k'+l}{2}$. Relation 4 of Definition \ref{DefinitionChenBrauer(e,e,3)} is $t_xs_3t_lt_0s_3\ e_{1,2;k} = e_{1,2;k'}\ t_xs_3t_lt_0s_3$. We need to show that $T_xS_3T_lT_0S_3\ E_k = E_{k'}\ T_xS_3T_lT_0S_3$ in Br$(e,e,3)$, that is $$T_xS_3T_lT_0S_3\ E_k\ S_3T_0T_lS_3T_x = E_{k'}.$$ We have $T_xS_3T_lT_0S_3\ E_k\ S_3T_0T_lS_3T_x = T_xS_3T_lT_0T_k\ F_3\ T_kT_0T_lS_3T_x$ by Equation \eqref{LastEquation} of Proposition \ref{PropAdditionalRelations2} and $T_lT_0T_k = \underset{2l-1}{\underbrace{T_1T_0 \cdots T_1}}\ T_0\ \underset{2k-1}{\underbrace{T_1T_0 \cdots T_1}} = \underset{2(k+l)-1}{\underbrace{T_1T_0 \cdots T_1}} = T_{k+l}$. Hence
$T_xS_3T_lT_0T_k\ F_3\ T_kT_0T_lS_3T_x = T_xS_3T_{k+l}\ F_3\ T_{k+l}S_3T_x = T_xS_3S_3\ E_{k+l}\ S_3S_3T_x = T_x\ E_{k+l}\ T_x$. We prove that $T_x\ E_{k+l}\ T_x = E_{k'}$.

Suppose $k+l$ even. Since $k+k'+l$ is even, we have $k'$ even. We have\\ $E_{k+l} = \underset{k+l-1}{\underbrace{T_1T_0 \cdots T_1}}\ E_0\ \underset{k+l-1}{\underbrace{T_1T_0 \cdots T_1}}$ and $T_x\ E_{k+l}\ T_x =\\ \underset{k+k'+l-1}{\underbrace{T_1T_0 \cdots T_1}}\ \underset{k+l-1}{\underbrace{T_1T_0 \cdots T_1}}\ E_0\ \underset{k+l-1}{\underbrace{T_1T_0 \cdots T_1}}\ \underset{k+k'+l-1}{\underbrace{T_1T_0 \cdots T_1}} = \underset{k'}{\underbrace{T_1T_0 \cdots T_0}}\ E_0\ \underset{k'}{\underbrace{T_0T_1 \cdots T_1}} = \\ \underset{k'-1}{\underbrace{T_1T_0 \cdots T_1}}\ E_0\ \underset{k'-1}{\underbrace{T_1T_0 \cdots T_1}} = E_{k'}$.

The case $k+l$ odd is done in the same way.\\

Case (1.4) is similar to Case (1.2).\\

Consider Case (1.5). We have $k-k'-l$ is even and $w = t_xt_0s_3t_lt_0s_3$ with \mbox{$x = \frac{k'-k-l}{2}$.} Relation 4 of Definition \ref{DefinitionChenBrauer(e,e,3)} is $t_xt_0s_3t_lt_0s_3\ e_{1,2;k} = e_{1,2;k'}\ t_xt_0s_3t_lt_0s_3$. We prove the following relation in Br$(e,e,3)$: $$T_xT_0S_3T_lT_0S_3\ E_k\ S_3T_0T_lS_3T_0T_x = E_{k'}.$$ Actually, $T_xT_0S_3T_lT_0S_3\ E_k\ S_3T_0T_lS_3T_0T_x = T_xT_0S_3T_lT_0T_k\ F_3\ T_kT_0T_lS_3T_0T_x =\\ T_xT_0S_3T_{k+l}\ F_3\ T_{k+l}S_3T_0T_x = T_xT_0S_3^2\ E_{k+l}\ S_3^2T_0T_x = T_xT_0\ E_{k+l}\ T_0T_x$. One can easily check that for both cases $k+l$ even (then $k'$ is even) and $k+l$ odd (then $k'$ is odd) we have $T_xT_0\ E_{k+l}\ T_0T_x = E_{k'}$.\\

Case (1.6) is similar to Case (1.5).\\

Consider Case (2.1). In Br$(e,e,3)$, we show that $$T_{k'}T_{l}T_{0}S_3T_{\frac{k'-k-l}{2}}T_0\ E_k\ T_0T_{\frac{k'-k-l}{2}}S_3T_{0}T_lT_{k'} = F_3.$$ This is the first identity of Lemma \ref{LemmaRelBrJet3}.\\

Cases (2.2), (2.4), and (2.5) are similar to Case (2.1).\\

Consider Case (2.3). The identity we prove in Br$(e,e,3)$ for this case is given by the second item of Lemma \ref{LemmaRelBrJet2}.\\

Cases (3.1), (3.2), (3.4), and (3.5) are done in the same way as the first item of Lemma \ref{LemmaRelBrJet3}.\\ 

Case (3.3) is similar to the first item of Lemma \ref{LemmaRelBrJet2}.\\

Cases (4.1) and (4.3) are done in the same way as the second item of Lemma \ref{LemmaRelBrJet1}.\\

Consider Case (4.2).  In Br$(e,e,3)$, we prove 
$$T_{\frac{-k-k'-l}{2}}T_0S_3T_{\frac{k+k'-l}{2}}T_k\ F_3\ T_kT_{\frac{k+k'-l}{2}}S_3T_0T_{\frac{-k-k'-l}{2}} = T_0T_{k'}\ F_3\ T_{k'}T_0.$$

Suppose $k \leq k'-l$. We have $T_{\frac{k+k'-l}{2}}T_k = \underset{k'-l-k}{\underbrace{T_1T_0 \cdots T_0}}$. Similarly, we have $T_kT_{\frac{k+k'-l}{2}} = \underset{k'-k-l}{\underbrace{T_0T_1 \cdots T_1}}$, $T_{k'}T_0T_{\frac{-k-k'-l}{2}}T_0 = \underset{k'-k-l}{\underbrace{T_1T_0 \cdots T_0}}$, and $T_0T_{\frac{-k-k'-l}{2}}T_0T_{k'} = \underset{k'-k-l}{\underbrace{T_0T_1 \cdots T_1}}$. Then we should prove

$$\underset{k'-k-l}{\underbrace{T_1T_0 \cdots T_0}}\ S_3\ \underset{k'-k-l}{\underbrace{T_1T_0 \cdots T_0}}\ F_3\ \underset{k'-k-l}{\underbrace{T_0T_1 \cdots T_1}}\ S_3\ \underset{k'-k-l}{\underbrace{T_0T_1 \cdots T_1}} = F_3.$$
Set $k'-k-l = 2x$. The equation we prove is equivalent to $$T_xT_0S_3T_xT_0\ F_3\ T_0T_xS_3T_0T_x = F_3.$$
The left-hand side is equal to $T_xT_0S_3T_xS_3E_0S_3T_xS_3T_0T_x = $
$T_xT_0T_xS_3T_xE_0T_xS_3T_xT_0T_x\\ = T_{2x}S_3T_x\ E_0\ T_xS_3T_{2x} = T_{2x}S_3\ E_{2x}\ S_3T_{2x} = T_{2x}^2 F_3 T_{2x}^2 = F_3$.

The case when $k > k'-l$ is done in the same way.\\

Case (4.4) is done in the same way as Case (4.2).\\

Consider Case (4.5). In Br$(e,e,3)$, we show

$$T_{k'}T_0T_{\frac{k-k'-l}{2}}T_0S_3T_{\frac{k'-k-l}{2}}T_0S_3T_0T_k\ F_3\ T_kT_0S_3T_0T_{\frac{k-k'-l}{2}}S_3T_0T_{\frac{k-k'-l}{2}}T_0T_{k'} = F_3.$$

For $0 \leq x \leq e-1$, we have $S_3(T_1T_0)^x S_3T_0 = S_3(T_1T_0)^{x-1}T_1T_0S_3T_0 =\\ S_3(T_1T_0)^{x-1}T_1S_3T_0S_3 = S_3T_xS_3T_0S_3 = T_xS_3T_xT_0S_3$.

 Also for $0 \leq x \leq e-1$, we have $S_3(T_1T_0)^x S_3T_1 = S_3(T_1T_0)^{x-1}T_2T_1S_3T_1 = S_3(T_1T_0)^{x-1}T_1T_0T_1S_3T_1S_3 = S_3T_{x+1}S_3T_1S_3 = T_{x+1}S_3T_{x+1}T_1S_3$.
 
Similarly, one can prove that for $0 \leq x \leq e-1$, we have $T_0S_3(T_0T_1)^xS_3 = S_3T_0T_xS_3T_x$ and $T_1S_3(T_0T_1)^xS_3 = S_3T_0T_xS_3T_{x+1}$.\\

Suppose $k'-k-l > 0$. Using the previous relations, we have $S_3T_{\frac{k'-k-l}{2}}T_0S_3T_0T_k = S_3(T_1T_0)^{\frac{k'-k-l}{2}}S_3T_0T_k = T_{\frac{k'-k-l}{2}}S_3T_{\frac{k'-k-l}{2}}T_0S_3\ \underset{2k-1}{\underbrace{T_1T_0 \cdots T_1}} =\\ T_{\frac{k'-k-l}{2}}S_3(T_1T_0)^{\frac{k'-k-l}{2}}S_3T_1\ \underset{2k-2}{\underbrace{T_0T_1 \cdots T_1}} =\\ T_{\frac{k'-k-l}{2}}T_{\frac{k'-k-l}{2} +1} S_3 T_{\frac{k'-k-l}{2} +1}T_1S_3\ \underset{2k-2}{\underbrace{T_0T_1 \cdots T_1}} = \cdots = (T_0T_1)^kS_3T_{\frac{k'-k-l}{2}}T_0S_3$.\\
Also we get $T_kT_0S_3T_0T_{\frac{k'-k-l}{2}}S_3 = \underset{2k-1}{\underbrace{T_1T_0 \cdots T_1}}\ T_0 S_3(T_0T_1)^{\frac{k'-k-l}{2}}S_3 =\\ \underset{2k-1}{\underbrace{T_1T_0 \cdots T_1}}\ S_3 T_0T_{\frac{k'-k-l}{2}}S_3T_{\frac{k'-k-l}{2}} = \underset{2k-2}{\underbrace{T_1T_0 \cdots T_0}}\ S_3 T_0T_{\frac{k'-k-l}{2}}\ S_3T_{\frac{k'-k-l}{2}+1}T_{\frac{k'-k-l}{2}} = \\ \underset{2k-2}{\underbrace{T_1T_0 \cdots T_0}}\ S_3T_0T_{\frac{k'-k-l}{2}}S_3T_1T_0 = \cdots = S_3T_0T_{\frac{k'-k-l}{2}}S_3(T_1T_0)^k$. The equation we prove is then
$$T_{k'}T_0T_{\frac{k-k'-l}{2}}T_0(T_0T_1)^kS_3T_{\frac{k'-k-l}{2}}T_0\ F_3\ T_0T_{\frac{k'-k-l}{2}}S_3(T_1T_0)^kT_0T_{\frac{k-k'-l}{2}}T_0T_{k'} = F_3.$$
This is Relation 2 of Lemma \ref{LemmaRelBrJet3}. Note that the case when $k'-k-l < 0$ is similar after replacing $T_{\frac{k'-k-l}{2}}$ by $T_{\frac{k'-k-l}{2} + e}$, where $0 \leq \frac{k'-k-l}{2} + e \leq e-1$.\\

Case (4.6) is similar to Case (4.5).\\

Consider Case (5.1). In Br$(e,e,3)$, we show that
$$T_{\frac{k+k'+l}{2}}S_3T_{\frac{k+k'-l}{2}}T_k\ F_3\ T_kT_0T_0T_{\frac{k+k'-l}{2}}S_3T_{\frac{k+k'+l}{2}} = T_{k'}\ F_3\ T_{k'}.$$
This is similar to Case (4.2). Also Case (5.2) is similar to Case (5.1).\\

Consider Case (5.4). In Br$(e,e,3)$, we show that 
$$T_{\frac{k'-k+l}{2}}S_3T_{\frac{k'-k-l}{2}}T_0S_3T_0T_k\ F_3\ T_kT_0S_3T_0T_{\frac{k'-k-l}{2}}S_3T_{\frac{k'-k+l}{2}} = T_{k'}\ F_3\ T_{k'}.$$
This is done in the same way as Case (4.5). Also Case (5.5) is similar to Case (5.4).\\

Consider Case (5.3). In Br$(e,e,3)$, we show 
$$T_0T_kF_3T_kT_0\ T_{k'}F_3T_{k'} = T_{k'-k}\ T_{k'}F_3T_{k'} = T_0T_kF_3T_kT_0\ T_{k'-k}.$$
This is Relation 1 of Lemma \ref{LemmaRelBrJet2}.\\

Note that Cases (6.2) and (6.4) are similar. They are done in the same way as Case (4.2). Also, Cases (6.5) and (6.6) are similar. They are done in the same way as Case (4.5). Case (6.3) is similar to Case (6.1). Consider Case (6.1):\\
In Br$(e,e,3)$, we show
$$T_kF_3T_k\ T_{k'}F_3T_{k'} = T_{\frac{k+k'}{2}}S_3T_{\frac{k+k'}{2}}\ T_{k'}F_3T_{k'} = T_kF_3T_k\ T_{\frac{k+k'}{2}}S_3T_{\frac{k+k'}{2}}.$$
This is Relation 2 of Lemma \ref{LemmaRelBrJet1}.

\end{proof}

The following proposition is a direct consequence of Propositions \ref{PropositionIsomChenPsi} and \ref{PropositionIsomChenPhi}.

\begin{proposition}\label{PropositionChenIsomBr(e,e,3)}

The algebra \emph{Br}$(e,e,3)$ is isomorphic to $\mathcal{B}(e,e,3)$ for all $e \geq 3$ and $e$ odd.

\end{proposition}

\begin{proof}

It is readily checked that $\phi \circ \psi$ is equal to the identity morphism on Br$(e,e,3)$. It is also straightforward to check that $\psi \circ \phi$ is equal to the identity morphism on $\mathcal{B}(e,e,3)$. The only non-trivial cases are to check that $\psi \circ \phi(e_{2,3;k}) = e_{2,3;k}$ and $\psi \circ \phi(e_{1,3;k}) = e_{1,3;k}$. 

For the first case, we have  $\psi \circ \phi(e_{2,3;k}) = \psi(T_0T_kF_3T_kT_0) = t_0t_k e_{2,3;0} t_kt_0$. Since $t_kt_0s_{2,3;k}t_0t_k$ is equal to $s_{2,3;0}$, by Relation 4 of Definition \ref{DefinitionChenBrauer(e,e,3)}, we have $t_kt_0 e_{2,3;k} = e_{2,3;0} t_kt_0$, that is $e_{2,3;k} = t_0t_k e_{2,3;0} t_kt_0$. Hence $\psi \circ \phi(e_{2,3;k})$ is equal to $e_{2,3;k}$.

For the second case, we have $\psi \circ \phi(e_{1,3;k}) = \psi(T_kF_3T_k) = t_k e_{2,3;0} t_k$. Since $s_{1,3;k} = t_ks_{2,3;0}t_k$, by Relation 4 of Definition \ref{DefinitionChenBrauer(e,e,3)}, we have $t_k e_{2,3;0} t_k = e_{1,3;k}$. Hence\\ $\psi \circ \phi(e_{1,3;k})$ is equal to $e_{1,3;k}$. 

In conclusion, the algebra Br$(e,e,3)$ is isomorphic to $\mathcal{B}(e,e,3)$ for all $e \geq 3$ and $e$ odd.

\end{proof}

\begin{cor}

The set $\{wE_k\} \cup \{wF_3\}$ for $w \in G(e,e,3)$ and $0 \leq k \leq e-1$ is a generating set of \emph{Br}$(e,e,3)$ for $e$ odd.

\end{cor}

\begin{proof}

By Proposition \ref{PropositionChenIsomBr(e,e,3)}, since Br$(e,e,3)$ is isomorphic to $\mathcal{B}(e,e,3)$, we get the generating set of Br$(e,e,3)$ described by Chen in Theorem 5.1 of \cite{ChenFlatConnectionsBrauerAlgebras}.

\end{proof}

In Proposition \ref{PropositionChenIsomBr(e,e,3)}, we proved that the Brauer algebra given in Definition \ref{DefinitionBr(e,e,n)eodd} is isomorphic to the Brauer-Chen algebra for $n = 3$ and $e$ odd. Even if this isomorphism corresponds to a small value of $n$ and $e$ odd, it is tempting to think that the Brauer-Chen algebra associated to a complex reflection group $G(e,e,n)$ defined by Chen in \cite{ChenFlatConnectionsBrauerAlgebras} is isomorphic to the Brauer algebra Br$(e,e,n)$ given in Definitions \ref{DefinitionBr(e,e,n)eodd} and \ref{DefinitionBr(e,e,n)eeven} for all $e$ and $n$. This enables us to ask the following question: Is the Brauer-Chen algebra isomorphic to Br$(e,e,n)$ for all $e$ and $n$?

\section{Constructing Krammer's representations}

Recall that in type ADE of Coxeter groups, the generalized Krammer's representations can be constructed via the BMW algebras. We attempt to construct explicit linear representations for the complex braid groups $B(e,e,n)$ by using the algebra BMW$(e,e,n)$. We call these representations the Krammer's representations for $B(e,e,n)$. We adopt a computational approach using the package GBNP (version 1.0.3) of GAP4 (\cite{GBNP}). We are able to construct Krammer's representations for $B(3,3,3)$ and $B(4,4,3)$. We also provide some conjectures about the BMW algebra and about Krammer's representations, see Conjectures \ref{ConjectureRho3and4faithful} and \ref{ConjectureStructureBMW}. We use the platform MATRICS (\cite{MATRICS}) of Universit\'e de Picardie Jules Verne for our heuristic computations. We restrict to the case $n = 3$ in order to make these computations more efficient.\\

The algebra BMW$(e,e,n)$ is defined by a presentation with generators and relations, see Definitions \ref{DefinitionBMWB(e,e,n)eOdd} and \ref{DefinitionBMWB(e,e,n)eEven}. The list of relations of BMW$(e,e,3)$ for a given $e$ is defined with GBNP as a list \verb|L| of non-commutative polynomials. Once a Gröbner basis of \verb|L| is computed, BMW$(e,e,3)$ is the quotient of the free algebra \mbox{$<T_0,T_1,S_3,E_0,E_1,F_3>$} by the two-sided ideal generated by the list of the non-commutative polynomials (or its Gröbner basis). The indeterminates $m$ and $l$ of BMW$(e,e,3)$ are specialized over $\mathbb{Q}$ or over a finite field when the computation over the rationals is very heavy. For the implementation, see Appendix B.\\

We compute a basis of the quotient algebra BMW$(e,e,3)$ by using the function \verb|BaseQA|. The dimension of BMW$(e,e,3)$ can be computed with \verb|DimQA|. Right multiplication by an element of BMW$(e,e,3)$ is a linear transformation $A_r: BMW \longrightarrow End_{\mathbb{Q}}(BMW)$. The matrix of this linear transformation with respect to the base is computed by using the function \verb|MatrixQA|. Denote $R_1$, $R_2$, $\cdots$, $R_6$ the images by $A_r$ of $T_0$, $T_1$, $S_3$, $E_0$, $E_1$, and $F_3$, respectively. We define the left multiplication to be $A_l: BMW \longrightarrow End_{\mathbb{Q}}(BMW)$ with $L_1$, $L_2$, $\cdots$, $L_6$ the images by $A_l$ of $T_0$, $T_1$, $S_3$, $E_0$, $E_1$, and $F_3$, respectively.\\

\textbf{Krammer's representation for $B(3,3,3)$}\\

Here we explain the algorithm that enables us to construct Krammer's representation for $B(3,3,3)$, see Appendix B for the implementation.\\

For many specializations of $m$ and $l$ over $\mathbb{Q}$, by using the function \verb|DimQA|, we get that the dimension of BMW$(3,3,3)$ is equal to $297$. We remark that $297 = 54 + 3 \times 9^2 =\#(G(3,3,3)) + 3 \times (\#\mathcal{R})^2$, where $\mathcal{R}$ is the set of reflections of the complex reflection group $G(3,3,3)$.\\

After computing $R_1$, $R_2$, $\cdots$, and $R_6$ with GBNP, we define the matrix algebra $A$ over $\mathbb{Q}$ generated by $R_1$, $R_2$, $\cdots$, and $R_6$. Then we compute the central idempotents of $A$ by using the function \verb|CentralIdempotentsOfAlgebra| of GAP4. We get a central idempotent of trace equal to $243 = 3 \times 9^2$. Denote this element by $i_1$.\\

We use the generator of the center of the complex braid group $B(3,3,3)$ that is $T_1T_0S_3T_1T_0S_3$. We compute its image by $A_r$ that we denote by $N_z$. We compute $i_1 N_z$ in the basis $(i_1,i_1N_z,i_1{N_z}^2)$. For many specializations of the parameters $m$ and $l$, we clearly have $i_1 {N_z}^3 = \frac{1}{l^2}i_1$ and we guess that the minimal polynomial of $i_1 N_z$ is $X^3 - \frac{1}{l^2}$. One can set $l = \lambda^3$ for a rational specialization of $\lambda$ and redefine the algebra $A$ over $\mathbb{Q}$ and compute its central idempotents. We get an idempotent element of trace equal to $81$ and $i_1 A$ is the algebra corresponding to Krammer's representation. This method will be used for $B(4,4,3)$. However, we find a faster way to construct the representation for $B(3,3,3)$ without computing the central idempotent elements. We describe this method now.\\

In BMW$(e,e,3)$, we have $E_i = \frac{l}{m}T_i^2 + l T_i - \frac{l}{m}$ and $E_i = x E_i$ for $i = 0, 1$, where $x = \frac{ml-l^2+1}{ml}$. Let $E_i = \frac{ml-l^2+1}{ml}$, we get $ \frac{l}{m}T_i^2 + l T_i - \frac{l}{m} = \frac{ml-l^2+1}{ml}$, that is $\frac{l}{m}T_i^2 + lT_i -1-\frac{1}{lm} = 0$. We factorize this equation and get $(T_i - \frac{1}{l})(T_i +m+\frac{1}{l}) = 0$. We determine $$Ker(N_z - \lambda^{-2}) \cap Ker(R_1 - \lambda^{-3}).$$ We get a vector space of dimension $9$ stable by left multiplication and define Krammer's representation ${\rho}_3$ by providing $A = \rho_3(T_0)$, $B = \rho_3(T_1)$, and $C = \rho_3(S_3)$ in $\mathcal{M}_{9\times 9}(\mathbb{Q}(m,\lambda))$ after interpolation of the coefficients over $\mathbb{Q}$ in order to get the corresponding coefficients over $\mathbb{Q}(m,\lambda)$. We checked that $A$, $B$, and $C$ satisfy all the relations of $BMW(3,3,3)$. The representation is defined over $\mathbb{Q}(m,\lambda)$ as follows. 

\begin{center}
$A = \begin{bmatrix}

0 & 0 & 1 & 0 & 0 & 0 & 0 & 0 & 0\\

0 & 0 & 0 & 1 & 0 & 0 & 0 & 0 & 0\\

1 & 0 & -m & 0 & 0 & 0 & \frac{m(\lambda -m)}{\lambda^3} & 0 & 0\\

0 & 1 & 0 & -m & 0 & 0 & \frac{m}{\lambda^6} & 0 & 0\\

0 & m^2 & 0 & m & 0 & 1 & \frac{m}{\lambda^4} & 0 & 0\\

0 & -m(m^2+1) & 0 & -m^2 & 1 & -m & 0 & 0 & 0\\

0 & 0 & 0 & 0 & 0 & 0 & \frac{1}{\lambda^3} & 0 & 0\\

-\frac{m^2}{\lambda^3} & -m(m^2+1) & -\frac{m}{\lambda^3} & -m^2 & 0 & 0 & A_{8,7} & 0 & 1\\

\frac{m(m^2+1)}{\lambda^3} & m^2(m^2+2) & \frac{m^2}{\lambda^3} & m(m^2+1) & 0 & 0 & A_{9,7} & 1 & -m\\

\end{bmatrix}$
\end{center}

\bigskip

\noindent with $A_{8,7} = \frac{m(\lambda^4 -m\lambda^3 -m(m^2+1)\lambda + m^2)}{\lambda^7}$ and $A_{9,7} = \frac{-m(\lambda^4 m - \lambda^2 m + \lambda m^2 + \lambda -m)}{\lambda^7}$.

\begin{center}
\mbox{$B = \begin{bmatrix}

\frac{1}{\lambda^3} & 0 & 0 & 0 & 0 & 0 & 0 & 0 & 0\\

0 & -m & 0 & 0 & 0 & \frac{1}{\lambda -m} & 0 & 0 & 0\\

0 & 0 & 0 & 0 & 0 & 0 & \lambda-m & 0 & 0\\

\frac{-m}{\lambda^3} & 0 & \frac{-1}{\lambda^3} & -m & 0 & 1 & \frac{1}{\lambda^3} & 0 & \frac{\lambda}{m(\lambda-m)}\\

\frac{m(\lambda + m)}{\lambda^4} & m\lambda & 0 & \lambda & \lambda -m & \frac{m(\lambda m + 1)}{\lambda-m} & 0 & \frac{\lambda}{m} & \frac{m\lambda}{\lambda-m}\\

\frac{m}{\lambda^4} & \lambda-m & 0 & 0 & 0 & 0 & 0 & 0 & 0\\

\frac{m}{\lambda-m} & 0 & \frac{1}{\lambda-m} & 0 & 0 & 0 & -m & 0 & 0\\

B_{8,1} & B_{8,2} & B_{8,3} & B_{8,4} & B_{8,5} & B_{8,6} & B_{8,7} & B_{8,8} & B_{8,9}\\

B_{9,1} & B_{9,2} & \frac{-m}{\lambda^4} & \frac{m(\lambda-m)}{\lambda} & 0 & 0 & \frac{m(\lambda-m)}{\lambda^3} & 0 & 0\\

\end{bmatrix}$}
\end{center}

\bigskip

\noindent with $B_{8,1} = \frac{-m^3(\lambda^3 -2m\lambda^2 - \lambda +m)}{\lambda^6(\lambda-m)}$, $B_{8,2} = \frac{-m^2(\lambda^2 - \lambda m + m^2 + 1)}{\lambda}$, $B_{8,3} = \frac{m(\lambda^2-\lambda m + m^2)}{\lambda^4(\lambda -m)}$, $B_{8,4} = \frac{-m(\lambda^2-\lambda m + m^2)}{\lambda}$, $B_{8,5} = \frac{m(-\lambda^2+\lambda m + 1)}{\lambda}$, $B_{8,6} = \frac{-m(\lambda m^2 + \lambda +m)}{\lambda -m}$, $B_{8,7} = \frac{-m(m^2+1)}{\lambda^3}$, $B_{8,8} = -\lambda$, $B_{8,9} = \frac{-\lambda(m^2+1)}{\lambda -m}$, $B_{9,1} = \frac{-2m^2(\lambda^2-1)}{\lambda^6}$, and $B_{9,2} = \frac{m(-\lambda^2 + 2\lambda m - m^2)}{\lambda}$.

\begin{center}
\mbox{$C = \begin{bmatrix}

-m & 0 & 0 & 0 & 0 & \lambda & 0 & 0 & 0\\

0 & \frac{1}{\lambda^3} & 0 & 0 & 0 & 0 & 0 & 0 & 0\\

m^2 & m\lambda^3 & 0 & \lambda^3 & \lambda^3 & -m\lambda^3 & -m & \frac{\lambda^3}{m} & -\lambda^3\\

0 & 0 & 0 & 0 & 0 & 0 & \frac{1}{\lambda^3} & 0 & 0\\

\frac{-m}{\lambda} & \frac{m(\lambda m +1)}{\lambda^3} & \frac{-1}{\lambda} & 0 & 0 & \lambda^2 & \frac{1}{\lambda} & 0 & \frac{\lambda^2}{m}\\

\frac{1}{\lambda} & \frac{m(\lambda-m)}{\lambda^3} & 0 & 0 & 0 & 0 & 0 & 0 & 0\\

0 & m\lambda^3 & 0 & \lambda^3 & 0 & 0 & -m & 0 & 0\\

C_{8,1} & C_{8,2} & C_{8,3} & 0 & \frac{m^2}{\lambda^2} & -\lambda^2m-m^3 & C_{8,7} & 0 & -\lambda^2-m^2\\

C_{9,1} & C_{9,2} & \frac{m^2}{\lambda^3} & 0 & \frac{m(\lambda^2+1)}{\lambda^2} & \frac{-m^2(2\lambda^2-1)}{\lambda^2} & \frac{-m^2}{\lambda^3} & 1 & -2m\\

\end{bmatrix}$}
\end{center}

\bigskip

\noindent with $C_{8,1} = \frac{m^2(\lambda^2+m^2)}{\lambda^3}$, $C_{8,2} = \frac{-m^2(\lambda^2m + m^3 + \lambda + m)}{\lambda^4}$, $C_{8,3} = \frac{m(\lambda^2+m^2+1)}{\lambda^3}$,\\ $C_{8,7} = \frac{-m(\lambda^2+m^2+1)}{\lambda^3}$, $C_{9,1} = \frac{-m(\lambda^2-m^2)}{\lambda^3}$, and $C_{9,2} = \frac{-m^2(\lambda^2-\lambda m + m^2)}{\lambda^4}$.\\

Absolute irreducibility being an open condition (see \cite{LubotzkyMagid}), it is sufficient to check it for one arbitrary specialization. For a rational specialization of $m$ and $\lambda$ (for example $m = 2$ and $\lambda = 17$ as in Appendix B), we check that the dimension of the matrix algebra generated by $A$, $B$, and $C$ is equal to $81$. Hence $\rho_3$ is absolutely irreducible over $\mathbb{Q}(m,\lambda)$.\\

We get two other irreducible Krammer's representations $\rho_3'$ and $\rho_3''$ for $B(3,3,3)$ by replacing $\lambda$ by $\zeta_3 \lambda$ and $\zeta_3^2 \lambda$, respectively. The trace of the matrix $BAC$ is equal to $\frac{1}{\lambda}$. Hence $\rho_3$, $\rho_3'$, and $\rho_3''$ are pairwise non-isomorphic. We deduce that if BMW$(3,3,3)$ is of dimension $297$ over $\mathbb{Q}(m,l)$, we know all its irreducible representations: those from the representations of the Hecke algebra $H(3,3,3)$ and those that correspond to the Krammer's representations. In particular, it is semi-simple. This is one of the reasons that led us to state Conjecture \ref{ConjectureStructureBMW} at the end of this section.\\ 

Let $m = r-\frac{1}{r}$ and $\lambda = \frac{1}{rt}$, we checked that the restriction of the representation $\rho_3$ to a parabolic subgroup of type $A_2$ is isomorphic to the direct sum $\mathcal{K}$ of the Cohen-Gijsbers-Wales representation of type $A_2$ and of the Hecke algebra representation associated to the permutation action of the symmetric group $S_3$ on $\mathcal{R} \setminus \{s_{1,2;0},s_{2,3;0}, s_{1,3;0}\}$, where $\mathcal{R}$ is the set of reflections of $G(3,3,3)$. Note that this property is satisfied by the monodromy representation, see Proposition 4.7 in \cite{MarinLinearity}. The representation $\mathcal{K}$ is as follows.\\

\begin{center}
\mbox{$ \mathcal{K}(T_0) = \begin{bmatrix}

t^3r^3 & 0 & t^3(r^4-r^2) & 0 & 0 & 0 & 0 & 0 & 0 \\
0 & 1/r-r & 1 & 0 & 0 & 0 & 0 & 0 & 0 \\
0 & 1 & 0 & 0 & 0 & 0 & 0 & 0 & 0\\
0 & 0 & 0 & 1/r-r & 0 & 0 & 1 & 0 & 0\\
0 & 0 & 0 & 0 & 0 & 0 & 0 & 1 & 0\\
0 & 0 & 0 & 0 & 0 & 1/r-r & 0 & 0 & 1\\
0 & 0 & 0 & 1 & 0 & 0 & 0 & 0 & 0\\
0 & 0 & 0 & 0 & 1 & 0 & 0 & 1/r-r & 0\\
0 & 0 & 0 & 0 & 0 & 1 & 0 & 0 & 0\\

\end{bmatrix}$}
\end{center}

\begin{center}
\mbox{$ \mathcal{K}(S_3) = \begin{bmatrix}

1/r-r & 0 & 1 & 0 & 0 & 0 & 0 & 0 & 0 \\
0 & t^3r^3 & t^3(r^4-r^2) & 0 & 0 & 0 & 0 & 0 & 0 \\
1 & 0 & 0 & 0 & 0 & 0 & 0 & 0 & 0\\
0 & 0 & 0 & 1/r-r & 0 & 0 & 0 & 1 & 0\\
0 & 0 & 0 & 0 & 1/r-r & 0 & 0 & 0 & 1\\
0 & 0 & 0 & 0 & 0 & 0 & 1 & 0 & 0\\
0 & 0 & 0 & 0 & 0 & 1 & 1/r-r & 0 & 0\\
0 & 0 & 0 & 1 & 0 & 0 & 0 & 0 & 0\\
0 & 0 & 0 & 0 & 1 & 0 & 0 & 0 & 0\\

\end{bmatrix}$}
\end{center}

\bigskip


\textbf{Krammer's representation for $B(4,4,3)$}\\

The method used for $B(4,4,3)$ is similar to the one used for $B(3,3,3)$. We get for many specializations of the parameters of the BMW algebra $v$ and $l$ that the dimension of BMW$(4,4,3)$ is equal to $384$. We remark that $384 = 96 + 2\times 12^2 = \#(G(4,4,3)) + 2\times (\#\mathcal{R})^2$, where $\mathcal{R}$ is the set of reflections of $G(4,4,3)$.\\

In order to get a model of Krammer's representation over $\mathbb{Q}$ using GAP, we set a rational specialization of $v$ and $l$ such that $v = \mu^2$ and $l = \lambda^2$ ($v$ and $l$ are square numbers). We compute the central idempotents and search for the central idempotent $i_1$ of trace $12^2 = 144$. Hence $i_1A$ is the algebra corresponding to Krammer's representation. This enables us to construct the Krammer's representation of $B(4,4,3)$ that we denote by $\rho_4$, see Appendix B for the implementation. We conjecture that this Krammer's representation is faithful, see Conjecture \ref{ConjectureRho3and4faithful}. We provide the matrices $A = \rho_4(T_0)$, $B = \rho_4(T_1)$, and $C = \rho_4(S_3)$ over $\mathbb{Q}(\mu,\lambda)$ that we get after interpolation of many rational specializations of $\mu$ and $\lambda$. As for the case of $B(3,3,3)$, we checked that $A$, $B$, and $C$ satisfy all the relations of BMW$(4,4,3)$ and that $\rho_4$ is absolutely irreducible over $\mathbb{Q}(\mu, \lambda)$ since the dimension of the corresponding matrix algebra for a given specialization of the parameters ($\mu = 2$ and $\lambda = 3$ as in Appendix B) is equal to $144$.\\

\noindent The matrix $A$ is equal to\\
$[\frac{1}{\lambda^2}, 0, 0, 0, 0, 0, 0, 0, 0, 0, 0, 0]$\\
$[0, 0, 0, 0, 0, 1, 0, 0, 0, 0, 0, 0]$\\
$[\frac{(\mu^4-1)}{\mu^4}, 0, \frac{1}{\mu^2}, 0, 0, 0, 0, 0, 0, 0, 0, 0]$\\
$[0, 0, 0, 0, 0, 0, 0, 0, 0, 0, 1, 0]$\\
$[\frac{(\mu^4-1)(\mu^4\lambda^2+\mu^3\lambda-\lambda^2)}{(\mu^5\lambda^3)}, 0, 0, 0, 0, 0, 0, 0, \frac{\lambda}{\mu}, 0, 0, 0]$\\
$[\frac{-(\mu^4-1)(\mu^5\lambda-\mu^2\lambda^2-\mu \lambda)}{(\mu^5\lambda^3)}, 1, 0, 0, 0, \frac{(-\mu^4+1)}{\mu^2}, 0, 0, 0, 0, 0, 0]$\\
$[0, 0, 0, 0, 0, 0, 0, 0, 0, 0, 0, 1]$\\
$[\frac{(-\mu^{16}+4\mu^{12}-5\mu^8+3\mu^4-1)}{(\mu^8\lambda^2)}, 0, 0, 0, \frac{(\mu^4-1)}{\mu \lambda}, 0, 0, 0, \frac{-(\mu^4-1)^2}{\mu^4}, \frac{\mu}{\lambda}, 0, 0]$\\
$[\frac{-(\mu^4-1)^3}{(\mu^6\lambda^2)}, 0, 0, 0, \frac{\mu}{\lambda}, 0, 0, 0, \frac{(-\mu^4+1)}{\mu^2}, 0, 0, 0]$\\
$[\frac{-(\mu^4-1)^3}{(\mu^7\lambda)}, 0, 0, 0, 0, 0, 0, \frac{\lambda}{\mu}, \frac{-(\mu^4-1)\lambda}{\mu^3}, \frac{(-\mu^4+1)}{\mu^2}, 0, 0]$\\
$[\frac{(\mu^4-1)(-\mu^4+\mu \lambda+1)}{(\mu^5\lambda)}, 0, 0, 1, 0, 0, 0, 0, 0, 0, \frac{(-\mu^4+1)}{\mu^2}, 0]$\\
$[0, 0, 0, 0, 0, 0, 1, 0, 0, 0, 0, \frac{(-\mu^4+1)}{\mu^2}]$.\\

\noindent The matrix $B$ is equal to\\
$[0, 0, 1, 0, 0, 0, 0, 0, 0, 0, 0, 0]$\\
$[0, 0, 0, 1, 0, 0, 0, 0, 0, 0, 0, 0]$\\
$[1, 0, \frac{(-\mu^4+1)}{\mu^2}, 0, 0, 0, 0, 0, 0, 0, 0, 0]$\\
$[0, 1, 0, \frac{(-\mu^4+1)}{\mu^2}, 0, 0, \frac{(\mu^4-1)}{\mu^6}, 0, 0, 0, 0, 0]$\\
$[\frac{(\mu^4-1)^2}{\mu^4}, \frac{(\mu^4-1)}{\mu^2}, \frac{(\mu^4-1)}{\mu^2}, 0, 0, 1, \frac{(\mu^4-1)}{(\mu^5\lambda)}, 0, 0, 0, 0, 0]$\\
$[\frac{(-\mu^{12}+2\mu^8-2\mu^4+1)}{\mu^6}, \frac{-(\mu^4-1)^2}{\mu^4}, \frac{-(\mu^4-1)^2}{\mu^4}, \frac{(-\mu^4+1)}{\mu^2}, 1, \frac{(-\mu^4+1)}{\mu^2}, 0, 0, 0, 0, 0, 0]$\\
$[0, 0, 0, 0, 0, 0, \frac{1}{\lambda^2}, 0, 0, 0, 0, 0]$\\
$[0, 0, 0, 0, 0, 0, \frac{(\mu^4-1)(\mu^5\lambda+\mu^4-1)}{(\mu^8\lambda^2)}, 0, 1, 0, 0, 0]$\\
$[0, 0, 0, 0, 0, 0, 0, 1, \frac{(-\mu^4+1)}{\mu^2}, 0, 0, 0]$\\
$[0, 0, 0, 0, 0, 0, \frac{(\mu^4-1)}{(\mu^7\lambda)}, 0, 0, 0, 1, 0]$\\
$[0, 0, 0, 0, 0, 0, 0, 0, 0, 1, \frac{(-\mu^4+1)}{\mu^2}, 0]$\\
$[0, 0, 0, 0, 0, 0, \frac{(\mu^4-1)}{\mu^4}, 0, 0, 0, 0, \frac{1}{\mu^2}]$.\\

\noindent The matrix $C$ is equal to\\
$[\frac{(-\mu^4+1)}{\mu^2}, \frac{(-\mu^4+1)}{(\mu^4-\mu \lambda-1)}, 0, 0, 0, \frac{-\mu^2}{(\mu^4-\mu \lambda-1)}, 0, 0, 0, 0, 0, 0]$\\
$[0, \frac{1}{\lambda^2}, 0, 0, 0, 0, 0, 0, 0, 0, 0, 0]$\\
$[0, \frac{\mu^8\lambda-\mu^7-2\mu^4\lambda+\mu^3+\lambda}{\mu^2\lambda(\mu^4-\mu \lambda-1)}, \frac{-\mu^4+1}{\mu^2}, 0, 0, \frac{\mu^4-1}{\mu^4-\mu \lambda-1}, \frac{\mu^4-1}{\mu^4(\mu^4-\mu \lambda-1)}, 0, 0, \frac{-\mu^3}{\lambda(\mu^4-\mu \lambda-1)}, 0, 0]$\\
$[0, 0, 0, 0, 0, 0, \frac{\lambda^2}{\mu^4}, 0, 0, 0, 0, 0]$\\
$[0, C_{5,2} , 0, \frac{\mu^4-1}{\mu \lambda}, \frac{-\mu^4+1}{\mu^2}, \frac{-\mu^4+1}{\mu^4-\mu \lambda-1}, \frac{\lambda(\mu^8-2\mu^4+1)}{\mu^7(\mu^4-\mu \lambda-1)}, 0, 0, \frac{-\mu^4+1}{\mu^4-\mu \lambda-1}, \frac{\mu}{\lambda}, 0]$\\
$[\frac{-\mu^4+\mu \lambda+1}{\mu^2}, 0, 0, 0, 0, 0, 0, 0, 0, 0, 0, 0]$\\
$[0, \frac{\mu^2(\mu^4-1)}{\lambda^2}, 0, \frac{\mu^4}{\lambda^2}, 0, 0, \frac{-\mu^4+1}{\mu^2}, 0, 0, 0, 0, 0]$\\
$[\frac{-\mu^{12}+2\mu^8-2\mu^4+1}{\mu^6}, C_{8,2} , \frac{-\mu^8+2\mu^4-1}{\mu^4}, \frac{\mu^8-2\mu^4+1}{\mu^4\lambda^2}, 0, \frac{-\mu^4+1}{\mu^2(\mu^4-\mu \lambda-1)}, C_{8,7} , \frac{1}{\mu^2}, 0, \frac{-\mu^8+2\mu^4-1}{\lambda(\mu^4-\mu \lambda-1)\mu^3}, 0, 0]$\\
$[\frac{-\mu^{12}+3\mu^8-3\mu^4+1}{\mu^8}, C_{9,2} , \frac{-\mu^8+2\mu^4-1}{\mu^6}, 0, 0, 0, \frac{\mu^8-2\mu^4+1}{\mu^8(\mu^4-\mu \lambda-1)}, 0, 0, \frac{-\mu^4+1}{\lambda(\mu^4-\mu \lambda-1)\mu}, 0, \frac{1}{\mu^2}]$\\
$[\frac{-\lambda(\mu^8-\mu^5\lambda-2\mu^4+\mu \lambda+1)}{\mu^5}, \frac{\mu^8-2\mu^4+1}{\mu^5\lambda}, \frac{-\lambda(\mu^4-\mu \lambda-1)}{\mu^3}, \frac{\mu^4-1}{\mu^3\lambda}, 0, 0, \frac{-\lambda(\mu^8-2\mu^4+1)}{\mu^9}, 0, 0, 0, 0, 0]$\\
$[\frac{-\lambda(\mu^8\lambda+\mu^7-2\mu^4\lambda-\mu^3+\lambda)}{\mu^6}, 0, \frac{-(\mu^4-1)\lambda^2}{\mu^4}, 0, \frac{\lambda}{\mu}, 0, \frac{-(\mu^4-1)\lambda^2}{\mu^6}, 0, 0, 0, 0, 0]$\\
$[\frac{-\mu^8+2\mu^4-1}{\mu^4}, 0, \frac{-\mu^4+1}{\mu^2}, 0, 0, 0, 0, 0, \mu^2, 0, 0, \frac{-\mu^4+1}{\mu^2}]$.\\

\noindent with $C_{5,2} = \frac{(\mu-1)^2(\mu+1)^2(\mu^2+1)^2(\mu^4-2\mu^2\lambda^2-\mu \lambda-1)}{\mu^4\lambda^2(\mu^4-\mu \lambda-1)},$\\
 $C_{8,2} = \frac{(\mu -1)(\mu +1)(\mu^2+1)(\mu^{12}\lambda+\mu^{11}-2\mu^9\lambda^2-3\mu^8\lambda-\mu^6\lambda^3-2\mu^7+3\mu^5\lambda^2+3\mu^4\lambda+\mu^2\lambda^3+\mu^3-2\mu \lambda^2-\lambda)}{\mu^6\lambda^3(\mu^4-\mu \lambda-1)},$\\
$C_{8,7} =  \frac{-(\mu-1)^2(\mu+1)^2(\mu^2+1)^2(\mu^8-\mu^5\lambda-2\mu^4+1)}{\mu^{10}(\mu^4-\mu \lambda-1)},$\\
$C_{9,2} = \frac{(\mu-1)(\mu+1)(\mu^2+1)(\mu^6-\mu^4\lambda^2-\mu^3\lambda-\mu^2+\lambda^2)}{\mu^3\lambda^3(\mu^4-\mu \lambda-1)}$.\\

\medskip

Recall that Chen studied the BMW algebras for the dihedral groups in \cite{ChenBMW2017} and was able to construct irreducible representations for the Artin-Tits groups associated to the dihedral groups. He also conjectured that these representations are isomorphic to the monodromy representations. Each of our explicit representations $\rho_3$ and $\rho_4$ is constructed via the corresponding BMW algebra. We proved that these representations are absolutely irreducible and have dimension the number of reflections of the complex refection group. We also studied the restriction of $\rho_3$ to parabolic subgroups. All these properties are satisfied for the generalized Krammer's representations. That's why we think that $\rho_3$ and $\rho_4$ are good candidates to be called the Krammer's representations for $B(3,3,3)$ and $B(4,4,3)$, respectively. We also think that each of these representations is isomorphic to the monodromy representation constructed by Marin in \cite{MarinLinearity}, where he conjectured that the monodromy representation is faithful, see Conjecture 6.3 in \cite{MarinLinearity}. This enables us to propose the following conjecture about the explicit Krammer's representations $\rho_3$ and $\rho_4$.

\begin{conjecture}\label{ConjectureRho3and4faithful}

The Krammer's representations $\rho_3$ and $\rho_4$ are faithful.\\

\end{conjecture}

Our method uses the computation of a Gröbner basis from the list of polynomials that describe the relations of the algebra BMW$(e,e,n)$. It it is not possible to make these heavy computations for $e \geq 5$ when $n = 3$. However, we try to compute the dimension of the BMW algebra over a finite field. We were able to compute the dimension of BMW$(5,5,3)$ and BMW$(6,6,3)$ over a finite field for many specializations for $m$ and $l$, see Appendix B for the implementation. The computations for BMW$(6,6,3)$ took 40 days on the platform MATRICS of Universit\'e de Picardie Jules Verne (\cite{MATRICS}). For all the different specializations, we get that the dimension of BMW$(5,5,3)$ is equal to $1275$ and the dimension of BMW$(6,6,3)$ is equal to $1188$. Recall that we get $297$ for BMW$(3,3,3)$ and $384$ for BMW$(4,4,3)$. We remark that for $e = 3$ or $e = 5$, we get that the dimension of BMW$(e,e,3)$ is equal to $\#(G(e,e,3)) + e\times (\#\mathcal{R})^2$, where $\mathcal{R}$ is the set of reflections of $G(e,e,3)$. Also, we remark that for $e = 4$ and $e = 6$, we get that the dimension of BMW$(e,e,3)$ is equal to $\#(G(e,e,3)) + \frac{e}{2} \times (\#\mathcal{R})^2$.\\

Based on the previous experimentations, we conclude this section by proposing the following conjecture.

\begin{conjecture}\label{ConjectureStructureBMW}

Let $K = \overline{\mathbb{Q}(m,l)}$ and let $H(e,e,3)$ be the Hecke algebra associated to $G(e,e,3)$ over $K$. Let $N$ be the number of reflections of the complex reflection group $G(e,e,3)$. 

Over $K$, the algebra \emph{BMW}$(e,e,3)$ is semi-simple, isomorphic to
$$\left\{\begin{array}{ll}
H(e,e,3) \oplus \left( \mathcal{M}_{N}(K) \right)^{e} & if\ e\ is\ odd,\\
H(e,e,3) \oplus \left( \mathcal{M}_{N}(K) \right)^{e/2} & if\ e\ is\ even,
\end{array}\right.$$
where $\mathcal{M}_{N}(K)$ is the matrix algebra over $K$ of dimension $N^2$. In particular, the algebra \emph{BMW}$(e,e,3)$ has dimension
$$\left\{\begin{array}{ll}
\#(G(e,e,3)) + e\times N^2 & if\ e\ is\ odd,\\
\#(G(e,e,3)) + \frac{e}{2} \times N^2 & if\ e\ is\ even.
\end{array}\right.$$

\end{conjecture}

%% file: annexe.tex
\begin{appendices}

\chapter{Monoid and homology implementations}

We start by explaining the way to implement the Garside monoids $B^{\oplus k}(e,e,n)$ into GAP by using the package CHEVIE (\cite{CHEVIEJMichel}) of GAP3. We provide the implementation of these monoids by defining the function \verb|CorranPicantinMonoid| that can be found in \cite{GAP3CPMichelNeaime}.\\

\noindent Recall that for $1 \leq k \leq e-1$, the monoid $B^{\oplus k}(e,e,n)$ is defined by the following presentation.

\begin{itemize}

\item Generating set: $\widetilde{X} = \{ \tilde{t}_{0}, \tilde{t}_{1}, \cdots, \tilde{t}_{e-1}, \tilde{s}_{3}, \cdots, \tilde{s}_{n}\}$.
\item Relations:$\left\{ \begin{array}{ll}

\tilde{s}_{i}\tilde{s}_{j}\tilde{s}_{i} = \tilde{s}_{j}\tilde{s}_{i}\tilde{s}_{j} & for\ |i-j|=1,\\
\tilde{s}_{i}\tilde{s}_{j} = \tilde{s}_{j}\tilde{s}_{i} & for\ |i-j| > 1,\\
\tilde{s}_{3}\tilde{t}_{i}\tilde{s}_{3} = \tilde{t}_{i}\tilde{s}_{3}\tilde{t}_{i} & for\ i \in \mathbb{Z}/e\mathbb{Z},\\
\tilde{s}_{j}\tilde{t}_{i} = \tilde{t}_{i}\tilde{s}_{j} & for\ i \in \mathbb{Z}/e\mathbb{Z}\ and\ 4 \leq j \leq n,\ and\\
\tilde{t}_{i}\tilde{t}_{i-k} = \tilde{t}_{j}\tilde{t}_{j-k} & for\ i, j \in \mathbb{Z}/e\mathbb{Z}.

\end{array}\right.$

\end{itemize}

\noindent Also recall that $B^{\oplus k}(e,e,n)$ is isomorphic to the interval monoid $M[1,\lambda^k]$ constructed in Chapter \ref{ChapterIntervalGarside}, where $\lambda \in G(e,e,n)$ is the diagonal matrix such that $\lambda[i,i] = \zeta_e$ for $2 \leq i \leq n$ and $\lambda[1,1] = \zeta_e^{-(n-1)}$. We showed that $B^{\oplus k}(e,e,n)$ is a Garside monoid with Garside element $\underline{\lambda^k}$ and with simples $\underline{[1,\lambda^k]}$, a copy of the interval $[1,\lambda^k] \subset G(e,e,n)$, see Theorem \ref{TheoremIntervalStructure}.\\

In order to construct these Garside monoids, we create a record \verb|M| containing some specific operations for the monoid $B^{\oplus k}(e,e,n)$ and then call \verb|CompleteGarsideRecord(M)|. Since we are dealing with interval monoids, we provide a second argument\\ \verb|rec(interval:=true)| to \verb|CompleteGarsideRecord|. We name the function\\ \verb|CorranPicantinMonoid|. It takes three arguments: $e$, $n$, and $k$ and defines the monoid $B^{\oplus k}(e,e,n)$. Note that the simples of the interval structure of  $B^{\oplus k}(e,e,n)$ are defined as elements of the complex reflection group $G(e,e,n)$.\\

First, we define the list \verb|M.atoms| of the atoms of $B^{\oplus k}(e,e,n)$. It consists of $t_0$, $t_1$, $\cdots$, $t_{e-1}$, $s_3$, $\cdots$, $s_n$ that are elements of $G(e,e,n)$. Next, we define the operation \verb|M.IsLeftDescending|$(s,i)$ that tells whether \verb|M.atoms|$[i]$ divides the simple $s$ on the left. To do this, we use the matrix form of the elements of $G(e,e,n)$ and apply Proposition \ref{PropLengthdecreas}. Since the atoms satisfy (\verb|M.atoms|$[i]$)$^2 = 1$, we automatically have the operation \verb|M.IsRightDescending|$(s,i)$ that tells whether \verb|M.atoms|$[i]$ divides the simple $s$ on the right. Moreover, we define the operation \verb|M.IsRightAscending|$(s,i)$ that tells whether the product of the simple $s$ by the atom \verb|M.atoms|$[i]$ is still simple. This is simply done by calling the function \verb|M.IsLeftDescending|$(s^{-1}$\verb|M.delta|,$i)$ where \verb|M.delta| is the Garside element $\lambda^k \in G(e,e,n)$. Other operations should also be defined like \verb|M.identity| that is the identity element of the monoid and \verb|M.stringDelta| that defines how the Garside element should be printed in normal forms.\\

Finally, we define the Reverse operation for the simples. Since $B^{\oplus k}(e,e,n)$ is an interval monoid, if (\verb|M.delta|)$^2$=\verb|M.identity|, the Reverse operation is just the function $x \mapsto x^{-1}$ and is then automatically defined. This is the case when $e =2k$. In the other cases, the function $x \mapsto x^{-1}$ sends \verb|CorranPicantinMonoid|$(e,n,k)$ into \verb|CorranPicantinMonoid|$(e,n,e-k)$. Thus, we equip \verb|M| with a field \verb|M.revMonoid|  containing this monoid. The implementation is as follows.\\

\begin{small}
\begin{Verbatim}[frame=single]
CorranPicantinMonoid := function ( arg )
    local  e, n, k, M, W, i, deltamat, r, cr;
    e := arg[1];
    n := arg[2];
    if Length( arg ) = 2  then
        k := 1;
    else
        k := arg[3] mod e;
        if k = 0  then
            Error( "k must not be divisible by e" );
        fi;
    fi;
    r := [ Concatenation( [ -1 * E( e ), 1 ], [ 1 .. (n - 2) ] * 0 ) ];
    for i  in [ 1 .. n - 1 ]  do
        Add( r, Concatenation( [ 1 .. (i - 1) ] * 0, [ -1, 1 ], 
           [ 1 .. (n - i - 1) ] * 0 ) );
    od;
    cr := Copy( r );
    cr[1][1] := -1 * E( e ) ^ (-1 * 1);
    W := PermRootGroup( r, cr );
    M := rec(
        group := W,
        operations := rec(
            Print := function ( M )
                  Print( "CorranPicantinMonoid(", e, ",", e, ",", n, ")" );
                  if k <> 1  then
                      Print( "[", k, "]" );
                  fi;
              end ) );
    M.atoms := [ W.2, W.1 ];
    for i  in [ 2 .. e - 1 ]  do
        Add( M.atoms, M.atoms[(i - 1)] ^ M.atoms[i] );
    od;
    Append( M.atoms, W.generators{[ 3 .. n ]} );
    M.mat := function ( s )
          return MatXPerm( W, s );
      end;
    M.IsLeftDescending := function ( simp, i )
          local  m, c1, c2;
          m := M.mat( simp );
          if i <= e  then
              c1 := PositionProperty( m[1], function ( x )
                      return x <> 0;
                  end );
              c2 := PositionProperty( m[2], function ( x )
                      return x <> 0;
                  end );
              if c1 < c2  then
                  return m[2][c2] <> 1;
              else
                  return m[1][c1] = E( e ) ^ (-1 * (i - 1));
              fi;
          fi;
          i := i - e + 2;
          c1 := PositionProperty( m[i - 1], function ( x )
                  return x <> 0;
              end );
          c2 := PositionProperty( m[i], function ( x )
                  return x <> 0;
              end );
          if c1 < c2  then
              return m[i][c2] <> 1;
          else
              return m[i - 1][c1] = 1;
          fi;
      end;
    M.IsRightDescending := function ( simp, i )
          return M.IsLeftDescending( simp ^ (-1 * 1), i );
      end;
    M.IsRightAscending := function ( simp, i )
          return M.IsLeftDescending( simp ^ (-1 * 1) * M.delta, i );
      end;
    M.IsLeftAscending := function ( simp, i )
          return M.IsRightDescending( M.delta * simp ^ (-1 * 1), i );
      end;
    M.identity := W.identity;
    deltamat := IdentityMat( n ) * E( e ) ^ k;
    deltamat[1][1] := E( e ) ^ (k * (1 - n));
    M.delta := PermMatX( W, deltamat );
    M.stringDelta := "d";
    M.orderDelta := Order( W, M.delta );
    M.showmat := function ( s )
          local  t;
          if IsPerm( s )  then
              PrintArray( M.mat( s ) );
          else
              for t  in s.elm  do
                  M.showmat( t );
              od;
          fi;
      end;
    if k <> e - k  then
        if Length( arg ) = 4  then
            M.revMonoid := arg[4];
        else
            M.revMonoid := CorranPicantinMonoid( e, n, e - k, M );
        fi;
        M.Reverse := function ( b )
              local  res, s;
              if Length( b.elm ) = 0  then
                  return M.revMonoid.Elt( [  ], b.pd );
              fi;
              res := [  ];
              for s  in List( Reversed( b.elm ), function ( y )
                        return M.revMonoid.DeltaAction( y ^ (-1 * 1), b.pd );
                    end )  do
                  res := M.revMonoid.AddToNormal( res, s );
              od;
              return GarsideEltOps.Normalize( M.revMonoid.Elt( res, b.pd ) );
          end;
    fi;
    CompleteGarsideRecord( M, rec ( interval := true ) );
    return M;
end;	
\end{Verbatim}
\end{small}

\bigskip

\noindent For example, M := \verb|CorranPicantinMonoid|(3,3,1) defines the classical Garside monoid of Corran and Picantin $B^{\oplus}(e,e,n)$ and M := \verb|CorranPicantinMonoid|(10,3,2) defines the Garside monoid $B^{\oplus 2}(10,10,3)$.\\

Now, we provide an application of this implementation in order to compute the second integral homology group of $B^{(k)}(e,e,n)$, the group of fractions of the monoid $B^{\oplus k}(e,e,n)$, see Proposition \ref{PropH2ofBkeen}. One needs to implement the differential $d_r$ after defining the two $\mathbb{Z}$-module homomorphisms $s_r$ and $u_r$, see Definition \ref{DefintionDifferentialHomology}. This is a general procedure when we compute the homology groups of Garside structures using the Dehornoy-Lafont complexes. On his webpage \cite{MarinCodeHomology}, Marin published the implementation (into GAP3) that he used in his paper \cite{MarinHomologyComputationsII} for the homology computations for some complex braid groups. We use the same file with some changes in order to be compatible with the new distribution of the package CHEVIE of GAP3 and the implementation of the monoid $B^{\oplus k}(e,e,n)$. The implementation (with comments) is given at the end of this Appendix. This allows us to compute the matrices of the differentials $d_2$ and $d_3$. We denote these matrices by \verb|M2| and \verb|M3|, respectively.\\

In order to compute the second integral homology group, one can use the package HAP of GAP4, see \cite{MarinCodeHomology} for the implementation. We simply use SAGE to compute the second integral homology group of $B^{(k)}(e,e,n)$. For example,\\ if M := \verb|CorranPicantinMonoid|$(e,n,k)$ for some parameters $e$, $n$, and $k$, we get two matrices \verb|M2| and \verb|M3| from GAP3. Using SAGE, we put the following commands in order to get the second integral homology group.

\begin{Verbatim}[frame=single] 
M2 = matrix(ZZ,M2)
M3 = matrix(ZZ,M3)
C = ChainComplex({3:M3,2:M2},degree_of_differential = -1)
C.homology(2)
\end{Verbatim}

\bigskip

\begin{center}
\textbf{Homology computation: the matrices of the differentials}
\end{center}
\begin{small}
\begin{Verbatim}[frame = single]
# Define an interval monoid M

M := CorranPicantinMonoid(10,3,2);

# Define the atoms and the simples of the Garside monoid

allatoms := M.atoms;;

allsimples := LeftDivisorsSimple(M,M.delta);;

# The next function returns the left lcm of a list of elements in M

leftlcmlist := function ( L )
    local  res;
    if Length( L ) = 0  then
        return M.B( [  ] );
    elif Length( L ) = 1  then
        return L[1];
    elif Length( L ) = 2  then
        res := LeftLcm( L[1], L[2] )[1];
        return res;
    else
        res := LeftLcm( L[1], leftlcmlist( L{[ 2 .. Length( L ) ]} ) )[1];
        return res;
    fi;
end;

# A chain is a list of elements of the form [c,[m,cell]]
# where m is an element of M, cell is an element of the basis of the chain,
# c is an integer coefficient that describes the linear combination

# The next function left multiples a chain by an element of M

leftmultipliemlst := function ( lst, m )
    local  res, i;
    res := ShallowCopy( lst );
    for i  in [ 1 .. Length( res ) ]  do
        res[i][2][1] := m * res[i][2][1];
    od;
    return res;
end;

# The next function add to a chain the element chc := [c,[m,cel]]

addelmtochain := function ( lst, chc )
    local  res, i, fait;
    res := ShallowCopy( lst );
    fait := false;
    for i  in [ 1 .. Length( res ) ]  do
        if res[i][2] = chc[2]  then
            res[i][1] := res[i][1] + chc[1];
            fait := true;
        fi;
    od;
    if fait = true  then
        return Filtered( res, function ( y )
                return y[1] <> 0;
            end );
    else
        Add( res, chc );
        return res;
    fi;
end;

# The next function returns the list of all the n-cells of M

cellsM := function ( n )
    local  pr, res, la;
    if n = 1  then
        return List( [ 1 .. Length( M.atoms ) ], function ( y )
                return [ M.B( y ) ];
            end );
    elif n = 0  then
        return [ [  ] ];
    else
        pr := cellsM( n - 1 );
        la := List( [ 1 .. Length( M.atoms ) ], M.B );
        res := Concatenation( List( la, function ( a )
                  return List( pr, function ( y )
                          return Concatenation( [ a ], y );
                      end );
              end ) );
        res := Filtered( res, function ( y )
                return Position( la, y[1] ) < Position( la, y[2] );
            end );
        res := Filtered( res, function ( y )
                return 
                 y[1] = M.B( First( [ 1 .. Length( M.atoms ) ], function (a)
                            return (RightGcd( leftlcmlist( y ), M.B( a ) )[1] 
                              <> M.B( [  ] ));
                        end ) );
            end );
        return res;
    fi;
end;

# Next, we define S_n, R_n, and D_n

allfunctions := function ( value1, value2, cl, n )
    local  ch, lst, ce, co, lcmA, it, alpha, alphaA, yy, ret1, ret2, res, i, 
    res2, g, A;
    if value2 = "lst"  then
        lst := ShallowCopy( cl );
        if Length( lst ) < 1  then
            Error( "Error in arguments.\n" );
        fi;
        res := allfunctions( value1, "ch", lst[1][2], n );
        if lst[1][1] <> 1  then
            for i  in [ 1 .. Length( res ) ]  do
                res[i][1] := res[i][1] * lst[1][1];
            od;
        fi;
        res := Filtered( res, function ( y )
                return y[1] <> 0;
            end );
        if Length( lst ) = 1  then
            return res;
        else
            res2 := allfunctions( value1, "lst", lst{[ 2 .. Length( lst ) ]}, 
               n );
            for g  in res2  do
                res := addelmtochain( res, g );
            od;
            res := Filtered( res, function ( y )
                    return y[1] <> 0;
                end );
            return res;
        fi;
    else
        ch := ShallowCopy( cl );
        if n <> Length( ch[2] )  then
            Error( "Error in dimension.\n" );
        fi;
        if value1 = "S"  then
            ce := ch[2];
            co := ch[1];
            lcmA := leftlcmlist( ce );
            it := co * lcmA;
            if it = M.B( [  ] )  then
                return [  ];
            fi;
            alpha := M.B( First( [ 1 .. Length( M.atoms ) ], function ( a )
                      return RightGcd( it, M.B( a ) )[1] <> M.B( [  ] );
                  end ) );
            alphaA := LeftLcm( alpha, lcmA )[1] * lcmA ^ -1;
            yy := co * alphaA ^ -1;
            if n > 0  then
                if alpha = ce[1]  then
                    return [  ];
                fi;
            fi;
            ret1 := [ 1, [ yy, Concatenation( [ alpha ], ce ) ] ];
            ret2 
             := allfunctions( "S", "lst", leftmultipliemlst( allfunctions( 
                   "R", "ch", [ alphaA, ce ], n ), yy ), n );
            return addelmtochain( ret2, ret1 );
        elif value1 = "R"  then
            if n = 0  then
                return [ [ 1, [ M.B( [  ] ), [  ] ] ] ];
            else
                return 
                 allfunctions( "S", "lst", allfunctions( "D", "ch", ch, n ), 
                   n - 1 );
            fi;
        elif value1 = "D"  then
            if n = 0  then
                return [  ];
            else
                alpha := ch[2][1];
                if n = 1  then
                    return 
                     leftmultipliemlst( addelmtochain( [ [ 1, [ alpha, [  ] ] 
                               ] ], [ -1, [ M.B( [  ] ), [  ] ] ] ), ch[1] );
                else
                    A := ShallowCopy( ch[2]{[ 2 .. Length( ch[2] ) ]} );
                    lcmA := leftlcmlist( A );
                    alphaA := LeftLcm( alpha, lcmA )[1] * lcmA ^ -1;
                    ret1 := [ 1, [ alphaA, A ] ];
                    ret2 := allfunctions( "R", "ch", [ alphaA, A ], n - 1 );
                    ret2 := List( ret2, function ( y )
                            return [ y[1] * -1, y[2] ];
                        end );
                    return leftmultipliemlst( addelmtochain( ret2, ret1 ), 
                       ch[1] );
                fi;
            fi;
        else
            Error( "Error in arguments.\n" );
        fi;
    fi;
end;

# The next function computes the matrix of d_n

diffmatr := function ( n )
    local  bseB, bseA, res, g, pr, vv, h;
    bseB := cellsM( n - 1 );
    Print( "cellsM(", n - 1, ")=", Length( bseB ), "\n" );
    bseA := cellsM( n );
    Print( "cellsM(", n, ")=", Length( bseA ), "\n" );
    res := [  ];
    for g  in bseA  do
        pr := allfunctions( "D", "ch", [ M.B( [  ] ), g ], n );
        vv := [  ];
        for h  in bseB  do
            Add( vv, Sum( List( Filtered( pr, function ( y )
                        return y[2][2] = h;
                    end ), function ( y )
                      return y[1];
                  end ) ) );
        od;
        Add( res, ShallowCopy( vv ) );
    od;
    return res;
end;

# The matrices we need to compute the 
# second integral homology group with SAGE are:

M2 := TransposedMat(diffmatr(2));

M3 := TransposedMat(diffmatr(3));
\end{Verbatim}
\end{small}

\newpage~
\thispagestyle{empty}

\chapter{Implementation for Krammer's representations}

In this appendix, we provide the implementation of the BMW algebra for types $(3,3,3)$, $(4,4,3)$, $(5,5,3)$, and $(6,6,3)$ using the GBNP package of GAP4, see \cite{GBNP}. We also provide the method used in order to construct Krammer's representations for the complex braid groups $B(3,3,3)$ and $B(4,4,3)$. For the computations, we use the platform MATRICS of Universit\'e de Picardie Jules Verne, see \cite{MATRICS}.\\

Recall that the BMW algebra is defined by a presentation with generators and relations, see Definitions \ref{DefinitionBMWB(e,e,n)eOdd} and \ref{DefinitionBMWB(e,e,n)eEven}. First, we load the package GBNP and set the standard infolevel \verb|InfoGBNP| and the time infolevel \verb|InfoGBNPTime|. The variables are $t_0$, $t_1$, $s_3$, $e_0$, $e_1$, and $f_3$, in this order. In order to have the results printed out with these symbols, we invoke \verb|GBNP.ConfigPrint|. Next, we enter the relations of the BMW algebra. This is done in NP (non-commutative polynomial) form. The inderminates $m$ and $l$ in the coefficient ring of the BMW algebra are specialized (over $\mathbb{Q}$ or a finite field) in order to make the computations possible. We print the relations using \verb|PrintNPList|. Next, we compute a Gröbner basis of the set of relations using the function \verb|Grobner| and then calculate a basis of the quotient algebra with \verb|BaseQA|. This is done for types $(3,3,3)$, $(4,4,3)$, $(5,5,3)$, and $(6,6,3)$ below. For types $(3,3,3)$ and $(4,4,3)$, we define the right and left multiplication of the corresponding matrix algebra and apply the method explained in Section 5.4 of Chapter 5 in order to get the Krammer's representations over $\mathbb{Q}$ for $B(3,3,3)$ and $B(4,4,3)$ for a certain specialization over $\mathbb{Q}$ of $m$ and $l$. This is given in \verb|KrammerRep| of the following codes.

\bigskip

\begin{center}
\textbf{The case of BMW$(3,3,3)$ and Krammer's representation of $B(3,3,3)$,\\ terminates after 10 minutes}
\end{center}
\begin{small}
\begin{Verbatim}[frame=single] 
LoadPackage("GBNP");

SetInfoLevel(InfoGBNP,1);
SetInfoLevel(InfoGBNPTime,1);

GBNP.ConfigPrint("t0","t1","s3","e0","e1","f3","h1");

Print("Define the algebra BMW(3,3,3) by generators and relations: \n");

m := 2; la := 17; l := la^3;

M0 := [[[[7,2],[2,7]],[1,-1]],
     [[[7,2],[]],[1,-1]]];;

M1 := [[[[1,2,1],[2,1,2]],[1,-1]],
     [[[1,3,1],[3,1,3]],[1,-1]],
     [[[2,3,2],[3,2,3]],[1,-1]],
     [[[3,2,1,7,3,2],[2,1,7,3,2,1]],[1,-1]]];;

M2 := [[[[4],[1,1],[1],[]],[1,-l/m,-l,l/m]],
     [[[5],[2,2],[2],[]],[1,-l/m,-l,l/m]],
     [[[6],[3,3],[3],[]],[1,-l/m,-l,l/m]]];;

M3 := [[[[1,4],[4]],[1,-1/l]],
     [[[2,5],[5]],[1,-1/l]],
     [[[3,6],[6]],[1,-1/l]]];;

M4 := [[[[4,2,4],[4]],[1,-l]],
     [[[5,1,5],[5]],[1,-l]],
     [[[5,3,5],[5]],[1,-l]],
     [[[6,2,6],[6]],[1,-l]],
     [[[4,3,4],[4]],[1,-l]],
     [[[6,1,6],[6]],[1,-l]],
     [[[6,2,1,7,6],[6]],[1,-l]],
     [[[4,7,3,2,4],[4]],[1,-l]]];;			

L := Concatenation(M0,M1,M2,M3,M4);;

PrintNPList(L);

GL := Grobner(L);;

BGL := BaseQA(GL,7,0);;

N := Length(BGL);;

bwords := List(BGL,y->y[1][1]);;

tt0 := MatrixQA(1,BGL,GL);;
tt1 := MatrixQA(2,BGL,GL);;
ss3 := MatrixQA(3,BGL,GL);;
ee0 := MatrixQA(4,BGL,GL);;
ee1 := MatrixQA(5,BGL,GL);;
ff3 := MatrixQA(6,BGL,GL);;

gens := [tt0,tt1,ss3,ee0,ee1,ff3];;

alg := Algebra(Rationals,gens);;

vectb := function ( i )
    local  res;
    res := List( [ 1 .. N ], function ( y )
            return 0;
        end );
    res[i] := 1;
    return res;
end
end;

multr := function ( mot )
    if Length( mot ) = 0  then
        return IdentityMat( N );
    fi;
    return Product( gens{mot} );
end;

gensleft := List([1..Length(gens)],i->
 List(bwords,w->vectb(1)*multr(Concatenation([i],w))));;

Print("End of the first step:
 Computation of the left and right multiplication.\n");

zz:=Product(gens{[2,1,3]})^2;;
zzml2:=zz-(1/la^2)*IdentityMat(N);;
rs1ml:=gens[1]-(1/la^3)*IdentityMat(N);;
eq:=TransposedMat(Concatenation(List([zzml2,rs1ml],TransposedMat)));;
neq:=NullspaceMat(eq);;

Print("End of the second step: Computation of the basis.\n");

KrammerRep:=List(gensleft,g->List(neq,v->SolutionMat(neq,v*g)));;

Print("End of the third step: Krammer's representation.\n");
\end{Verbatim}
\end{small}

\bigskip

\begin{center}
\textbf{The case of BMW$(4,4,3)$ and Krammer's representation of $B(4,4,3)$,\\ terminates after 1 day}
\end{center}
\begin{small}
\begin{Verbatim}[frame=single] 
LoadPackage("GBNP");

SetInfoLevel(InfoGBNP,1);
SetInfoLevel(InfoGBNPTime,1);

GBNP.ConfigPrint("t0","t1","s3","e0","e1","f3","h0","h1");;

Print("Define the algebra BMW(4,4,3) by generators and relations. \n");

v := 4; m := v-1/v; l := 9;

M0 := [[[[7,1],[1,7]],[1,-1]],
     [[[7,1],[]],[1,-1]],
     [[[8,2],[2,8]],[1,-1]],
     [[[8,2],[]],[1,-1]]];;

M1 := [[[[1,3,1],[3,1,3]],[1,-1]],
     [[[2,3,2],[3,2,3]],[1,-1]],
     [[[1,2,1,2],[2,1,2,1]],[1,-1]],
     [[[3,2,1,8,3],[2,1,8,3,2,1,8]],[1,-1]],
     [[[3,2,1,2,7,8,3],[2,1,2,7,8,3,2,1,2,7,8]],[1,-1]]];;

M2 := [[[[4],[1,1],[1],[]],[1,-l/m,-l,l/m]],
     [[[5],[2,2],[2],[]],[1,-l/m,-l,l/m]],
     [[[6],[3,3],[3],[]],[1,-l/m,-l,l/m]]];;

M3 := [[[[1,4],[4]],[1,-1/l]],
     [[[2,5],[5]],[1,-1/l]],
     [[[3,6],[6]],[1,-1/l]]];;

M4 := [[[[6,1,6],[6]],[1,-l]],
     [[[4,3,4],[4]],[1,-l]],
     [[[6,2,6],[6]],[1,-l]],
     [[[5,3,5],[5]],[1,-l]],
     [[[6,2,1,8,6],[6]],[1,-l]],
     [[[2,4,8,3,2,4],[2,4]],[1,-l]],
     [[[6,2,1,2,7,8,6],[6]],[1,-l]],
     [[[5,7,8,3,2,1,5],[5]],[1,-l]]];;

M5 := [[[[2,1,2,4],[4,2,1,2]],[1,-1]],
     [[[1,2,1,5],[5,1,2,1]],[1,-1]],
     [[[1,2,1,2,4],[4]],[1,-1/(l*v)]],
     [[[1,2,1,2,5],[5]],[1,-1/(l*v)]],
     [[[4,2,4],[4]],[1,-(1/v)-l]],
     [[[5,1,5],[5]],[1,-(1/v)-l]],
     [[[4,5],[]],[1,0]],
     [[[5,4],[]],[1,0]]];;

L := Concatenation(M0,M1,M2,M3,M4,M5);;

PrintNPList(L);

GL := Grobner(L);;

BGL := BaseQA(GL,8,0);;

N := Length(BGL);;

bwords := List(BGL,y->y[1][1]);;

tt0 := MatrixQA(1,BGL,GL);;
tt1 := MatrixQA(2,BGL,GL);;
ss3 := MatrixQA(3,BGL,GL);;
ee0 := MatrixQA(4,BGL,GL);;
ee1 := MatrixQA(5,BGL,GL);;
ff3 := MatrixQA(6,BGL,GL);;

gens := [tt0,tt1,ss3,ee0,ee1,ff3];;

alg := Algebra(Rationals,gens);;

vectb := function ( i )
    local  res;
    res := List( [ 1 .. N ], function ( y )
            return 0;
        end );
    res[i] := 1;
    return res;
end
end;

multr := function ( mot )
    if Length( mot ) = 0  then
        return IdentityMat( N );
    fi;
    return Product( gens{mot} );
end;

gensleft := List([1..Length(gens)],i->
 List(bwords,w->vectb(1)*multr(Concatenation([i],w))));;

Print("End of the first step: Computation of the left
and right multiplication.\n");

idc := CentralIdempotentsOfAlgebra(alg);;

Print("End of the second step: Computing the central idempotents.\n");

rs1ml := (idc[1]*gens[1])-(1/l)*IdentityMat(N);;
neq := NullspaceMat(rs1ml);;

KrammerRep := List(gensleft,g->List(neq,v->SolutionMat(neq,v*g)));;

Print("End of the third step: Krammer's representation.\n");
\end{Verbatim}
\end{small}

\bigskip

\newpage

\begin{center}
\textbf{The case of BMW$(5,5,3)$, terminates after 4 days}
\end{center}
\begin{small}
\begin{Verbatim}[frame=single] 
LoadPackage("GBNP");

SetInfoLevel(InfoGBNP,1);
SetInfoLevel(InfoGBNPTime,1);

GBNP.ConfigPrint("t0","t1","s3","e0","e1","f3","h0","h1");

p := 103; F := GF(p); m := 263; l := 151;

r := NumeratorRat(l/m)*Z(p)^0/DenominatorRat(l/m)*Z(p)^0;
s := NumeratorRat(1/l)*Z(p)^0/DenominatorRat(1/l)*Z(p)^0;
m := NumeratorRat(m)*Z(p)^0/DenominatorRat(m)*Z(p)^0;
l := NumeratorRat(l)*Z(p)^0/DenominatorRat(l)*Z(p)^0;

M0 := [[[[7,1],[1,7]],[1,-1]],
     [[[7,1],[]],[1,-1]],
     [[[8,2],[2,8]],[1,-1]],
     [[[8,2],[]],[1,-1]]];;

M1 := [[[[1,3,1],[3,1,3]],[1,-1]],
     [[[2,3,2],[3,2,3]],[1,-1]],
     [[[1,2,1,2,1],[2,1,2,1,2]],[1,-1]],
     [[[3,2,1,8,3],[2,1,8,3,2,1,8]],[1,-1]],
     [[[3,2,1,2,7,8,3],[2,1,2,7,8,3,2,1,2,7,8]],[1,-1]],
     [[[3,2,1,2,1,8,7,8,3],
     [2,1,2,1,8,7,8,3,2,1,2,1,8,7,8]],[1,-1]]];;

M2 := [[[[6],[3,3],[3],[]],[1,-r,-l,r]],
     [[[4],[1,1],[1],[]],[1,-r,-l,r]],
     [[[5],[2,2],[2],[]],[1,-r,-l,r]]];;

M3 := [[[[1,4],[4]],[1,-s]],
     [[[2,5],[5]],[1,-s]],
     [[[3,6],[6]],[1,-s]]];;

M4 := [[[[6,1,6],[6]],[1,-l]],
     [[[4,3,4],[4]],[1,-l]],
     [[[4,2,4],[4]],[1,-l]],
     [[[4,2,1,2,4],[4]],[1,-l]],
     [[[5,1,5],[5]],[1,-l]],
     [[[5,1,2,1,5],[5]],[1,-l]],
     [[[6,2,6],[6]],[1,-l]],
     [[[5,3,5],[5]],[1,-l]],
     [[[6,2,1,8,6],[6]],[1,-l]],
     [[[2,4,8,3,2,4],[2,4]],[1,-l]],
     [[[6,2,1,2,7,8,6],[6]],[1,-l]],
     [[[2,1,5,7,8,3,2,1,5],[2,1,5]],[1,-l]],
     [[[6,2,1,2,1,8,7,8,6],[6]],[1,-l]],
     [[[2,1,2,4,8,7,8,3,2,1,2,4],[2,1,2,4]],[1,-l]]];;

M5 := [[[[4,5,4],[4]],[1,-1]],
     [[[5,4,5],[5]],[1,-1]],
     [[[4,8,7,8,4],[4]],[1,-s]],
     [[[5,7,8,7,5],[5]],[1,-s]]];;

L := Concatenation(M0,M1,M2,M3,M4,M5);;

PrintNPList(L);

L := List(L,y->[y[1],Z(p)^0*y[2]]);;

GL := Grobner(L);;

BGL := BaseQA(GL,8,0);;

Print("For p equal to ",p,"the length of BGL is ", Length(BGL),". \n");
\end{Verbatim}
\end{small}

\bigskip

\begin{center}
\textbf{The case of BMW$(6,6,3)$, terminates after 40 days}
\end{center}
\begin{small}
\begin{Verbatim}[frame=single] 
LoadPackage("GBNP");

SetInfoLevel(InfoGBNP,1);
SetInfoLevel(InfoGBNPTime,1);

GBNP.ConfigPrint("t0","t1","s3","e0","e1","f3","h0","h1");

p := 103; F := GF(p); v := 2; m := v-1/v; l := 1;

r := NumeratorRat(l/m)*Z(p)^0/DenominatorRat(l/m)*Z(p)^0;
s := NumeratorRat(1/l)*Z(p)^0/DenominatorRat(1/l)*Z(p)^0;
u := NumeratorRat(1/v)*Z(p)^0/DenominatorRat(1/v)*Z(p)^0;
m := NumeratorRat(m)*Z(p)^0/DenominatorRat(m)*Z(p)^0;
l := NumeratorRat(l)*Z(p)^0/DenominatorRat(l)*Z(p)^0;

M0 := [[[[7,1],[1,7]],[1,-1]],
     [[[7,1],[]],[1,-1]],
     [[[8,2],[2,8]],[1,-1]],
     [[[8,2],[]],[1,-1]]];;

M1 := [[[[1,3,1],[3,1,3]],[1,-1]],
     [[[2,3,2],[3,2,3]],[1,-1]],
     [[[1,2,1,2,1,2],[2,1,2,1,2,1]],[1,-1]],
     [[[3,2,1,8,3],[2,1,8,3,2,1,8]],[1,-1]],
     [[[3,2,1,2,7,8,3],[2,1,2,7,8,3,2,1,2,7,8]],[1,-1]],
     [[[3,2,1,2,1,8,7,8,3],
     [2,1,2,1,8,7,8,3,2,1,2,1,8,7,8]],[1,-1]],
     [[[3,2,1,2,1,2,7,8,7,8,3],
     [2,1,2,1,2,7,8,7,8,3,2,1,2,1,2,7,8,7,8]],[1,-1]]];;

M2 := [[[[6],[3,3],[3],[]],[1,-r,-l,r]],
     [[[4],[1,1],[1],[]],[1,-r,-l,r]],
     [[[5],[2,2],[2],[]],[1,-r,-l,r]]];;

M3 := [[[[1,4],[4]],[1,-s]],
     [[[2,5],[5]],[1,-s]],
     [[[3,6],[6]],[1,-s]]];;

M4 := [[[[6,1,6],[6]],[1,-l]],
     [[[4,3,4],[4]],[1,-l]],
     [[[6,2,6],[6]],[1,-l]],
     [[[5,3,5],[5]],[1,-l]],
     [[[6,2,1,8,6],[6]],[1,-l]],
     [[[4,8,3,2,4],[4]],[1,-l]],
     [[[6,2,1,2,7,8,6],[6]],[1,-l]],
     [[[5,7,8,3,2,1,5],[5]],[1,-l]],
     [[[6,2,1,2,1,8,7,8,6],[6]],[1,-l]],
     [[[4,8,7,8,3,2,1,2,4],[4]],[1,-l]],
     [[[6,2,1,2,1,2,7,8,7,8,6],[6]],[1,-l]],
     [[[5,7,8,7,8,3,2,1,2,1,5],[5]],[1,-l]]];;

M5 := [[[[2,1,2,1,2,4],[4,2,1,2,1,2]],[1,-1]],
     [[[1,2,1,2,1,5],[5,1,2,1,2,1]],[1,-1]],
     [[[1,2,1,2,1,2,4],[4]],[1,-s*u]],
     [[[1,2,1,2,1,2,5],[5]],[1,-s*u]],
     [[[4,2,4],[4]],[1,-u-l]],
     [[[5,1,5],[5]],[1,-u-l]],
     [[[4,5],[]],[1,0]],
     [[[5,4],[]],[1,0]]];;

L := Concatenation(M0,M1,M2,M3,M4,M5);;

PrintNPList(L);

L := List(L,y->[y[1],Z(p)^0*y[2]]);;

GL := Grobner(L);;

BGL := BaseQA(GL,8,0);;

Print("For p equal to ",p,", the length of BGL is ", Length(BGL),".\n");
\end{Verbatim}
\end{small}

\end{appendices}

%% file: compterendu.tex
\chapter*{Compte-rendu \footnote{Lorsque la langue de rédaction de la dissertation est l'anglais, l'Université de Caen Normandie nous demande d'inclure un compte-rendu de la thèse rédigé en français.}}

\addcontentsline{toc}{chapter}{Compte-rendu}

Dans ce compte-rendu nous mettons en évidence les motivations, les enjeux et les apports de cette thèse en exposant ses principaux résultats. Nous commençons par rappeler les définitions des groupes de réflexions complexes, des groupes de tresses complexes et des algèbres de Hecke.

\section*{Groupes de réflexions complexes}

Nous rappelons la définition d'un groupe de réflexions complexes et la classification de Shephard et Todd des groupes finis de réflexions complexes. Nous rappelons aussi la définition des groupes de Coxeter et la classification des groupes de Coxeter finis et irréductibles.\\

Soit $n$ un entier positif non nul et soit $V$ un $\mathbb{C}$-espace vectoriel de dimension $n$.

\begin{definitionf}

Un élément $s$ de $GL(V)$ est appelé une réflexion si $Ker(s-1)$ est un hyperplan et si $s^d = 1$ pour un entier $d \geq 2$.

\end{definitionf}

Soit $W$ un sous-groupe fini de $GL(V)$.

\begin{definitionf}

W est un groupe de réflexions complexes si $W$ est engendré par l'ensemble $\mathcal{R}$ des réflexions de $W$.

\end{definitionf}

On dit que $W$ est irréductible si $V$ est une représentation linéaire irréductible de $W$. Tout groupe de réflexions complexes peut être écrit comme produit direct d'irréductibles (cf. Proposition 1.27 de \cite{UnitaryReflectionGroupsLehrerTaylor}). Donc on peut se restreindre à l'étude des groupes de réflexions complexes irréductibles. Ces derniers ont été classifiés par Shephard et Todd en 1954 (cf. \cite{ShephardTodd}). La classification est comme suit.

\begin{propositionf}

Soit $W$ un groupe de réflexions complexes irréductible. À conjugaison près, $W$ appartient à l'un des cas suivants :

\begin{itemize}
\item La série infinie $G(de,e,n)$ qui dépend de trois paramètres entiers strictement positifs $d$, $e$ et $n$ (cf. la définition suivante).
\item Les $34$ groupes exceptionnels $G_4$, $\cdots$, $G_{37}$.
\end{itemize}

\end{propositionf}  

Nous donnons la définition de la série infinie $G(de,e,n)$ ainsi que certaines de ses propriétés. Pour la définition des $34$ groupes exceptionnels, consulter \cite{ShephardTodd}.

\begin{definitionf}

$G(de,e,n)$ est le groupe des matrices monomiales de tailles $n \times n$ telles que

\begin{itemize}
\item les coefficients non nuls de chaque matrice sont des racines $de$-ième de l'unité et
\item le produit des coefficients non nuls est une racine $d$-ième de l'unité.

\end{itemize}

\end{definitionf}

\begin{remarkf}

Le groupe $G(1,1,n)$ est irréductible sur $\mathbb{C}^{n-1}$ et le groupe $G(2,2,2)$ n'est pas irréductible donc il est exclu de la classification de Shephard et Todd.

\end{remarkf}

Soit $s_i$ la matrice de la transposition $(i,i+1)$ pour $1 \leq i \leq n-1$, $t_e = \begin{pmatrix}
0 & \zeta_e^{-1} & 0\\
\zeta_e & 0 & 0\\
0 & 0 & I_{n-2}
\end{pmatrix}$ et $u_d = \begin{pmatrix}
\zeta_d & 0\\
0 & I_{n-1}
\end{pmatrix}$, où $I_k$ est la matrice identité de taille $k \times k$ et $\zeta_l$ est la racine $l$-ième de l'unité égale à $exp(2i\pi\l/l)$. Le résultat suivant peut être trouvé dans la section 3 du chapitre 2 de \cite{UnitaryReflectionGroupsLehrerTaylor}.

\begin{propositionf}

L'ensemble des générateurs du groupe de réflexions complexes $G(de,e,n)$ est comme suit.
\begin{itemize}
\item Le groupe $G(e,e,n)$ est engendré par les réflexions $t_e$, $s_1$, $s_2$, $\cdots$, $s_{n-1}$.
\item Le groupe $G(d,1,n)$ est engendré par les réflexions $u_d$, $s_1$, $s_2$, $\cdots$, $s_{n-1}$.
\item Pour $d \neq 1$ et $e \neq 1$, le groupe $G(de,e,n)$ est engendré par les réflexions $u_d$, $t_{de}$, $s_1$, $s_2$, $\cdots$, $s_{n-1}$. 
\end{itemize}
\end{propositionf}

Considérons maintenant $W$ un groupe de réflexions réels. Par un théorème de Coxeter, il est connu que tout groupe fini de réflexions réel est isomorphe à un groupe de Coxeter. La définition des groupes de Coxeter par une présentation par générateurs et relations est la suivante.

\begin{definitionf}

Supposons que $W$ est un groupe et que $S$ est un sous-ensemble de $W$. Pour $s$ et $t$ dans $S$, soit $m_{st}$ égal à l'ordre de $st$ dans $W$ si cet ordre est fini et égal à $\infty$ sinon. On dit que $(W,S)$ est un système de Coxeter et que $W$ est un groupe de Coxeter si $W$ admet la présentation avec ensemble générateur égal à $S$ et relations : \begin{itemize}

\item relations quadratiques: $s^2=1$ pour tout $s \in S$ et
\item relations de tresses: $\underset{m_{st}}{\underbrace{sts\cdots}}=\underset{m_{st}}{\underbrace{tst\cdots}}$  pour $s,t \in S$, $s \neq t$ et $m_{st} \neq \infty$.

\end{itemize}

\end{definitionf}

La présentation d'un groupe de Coxeter peut être décrite par un diagramme où les noeuds sont les générateurs de la présentation qui appartiennent à $S$ et les arêtes décrivent les relations entre ces générateurs. Nous suivons les conventions standard des diagrammes de Coxeter. La classification des groupes de Coxeter finis irréductibles comporte :

\begin{tabular}{ll}

-- Type $A_n$ (le groupe symétrique $S_{n+1}$): &  

\begin{tikzpicture}
\draw[fill=black] 
(0,0)                         
      circle [radius=.1] node [label=above:$s_1$] (1) {} 
(1,0) 
      circle [radius=.1] node [label=above:$s_2$] (2) {} 
(3,0)
      circle [radius=.1] node [label=above:$s_n$] (3) {}; 

\draw[-] (1) to (2);
\draw[dashed,-] (2) to (3);

\end{tikzpicture}\\

-- Type $B_n$: & 

\begin{tikzpicture}

\draw[fill=black] 
(0,0)                         
      circle [radius=.1] node [label=above:$s_1$] (1) {}
(1,0) 
      circle [radius=.1] node [label=above:$s_2$] (2) {} 
(2,0)
      circle [radius=.1] node [label=above:$s_3$] (3) {}
(4,0)
	  circle [radius=.1] node [label=above:$s_n$] (4) {}; 

\draw[thick,double,-] (1) to (2);
\draw[-] (2) to (3);
\draw[dashed,-] (3) to (4);

\end{tikzpicture}\\

-- Type $D_n$: &

\begin{tabular}{ll}
\begin{tikzpicture}

\draw[fill=black] 
(-1,0.5)                         
      circle [radius=.1] node [label=left:$s_1$] (1) {}
(-1,-0.5) 
      circle [radius=.1] node [label=left:$s_2$] (2) {} 
(0,0)
      circle [radius=.1] node [label=below:$s_3$] (3) {}
(1,0)
	  circle [radius=.1] node [label=below:$s_4$] (4) {}
(3,0)
	  circle [radius=.1] node [label=below:$s_n$] (5) {}; 

\draw[-] (1) to (3);
\draw[-] (2) to (3);
\draw[-] (3) to (4);
\draw[dashed,-] (4) to (5);

\end{tikzpicture} & \\
 & 

\end{tabular}\\

-- Type $I_2{(e)}$ (le groupe diédral): & \begin{tikzpicture}

\draw[fill=black] 
(0,0)                         
      circle [radius=.1] node [label=above:$s_1$] (1) {}
    
(1,0) 
      circle [radius=.1] node [label=above:$s_2$] (2) {};

\draw[-,label=$e$] (1) to (2);

\node at (0.5,0.15) {$e$};

\end{tikzpicture}\\

 & \\

-- $H_3$, $F_4$, $H_4$, $E_6$, $E_7$ et $E_8$. & \\

\end{tabular}

\begin{remarkf}

En utilisant la notation des groupes qui apparaissent dans la classification de Shephard et Todd, nous avons : type $A_{n-1}$ est $G(1,1,n)$, type $B_n$ est $G(2,1,n)$, type $D_n$ est $G(2,2,n)$ et type $I_2(e)$ est $G(e,e,2)$. Pour les groupes exceptionnels, on a $H_3 = G_{23}$, $F_4 = G_{28}$, $H_4 = G_{30}$, $E_6 = G_{35}$, $E_7 = G_{36}$ et $E_8 = G_{37}$.

\end{remarkf}

\section*{Groupes de tresses complexes}

Soit $W$ un groupe de Coxeter. On définit le groupe d'Artin-Tits $B(W)$ associé à $W$ comme suit.

\begin{definitionf}

Le groupe d'Artin-Tits $B(W)$ associé à $W$ est défini par une présentation avec un ensemble générateur $\widetilde{S}$ en bijection avec l'ensemble générateur $S$ du groupe de Coxeter et les relations sont seulement les relations de tresses $\underset{m_{st}}{\underbrace{\tilde{s}\tilde{t}\tilde{s}\cdots}}=\underset{m_{st}}{\underbrace{\tilde{t}\tilde{s}\tilde{t}\cdots}}$ pour $\tilde{s}, \tilde{t} \in \widetilde{S}$ et $\tilde{s} \neq \tilde{t}$, où $m_{st} \in \mathbb{Z}_{\geq 2}$ est l'ordre de $st$ dans $W$.

\end{definitionf}

Soit $W = S_n$ le groupe symétrique avec $n \geq 2$.
Le groupe d'Artin-Tits associé à $S_n$ est le groupe de tresse `classique' noté $B_n$. Le diagramme suivant décrit la présentation par générateurs et relations de $B_n$.

\begin{center}
\begin{tikzpicture}

\node[draw, shape=circle, label=above:$\tilde{s}_1$] (1) at (0,0) {};
\node[draw, shape=circle, label=above:$\tilde{s}_2$] (2) at (1,0) {};
\node[draw, shape=circle,label=above:$\tilde{s}_{n-2}$] (n-2) at (4,0) {};
\node[draw,shape=circle,label=above:$\tilde{s}_{n-1}$] (n-1) at (5,0) {};

\draw[thick,-] (1) to (2);
\draw[thick,dashed,-] (2) to (n-2);
\draw[thick,-] (n-2) to (n-1);

\end{tikzpicture}
\end{center}

Brou\'e, Malle et Rouquier \cite{BMR} ont réussi à associer un groupe de tresses complexes pour chaque groupe de réflexions complexes. Ceci généralise la notion des groupes d'Artin-Tits associés aux groupes de réflexions réels. Nous fournissons la construction de ces groupes de tresses complexes. Tous les résultats peuvent être trouvés dans \cite{BMR}.\\

Soit $W < GL(V)$ un groupe fini de réflexions complexe. Soit $\mathcal{R}$ l'ensemble des réflexions de $W$. Soit $\mathcal{A} = \{Ker(s-1)\ |\ s \in \mathcal{R} \}$ l'arrangement des hyperplans et $X = V\setminus\bigcup\mathcal{A}$ le complémentaire des hyperplans. Le groupe de réflexion complexes $W$ agit naturellement sur $X$. Soit $p: X \rightarrow X/W$ la surjection canonique. Par un théorème de Steinberg (cf. \cite{SteinbergFreeAction}), cette action est libre. Donc, elle définit un recouvrement galoisien $X \rightarrow X/W$ qui donne lieu à la suite exacte suivante. Soit $x \in X$, on a 

$$1 \longrightarrow \pi_1(X,x) \longrightarrow \pi_1(X/W, p(x)) \longrightarrow W \longrightarrow 1.$$

Ceci nous permet de donner la définition suivante.

\begin{definitionf}

On définit $P := \pi_1(X,x)$ le groupe de tresses complexes pur et $B := \pi_1(X/W, p(x))$ le groupe de tresses complexes associés à $W$.

\end{definitionf}

Soit $s \in \mathcal{R}$ et $H_s$ l'hyperplan correspondant. On définit le lacet $\sigma_s \in B$ donné par le chemin $\gamma$ dans $X$. Choisissons un point $x_0$ `proche de $H_s$ et loin des autres hyperplans'. On définit $\gamma$ le chemin dans $X$ égal à $s.({\widetilde{\gamma}}^{-1}) \circ \gamma_0 \circ \widetilde{\gamma}$, où $\widetilde{\gamma}$ est n'importe quel chemin dans $X$ de $x$ à $x_0$, $s.({\widetilde{\gamma}}^{-1})$ est l'image de $\widetilde{\gamma}^{-1}$ sous l'action de $s$ et $\gamma_0$ est un chemin dans $X$ de $x_0$ à $s.x_0$ autour de l'hyperplan $H_s$. Le chemin $\gamma$ est illustré par la figure suivante.

\begin{center}
\begin{tikzpicture}
[x={(0.866cm,-0.5cm)}, y={(0.866cm,0.5cm)}, z={(0cm,1cm)}, scale=0.9,
>=stealth, %
inner sep=0pt, outer sep=2pt,%
axis/.style={thick,->},
wave/.style={thick,color=#1,smooth},
polaroid/.style={fill=black!60!white, opacity=0.3},
]

\coordinate (O) at (0, 0, 0);
\draw[thick,dashed] (O) -- +(0,  2.7, 0) ;
\draw[thick,dashed] (O) -- +(0,  -2.7, 0);
\draw(0.3,1.9,0) node [right]{$s\cdot x$};
\draw(0.4,-2.2,0) node [left] {$x$};
 \draw[thick] (-2,0,0) -- (O);
\draw[thick] (O) -- +(4.35, 0,   0) ;

\draw [thick] (4.6,0,0)-- +(2.5, 0,   0) node [right] {$H_s$};
\draw (0,2,0)  node{\textbullet} ;
\draw (0,-2,0)  node{\textbullet} ;
\draw[thick,dashed] (5,0,0) -- +(0,  1.9, 0) node[right]{${H_{s}}^{\perp}$};
\draw (5, 0.6, 0) node{\textbullet}; 
\draw[thick,dashed] (5,0,0) -- +(0,  -1.9, 0) ;
\draw (5, -0.6, 0) node{\textbullet} ;
\draw (5.6,-1.1,0.1) node [left] {$x_0$};
\draw (5.8, 1.2,-0.1) node [left] {$s\cdot x_0$};
\draw (5,0,0) node{\textbullet};
\draw (5.5,-0.2,0) node[left]{0};

\draw[fermion, blue] (5,-0.6,0) arc [start angle=190,end angle=23, x radius=0.55cm, y radius=1cm];
\draw[thick, dashed] (5,-0.6,0) arc [start angle=-162, end angle=21, x radius=0.55cm, y radius=1cm];

\path[fermion,draw,red] (0,-2,0) .. controls (4,1,0) and (0.65,-3.7,0) .. (5,-0.6,0);
\path [fermionbar, draw,red](0,2,0).. controls (4,-1,0) and (0.65,3.7,0) .. (5,0.6,0);

\node at (2.5, -1.8,0) {$\widetilde{\gamma}$};
\node at (2.5, 2,0) {$s.({\widetilde{\gamma}}^{-1})$};
\node at (3.8, 0.5,0) {$\gamma_0$};

\end{tikzpicture}
\end{center}

On appelle $\sigma_s$ une réflexion tressée associée à $s$. Les réflexions tressées satisfont la propriété suivante (cf. \cite{BMR} pour une preuve).

\begin{propositionf}

Soient $s_1$ et $s_2$ deux réflexions conjuguées dans $W$ et soient $\sigma_1$ and $\sigma_2$ deux réflexions tressées associées à $s_1$ et $s_2$, respectivement. Les réflexions tressées $\sigma_1$ et $\sigma_2$ sont conjuguées dans $B$.

\end{propositionf}

Une réflexion $s$ est dite distinguée si son unique valeur propre non-triviale est $exp(2i\pi/o(s))$, où $o(s)$ est l'ordre de $s$ dans le groupe de réflexions complexes. On peut associer une réflexion tressée $\sigma_s$ à chaque réflexion distinguée $s$. Dans ce cas, on appelle $\sigma_s$ une réflexion tressée distinguée associée à $s$. On a le résultat suivant (cf. \cite{BMR} pour une preuve). 

\begin{propositionf}

Le groupe de tresses complexes $B$ est engendré par les réflexions tressées distinguées associées aux réflexions distinguées de $W$.

\end{propositionf}

\begin{remarkf}

Par un théorème de Brieskorn \cite{BrieskornArtinTitsFundamentalGroup}, le groupe de tresses complexes associé à un groupe de Coxeter fini $W$ est isomorphe au groupe d'Artin-Tits $B(W)$ associé à $W$.

\end{remarkf}

Une propriété importante des groupes de tresses complexes est qu'ils peuvent être définis par des présentations finis par générateurs et relations qui ressemblent aux présentations des groupes de Coxeter. Ceci généralise le cas des groupes de tresses complexes associés aux groupes finis de Coxeter (consulter la remarque précédente). Des présentations des groupes de tresses complexes associés à la série infinie $G(de,e,n)$ et à certains des groupes exceptionnels peuvent être trouvées dans \cite{BMR}. Des présentations des groupes de tresses complexes associés aux autres groupes exceptionnels sont données dans \cite{BessisMichelPresentationsExcepBraidGroups}, \cite{BessisKPi1} et \cite{MalleMichelRepresentationsHecke}.

\section*{Algèbres de Hecke}

Généralisant des résultats précédents dans \cite{BroueMalleZyklotomischeHeckealgebren}, Brou\'e, Malle et Rouquier ont réussi à généraliser d'une manière naturelle la définition de l'algèbre de Hecke (ou Iwahori-Hecke) des groupes de réflexions réels à n'importe quel groupe de réflexions complexes. En effet, ils ont défini ces algèbres de Hecke en utilisant leur définition des groupes de tresses complexes que nous avons déjà rappelée. Nous donnons maintenant la définition de ces algèbres de Hecke ainsi que certaines de leurs propriétés.\\

Soit $W$ un groupe de réflexions complexes et $B$ le groupe de tresses complexes associé. Soit $R=\mathbb{Z}[a_{s,i},a_{s,0}^{-1}]$ où $s$ appartient à un système de représentants des classes de conjugaison de réflexions distinguées dans $W$ et $0 \leq i \leq o(s)-1$, où $o(s)$ est l'ordre de $s$ dans $W$. On choisit une réflexion tressée distinguée $\sigma_s$ pour chaque réflexion distinguée $s$. La définition d'une réflexion tressée distinguée a été donnée avant.

\begin{definitionf}

L'algèbre de Hecke $H(W)$ associée à $W$ est le quotient de l'algèbre du groupe $RB$ par l'idéal engendré par les relations $${\sigma_s}^{o(s)} = \sum\limits_{i=0}^{o(s)-1} a_{s,i}{\sigma_s}^{i}$$ où $\sigma_s$ est la réflexion tressée distinguée associée à $s$ et $s$ appartient à un système de représentants des classes de conjugaison de réflexions distinguées dans $W$.

\end{definitionf}

\begin{remarkf}

Du fait que deux réflexions tressées distinguées associées à deux réflexions conjuguées dans $W$ sont conjuguées dans $B$ et du fait que les relations qui définissent l'algèbre de Hecke sont des relations polynomiales sur les réflexions tressées, la définition précédente de l'algèbre de Hecke ne dépend pas ni du choix du système de représentants des classes de conjugaison des réflexions distinguées dans $W$, ni du choix de la réflexion tressée distinguée associée à chaque réflexion distinguée. De plus, elle coincide avec la définition usuelle des algèbres de Hecke (cf. \cite{BMR} et \cite{MarinBMRFreenessConjecture}).

\end{remarkf}

\begin{remarkf}

Soit $W$ un groupe de Coxeter fini d'ensemble générateur $S$. L'algèbre de Hecke (ou Iwahori-Hecke) associée à $W$ est définie sur $R=\mathbb{Z}[a_{1},a_{0}^{-1}]$ par une présentation avec un ensemble générateur $\Sigma$ en bijection avec $S$ et les relations sont $\underset{m_{st}}{\underbrace{\sigma_s \sigma_t \sigma_s \cdots}}=\underset{m_{st}}{\underbrace{\sigma_t \sigma_s \sigma_t \cdots}}$ avec les relations polynomiales $\sigma_s^2 = a_1 \sigma_s + a_0$ pour tout $s \in S$.

\end{remarkf}

Une conjecture importante (la conjecture de liberté de BMR) sur $H(W)$ a été donnée dans \cite{BMR}:

\begin{conjecturef}

L'algèbre de Hecke $H(W)$ est un $R$-module libre de rang $|W|$.

\end{conjecturef}

Cette conjecture a été utilisée par un certain nombre d'auteurs en tant qu'hypothèse. Par exemple, Malle l'a utilisé pour prouver que les caractères de $H(W)$ prennent leurs valeurs dans un corps spécifique (cf. \cite{MalleCharactersHecke}). Notons que la validité de cette conjecture implique que $H(W) \otimes_R F$ est isomorphe à l'algèbre du groupe $FW$, où $F$ est la clôture algébrique du corps des fractions de $R$ (cf. \cite{MarinBMRFreenessConjecture}).\\

Cette conjecture est vraie pour les groupes de réflexions réels (cf. Lemma 4.4.3 de \cite{GeckPfeifferBook}). De plus, par un travail de Ariki \cite{ArikiHecke} et Ariki-Koike \cite{ArikiKoikeHecke}, cette conjecture est aussi vraie pour la série infinie $G(de,e,n)$. Marin a prouvé la conjecture pour $G_4$, $G_{25}$, $G_{26}$ et $G_{32}$ dans \cite{MarinCubicHeckeAlgebra5Strands} et \cite{MarinBMRFreenessConjecture}. Marin et Pfeiffer l'ont prouvé pour $G_{12}$, $G_{22}$, $G_{24}$, $G_{27}$, $G_{29}$, $G_{31}$, $G_{33}$ et $G_{34}$ dans \cite{MarinPfeifferBMRFreeness}. Dans sa thèse de doctorat et dans l'article qui a suivi (cf. \cite{ChavliThesis} et \cite{ChavliArticlePhD}), Chavli a pouvé la validité de cette conjecture pour $G_{5}$, $G_{6}$, $\cdots$, $G_{16}$. Récemment, Marin a prouvé la conjecture pour $G_{20}$ and $G_{21}$ (cf. \cite{MarinG20G21}) et finalement Tsushioka pour $G_{17}$, $G_{18}$ et $G_{19}$ (cf. \cite{TsuchiokaBMRfreenessConjecture}). Il s'ensuit alors le théorème suivant :

\begin{theoremf}

L'algèbre de Hecke $H(W)$ est un $R$-module libre de rang $|W|$.

\end{theoremf}

La raison pour laquelle nous rappelons la conjecture de liberté de BMR est que nous en donnons dans cette thèse une nouvelle preuve dans le cas des séries infinies des groupes de réflexions complexes de type $G(e,e,n)$ et $G(d,1,n)$.

\section*{Formes normales géodésiques pour $G(de,e,n)$}

Rappelons que $G(e,e,n)$ est le groupe des matrices monomiales de taille $n \times n$ où les coefficients de ces matrices sont des racines $e$-ième de l'unité et le produit des coefficients non nuls de chaque matrice est égal à $1$. Soient $e \geq 1$ et $n > 1$. On rappelle la présentation de Corran-Picantin de $G(e,e,n)$ donnée dans \cite{CorranPicantin}.

\begin{definitionf}

Le groupe de réflexions complexes $G(e,e,n)$ peut être défini par une présentation ayant comme ensemble générateur $\textbf{X} = \{\textbf{t}_i \ |\ i \in \mathbb{Z}/e\mathbb{Z}\} \cup \{\textbf{s}_3, \textbf{s}_4, \cdots, \textbf{s}_n \}$ et les relations sont les suivantes.

\begin{enumerate}

\item $\textbf{t}_i \textbf{t}_{i-1} = \textbf{t}_j \textbf{t}_{j-1}$ pour $i, j \in \mathbb{Z}/e\mathbb{Z}$,
\item $\textbf{t}_i \textbf{s}_3 \textbf{t}_i = \textbf{s}_3 \textbf{t}_i \textbf{s}_3$ pour $i \in \mathbb{Z}/e\mathbb{Z}$,
\item $\textbf{s}_j \textbf{t}_i = \textbf{t}_i \textbf{s}_j$ pour $i \in \mathbb{Z}/e\mathbb{Z}$ et $4 \leq j \leq n$,
\item $\textbf{s}_i \textbf{s}_{i+1} \textbf{s}_i = \textbf{s}_{i+1} \textbf{s}_i \textbf{s}_{i+1}$ pour $3 \leq i \leq n-1$, 
\item $\textbf{s}_i \textbf{s}_j = \textbf{s}_j \textbf{s}_i$ pour $|i-j| > 1$ et
\item $\textbf{t}_i^2=1$ pour $i \in \mathbb{Z}/e\mathbb{Z}$ et $\textbf{s}_j^2=1$ pour $3 \leq j \leq n$.

\end{enumerate}

\end{definitionf}

Les matrices de $G(e,e,n)$ qui correspondent à l'ensemble des générateurs $\mathbf{X}$ de la présentation sont données par $\mathbf{t}_i \longmapsto t_i:= \begin{pmatrix}

0 & \zeta_{e}^{-i} & 0\\
\zeta_{e}^{i} & 0 & 0\\
0 & 0 & I_{n-2}\\

\end{pmatrix}$ \mbox{pour $0 \leq i \leq e-1$} et $\mathbf{s}_j \longmapsto s_j:= \begin{pmatrix}

I_{j-2} & 0 & 0 & 0\\
0 & 0 & 1 & 0\\
0 & 1 & 0 & 0\\
0 & 0 & 0 & I_{n-j}\\

\end{pmatrix}$ pour $3 \leq j \leq n$. Pour éviter les confusions, nous utilisons des lettres en gras pour désigner les générateurs de la présentation et des lettres normales pour désigner les matrices correspondantes. Notons par $X$ l'ensemble  $\{t_0,t_1, \cdots, t_{e-1},s_3, \cdots, s_n\}$.\\

Cette présentation est décrite par le diagramme suivant où le cercle pointillé décrit la relation $1$ dans la définition de la présentation. Les autres arêtes suivent les conventions standard des diagrammes des groupes de Coxeter.

\begin{center}
\begin{tikzpicture}[yscale=0.8,xscale=1,rotate=30]

\draw[thick,dashed] (0,0) ellipse (2cm and 1cm);

\node[draw, shape=circle, fill=white, label=above:$\mathbf{t}_0$] (t0) at (0,-1) {\begin{tiny} 2 \end{tiny}};
\node[draw, shape=circle, fill=white, label=above:$\mathbf{t}_1$] (t1) at (1,-0.8) {\begin{tiny} 2 \end{tiny}};
\node[draw, shape=circle, fill=white, label=right:$\mathbf{t}_2$] (t2) at (2,0) {\begin{tiny} 2 \end{tiny}};
\node[draw, shape=circle, fill=white, label=above:$\mathbf{t}_i$] (ti) at (0,1) {\begin{tiny} 2 \end{tiny}};
\node[draw, shape=circle, fill=white, label=above:$\mathbf{t}_{e-1}$] (te-1) at (-1,-0.8) {\begin{tiny} 2 \end{tiny}};

\draw[thick,-] (0,-2) arc (-180:-90:3);

\node[draw, shape=circle, fill=white, label=below left:$\mathbf{s}_3$] (s3) at (0,-2) {\begin{tiny} 2 \end{tiny}};

\draw[thick,-] (t0) to (s3);
\draw[thick,-,bend left] (t1) to (s3);
\draw[thick,-,bend left] (t2) to (s3);
\draw[thick,-,bend left] (s3) to (te-1);

\node[draw, shape=circle, fill=white, label=below:$\mathbf{s}_4$] (s4) at (0.15,-3) {\begin{tiny} 2 \end{tiny}};
\node[draw, shape=circle, fill=white, label=below:$\mathbf{s}_{n-1}$] (sn-1) at (2.2,-4.9) {\begin{tiny} 2 \end{tiny}};
\node[draw, shape=circle, fill=white, label=right:$\mathbf{s}_{n}$] (sn) at (3,-5) {\begin{tiny} 2 \end{tiny}};

\node[fill=white] () at (1,-4.285) {$\cdots$};
\end{tikzpicture}
\end{center}

\begin{remarkf}\

\begin{enumerate}
\item Pour $e = 1$ et $n \geq 2$, on retrouve la présentation classique du groupe symétrique $S_n$.
\item Pour $e = 2$ et $n \geq 2$, on retrouve la présentation classique du groupe de Coxeter de type $D_n$.
\item Pour $e \geq 2$ et $n = 2$, on retrouve la présentation duale du groupe diédral $I_2(e)$, consulter \cite{PicantinPresentationsDualMonoids}. 

\end{enumerate}

\end{remarkf}

\begin{remarkf}

Dans leur article \cite{CorranPicantin}, Corran et Picantin ont prouvé que si on enlève les relations quadratiques de la présentation de $G(e,e,n)$, on obtient une présentation du groupe de tresses complexes $B(e,e,n)$. Ils ont aussi prouvé que cette présentaton fournit une structure de Garside de $B(e,e,n)$. La notion de structure de Garside sera développée par la suite.

\end{remarkf}

Notre but est de représenter chaque élément $w \in G(e,e,n)$ par un mot réduit sur $\mathbf{X}$. Ceci nécessite l'élaboration d'une technique combinatoire qui détermine une expression réduite sur $\mathbf{X}$ qui représente $w$. Nous avons introduit un algorithme qui produit un mot $R\!E(w)$ sur $\mathbf{X}$ pour chaque élément $w$ de $G(e,e,n)$. Nous avons prouvé que $R\!E(w)$ est une expression réduite sur $\mathbf{X}$ qui représente $w$. Nous allons donner une idée de cet algorithme.\\

Soit $w_n := w \in G(e,e,n)$. Pour $i$ allant de $n$ jusqu'à $2$, l'étape $i$ de l'algorithme transforme la matrice diagonale par blocs $\left(
\begin{array}{c|c}
w_i & 0 \\
\hline
0 & I_{n-i}
\end{array}
\right)$ en une matrice diagonale par blocs de la forme $\left(
\begin{array}{c|c}
w_{i-1} & 0 \\
\hline
0 & I_{n-i+1}
\end{array}
\right) \in G(e,e,n)$ avec $w_1 = 1$. En effet, pour $2 \leq i \leq n$, il existe un unique $c$ avec $1 \leq c \leq n$ tel que $w_i[i,c] \neq 0$. À chaque étape de l'algorithme, si  $w_i[i,c] =1$, on décale le $1$ jusqu'à la position diagonale $[i,i]$ par une mutiplication à droite par des matrices de transpositions. Si $w_i[i,c] \neq 1$, on décale cette racine de l'unité à la première colonne par une multiplication à droite par des transpositions, puis on la transforme en $1$ par une multiplication par un élément de $\{t_0, t_1, \cdots, t_{e-1} \}$. Enfin, on décale le $1$ obtenu en position $[i,2]$ à la position diagonale $[i,i]$ grâce à une multiplication à droite par des matrices de transpositions du groupe symétrique $S_n$. Nous avons le lemme suivant :

\begin{lemmaf}

Pour $2 \leq i \leq n$, supposons que $w_i[i,c] \neq 0$. Le bloc $w_{i-1}$ est obtenu
\begin{itemize}

\item en enlevant la ligne $i$ et la colonne $c$ de $w_i$ et puis
\item en multipliant la première colonne de la nouvelle matrice par $w_i[i,c]$.
\end{itemize}

\end{lemmaf}

L'algorithme précédent se généralise aux autres cas de la série infinie des groupes de réflexions complexes $G(de,e,n)$ pour $d > 1$, $e \geq 1$ et $n \geq 2$. Nous avons utilisé la présentation de Corran-Lee-Lee des groupes $G(de,e,n)$, cf. \cite{CorranLeeLee}. Elle est définie comme suit.

\begin{definitionf}

Le groupe de réflexions complexes $G(de,e,n)$ est défini par une présentation ayant comme ensemble des générateurs $\mathbf{X} = \{ \mathbf{z}\} \cup \{ \mathbf{t}_i \ | \ i \in \mathbb{Z}/de\mathbb{Z}\} \cup \{\mathbf{s}_3, \mathbf{s}_4, \cdots, \mathbf{s}_n \}$ et les relations sont :

\begin{enumerate}

\item $\mathbf{z} \mathbf{t}_i = \mathbf{t}_{i-e} \mathbf{z}$ pour $i \in \mathbb{Z}/de\mathbb{Z}$,
\item $\mathbf{z}  \mathbf{s}_j=\mathbf{s}_j \mathbf{z}$ pour $3 \leq j \leq n$,
\item $\mathbf{t}_i \mathbf{t}_{i-1} = \mathbf{t}_j \mathbf{t}_{j-1}$ pour $i, j \in \mathbb{Z}/de\mathbb{Z}$,
\item $\mathbf{t}_i \mathbf{s}_3 \mathbf{t}_i = \mathbf{s}_3 \mathbf{t}_i \mathbf{s}_3$ pour $i \in \mathbb{Z}/de\mathbb{Z}$,
\item $\mathbf{s}_j \mathbf{t}_i = \mathbf{t}_i \mathbf{s}_j$ pour $i \in \mathbb{Z}/de\mathbb{Z}$ et $4 \leq j \leq n$,
\item $\mathbf{s}_i \mathbf{s}_{i+1} \mathbf{s}_i = \mathbf{s}_{i+1} \mathbf{s}_i \mathbf{s}_{i+1}$ pour $3 \leq i \leq n-1$,
\item $\mathbf{s}_i \mathbf{s}_j = \mathbf{s}_j \mathbf{s}_i$ pour $|i-j| > 1$ et
\item $\mathbf{z}^d=1$, $\mathbf{t}_i^2=1$ pour $i \in \mathbb{Z}/de\mathbb{Z}$ et $\mathbf{s}_j^2=1$ pour $3 \leq j \leq n$.

\end{enumerate}

\end{definitionf}

Les générateurs de cette présentation correspondent aux matrices $n \times n$ suivantes :
Le générateur $\mathbf{t}_i$ est représenté par la matrice $t_i =
\begin{pmatrix}

0 & \zeta_{de}^{-i} & 0\\
\zeta_{de}^{i} & 0 & 0\\
0 & 0 & I_{n-2}\\

\end{pmatrix}$
pour $i \in \mathbb{Z}/de\mathbb{Z}$, $\mathbf{z}$ par la matrice diagonale $z = Diag(\zeta_d,1,\cdots,1)$ où
$\zeta_d = exp(2i \pi / d)$, 
et $\mathbf{s}_j$ par la matrice de la transposition $s_j = (j-1,j)$ pour $3 \leq j \leq n$.\\

Cette présentation peut être décrite par le diagramme suivant. La flèche courbée au-dessous de $\mathbf{z}$ décrit la relation $1$ de la définition précédente.

\begin{center}
\begin{tikzpicture}[yscale=0.8,xscale=1,rotate=30]

\draw[thick,dashed] (0,0) ellipse (2cm and 1cm);

\node[draw, shape=circle, fill=white, label=above:\begin{small}$\mathbf{z}$\end{small}] (z) at (0,0) {\begin{tiny}$d$\end{tiny}};
\node[draw, shape=circle, fill=white, label=above:\begin{small}$\mathbf{t}_0$\end{small}] (t0) at (0,-1) {\begin{tiny}$2$\end{tiny}};
\node[draw, shape=circle, fill=white, label=above:\begin{small}$\mathbf{t}_1$\end{small}] (t1) at (1,-0.8) {\begin{tiny}$2$\end{tiny}};
\node[draw, shape=circle, fill=white, label=right:\begin{small}$\mathbf{t}_2$\end{small}] (t2) at (2,0) {\begin{tiny}$2$\end{tiny}};
\node[draw, shape=circle, fill=white, label=above:$\mathbf{t}_i$] (ti) at (0,1) {\begin{tiny}$2$\end{tiny}};
\node[draw, shape=circle, fill=white, label=above:\begin{small}$\mathbf{t}_{de-1}$\end{small}] (te-1) at (-1,-0.8) {\begin{tiny}$2$\end{tiny}};

\draw[->,bend left=45] (-0.3,0.26) to (0.3,0.26);
\draw[thick,-] (0,-2) arc (-180:-90:3);

\node[draw, shape=circle, fill=white, label=below left:$\mathbf{s}_3$] (s3) at (0,-2) {\begin{tiny}$2$\end{tiny}};

\draw[thick,-] (t0) to (s3);
\draw[thick,-,bend left] (t1) to (s3);
\draw[thick,-,bend left] (t2) to (s3);
\draw[thick,-,bend left] (s3) to (te-1);

\node[draw, shape=circle, fill=white, label=below:$\mathbf{s}_4$] (s4) at (0.15,-3) {\begin{tiny}$2$\end{tiny}};
\node[draw, shape=circle, fill=white, label=below:$\mathbf{s}_{n-1}$] (sn-1) at (2.2,-4.9) {\begin{tiny}$2$\end{tiny}};
\node[draw, shape=circle, fill=white, label=right:$\mathbf{s}_{n}$] (sn) at (3,-5) {\begin{tiny}$2$\end{tiny}};

\node[fill=white] () at (1,-4.285) {$\cdots$};

\end{tikzpicture}
\end{center}

\begin{definitionf}

Soit $e=1$. La présentation de la définition précédente est équivalente à la présentation classique du groupe de réflexions complexes $G(d,1,n)$ qui est décrite par le diagramme suivant.

\begin{center}
\begin{tikzpicture}

\node[draw, shape=circle, label=above:$\mathbf{z}$] (1) at (0,0) {$d$};
\node[draw, shape=circle,label=above:$\mathbf{s}_2$] (2) at (2,0) {$2$};
\node[draw, shape=circle,label=above:$\mathbf{s}_3$] (3) at (4,0) {$2$};
\node[draw, shape=circle,label=above:$\mathbf{s}_{n-1}$] (n-1) at (6,0) {$2$};
\node[draw,shape=circle,label=above:$\mathbf{s}_n$] (n) at (8,0) {$2$};

\draw[thick,-,double] (1) to (2);
\draw[thick,-] (2) to (3);
\draw[dashed,-,thick] (3) to (n-1);
\draw[thick,-] (n-1) to (n);

\end{tikzpicture}
\end{center}

\end{definitionf}

\begin{remarkf}

Pour $d=2$, la présentation décrite par le diagramme précédent est la présentation classique du groupe de Coxeter de type $B_n$.

\end{remarkf}

Les algorithmes qui produisent des formes normales géodésiques pour $G(de,e,n)$ et $G(d,1,n)$ sont les algorithmes $2$ et $3$ que nous avons donnés dans le chapitre 2 de cette dissertation. Nous ne rappelons pas ces deux algorithmes dans ce compte-rendu. Ils sont similaires à l'algorithme qui produit des formes normales géodésiques pour les groupes $G(e,e,n)$.\\

Nous allons maintenant donner les éléments de $G(e,e,n)$ de longueur maximale. Il s'agit d'une application directe de l'algorithme qui produit les formes normales géodésiques de $G(e,e,n)$.

\begin{propositionf}

Soit $e > 1$ et $n \geq 2$. La longueur maximale d'un élément de $G(e,e,n)$ est $n(n-1)$. Elle est réalisée pour les matrices diagonales $w$ telle que $w[i,i]$ est une racine $e$-ième de l'unité qui est différente de $1$ pour $2 \leq i \leq n$. Un mot réduit qui représente cet élément est de la forme $$(\mathbf{t}_{k_2}\mathbf{t}_0)(\mathbf{s}_3\mathbf{t}_{k_3}\mathbf{t}_0\mathbf{s}_3)\cdots (\mathbf{s}_n \cdots \mathbf{s}_3\mathbf{t}_{k_n}\mathbf{t}_0\mathbf{s}_3 \cdots \mathbf{s}_n),$$ avec $1 \leq k_2, \cdots, k_n \leq e-1$. Le nombre d'éléments de cette forme est $(e-1)^{(n-1)}$.

\end{propositionf}

\noindent Notons par $\lambda$ l'élément  $\begin{pmatrix}

(\zeta_e^{-1})^{(n-1)} & & & \\
 & \zeta_e & & \\
 & & \ddots & \\
 & & & \zeta_e\\

\end{pmatrix} \in G(e,e,n)$.

\begin{examplef}

Nous avons $R\!E(\lambda) = (\mathbf{t}_{1}\mathbf{t}_0)(\mathbf{s}_3\mathbf{t}_{1}\mathbf{t}_0\mathbf{s}_3)\cdots (\mathbf{s}_n \cdots \mathbf{s}_3\mathbf{t}_{1}\mathbf{t}_0\mathbf{s}_3 \cdots \mathbf{s}_n)$. Donc ${\ell}(\lambda) = n(n-1)$ ce qui correspond à la longueur maximale d'un élément de $G(e,e,n)$.

\end{examplef}

\begin{remarkf}

Considérons le groupe $G(1,1,n)$ qui est le groupe symétrique $S_n$. Il existe un unique élément de longueur maximale dans $S_n$ qui est de la forme $$\mathbf{t}_0 (\mathbf{s}_3\mathbf{t}_0) \cdots (\mathbf{s}_{n-1}\cdots \mathbf{s}_3\mathbf{t}_0) (\mathbf{s}_{n}\cdots \mathbf{s}_3\mathbf{t}_0).$$ Ceci correspond au nombre maximal d'étapes de l'algorithme. La longeur de cet élément est $n(n-1)/2$ ce qui est déjà connu pour le groupe symétrique $S_n$, consulter Example 1.5.4 de \cite{GeckPfeifferBook}.

\end{remarkf}

\begin{remarkf}

Considérons le groupe $G(2,2,n)$ qui est le groupe de Coxeter de type $D_n$. Nous avons $e = 2$. Donc, par la proposition précédente, le nombre d'éléments de longueur maximale est égal à $(e-1)^{(n-1)} = 1$. Ainsi, il existe un unique élément de longueur maximale dans $G(2,2,n)$ qui est de la forme $$(\mathbf{t}_{1}\mathbf{t}_0)(\mathbf{s}_3\mathbf{t}_{1}\mathbf{t}_0\mathbf{s}_3) \cdots (\mathbf{s}_n \cdots \mathbf{s}_3\mathbf{t}_{1}\mathbf{t}_0\mathbf{s}_3 \cdots \mathbf{s}_n).$$ La longueur de cet élément est $n(n-1)$ ce qui est déjà connu pour les groupes de Coxeter de type $D_n$, consulter Example 1.5.5 de \cite{GeckPfeifferBook}.

\end{remarkf}

\section*{Algèbres de Hecke pour $G(e,e,n)$ et $G(d,1,n)$}

Nous allons donner la définition des algèbres de Hecke associés aux groupes de réflexions complexes $G(e,e,n)$ et $G(d,1,n)$. Nous utilisons la présentation de Corran et Picantin du groupe de tresses complexes $B(e,e,n)$ et nous utilisons la présentation classique du groupe d'Artin-Tits $B(d,1,n)$.

\begin{definitionf}

Soient $e \geq 1$ et $n \geq 2$. Soit $R_0 = \mathbb{Z}[a]$. On définit une algèbre de Hecke associative et unitaire $H(e,e,n)$ comme quotient de l'algèbre du groupe $R_0(B(e,e,n))$ par les relations suivantes :

\begin{enumerate}

\item $t_i^2 - at_i - 1 = 0$ pour $0 \leq i \leq e-1$,
\item $s_j^2 - as_j - 1 = 0$ pour $3 \leq j \leq n$.

\end{enumerate}

Donc, une présentation de $H(e,e,n)$ est obtenue en ajoutant les relations précédentes aux relations de la présentation de Corran et Picantin de $B(e,e,n)$.

\end{definitionf}

\begin{definitionf}

Soient $d > 1$ et $n \geq 2$. Soit $R_0 = \mathbb{Z}[a,b_1,b_2, \cdots, b_{d-1}]$. On définit une algèbre de Hecke associative et unitaire $H(d,1,n)$ comme quotient de l'algèbre du groupe $R_0(B(d,1,n))$ par les relations suivantes :

\begin{enumerate}

\item $z^d -b_1 z^{d-1} - b_2 z^{d-2} - \cdots - b_{d-1} z -1 = 0$,
\item $s_j^2 - as_j - 1 = 0$ pour $2 \leq j \leq n$.

\end{enumerate}

Donc, une présentation de $H(d,1,n)$ est obtenue en ajoutant les relations précédentes aux relations de la présentation classique de $B(d,1,n)$.

\end{definitionf}

Dans les définitions précédentes, nous avons utilisé l'anneau des polynômes $R_0$ au lieu de l'anneau usuel des polynômes de Laurent $R$ que nous avons déjà introduit dans la définition générale des algèbres de Hecke. En effet, par Proposition 2.3 $(ii)$ de \cite{MarinG20G21}, la conjecture de liberté de BMR (appliquée au cas de $G(e,e,n)$) est équivalente au fait que $H(e,e,n)$ est un $R_0$-module de rang $|G(e,e,n)|$. La même chose est valable pour $G(d,1,n)$. Donc, en prouvant que $H(e,e,n)$ et $H(d,1,n)$ sont des $R_0$-modules libres de rang $|G(e,e,n)|$ et $|G(d,1,n)|$, respectivement, on obtient une nouvelle preuve de la conjecture de liberté de BMR pour les cas de $G(e,e,n)$ et $G(d,1,n)$.\\

On utilise la forme normale géodésique que nous avons introduite précédemment pour $G(e,e,n)$ afin de construire une base de $H(e,e,n)$ qui est différente de celle de Ariki définie dans \cite{ArikiHecke}. On introduit les sous-ensembles suivants de $H(e,e,n)$.

\begin{center}
\begin{tabular}{lllll}
$\Lambda_2 =$ & $\{$ & $1$,\\
& & $t_k$ & pour $0 \leq k \leq e-1$, & \\
 & & $t_kt_0$ & pour $1 \leq k \leq e-1$ & $\}$,
\end{tabular}
\end{center}
et pour $3 \leq i \leq n$,
\begin{center}
\begin{tabular}{lllll}
$\Lambda_i =$ & $\{$ & $1$, & & \\
 & & $s_i \cdots s_{i'}$ & pour $3 \leq i' \leq i$, & \\ 
 & & $s_i \cdots s_{3}t_k$ & pour $0 \leq k \leq e-1$, & \\
 & & $s_i \cdots s_{3}t_ks_2 \cdots s_{i'}$ & pour $1 \leq k \leq e-1$ et $2 \leq i' \leq i$ & $\}.$\\
\end{tabular}
\end{center}

Définissons $\Lambda = \Lambda_2 \cdots \Lambda_n$ l'ensemble des produits  $a_2 \cdots a_n$ où $a_2 \in \Lambda_2, \cdots, a_n \in \Lambda_n$. Nous avons prouvé le théorème suivant :

\begin{theoremf}

L'ensemble $\Lambda$ fournit une $R_0$-base de l'algèbre de Hecke $H(e,e,n)$.

\end{theoremf}

\begin{remarkf}

Comme conséquence du théorème précédent et par la proposition 2.3 $(ii)$ de \cite{MarinG20G21}, nous obtenons une nouvelle preuve de la conjecture de liberté de BMR pour les groupes de réflexions complexes $G(e,e,n)$.

\end{remarkf}

\begin{remarkf}

Notre base de $H(e,e,n)$ est construite naturellement à partir des formes normales géodésiques de $G(e,e,n)$. Elle ne coincide jamais avec la base de Ariki pour les algèbres de Hecke associées aux groupes de réflexions complexes $G(e,e,n)$ comme le montre l'exemple suivant. Considérons l'élément $t_1t_0.\ t_0$ qui appartient à la base d'Ariki. Dans notre base, il est égal à la combinaison linéaire non triviale $at_1t_0 +t_1$ où $t_1t_0$ et $t_1$ sont deux éléments distincts de notre base. 

\end{remarkf}

\vspace{0.2cm}

%
%
%
%

Nous revenons maintenant au cas de l'algèbre de Hecke $H(d,1,n)$ définie avant par une présentation par des générateurs et relations. Nous allons construire une base de $H(d,1,n)$ qui est différente de celle de Ariki et Koike, cf. \cite{ArikiKoikeHecke}. On introduit les sous-ensembles suivants de $H(d,1,n)$. 

\begin{center}
\begin{tabular}{lllll}
$\Lambda_1 =$ & $\{$ & $z^k$ & pour $0 \leq k \leq d-1$ & $\}$,
\end{tabular}
\end{center}
et pour $2 \leq i \leq n$,
\begin{center}
\begin{tabular}{lllll}
$\Lambda_i =$ & $\{$ & $1$, & & \\
 & & $s_i \cdots s_{i'}$ & pour $2 \leq i' \leq i$, & \\ 
 & & $s_i \cdots s_{2}z^k$ & pour $1 \leq k \leq d-1$, & \\
 & & $s_i \cdots s_{2}z^ks_2 \cdots s_{i'}$ & pour $1 \leq k \leq d-1$ et $2 \leq i' \leq i$ & $\}.$\\
\end{tabular}
\end{center}

Définissons $\Lambda = \Lambda_1\Lambda_2 \cdots \Lambda_n$ l'ensemble des produits $a_1a_2 \cdots a_n$ où $a_1 \in \Lambda_1, \cdots,\\ a_n \in \Lambda_n$. Nous avons démontré le théorème suivant :

\begin{theoremf}

L'ensemble $\Lambda$ fournit une $R_0$-base de l'algèbre de Hecke $H(d,1,n)$.

\end{theoremf}

\begin{remarkf}

On remarque que pour $d$ et $n$ au-moins égaux à $2$, notre base de $H(d,1,n)$ ne coïncide jamais avec la base de Ariki-Koike comme le montre l'exemple suivant. Considérons l'élément $s_2zs_2^2 = s_2zs_2.s_2$ qui appartient à la base de Ariki-Koike. Dans notre base, il est égal à la combinaison linéaire $as_2zs_2 + s_2z$, où $s_2zs_2$ et $s_2z$ sont deux éléments distincts de notre base de $H(d,1,n)$. 

\end{remarkf}

\section*{Monoïdes et groupes de Garside}

Dans sa thèse de Doctorat, défendue en 1965 \cite{GarsideThesis} et dans l'article qui a suivi \cite{GarsideArt}, Garside a résolu le Problème de Conjugaison pour le groupe de tresses `classique' $B_n$ en introduisant un sous-monoïde $B_{n}^{+}$ de $B_n$ et un élément $\Delta_n$ de $B_{n}^{+}$ qu'il avait appelé élément fondamental. Il a prouvé l'existence d'une forme normale pour chaque élément de $B_n$. Au début des années 70, Brieskorn-Saito et Deligne ont remarqué que les résultats de Garside s'étendent à tous les groupes d'Artin-Tits, consulter \cite{GarsideBrieskornSaito} et \cite{GarsideDeligne}. À la fin des années 90 et après avoir lister les propriétés abstraites de $B_n^{+}$ et de l'élément fondamental $\Delta_n$, Dehornoy et Paris \cite{GarsideDehornoyParis} ont défini la notion de groupes Gaussiens et groupes de Garside, ce qui a mené à la théorie de Garside. Pour une étude détaillée sur les structures de Garside, le lecteur est invité à lire \cite{GarsideBookPatrick}. Nous fournissons ici les préliminaires nécessaires sur les structures de Garside.\\

Soit $M$ un monoïde. Sous certaines hypothèses sur $M$, plus précisément les hypothèses 1 et 2 de la définition suivante d'un monoïde de Garside, nous pouvons définir un ordre partiel sur $M$ comme suit.

\begin{definitionf}

Soient $f, g \in M$. On dit que $f$ divise à gauche $g$ ou tout simplement $f$ divise $g$ s'il n'y a pas de confusion, et on écrit $f \preceq g$, si $fg'=g$ pour un certain $g' \in M$. De la même manière, on dit que $f$ divise à droite $g$, et on écrit $f \preceq_r g$, si $g'f=g$ pour un certain $g' \in M$. 

\end{definitionf}

Nous donnons maintenant la définition d'un monoïde et groupe de Garside.

\begin{definitionf}

Un monoïde de Garside est une paire $(M, \Delta)$ où $M$ est un monoïde et

\begin{enumerate}

\item $M$ est simplifiable c'est-à-dire $fg=fh \Longrightarrow g=h$ et $gf=hf \Longrightarrow g=h$ pour $f,g,h \in M$,

\item il existe $\lambda : M \longrightarrow \mathbb{N}$ t.q. $\lambda(fg) \geq \lambda(f) + \lambda(g)$ et $g \neq 1 \Longrightarrow \lambda(g) \neq 0$,

\item deux éléments quelconques de $M$ ont un pgcd et un ppcm pour $\preceq$ et $\preceq_r$ et

\item $\Delta$ est un élément de Garside de $M$ c'est-à-dire l'ensemble de ses diviseurs à gauche coïncide avec l'ensemble de ses diviseurs à droite, engendre $M$ et est fini.

\end{enumerate} 

Les diviseurs de $\Delta$ s'appellent les \emph{simples} de $M$.

\end{definitionf}

Un monoïde quasi-Garside est une paire $(M,\Delta)$ qui satisfait les conditions de la définition précédente sauf pour la finitude de l'ensemble des diviseurs de $\Delta$.\\

Les hypothèses 1 et 3 de la définition précédente impliquent que les conditions de Ore sont satisfaites. Ainsi, il existe un groupe de fractions du monoïde $M$ dans lequel il se plonge. Ceci nous permet de donner la définition suivante.

\begin{definitionf}

Un groupe de Garside est le groupe des fractions d'un monoïde de Garside.

\end{definitionf}

Considérons une paire $(M,\Delta)$ qui satisfait les conditions de la définition précédente. La paire $(M,\Delta)$ et le groupe des fractions de $M$ fournissent ce qu'on appelle une structure de Garside de $M$.\\

Notons que l'un des aspects importants de la structure de Garside est l'existence d'une forme normale pour les éléments du groupe de Garside. De plus, beaucoup de problèmes comme le Problème des Mots et le Problème de Conjugaison sont résolus dans les groupes de Garside, ce qui rend leurs études intéressantes.\\

Soit $G$ un groupe fini engendré par un ensemble fini $S$. Il existe une manière de construire des structures de Garside à partir d'intervalles dans $G$. Nous allons expliquer ceci par la suite. Commençons par définir une relation d'ordre partiel sur $G$.

\begin{definitionf}

Soient $f, g \in G$. On dit que $g$ est un diviseur de $f$ ou que $f$ est un multiple de $g$ et on écrit $g \preceq f$ si $f = g h$ avec $h \in G$ et $\ell(f) = \ell(g) + \ell(h)$, où $\ell(f)$ est la longeur sur $S$ de $f \in G$.

\end{definitionf}

\begin{definitionf}

Pour $w \in G$, on définit le monoïde $M([1,w])$ par la présentation de monoïde avec

\begin{itemize}

\item ensemble générateur : $\underline{P}$ en bijection avec l'intervalle
\begin{center} $[1,w] := \{ f \in G \ | \ 1 \preceq f \preceq w \}$ et \end{center}

\item relations : $\underline{f}\ \underline{g} = \underline{h}$ si $f, g , h \in [1,w]$, $fg=h$ et $f \preceq h$ c'est-à-dire $\ell(f) + \ell(g) = \ell(h)$.

\end{itemize}

\end{definitionf}

\noindent De la même manière, on peut définir la relation d'ordre partielle sur $G$
\begin{center}$g \preceq_r f$ si et seulement si $\ell(f{g}^{-1}) + \ell(g) = \ell(f)$, \end{center}
ensuite définir l'intervalle $[1,w]_r$ et le monoïde $M([1,w]_r)$. 

\begin{definitionf}

Soit $w$ un élément dans $G$. On dit que $w$ est un élément équilibré de $G$ si $[1,w] = [1,w]_r$.

\end{definitionf}

Nous avons le théorème suivant qui est dû à Michel (consulter la section 10 de \cite{CorseJeanMichel} pour une preuve).

\begin{theoremf}

Si $w \in G$ est un élément équilibré et si $([1,w],\preceq)$ et $([1,w]_r, \preceq_r)$ sont des treillis, alors $(M([1,w]),\underline{w})$ est un monoïde de Garside avec simples $\underline{[1,w]}$, où $\underline{w}$ et $\underline{[1,w]}$ sont donnés dans la définition d'un monoïde d'intervalle.

\end{theoremf}

La construction précédente donne lieu à une structure de Garside. Le monoïde d'intervalle est $M([1,w])$. Quand $M([1,w])$ est un monoïde de Garside, son groupe des fractions existe et est noté par $G(M([1,w]))$. On l'appelle le groupe d'intervalle. Nous donnons un exemple classique de cette structure qui montre que les groupes d'Artin-Tits admettent des structures d'intervalles.

\begin{examplef}

Soit $W$ un groupe fini de Coxeter et $B(W)$ le groupe d'Artin-Tits associé à $W$. 

\begin{center} $W = <S\ |\ s^2=1, \ \underset{m_{st}}{\underbrace{sts\cdots}}=\underset{m_{st}}{\underbrace{tst\cdots}}$ pour $s, t \in S, s \neq t, m_{st} = o(st)>,$\\
$B(W) = <\widetilde{S}\ |\ \underset{m_{st}}{\underbrace{\tilde{s}\tilde{t}\tilde{s}\cdots}}=\underset{m_{st}}{\underbrace{\tilde{t}\tilde{s}\tilde{t}\cdots}}\ pour\ \tilde{s}, \tilde{t} \in \widetilde{S},\ \tilde{s} \neq \tilde{t}>$.
\end{center}
Prendre $G = W$ et $g=w_0$ l'élément de plus grande longueur sur $S$ dans $W$. On a $[1,w_0] = W$. Construisons le monoïde d'intervalle $M([1,w_0])$. On a $M([1,w_0])$ est le monoïde d'Artin-Tits $B^{+}(W)$, où $B^{+}(W)$ est le monoïde défini par la même présentation que $B(W)$. Donc, $B^{+}(W)$ est engendré par une copie $\underline{W}$ de $W$ avec $\underline{f}\ \underline{g} = \underline{h}$ si $fg=h$ et $\ell(f) + \ell(g) = \ell(h)$; $f,g$ et $h \in W$. Il est aussi connu que $w_0$ est équilibré et que $([1,w_0],\preceq)$ et $([1,w_0]_r, \preceq_r)$ sont des treillis. Par le théorème précédent, il s'ensuit le résultat suivant.

\end{examplef}

\begin{propositionf}

$(B^{+}(W),\underline{w_0})$ est un monoïde de Garside avec simples $\underline{W}$, où $\underline{w_0}$ et $\underline{W}$ sont donnés dans l'exemple précédent.

\end{propositionf}

Notre but est de construire des structures d'intervalles pour le groupe de tresses complexes $B(e,e,n)$. Comme les structures de Garside d'intervalles reposent sur la définition d'une fonction qui décrit la longueur dans le groupe de réflexions complexes correspondant, nous allons utiliser la forme normale géodésique que nous avons déjà définie auparavant pour le groupe $G(e,e,n)$ afin d'étudier les structures d'intervalles du groupe de tresses complexes $B(e,e,n)$.

\section*{Structures d'intervalles pour $B(e,e,n)$}

Il est prouvé par Bessis et Corran \cite{BessisCorranNonCrossingPartitions} en 2006 et par Corran et Picantin \cite{CorranPicantin} en 2009 que $B(e,e,n)$ admet des structures de Garside. Il est aussi prouvé dans \cite{CorranLeeLee} que $B(de,e,n)$ admet des structures de quasi-Garside. Nous sommes intéressés dans la construction des structures de Garside de $B(e,e,n)$ qui dérivent des intervalles dans le groupe de réflexions complexes $G(e,e,n)$.\\

La description d'une forme normale géodésique de $G(e,e,n)$ nous permet de déterminer les éléments équilibrés de $G(e,e,n)$ qui sont de longueur maximale et de déterminer aussi l'ensemble de leurs diviseurs. Soit $\lambda$ la matrice diagonale de $G(e,e,n)$ telle que $\lambda[i,i] = \zeta_e$ pour $2 \leq i \leq n$. 

\begin{propositionf}

Les éléments équilibrés de $G(e,e,n)$ de longueur maximale sont $\lambda^k$ pour $1 \leq k \leq e-1$.

\end{propositionf}

Nous avons défini une technique combinatoire simple qui permet de reconnaître si un élément de $G(e,e,n)$ est un diviseur de $\lambda^k$ ou pas en utilisant directement sa forme matricielle. Ceci nous permet de caractériser explicitement l'ensemble des diviseurs de $\lambda^k$ que nous avons noté par $D_k$.\\

Soient $\lambda^k$ un élément équilibré et $[1,\lambda^k]$ l'intervalle des diviseurs de $\lambda^k$ dans le groupe de réflexions complexes. Nous avons prouvé dans Proposition \ref{PropMatsumoto} une propriété pour chaque intervalle $[1,\lambda^k]$ qui est similaire à la propriété de Matsumoto pour un groupe de réflexions réels. Ensuite, nous avons prouvé que $[1,\lambda^k]$ est un treillis pour la division à gauche et à droite (consulter Corollary \ref{CorBothPosetsLattices} de cette dissertation). Par le théorème de Michel, nous obtenons alors le résultat suivant :

\begin{theoremf}

Nous avons que $(M([1,\lambda^k]), \underline{\lambda^k})$ est un monoïde de Garside d'intervalles avec simples $\underline{D_k}$ où $D_k$ est l'ensemble des diviseurs de $\lambda^k$. Son groupe des fractions existe et est noté par $G(M([1,\lambda^k]))$.

\end{theoremf}

De plus, nous avons prouvé que $M([1,\lambda^k])$ est isomorphe au monoïde que nous avons noté par $B^{\oplus k}(e,e,n)$. Ce dernier est défini comme suit :

\begin{definitionf}

Pour $1 \leq k \leq e-1$, on définit le monoïde $B^{\oplus k}(e,e,n)$ par une présentation de monoïde avec

\begin{itemize}

\item ensemble générateur : $\widetilde{X} = \{ \tilde{t}_{0}, \tilde{t}_{1}, \cdots, \tilde{t}_{e-1}, \tilde{s}_{3}, \cdots, \tilde{s}_{n}\}$ et
\item relations : $\left\{
\begin{array}{ll}

\tilde{s}_{i}\tilde{s}_{j}\tilde{s}_{i} = \tilde{s}_{j}\tilde{s}_{i}\tilde{s}_{j} & pour\ |i-j|=1,\\
\tilde{s}_{i}\tilde{s}_{j} = \tilde{s}_{j}\tilde{s}_{i} & pour\ |i-j| > 1,\\
\tilde{s}_{3}\tilde{t}_{i}\tilde{s}_{3} = \tilde{t}_{i}\tilde{s}_{3}\tilde{t}_{i} & pour\ i \in \mathbb{Z}/e\mathbb{Z},\\
\tilde{s}_{j}\tilde{t}_{i} = \tilde{t}_{i}\tilde{s}_{j} & pour\ i \in \mathbb{Z}/e\mathbb{Z}\ et\ 4 \leq j \leq n,\\
\tilde{t}_{i}\tilde{t}_{i-k} = \tilde{t}_{j}\tilde{t}_{j-k} & pour\ i, j \in \mathbb{Z}/e\mathbb{Z}.

\end{array}
\right.$

\end{itemize}

\end{definitionf}

Notons que le monoïde $B^{\oplus 1}(e,e,n)$ est le monoïde $B^{\oplus}(e,e,n)$ de Corran et Picantin, consulter \cite{CorranPicantin}. Notons par $B^{(k)}(e,e,n)$ le groupe des fractions (qui existe par la structure de Garside) de $B^{\oplus k}(e,e,n)$. L'un des résultats importants obtenus est le suivant :

\begin{theoremf}

$B^{(k)}(e,e,n)$ est isomorphe à $B(e,e,n)$ si et seulement si \mbox{$k\wedge e=1$.}

\end{theoremf}

Quand $k \wedge e \neq 1$, chaque groupe $B^{(k)}(e,e,n)$ est décrit comme un produit amalgamé de $k \wedge e$ copies du groupe de tresses complexes $B(e',e',n)$ avec $e'=e/e \wedge k$, au-dessus du sous-groupe d'Artin-Tits $B(2,1,n-1)$.

\begin{center}
\begin{tikzpicture}

\node[draw, shape=rectangle, label=below:{$B(e',e',n)$}] (1) at (0,0) {};
\node[draw, shape=rectangle, label=below:{$B(e',e',n)$}] (2) at (2,0) {};
\node[draw, shape=rectangle, label=below:{$B(e',e',n)$}] (3) at (-1,1) {};
\node[draw, shape=rectangle, label=right:{$B(2)$}] (4) at (1,1) {};
\node[draw, shape=rectangle, label=right:{$B(3)$}] (5) at (0,2) {};
\node[draw, shape=rectangle, label=left:{$B(e',e',n)$}] (6) at (-3,3) {};
\node[draw, shape=rectangle] (7) at (-1,3) {};
\node[draw, shape=rectangle, label=right:{$B(e \wedge k) = B^{(k)}(e,e,n)$}] (8) at (-2,4) {};

\draw[thick,-] (1) to (4);
\draw[thick,-] (2) to (4);
\draw[thick,-] (5) to (3);
\draw[thick,-] (5) to (4);
\draw[thick,dashed,-] (7) to (5);
\draw[thick,dashed,-] (-1.5,1.5) to (-2.5,2.5);
\draw[thick,-] (8) to (6);
\draw[thick,-] (8) to (7);

\end{tikzpicture}
\end{center}

De plus, nous avons calculé le deuxième groupe d'homologie sur $\mathbb{Z}$ de $B^{(k)}(e,e,n)$ en utilisant les complexes de Dehornoy-Lafont \cite{DehLafHomologyofGaussianGroups} et la méthode de \cite{CalMarHomologyComputations} afin de déduire que $B^{(k)}(e,e,n)$ n'est pas isomorphe à un certain $B(e,e,n)$ quand $k \wedge e \neq 1$. On rappelle le résulat suivant de \cite{CalMarHomologyComputations} :

\begin{itemize}

\item $H_2(B(e,e,3),\mathbb{Z}) \simeq \mathbb{Z}/e\mathbb{Z}$ où $e \geq 2$,
\item $H_2(B(e,e,4),\mathbb{Z}) \simeq \mathbb{Z}/e\mathbb{Z} \times \mathbb{Z}/2\mathbb{Z}$ quand $e$ est impair et $H_2(B(e,e,4),\mathbb{Z}) \simeq \mathbb{Z}/e\mathbb{Z} \times (\mathbb{Z}/2\mathbb{Z})^2$ quand $e$ est pair et
\item $H_2(B(e,e,n),\mathbb{Z}) \simeq \mathbb{Z}/e\mathbb{Z} \times \mathbb{Z}/2\mathbb{Z}$ quand $n \geq 5$ et $e \geq 2$.

\end{itemize}

Les monoïdes de Garside $B^{\oplus k}(e,e,n)$ ont été implémentés par Michel et l'auteur (consulter \cite{CHEVIEJMichel} et \cite{GAP3CPMichelNeaime}). Dans Appendix A, nous avons fourni cette implémentation et nous avons donné un code générique qui permet de calculer les groupes d'homologie d'une structure de Garside en utilisant les complexes de Dehornoy et Lafont. Nous avons appliqué ce code aux groupes $B^{(k)}(e,e,n)$ afin de vérifier le calcul de leurs groupes d'homologie.

\section*{Représentations de Krammer et algèbres de BMW}

Bigelow \cite{BigelowBraidGroupsLinear} et Krammer \cite{KrammerGroupB4,KrammerBraidGroupsLinear} ont prouvé que le groupe de tresses classique est linéaire c'est-à-dire qu'il existe une représentation linéaire fidèle de dimension finie du groupe de tresses classique $B_n$. Rappelons que $B_n$ est défini par une présentation avec comme ensemble générateur $\{s_1, s_2, \cdots, s_{n-1} \}$ et les relations sont $s_is_{i+1}s_i = s_{i+1}s_is_{i+1}$ pour $1 \leq i \leq n-2$, et $s_is_j = s_js_i$ pour $|i-j| > 1$. Cette présentation est illustrée par le diagramme suivant :

\begin{center}
\begin{tikzpicture}

\node[draw, shape=circle, label=above:$s_1$] (1) at (0,0) {};
\node[draw, shape=circle,label=above:$s_2$] (2) at (1,0) {};
\node[draw,shape=circle,label=above:$s_{n-2}$] (n-1) at (4,0) {};
\node[draw,shape=circle,label=above:$s_{n-1}$] (n) at (5,0) {};

\draw[thick,-] (1) to (2);
\draw[dashed,-,thick] (2) to (n-1);
\draw[thick,-] (n-1) to (n);

\end{tikzpicture}
\end{center}

La représentation de Krammer $\rho: B_{n} \longrightarrow GL(V)$ est définie sur un $\mathbb{R}(q,t)$-espace vectoriel $V$ de base $\{x_s \ |\ s \in \mathcal{R} \}$ indexée par l'ensemble des réflexions $(i,j)$ du groupe symétrique $S_n$ avec $1 \leq i < j \leq n-1$. Sa dimension est alors égale à $\#\mathcal{R} = \frac{n(n-1)}{2}$. Notons $x_{(i,j)}$ par $x_{i,j}$. La représentation est définie comme suit :

\begin{tabular}{ll}

 $s_{k}x_{k,k+1} = tq^{2}x_{k,k+1}$, &  \\
 $s_{k}x_{i,k} = (1-q)x_{i,k} + qx_{i,k+1}$, & $i<k$,\\
 $s_{k}x_{i,k+1} = x_{i,k} + tq^{(k-i+1)}(q-1)x_{k, k+1}$, & $i<k$,\\
 $s_{k}x_{k,j} = tq(q-1)x_{k,k+1} + qx_{k+1,j}$, & $k+1<j$,\\
 $s_{k}x_{k+1,j} = x_{k,j} + (1-q)x_{k+1,j}$, & $k+1 < j$,\\
 $s_{k}x_{i,j} = x_{i,j},$ & $i < j < k$ ou $k+1 < i < j$ et\\
 $s_{k}x_{i,j} = x_{i,j} + tq^{(k-i)}(q-1)^{2}x_{k,k+1}$, & $i < k < k+1 < j$.\\

\end{tabular}

\bigskip

Le critère de fidélité utilisé par Krammer peut être énoncé pour un groupe de Garside. Ce critère fournit des conditions nécessaires qui assurent qu'une représentation linéaire d'un groupe de Garside est fidèle. Soit $M$ un monoïde de Garside et $G(M)$ son groupe des fractions. Notons par $\Delta$ et $P$ l'élément de Garside et l'ensemble des simples de $M$, respectivement. Définissons $\alpha(x)$ le pgcd de $x$ et $\Delta$ pour $x \in M$. Soit $\rho : G(M) \longrightarrow GL(V)$ une représentation linéaire de dimension finie de $G(M)$ et soit $(C_x)_{x \in P}$ une famille de sous-ensembles de $V$ indexée par l'ensemble des simples $P$. Si les $C_x$ sont non vides et deux à deux disjoints et si $x C_y \subset C_{\alpha(xy)}$ pour tout $x \in M$ et $y \in P$, alors la représentation $\rho$ est fidèle.\\

La représentation de Krammer ainsi que la preuve de fidélité ont été généralisées à tous les groupes d'Artin-Tits de type sphérique par les travaux de \cite{DigneLinearity}, \cite{CohenWalesLinearity} et \cite{ParisArtinMonoidInjectGroup}. Notons aussi qu'une preuve simple de la fidélité a été donnée dans \cite{HeePreuveSimpleFidelite}. Marin a généralisé dans \cite{MarinLinearity} cette représentation à tous les groupes de $2$-réflexions. Sa représentation est définie analytiquement sur le corps des séries formelles de Laurent par la monodromie de certaines formes différentielles et elle est de dimension égale au nombre de réflexions dans le groupe de réflexions complexes correspondant. Il était conjecturé par Marin dans \cite{MarinLinearity} que cette représentation est fidèle. Elle a aussi été généralisée par Chen \cite{ChenFlatConnectionsBrauerAlgebras} à tous les groupes de réflexions.\\

Pour le type ADE des groupes de Coxeter, les représentations de Krammer généralisées peuvent être construites à partir de l'algèbre de BMW (Birman-Murakami-Wenzl). Nous avons détaillé ceci dans l'introduction de la dissertation. Dans ce compte-rendu, nous allons juste rappeler la définition de l'algèbre de BMW pour le cas ADE, consulter \cite{CohenGijsbersWalesBMW}.\\

\begin{definitionf}

Soit $W$ un groupe de Coxeter de type $A_n$, $D_n$ pour tout $n$, ou $E_n$ pour $n = 6, 7, 8$. L'algèbre de BMW associée à $W$ est la $\mathbb{Q}(l,x)$-algèbre unitaire d'ensemble générateur $\{ S_1, S_2, \cdots, S_n \} \cup \{ F_1, F_2, \cdots, F_n \}$ et les relations sont les relations de tresses avec les relations suivantes :

\begin{enumerate}

\item $mF_i = l(S_i^2 + m S_i -1)$ avec $m = \frac{l-l^{-1}}{1-x}$ pour tout $i$,
\item $S_iF_i = F_iS_i = l^{-1}F_i$ pour tout $i$ et
\item $F_iS_jF_i = l F_i$ pour tout $i,j$ quand $s_is_js_i = s_js_is_j$.

\end{enumerate}

\end{definitionf}

Pour la preuve de la proposition suivante, consulter Propositions 2.1 et 2.3 de \cite{CohenGijsbersWalesBMW}.

\begin{propositionf}

On a $S_i$ est inversible d'inverse $S_i^{-1} = S_i + m - mF_i$. Nous avons aussi $F_i^{2} = x F_i$ et $S_jS_iF_j = F_iS_jS_i = F_iF_j$ quand $s_is_js_i = s_js_is_j$.

\end{propositionf}

Si $l = 1$, cette algèbre de BMW dégénère à l'algèbre de Brauer suivante, consulter \cite{Brauer} et \cite{CohenFrenkWales}. Dans ce cas, on a $m = \frac{l-l^{-1}}{1-x} = 0$ et la relation 1 de la définition précédente devient $S_i^2 = 1$. Ainsi, on a $S_i^{-1} = S_i$. Par la proposition précédente, nous avons toujours la relation $F_i^{2} = x F_i$ dans l'algèbre de Brauer.

\begin{definitionf}

Soit $W$ un groupe de Coxeter de type $A_n$, $D_n$ pour tout $n$, ou $E_n$ pour $n = 6, 7, 8$. L'algèbre de Brauer associée à $W$ est la $\mathbb{Q}(x)$-algèbre unitaire d'ensemble générateur $\{ S_1, S_2, \cdots, S_n \} \cup \{ F_1, F_2, \cdots, F_n \}$ et les relations sont les relations de tresses avec les relations suivantes :

\begin{enumerate}

\item $S_i^2 = 1$ pour tout $i$,
\item $F_i^2 = x F_i$ pour tout $i$,
\item $S_iF_i = F_iS_i = F_i$ pour tout $i$, 
\item $F_iS_jF_i = F_i$ pour tout $i,j$ quand $s_is_js_i = s_js_is_j$ et
\item $S_jS_iF_j = F_iS_jS_i = F_iF_j$ quand $s_is_js_i = s_js_is_j$.

\end{enumerate}

\end{definitionf}

Chen a défini dans \cite{ChenFlatConnectionsBrauerAlgebras} une algèbre de Brauer qui généralise les travaux de \cite{Brauer} et \cite{CohenFrenkWales}. Complétant ses résultats dans \cite{ChenOldBMW}, Chen a aussi défini dans \cite{ChenBMW2017} une algèbre de BMW pour les groupes diédraux $I_2(e)$ pour tout $e \geq 2$ à partir de laquelle il a défini une algèbre de BMW pour tous les groupes de Coxeter qui dégénère à l'algèbre de Brauer qu'il avait déjà introduite dans \cite{ChenFlatConnectionsBrauerAlgebras}. Il est aussi prouvé dans \cite{ChenBMW2017} l'existence d'une représentation des groupes d'Artin-Tits associés aux groupes diédraux. La représentation est de dimension $e$ et elle est définie explicitement sur $\mathbb{Q}(\alpha,\beta)$ où $\alpha$ et $\beta$ dépendent des paramètres de l'algèbre de BMW. Il est aussi conjecturé dans \cite{ChenBMW2017} que cette représentation est isomorphe à la représentation monodromique construite par Marin dans \cite{MarinLinearity}.

\section*{Algèbres de BMW et de Brauer pour le type $(e,e,n)$}

Nous nous sommes inspirés par le monoïde de Corran et Picantin pour construire une algèbre de BMW pour le groupe de tresses complexes $B(e,e,n)$. Nous utilisons la définition de l'algèbre de BMW associée aux groupes diédraux \cite{ChenBMW2017} et la définition de BMW des groupes de Coxeter de type ADE \cite{CohenGijsbersWalesBMW}. Nous distinguons les cas où $e$ est impair et $e$ est pair.

\begin{definitionf}

Supposons que $e$ est impair. On définit l'algèbre BMW$(e,e,n)$ associée à $B(e,e,n)$ comme étant la $\mathbb{Q}(l,x)$-algèbre unitaire avec l'ensemble générateur $$\{T_i\ |\ i \in \mathbb{Z}/e\mathbb{Z} \} \cup \{S_3, S_4, \cdots, S_n \} \cup \{E_i\ |\ i \in \mathbb{Z}/e\mathbb{Z}\} \cup \{ F_3, F_4, \cdots, F_n\},$$ et les relations sont les relations de l'algèbre de BMW de type $A_{n-1}$ pour\\ $\{ T_i, S_3, S_4, \cdots, S_n \} \cup \{E_i,F_3, \cdots , F_n\}$ avec $0 \leq i \leq e-1$ avec les relations de type diédral suivantes :

\begin{enumerate}

\item $T_i = T_{i-1}T_{i-2}T_{i-1}^{-1}$ et $E_i = T_{i-1}E_{i-2}T_{i-1}^{-1}$ pour tout $i \in \mathbb{Z}/e\mathbb{Z}$, $i \neq 0,1$,
\item $\underset{e}{\underbrace{T_1T_0 \cdots T_1}}=\underset{e}{\underbrace{T_0T_1 \cdots T_0}}$,
\item $m E_i = l (T_i^2 + mT_i -1)$ pour $i = 0, 1$, où $m = \frac{l-l^{-1}}{1-x}$,
\item $T_iE_i = E_iT_i = l^{-1} E_i$ pour $i = 0, 1$, 
\item $E_1\underset{k}{\underbrace{T_0T_1 \cdots T_0}}E_1 = l E_1$, où $1 \leq k \leq e-2$, $k$ impair,
\item $E_0\underset{k}{\underbrace{T_1T_0 \cdots T_1}}E_0 = l E_0$, où $1 \leq k \leq e-2$, $k$ impair,
\item $\underset{e-1}{\underbrace{T_1T_0 \cdots T_0}}E_1 = E_0 \underset{e-1}{\underbrace{T_1T_0 \cdots T_0}}$,
\item $\underset{e-1}{\underbrace{T_0T_1 \cdots T_1}}E_0 = E_1\underset{e-1}{\underbrace{T_0T_1 \cdots T_1}}$.
\end{enumerate}

\end{definitionf}

\begin{definitionf}

Supposons que $e$ est pair. Soient $m = v-v^{-1}$ et $x$ tels que $m = \frac{l-l^{-1}}{1-x}$. On définit l'algèbre BMW$(e,e,n)$ associée à $B(e,e,n)$ comme étant la $\mathbb{Q}(l,v)$-algèbre unitaire avec l'ensemble générateur $$\{T_i\ |\ i \in \mathbb{Z}/e\mathbb{Z} \} \cup \{S_3, S_4, \cdots, S_n \} \cup \{E_i\ |\ i \in \mathbb{Z}/e\mathbb{Z}\} \cup \{ F_3, F_4, \cdots, F_n\},$$ et les relations sont les relations de l'algèbre de BMW de type $A_{n-1}$ pour\\ $\{ T_i, S_3, S_4, \cdots, S_n \} \cup \{E_i,F_3, \cdots , F_n\}$ avec $0 \leq i \leq e-1$ avec les relations de type diédral suivantes :

\begin{enumerate}

\item $T_i = T_{i-1}T_{i-2}T_{i-1}^{-1}$ et $E_i = T_{i-1}E_{i-2}T_{i-1}^{-1}$ pour tout $i \in \mathbb{Z}/e\mathbb{Z}$, $i \neq 0,1$,
\item $\underset{e}{\underbrace{T_1T_0 \cdots T_0}}=\underset{e}{\underbrace{T_0T_1 \cdots T_1}}$,
\item $m E_i = l (T_i^2 + mT_i -1)$ pour $i = 0, 1$,
\item $T_iE_i = E_iT_i = l^{-1} E_i$ pour $i = 0, 1$,
\item $E_1\underset{i}{\underbrace{T_0T_1 \cdots T_0}}E_1 = (v^{-1} + l) E_1$ pour $i = 4k+ 1 < e/2$ et $i = 4k+3 < e/2$,
\item $E_0\underset{i}{\underbrace{T_1T_0 \cdots T_1}}E_0 = (v^{-1} + l) E_0$ pour $i = 4k+ 1 < e/2$ et $i = 4k+3 < e/2$,
\item $\underset{e-1}{\underbrace{T_1T_0 \cdots T_1}}E_0 = E_0 \underset{e-1}{\underbrace{T_1T_0 \cdots T_1}} = v^{-1} E_0$,
\item $\underset{e-1}{\underbrace{T_0T_1 \cdots T_0}}E_1 = E_1\underset{e-1}{\underbrace{T_0T_1 \cdots T_0}} = v^{-1} E_1$,
\item $E_0AE_1 = E_1AE_0 = 0$, où $A$ est un mot sans carrés sur $\{T_0,T_1\}$ de longueur au plus $e-1$.

\end{enumerate}

\end{definitionf}

Pour des spécialisations $l = 1$ et $v = 1$, l'algèbre BMW$(e,e,n)$ dégénère en ce que nous avons appelé l'algèbre de Brauer associée. Elle est notée par Br$(e,e,n)$. Nous avons prouvé que Br$(e,e,n)$ coïncide avec l'algèbre de Brauer-Chen définie par Chen dans \cite{ChenFlatConnectionsBrauerAlgebras} pour $n = 3$ et $e$ impair. Donc Br$(e,e,3)$ fournit une présentation canonique de l'algèbre de Brauer-Chen pour $e$ impair. À la fin de la section 5.3 de la dissertation, nous posons la question si l'algèbre de Brauer Br$(e,e,n)$ que nous avons définie est isomorphe à l'algèbre de Brauer-Chen pour tout $e$ et $n$.

\section*{Représentations de Krammer du groupe de tresses complexes $B(e,e,n)$ et conjectures}

Rappelons qu'en type ADE des groupes de Coxeter, les représentations de Krammer généralisées se construisent via les algèbres de BMW. Nous tentons la construction de représentations linéaires explicites des groupes de tresses complexes $B(e,e,n)$ en utilisant l'algèbre BMW$(e,e,n)$. Nous appelons ces représentations : les représentations de Krammer de $B(e,e,n)$.\\

Nous adoptons une approche combinatoire et nous avons utilisé le package GBNP (version 1.0.3) de GAP4 (\cite{GBNP}) pour nos calculs. Notre méthode repose sur le calcul d'une base de Gröbner de la liste des polynômes qui définissent BMW$(e,e,n)$. Nous avons trouvé une méthode pour construire des représentations de Krammer explicites pour $B(3,3,3)$ et $B(4,4,3)$. Nous notons ces deux représentations par $\rho_3$ et $\rho_4$, respectivement. Dans la section 5.4 de la dissertation, nous donnons ces deux représentations en expliquant la manière de les construire. Dans Appendix B, nous avons fourni les codes qui ont permis de faire ces constructions. Nous avons utilisé la plate-forme MATRICS (\cite{MATRICS}) de l'Universit\'e de Picardie Jules Verne pour nos calculs heuristiques.\\

Rappelons que Chen a étudié les algèbres de BMW pour les groupes diédraux dans \cite{ChenBMW2017} et il a été capable de construire des représentations irréductibles des groupes d'Artin-Tits associés aux groupes diédraux. Il a aussi conjecturé que ces représentations sont isomorphes aux représentations monodromiques correspondantes. Nos représentations $\rho_3$ et $\rho_4$ sont construites à partir de l'algèbre de BMW. Nous avons prouvé qu'elles sont absolument irréductibles et sont de dimension égale au nombre de réflexions dans le groupe de réflexions complexes correspondant. Nous avons aussi étudié la restriction de $\rho_3$ aux sous-groupes paraboliques maximaux de $B(3,3,3)$. Toutes ces propriétés sont satisfaites pour les représentations de Krammer généralisées. C'est pour ces raisons que nous pensons que les représentations $\rho_3$ et $\rho_4$ sont de bons candidats pour être appelées les représentations de Krammer de $B(3,3,3)$ et $B(4,4,3)$, respectivement. Nous pensons aussi que ces deux représentations sont isomorphes aux représentations monodromiques de $B(3,3,3)$ et $B(4,4,3)$ construites par Marin dans \cite{MarinLinearity}, où il conjecture que la représentation monodromique est fidèle, consulter Conjecture 6.3 de \cite{MarinLinearity}. Ceci nous permet de proposer la conjecture suivante sur les représentations de Krammer explicites $\rho_3$ et $\rho_4$.

\begin{conjecturef}

Les représentations de Krammer $\rho_3$ et $\rho_4$ sont fidèles.\\

\end{conjecturef}

La méthode de construction de $\rho_3$ et $\rho_4$ repose sur le calcul d'une base de Gröbner de la liste des polynômes qui définissent les relations de BMW$(e,e,n)$. Ce n'est plus possible d'appliquer ces calculs très lourds pour $e \geq 5$ quand $n =3$. Cependant, nous avons essayé de calculer la dimension de l'algèbre BMW$(e,e,n)$ sur un corps fini. Nous étions capables de calculer les dimensions de BMW$(5,5,3)$ et BMW$(6,6,3)$ sur un corps fini pour beaucoup de spécialisations des paramètres $m$ et $l$. Les calculs pour BMW$(6,6,3)$ ont pris 40 jours sur la plate-forme MATRICS de l'Universit\'e de Picardie Jules Verne (\cite{MATRICS}). Pour toutes les différentes spécialisations, nous avons trouvé que la dimension de BMW$(5,5,3)$ est égale à $1275$ et celle de BMW$(6,6,3)$ est égale à $1188$. Notons que nous avons obtenu $297$ pour BMW$(3,3,3)$ et $384$ pour BMW$(4,4,3)$. On remarque que pour $e = 3$ et $e = 5$, nous obtenons que la dimension de BMW$(e,e,3)$ est égale à $\#(G(e,e,3)) + e\times (\#\mathcal{R})^2$, où $\mathcal{R}$ est l'ensemble des réflexions de $G(e,e,3)$. De plus, on remarque que pour $e = 4$ et $e = 6$, nous obtenons que la dimension de BMW$(e,e,3)$ est égale à $\#(G(e,e,3)) + \frac{e}{2} \times (\#\mathcal{R})^2$.\\

En se basant sur les précédentes expérimentations, nous proposons la conjecture suivante :

\begin{conjecturef}

Soit $K = \overline{\mathbb{Q}(m,l)}$ et soit $H(e,e,3)$ l'algèbre de Hecke associée à $G(e,e,3)$ sur $K$. Soit $N$ le nombre de réflexions du groupe de réflexions complexes $G(e,e,3)$. 

Sur $K$, l'algèbre \emph{BMW}$(e,e,3)$ est semi-simple, isomorphe à
$$\left\{\begin{array}{ll}
H(e,e,3) \oplus \left( \mathcal{M}_{N}(K) \right)^{e} & si\ e\ est\ impair,\\
H(e,e,3) \oplus \left( \mathcal{M}_{N}(K) \right)^{e/2} & si\ e\ est\ pair,
\end{array}\right.$$
où $\mathcal{M}_{N}(K)$ est l'algèbre des matrices sur $K$ de dimension $N^2$. En particulier, l'algèbre \emph{BMW}$(e,e,3)$ est de dimension
$$\left\{\begin{array}{ll}
\#(G(e,e,3)) + e\times N^2 & si\ e\ est\ impair,\\
\#(G(e,e,3)) + \frac{e}{2} \times N^2 & si\ e\ est\ pair.
\end{array}\right.$$

\end{conjecturef}